\documentclass[12pt]{UOthesis}

\usepackage{fancybox}
\usepackage{shadow}
\usepackage[all]{xy}
\usepackage{url}
\usepackage{rotating}
\usepackage{verbatim}
\usepackage{eufrak}
\usepackage{mathrsfs}
\usepackage{amsmath}
\usepackage{multirow}
\usepackage{enumitem}
\usepackage{mathtools}
\usepackage{multicol}
\usepackage{amssymb}
\usepackage{bbm}
\usepackage{array}

\allowdisplaybreaks

\NoLogo
\NoSignatures

\setUOname{Ga\"el Giordano}         
\setUOcpryear{2023}               
\setUOtitle{On the Use of the Kantorovich-Rubinstein Distance for Dimensionality Reduction}  

\phd                

{
  \newtheoremstyle{remarkstyle}{\topsep}{\topsep}{\rm}{}{\bfseries}{.}{.5em}{}
  \theoremstyle{remarkstyle}
  \newtheorem{rmk}[theo]{Remark}
  \newtheorem{egg}[theo]{Example}
  \newtheorem{pc}[theo]{Particular Case}
  \newtheorem{prpty}[theo]{Property}
  \newtheorem{ntn}[theo]{Notation}
  \newtheorem{hn}[theo]{Historical Note}
  \newtheorem{term}[theo]{Terminology}
  \newtheorem{clm}[theo]{Claim}
  
}
\newcommand{\RR}{\mathbb{R}}
\newcommand{\NNN}{\mathbb{N}}
\newcommand{\PP}{\mathbb{P}}
\newcommand{\ZZZ}{\mathbb{Z}}
\newcommand{\SSS}{\mathbb{S}}

\newcommand{\EE}{\mathbb{E}}
\newcommand{\OO}{\mathcal{O}}
\newcommand{\UU}{\mathcal{U}}
\newcommand{\ELL}{\mathcal{L}} 
\newcommand{\XX}{\mathcal{X}} 
\newcommand{\YY}{\mathcal{Y}} 

\newcommand{\NN}{\mathcal{N}} 
\newcommand{\GG}{\mathcal{G}}
\newcommand{\FF}{\mathcal{F}} 
\newcommand{\BB}{\mathcal{B}}
\newcommand{\JJ}{\mathcal{J}}
\newcommand{\DD}{\mathcal{D}}

\newcommand{\MM}{\mathcal{M}}
\newcommand{\EEE}{\mathcal{E}} 
 \newcommand{\LL}{\mathcal{L}}
\newcommand{\ZZ}{\mathcal{Z}} 
\newcommand{\HH}{\mathcal{H}} 
\newcommand{\PPP}{\mathcal{P}}
\newcommand{\RRR}{\mathcal{R}} 
\newcommand{\dps}{\displaystyle}
\newcommand{\diff}{\mathop{}\!\mathrm{d}}
\newcommand{\myeq}{\mathrel{\overset{\makebox[0pt]{\mbox{\normalfont\tiny\sffamily def}}}{=}}}
\newcommand{\comp}{\mathsf{c}}
\newcommand{\indicator}{\mathbbm{1}}

\DeclareMathOperator{\law}{law}
\DeclareMathOperator{\err}{err} 
\DeclareMathOperator{\Err}{Err}
\DeclareMathOperator{\supprt}{supp}
\DeclareMathOperator{\range}{ran}
\DeclareMathOperator{\diam}{diam}
\DeclareMathOperator*{\argmin}{Arg min}
\DeclareMathOperator*{\argmax}{Arg max}
\DeclareMathOperator{\lip}{Lip}
\DeclareMathOperator{\sgn}{sgn}

\begin{document}

%
%
\notachapter{Abstract}

The goal of this thesis is to study the use of the Kantorovich-Rubinstein distance as to build a descriptor of sample complexity in classification problems. The idea is to use the fact that the Kantorovich-Rubinstein distance is a metric in the space of measures that also takes into account the geometry and topology of the underlying metric space. We associate to each class of points a measure and thus study the geometrical information that we can obtain from the Kantorovich-Rubinstein distance between those measures. We show that a large Kantorovich-Rubinstein distance between those measures allows to conclude that there exists a 1-Lipschitz classifier that classifies well the classes of points. We also discuss the limitation of the Kantorovich-Rubinstein distance as a descriptor.

\cleardoublepage

\notachapter{R\'esum\'e}
Dans cette th\`ese, on \'etudie l'utilisation de la distance de Kantorovich-Rubinstein afin de construire des descripteurs de complexit\'e. Ces descripteurs aident \`a juger de la difficult\'e \`a s\'eparer un \'echantillon lors d'un probl\`eme de classification. On utilise le fait que la distance de Kantorovich-Rubinstein est une distance dans l'espace des mesures qui prend en compte la g\'eom\'etrie et la topologie de l'espace m\'etrique sous-jacent. On associe \`a chacune des deux classes de points une mesure et on \'etudie l'information g\'eom\'etrique que l'on obtient \`a partir de la distance de Kantorovich-Rubinstein entre ces mesures. On montre qu'une grande distance de Kantorovich-Rubinstein permet de conclure qu'il existe un classificateur 1-lipschitzien qui s\'epare avec une grande pr\'ecision les deux classes de points. On termine la th\`ese par une discussion sur les limites de l'utilisation de la distance de Kantorovich-Rubinstein comme descripteur de complexit\'e.

\cleardoublepage



\cleardoublepage

\notachapter{Acknowledgement}   

Merci \`a tous ceux qui m'ont soutenu durant les tr\`es nombreuses ann\'ees n\'ecessaires \`a la r\'edaction de cette th\`ese. Un merci tout sp\'ecial \`a ma famille qui a trouv\'e la patience pour me soutenir du premier au dernier jour.
Je tiens aussi \`a souligner l'aide apport\'ee par mes deux superviseurs, Dr Pestov et Dr Wells.

\cleardoublepage

\tableofcontents
\cleardoublepage


\ListOfTables
\cleardoublepage

%
%
%

%
\nonumchapter{Preface}

If one ever needs an example to illustrate the proverb ``Necessity is the mother of invention'', one should look no further than the field of \textit{optimal transport}. The field's birth can be traced to the publication, by the French geometer Gaspard Monge, of his famous work \textit{M\'emoire sur la th\'eorie des d\'eblais et des remblais} in 1781. Monge considered the following problem: suppose you have crushed stones that you need to extract from quarries and transport to construction sites. The location of the quarries and the construction sites are known. The volume extracted from each quarry and the volume needed at each construction site are also known. The problem is to decide for each quarry how to dispatch the crushed stones in such a way as to minimize the total transport cost. Monge assumed that the transport cost of one unit of mass along a certain distance was given by the product of the mass by the distance. It is to construct that optimal transportation strategy that Monge created the field of \textit{Optimal Transport}. Most of Monge's result were flawed (by current mathematical standards) and the field of optimal transport remained dormant for more than a century. The revival of Optimal Transport came in the 1930's when the Russian mathematician Kantorovich realised that an optimal transport problem was in fact a particular case of an optimal coupling problem. In his quest to solve optimal coupling problems, Kantorovich stated and proved a fundamental duality theorem. He also defined a very useful notion of distance between two measures: the total cost of the optimal coupling when the cost is chosen as a distance function on the underlying product space. This distance is called the Kantorovich-Rubinstein distance. Through out the second half of the twentieth century, statisticians and probabilists used the Kantorovich-Rubinstein distance in the study of many different fields of Mathematics and Physics. In this thesis, we study the use of the Kantorovich-Rubinstein distance to analyse the difficulty of a classification problem. This type of analysis tries to predict under which scenarios a given classifier succeeds or fails without proceeding with the classification per se. We then focus on constructing this Kantorovich-Rubinstein based complexity descriptor for a Genome Wide Association Study (GWAS) dataset. This thesis will be structured as follows: \\
Chapter \ref{chapter_number_Optimal_Transport_and _Duality_Theorems} introduces the Kantorovich Minimisation problem, its associated duality theorem and its particular case, the Kantorovich-Rubinstein theorem. It finishes with a proof of the existence of optimal solutions for the Kantorovich Minimisation problem. The results presented in this chapter are known.\\
Chapter \ref{Chapitre_KR_distance} focuses on the Kantorovich-Rubinstein distance as a distance between measures. Once the Kantorovich-Rubinstein distance is defined, we study the convergence of measures and its topological properties. Then, we construct explicit formulas for the Kantorovich-Rubinstein distance for particular Polish spaces. The vast majority of results stated in this chapter are known. There are nonetheless a few new corollaries and new proofs for the old results. \\
Chapter \ref{chapter_on _KR_distance_and_STT} is new material. It studies the link between commonly used test statistics for genetic association and the Kantorovich-Rubinstein distance.\\
Chapter \ref{chapter_K-Lipschitz_mappings_and_KR_distance} is also new material. It is a technical chapter that studies the functional properties of the Kantorovich-Rubinstein distance in order to compare the Kantorovich-Rubinstein distance between two measures and between their respective push-forward measures.\\
Chapter \ref{Classification_problem_background} introduces the basic notions of classification problems as a particular case of statistical learning problems. We thus formalize the notions of loss functions and real-valued classification functions, overfitting and its link to generalization bound with the Rademacher averages as particular case of representational capacity. We end this overview of learning theory for classification problems with an introduction of margin theory.\\
Chapter \ref{KR_distance_upper_bound_of_classifiers} is the central chapter of the thesis and only contains new material. It studies thoroughly the association between the Kantorovich-Rubinstein distance and the risk functional for particular choices of loss functions.\\
Chapter \ref{KR_score_appplied_to_gwas} introduces the notion of sample complexity and explains why the Kantorovich-Rubinstein distance can be used as a descriptor of a sample complexity. Then, using the fact that the Kantorovich-Rubinstein distance can be used as a descriptor of a sample complexity, we show that the Kantorovich-Rubinstein distance is a good choice of evaluation criterion function in feature selection algorithms that could be of great interest to reduce the dimensionality of GWAS datasets such as the Ottawa Heart Genomics Study dataset (OHGS).\\
Finally, chapter \ref{chapitre_conclusion}, the conclusion chapter, addresses limitations of the Kantorovich-Rubinstein distance as a sample complexity and gives precise modifications to (we hope) improve the descriptive capability of the Kantorovich-Rubinstein distance.\\
Lastly, there are four appendices. Appendix \ref{Elementary_results_appendix} features some known definitions and results used in the thesis but not directly related to optimal transport or classification problems. Appendix \ref{Feature_selection_algo_appndix} gives a brief overview of feature selection algorithms and groups them in different categories. Appendix \ref{Biological_and_Technical_Background_appendix} gives the biological, genetical and technical  information necessary to understand how GWAS datasets, and more particularly the OHGS dataset, are constructed. Appendix \ref{KR_distance_geometric_construction} gives a description of the Kantorovich-Rubinstein distance from a geometrical perspective. It is given in the Appendix because it opens up (we hope) to a generalisation of the results in this thesis.

\cleardoublepage

\pagenumbering{arabic}





\chapter{Kantorovich minimisation problem and  Duality Theorems}\label{chapter_number_Optimal_Transport_and _Duality_Theorems}

Chapter \ref{chapter_number_Optimal_Transport_and _Duality_Theorems} starts with the description and the characterisation of the optimisation problem called the \textit{Kantorovich Minimisation problem}. Then, we establish a theorem of major importance to the field, the Kantorovich Duality Theorem and study its particular case of interest, the Kantorovich-Rubinstein Theorem. The chapter ends with the a proof of the existence of optimal solutions for the optimisation problem under certain regularity assumptions.

The important results presented in this chapter are found in two books by Villani: \textit{Topics in Optimal Transportation} \cite{VillaniTopicsOptimalTransportation} and  \textit{Optimal transport, old and new} \cite{VillaniOptimalTrans}.

Throughout this chapter (and more generally, throughout this thesis), we will often omit to specify the $\sigma$-algebras associated to a space if no specific $\sigma$-algebra is required for the result to stand or if the $\sigma$-algebra used is clear from the context. Note that, generally, in the case of a topological space, the $\sigma$-algebra is the Borel $\sigma$-algebra.

\section{Couplings}

Couplings are very well-known in probability theory. Since they are central to the definition of the Kantorovich minimisation problem, we recall both the measure theoretical and the probabilistic definitions. We first start with the definition of a \textit{marginal} as it appears in the definition of a \textit{coupling}.

\begin{prop}\label{equivalence_marginals}
Let $(\XX,\mu)$ and $(\YY,\nu)$ be two probability spaces and let $\vartheta$ be probability measure on the product space $\XX\times\YY$. Then, the following three statements are equivalent:
\begin{itemize}
\item[(i)] $\mu(A) = \vartheta(A\times \YY) \mbox{ and } \nu(B) = \vartheta(\XX \times B) \mbox{ for all measurable subsets } A,B \subset \XX.$

\item[(ii)] $\pi_{\XX}(\vartheta)=\mu$ and $\pi_{\YY}(\vartheta)=\nu$ where $\pi_{\XX}$ and $\pi_{\YY}$ denote the natural projections of $\XX \times \YY$ onto $\XX$ and $\YY$ respectively.

\item[(iii)] For all integrable measurable functions $\varphi$, respectively $\psi$, on $\XX$, respectively $\YY$,
\begin{equation*}
\int_{\XX\times\YY}(\varphi(x) + \psi(y)) \diff\vartheta(x,y) = \int_{\XX} \varphi(x) \diff\mu(x) + \int_{\YY} \psi(y) \diff\nu(y).
\end{equation*}
Note that formally, one should write 
\begin{equation*}
\int_{\XX\times\YY}\big((\varphi \circ \pi_{\XX})(x,y) + (\psi \circ \pi_{\YY})(x,y) \big) \diff\vartheta(x,y) \quad \mbox{instead of} \quad \int_{\XX\times\YY}(\varphi(x) + \psi(y)) \diff\vartheta(x,y).
\end{equation*}
\end{itemize}
\end{prop}

\begin{defn}[Marginals]
Let $(\XX,\mu)$ and $(\YY,\nu)$ be two probability spaces and let $\vartheta$ be a probability measure on the product space $\XX\times\YY$. Then, the probability measure $\vartheta$ admits the measure $\mu$ and $\nu$ as marginals on $\XX$ and $\YY$ respectively, if $\mu$ and $\nu$ satisfy any of the conditions in Proposition \ref{equivalence_marginals}.
\end{defn}

\begin{proof}[Property \ref{equivalence_marginals}]
The proposition will be established in the following way: $(i) \Leftrightarrow (ii)$, $(ii) \Rightarrow (iii)$ and $(iii) \Rightarrow (i)$.\\

\noindent\underline{$(i) \Leftrightarrow(ii)$}: For a measurable subset $A\subset \XX$ we have: \\
$\pi_{\XX}(\vartheta)(A) = \vartheta(\pi^{-1}_{\XX}(A)) = \vartheta(A \times \YY) = \mu(A)$. \\
Likewise, for a measurable subset $B \subset \YY$ we obtain $\pi_{\YY}(\vartheta)(B) = \nu(B)$. \\
Conversely, we have $\mu(A) = \pi_{\XX}(\vartheta)(A) = \vartheta(\pi_{\XX}^{-1}(A)) = \vartheta(A \times \YY)$.\\
 Likewise, $\nu(B) = \vartheta)(\XX \times B)$.\\

\noindent\underline{$(ii) \Rightarrow (iii)$}: 
Let $\varphi$ (respectively $\psi$) be measurable functions on $\XX$ (respectively $\YY$). We have
\begin{equation*}
\int_{\XX\times\YY}\big(\varphi \circ \pi_{\XX})(x,y) + (\psi \circ \pi_{\YY})(x,y) \big) \diff\vartheta(x,y) = \int_{\XX\times\YY} (\varphi \circ \pi_{\XX})(x,y) + \int_{\XX\times\YY} (\psi \circ \pi_{\YY})(x,y) \diff\vartheta(x,y).
\end{equation*}
The change of variable formula yields:
\begin{equation*}
\int_{\XX\times\YY} (\varphi \circ \pi_{\XX})(x,y) + \int_{\XX\times\YY} (\psi \circ \pi_{\YY})(x,y) \diff\vartheta(x,y) = \int_{\XX} \varphi(x) \diff \pi_{\XX}(\vartheta)(x) + \int_{\YY} \psi(x) \diff \pi_{\YY}(\vartheta)(y).
\end{equation*}
Now, by $(ii)$, we have 
\begin{equation*}
\int_{\XX} \varphi(x) \diff \pi_{\XX}(\vartheta)(x) + \int_{\YY} \psi(x) \diff \pi_{\YY}(\vartheta)(y)= \int_{\XX} \varphi(x) \diff \mu(x) + \int_{\YY} \psi(x) \diff \nu(y).
\end{equation*}

\noindent\underline{$(iii) \Rightarrow (i)$}:  For any $A \subset \XX$,
\begin{align*}
\mu(A) &= \int_{\XX} 1_{A}(x) \diff \mu(x) \\
&= \int_{\XX\times\YY} \pi_{\XX}(1_{A})(x,y) \diff \vartheta(x,y) \qquad \mbox{by (iii)} \\
&= \int_{A\times\YY}(x,y) \diff\vartheta(x,y) \\
&= \vartheta(A\times\YY).
\end{align*}
Hence we have $\mu(A) = \vartheta(A\times\YY)$.\\
Likewise, for any $B\subset\XX$, we have $\nu(B) = \vartheta(\XX\times B)$.
\end{proof}

\begin{defn}[Couplings]\label{Coupling}
Let $(\XX,\mu)$ and $(\YY,\nu)$ be two probability spaces. \\
A coupling of $\mu$ and $\nu$ is a probability measure $\vartheta$ on the product space $\XX\times\YY$ such that $\vartheta$ admits $\mu$ and $\nu$ as marginals on $\XX$ and $\YY$, respectively. 
\end{defn}

\begin{rmk}
We denote by $\NN(\mu,\nu)$ the set of all couplings of $\mu$ and $\nu$. \\
Note that $\NN(\mu,\nu) \neq \emptyset$, because the product measure $\mu \times \nu$ belongs to $\NN(\mu,\nu)$.
\end{rmk}

\begin{rmk}\label{generalization_of_couplings}
 The Definition \ref{Coupling} of a coupling can be generalized to the case of any two bounded measures $\mu$ and $\nu$ such that $\mu(\XX) = \nu(\YY)$. Indeed, suppose that $\mu(\XX) = \nu(\YY) = k$ with $k < \infty$. Since $\mu(\XX) = \vartheta( \XX \times \YY)$ then $\vartheta(\XX \times \YY) = k$.\\
 Since a probability measure is a normalized bounded measure, and since all publications in the field of Optimal Transport Theory consider probability spaces, all results in this thesis will be written for probability spaces $\XX$ and $\YY$. Please keep in mind that these results remain true for measurable spaces equipped with bounded measures $\mu$ and $\nu$ such that $\mu(\XX) = \nu(\YY)$.
\end{rmk}

There also exists a classical probabilistic definition of a coupling. This definition relies on probability theory terminology. We thus first recall some basic definitions:

\begin{defn}[Random Variable]
Let $\XX$ be a measurable space and let $(\Omega_{\XX}, \PP_{\XX})$ be a probability space. \\
Then, a measurable map $X: \Omega_{\XX} \rightarrow \XX$ is a called a random variable. 
\end{defn}

\begin{defn}[Push-forward measures]
Let $\XX$ be a measurable space and let $(\Omega_{\XX}, \PP_{\XX})$ be a probability space. Let $X: \Omega_{\XX} \rightarrow \XX$  be a random variable. The push-forward measure $X(\PP_{\XX})$ of $\PP_{\XX}$ by $X$ is the measure on $\XX$ defined by 
\begin{equation*}
X(\PP_{\XX})(A) = \PP_{\XX} \big( \{ \omega \in \Omega_{\XX} ; \, X(\omega) \in A\}\big),
\end{equation*}
for all a measurable set $A$ in $\XX$. \\
The push-forward measure $X(\PP_{\XX})$ is also called the law of $X$. One writes $\law(X)$. 
\end{defn}

\begin{rmk} Let $\varphi$ be a map from a measured space $(\XX, \mu)$ to a space $\YY$. Several notations are used to write the push-forward of $\mu$ by $\varphi$, for example $\varphi \sharp \mu$ or $\varphi * (\mu)$. In this thesis, we will use the notation $\varphi(\mu)$.
\end{rmk}


\noindent We can now give the probabilistic definition of a \textit{coupling}:

\begin{defn}[Couplings in probabilistic formulation]
Let $\XX$ (respectively $\YY$) be a measurable space and $(\Omega_{\XX}, \PP_{\XX})$ (respectively $(\Omega_{\YY}, \PP_{\YY})$) be a probability space. \\
Let $X: \Omega_{\XX} \rightarrow \XX$ and $Y: \Omega_{\YY} \rightarrow \YY$ be two random variables.\\
A couple $(\tilde{X},\tilde{Y})$ of random variables on the probability space $\big( \Omega_{\XX} \times \Omega_{\YY}, \tilde{\PP} \big)$ is a coupling of $X$ and $Y$ if $\law(\tilde{X}) = \law(X)$ on $\XX$ and $\law(\tilde{Y}) = \law(Y)$ on $\YY$.

Hence, the push-forward measure $(\tilde{X}, \tilde{Y})(\tilde{\PP})$ on $\XX \times \YY$ is a coupling of $X(\PP_{\XX})$ and $Y(\PP_{\YY})$ in the sense of definition \ref{Coupling}.
\end{defn}

\noindent In the literature, many authors are content with a reduced terminology and refer to the pair of random variables $(\tilde{X},\tilde{Y})$ as a coupling of the probability measures $X(\PP_{\XX})$ and $Y(\PP_{\YY})$. Hence, we can often find the following definition:

\begin{defn}[Couplings in probabilistic formulation, version 2]\label{Coupling_probabilistic_Villani}
Let $(\XX, \mu)$ and $(\YY, \nu)$ be two probability spaces and let $X: \Omega_{\XX} \rightarrow \XX$ and $Y: \Omega_{\YY} \rightarrow \YY$ be two random variables. \\
Coupling $\mu$ and $\nu$ means constructing two random variables $X$ and $Y$ on some probability space $\big( \Omega, \PP \big)$ such that $\law(X) = \mu$, $\law(Y) = \nu$. The couple $(X,Y)$ is called a coupling of $(\mu,\nu)$.
\end{defn}

\begin{egg}[Coupling for atomic measures]\label{Coupling for atomic measures}
Let $\XX$ and $\YY$ be two spaces and define two atomic probability measures
\begin{equation*}
\mu = \sum_{i = 1}^{n} \mu_i \delta_{x_i} \quad \mbox{and} \quad \nu = \sum_{i = 1}^{m} \nu_i \delta_{y_i},
\end{equation*}
supported on $\{x_1, \ldots , x_n\} \in \XX$ and $\{y_1, \ldots , y_m\} \in \YY$, respectively. Then,
\begin{equation*}
\NN(\mu, \nu) = \Big\{ \sum_{i=1}^n \sum_{j=1}^m \vartheta_{ij} \delta_{(x_i,y_j)} : \, \forall i, \sum_{j=1}^m \vartheta_{ij} = \mu_i \quad \mbox{and} \quad \forall j, \sum_{i=1}^n \vartheta_{ij} = \nu_j \quad \mbox{with} \quad \vartheta_{ij} \geq 0 \Big\}.
\end{equation*}
\end{egg}

\section{Kantorovich Minimisation Problem}

We now have the required mathematical definitions to define the \textit{Kantorovich minimisation problem}.
\begin{defn}[Kantorovich minimisation problem]
Let $(\XX,\mu)$ and $(\YY,\nu)$ be two probability spaces and let $\NN(\mu,\nu)$ be the set of all couplings of $\mu$ and $\nu$. Let $c: \XX \times \YY \to \RR \cup \{\infty\}$ be a nonnegative measurable function, called the cost function.\\
For a coupling $\vartheta \in \NN(\mu,\nu)$, consider the functional $\mathrm{I}$ defined by
\begin{equation*}
\mathrm{I} (\vartheta)= \int_{\XX\times\YY} c(x,y) \diff \vartheta(x,y).
\end{equation*}
Then the \textit{Kantorovich minimisation problem} is to find:
\begin{equation*}
\inf \{ \mathrm{I}(\vartheta); \, \vartheta \in \NN(\mu,\nu) \}.
\end{equation*}
\end{defn}

In probability theory, one would write:\\
Find $\dps\inf \mathbb{E}\,c(X,Y)$, where the pair of random variables $(X,Y)$ runs over all possible couplings of $(\mu, \nu)$.

Of course, the solution of the Kantorovich minimisation problem depends on the cost function $c$. The cost function and the probability spaces here can be very general. In his books \cite{VillaniTopicsOptimalTransportation}, \cite{VillaniOptimalTrans}, Villani obtains nontrivial results as soon as $c$ is lower semi-continuous and $\XX,\YY$ are Polish spaces. 

\begin{defn}[Polish Space]
A Polish space $\XX$ is a topological space which is separable and completely metrizable.
\end{defn}
Recall that a completely metrizable space is a topological space $(\mathcal{X}, T)$ for which there exists at least one metric $d$ on $\mathcal{X}$ such that $(\mathcal{X}, d)$ is a complete metric space and $d$ induces the topology $T$.\\
In this thesis, we consider Polish spaces with their Borel $\sigma$-algebra.

\begin{defn}[Lower Semi-continuous Function]
Let $\XX$ be a topological space. A function $f:\XX \rightarrow [-\infty, \infty]$ is lower semi-continuous if, for any $\alpha \in \RR$, $f^{-1}((\alpha, +\infty]) = \{ x \in \XX; f(x) > \alpha \}$ is open in $\XX$.\\
\end{defn}

\begin{rmk}
If $\XX$ is a Polish space, a function $f: \XX \rightarrow [-\infty,\infty]$ is lower semi-continuous if, for all $x_{\circ} \in \XX$,
\begin{equation*}
f(x_{\circ}) \leq \liminf_{x\rightarrow x_{\circ}} f(x).
\end{equation*}
Moreover, each lower semi-continuous function $f: \XX \rightarrow \RR_+$ is the (pointwise) supremum of an increasing sequence of uniformly continuous nonnegative functions (see Villani p.26 \cite{VillaniTopicsOptimalTransportation})
\end{rmk}

\begin{hn} The Kantorovich minimisation problem is named after the Russian mathematician Leonid Vitaliyevich Kantorovich. Born in 1912, Kantorovich was a very gifted mathematician who made his reputation as a first-class researcher at the age of 18, and earned a position of professor at just 22 at the University of Leningrad. He worked in many areas of mathematics, with a strong taste for applications in economics, and theoretical computer science. In 1938 a laboratory consulted him for the solution of a production optimization problem, which he found out was representative of a whole class of linear problems arising in various areas of economics. 
He received the Nobel Prize in Economic Sciences in 1975. It was shared with Tjalling Koopmans, and it was given "for their contributions to the theory of optimum allocation of resources."
\end{hn}

\subsection{Kantorovich Duality Theorem}

In 1942, Kantorovich stated and proved, by means of functional analytical tools, a duality theorem that is both well-known and widely used. It is stated below. Note that, for a probability space $(\XX,\mu)$, the notation $L_1(\mu)$ represents all the integrable functions on $\XX$ with respect to $\mu$.

\begin{theo}[Kantorovich Duality Theorem - Part 1]\label{M-K_duality_thm}
Let $\mu$ and  $\nu$ be two probability measures on the Polish spaces $\XX$ and $\YY$ respectively. Let $c \,:\, \XX \times \YY \rightarrow \RR \cup \{+\infty\}$ be a nonnegative lower semi-continuous cost function.\\
Let $\Phi_c$ be the set of all pairs of measurable functions $(\varphi,\psi)\in L_1(\mu)\times L_1(\nu)$ satisfying
\begin{equation}\label{def_Phi_c}
\varphi(x) + \psi(y) \leq c(x,y),
\end{equation}
for $\mu$-almost all $x\in\XX$, $\nu$-almost all $y\in\YY$. \\

\noindent Then, for $(\varphi,\psi)\in\Phi_c$, we have
\begin{equation}\label{MK_duality}
\inf_{\vartheta\in\NN(\mu,\nu)} I[\vartheta] = \sup_{\Phi_c} J(\varphi,\psi),
\end{equation}
where 
\begin{equation*}
J(\varphi,\psi) = \int_\XX \varphi(x) \diff\mu(x) + \int_\YY \psi(y) \diff\nu(y) \quad \mbox{and} \quad I[\vartheta] = \int_{\XX\times \YY} c(x,y) \diff\vartheta(x,y). 
\end{equation*}

The right hand side of equation (\ref{MK_duality}) is known as the dual formulation of the Monge-Kantorovich minimisation problem,  which is itself often called the primal problem. 
\end{theo}

In the case of atomic measures with finite support, the Monge-Kantorovich theorem is equivalent to the well known linear programming duality:

\begin{cor}[Kantorovich Duality Theorem for atomic measures]\label{M-K_duality_thm_atomic_measures}
Let $\XX$ and $\YY$ be two Polish spaces. Let us define two atomic probability measures
\begin{equation*}
\mu = \sum_{i = 1}^{n} \mu_i \delta_{x_i} \quad \mbox{and} \quad \nu = \sum_{i = 1}^{m} \nu_i \delta_{y_i},
\end{equation*}
supported on $\{x_1, \ldots , x_n\} \in \XX$ and $\{y_1, \ldots , y_n\} \in \YY$, respectively.

\noindent Recall that any coupling $\vartheta$ of $\mu$ and $\nu$ is an atomic measure on $\XX \times \YY$ written as $\sum \sum \vartheta_{ij} \delta_{(x_i,y_j)}$, for $i = 1, \ldots, n$ and $j = 1,\ldots, m$.

Let $c \,:\, \XX \times \YY \rightarrow \RR \cup \{+\infty\}$ be a nonnegative lower semi-continuous cost function. Then, the functional 
\begin{equation*}
\min \Big\{ \sum_{i=1}^n \sum_{j=1}^m c(x_i,y_j) \vartheta_{ij} : \, \forall i, \sum_{j=1}^m \vartheta_{ij} = \mu_i \quad \mbox{and} \quad \forall j, \sum_{i=1}^n \vartheta_{ij} = \nu_j \quad \mbox{with} \quad \vartheta_{ij} \geq 0 \Big\}
\end{equation*}
admits the dual representation
\begin{equation*}
\max \Big\{ \sum_{i=1}^n \varphi_i \mu_i + \sum_{j=1}^m \psi_j \nu_j ; \, \varphi_i + \psi_j \leq c(x_i,y_j), \, \forall i,j\Big\}.
\end{equation*}
\end{cor}

\begin{rmk}[Regarding theorem \ref{M-K_duality_thm}]\label{coupling_almost_everywhere} 
Notice that inequality (\ref{def_Phi_c}) in Theorem \ref{M-K_duality_thm} holds $\vartheta$-a.e. For any $x_{\circ} \in \XX$ and $y_{\circ} \in \YY$, let $S_{x_{\circ}}$ and $S_{y_{\circ}}$ be measurable sets of measure zero on $\YY$ and $\XX$, respectively, on which inequality (\ref{def_Phi_c}) is not satisfied. We have
\begin{equation*}
S^{\comp}_{x_{\circ}} = \{ y \in \YY: \varphi(x_{\circ}) + \psi(y) \leq c(x_{\circ},y)\} \quad \mbox{and} \quad  S^{\comp}_{y_{\circ}} = \{ x \in \XX: \varphi(x) + \psi(y_{\circ}) \leq c(x,y_{\circ})\}.\end{equation*}
Then, equation (\ref{def_Phi_c}) holds for all $(x,y) \in S_{y_{\circ}}^{\comp} \times S_{x_{\circ}}^{\comp}$. Thus, for equation (\ref{def_Phi_c}) to hold $\vartheta$-a.e., we need to show that $\vartheta \big((S_{y_{\circ}}^{\comp} \times S_{x_{\circ}}^{\comp})^{\comp}\big) = 0$. Now, $(S_{y_{\circ}}^{\comp} \times S_{x_{\circ}}^{\comp})^{\comp} = (S_{y_{\circ}} \times \YY) \cup (\XX \times S_{x_{\circ}})$ and $\vartheta$ has marginals $\mu$ and $\nu$. Hence $\vartheta(S_{y_{\circ}} \times \YY) = 0$ and $\vartheta(\XX \times S_{x_{\circ}}) = 0$. Therefore, $\vartheta \big((S_{y_{\circ}}^{\comp} \times S_{x_{\circ}}^{\comp})^{\comp}\big) = 0$.
\end{rmk}

One can show that the value of the supremum of $J$ on $\Phi_c$ is the same as the value of the supremum of $J$ if one restricts $\Phi_c$ to the functions $(\varphi,\psi)$ that are bounded and continuous. It is not obvious that pairs of $L_1$ functions satisfying equation (\ref{def_Phi_c}) can be approximated by pairs of continuous function also satisfying (\ref{def_Phi_c}). More details are in Proposition \ref{Prop_for_M-K_duality_with_C_b_functions}:

\begin{prop}\label{Prop_for_M-K_duality_with_C_b_functions}
Let $\mu$ and  $\nu$ be probability measures on the Polish spaces $\XX$ and $\YY$ respectively. Let $c \,:\, \XX \times \YY \rightarrow \RR \cup \{+\infty\}$ be a nonnegative lower semi-continuous function.\\
Let $\Phi_c$ be defined as in \ref{M-K_duality_thm} and $\Phi^{'}_c$ be defined as $\Phi_c$ but restricted to the functions $(\varphi^{'},\psi^{'})$ which are bounded and continuous. Then,
\begin{equation}\label{double_inequality_Phic_Phic'_inf_I}
\sup_{\Phi^{'}_c} J(\varphi^{'},\psi^{'}) \leq \sup_{\Phi_c} J(\varphi,\psi) \leq \inf_{\vartheta\in\NN(\mu,\nu)}I[\vartheta].
\end{equation}
where $J(\varphi,\psi) = \int_\XX \varphi(x) \diff\mu(x) + \int_\YY \psi(y) \diff\nu(y)$ and $I[\vartheta] = \int_{\XX\times \YY} c(x,y) \diff\vartheta$.
\end{prop}

\begin{proof}[Propostion \ref{Prop_for_M-K_duality_with_C_b_functions}]
The inequality on the left of \ref{double_inequality_Phic_Phic'_inf_I} is trivial since $C_b(\XX) \times C_b(\YY) \subset L_1(\mu) \times L_1(\nu)$. Hence, we only have to consider the inequality on the right. Let $(\varphi,\psi) \in \Phi_c$ and let $\vartheta \in \NN(\mu,\nu)$. By Proposition \ref{equivalence_marginals}, we have 
\begin{equation*}
J(\varphi,\psi) = \int_{\XX\times\YY}(\varphi(x) + \psi(y)) \diff\vartheta(x,y).
\end{equation*}
Moreover, since $\varphi(x)+\psi(y) \leq c(x,y)$, for $\vartheta$-almost all $(x,y) \in \XX \times \YY$, we obtain
\begin{equation*}\label{Sup_C_b_functions_leq_Inf_I}
\int_{\XX\times\YY}(\varphi(x) + \psi(y)) \diff\vartheta(x,y) \leq \int_{\XX\times \YY} c(x,y) \diff\vartheta(x,y).
\end{equation*}
As the supremum is the lowest upper bound, we have $\sup_{\Phi_c} J(\varphi,\psi) \leq \inf_{\vartheta\in\NN(\mu,\nu)}I[\vartheta]$, which completes the proof.
\end{proof}

\noindent Now, it follows from Proposition \ref{Prop_for_M-K_duality_with_C_b_functions} that the duality
\begin{equation*} 
\sup_{\Phi^{'}_c} J(\varphi^{'},\psi^{'}) = \inf_{\vartheta\in\NN(\mu,\nu)}I[\vartheta] \quad \mbox{implies that} \quad \sup_{\Phi^{'}_c} J(\varphi^{'},\psi^{'}) = \sup_{\Phi_c} J(\varphi,\psi).
\end{equation*}

\noindent We can now give the general idea of the proof of the Monge-Kantorovich Theorem:\\

\begin{proof}[Theorem \ref{M-K_duality_thm}]
To prove that the infimum is at least as large as the supremum is easy. It has been proved in Proposition \ref{Prop_for_M-K_duality_with_C_b_functions}. 

The proof for the reverse inequality is much more complicated. It is separated in 3 steps by increasing order of generality. The first step assumes that $\XX$ and $\YY$ are compact and $c$ is continuous. It uses a minimax argument which is an argument of the form 
\begin{equation*}
\inf_{x\in \XX} \sup_{y\in \YY} \varphi(x,y) = \sup_{a\in \YY} \inf_{x\in \XX} \varphi(x,y).
\end{equation*}
The minimax argument allows to show that
\begin{equation}\label{Sup_C_b_functions_geq_Inf_I}
\inf_{\vartheta\in\NN(\mu,\nu)}I[\vartheta] \leq \sup_{\Phi^{'}_c} J(\varphi^{'},\psi^{'}).
\end{equation}
Putting together inequalities (\ref{double_inequality_Phic_Phic'_inf_I}) and (\ref{Sup_C_b_functions_geq_Inf_I}), one obtains
\begin{equation*}
\sup_{\Phi^{'}_c} J(\varphi^{'},\psi^{'}) = \inf_{\vartheta\in\NN(\mu,\nu)}I[\vartheta]  \quad  \mbox{and thus}  \quad  \sup_{\Phi_c} J(\varphi,\psi) = \inf_{\vartheta\in\NN(\mu,\nu)}I[\vartheta],
\end{equation*}
which proves the duality (\ref{MK_duality}) in this special case.

The last two steps consist in showing that the duality (\ref{MK_duality}) holds with relaxed assumptions, using approximation arguments. The second step relaxes the assumption of compactness while making the assumption that $c$ is bounded and uniformly continuous. Finally, in the third step, the condition of continuity on $c$ is relaxed to obtain the general result. To do so, one has to write $c = \sup c_n$, where $c_n$ is a nondecreasing sequence of bounded, nonnegative, uniformly continuous cost functions.

A detailed version of the proof is available in \textit{Topics in Optimal Transportation} \cite{VillaniOptimalTrans} and \textit{Optimal Transport: Old and New} \cite{VillaniTopicsOptimalTransportation}, both by Villani.
\end{proof}

\subsection{Existence of Optimal Solutions} \label{subsection_Existence_of_Optimal_Solutions}

As its name indicates, the Kantorovich Minimisation Problem requires to find an infimum. A natural question is therefore to ask whether this infimum is realised. That is to say, do there exist, over all couplings, minimisers of the Kantorovich Minimisation Problem?

In light of the Kantorovich Duality Theorem \ref{M-K_duality_thm}, it is also natural to ask whether the dual formulation achieves its supremum. That is, do there exist, over all pairs of functions in $\Phi_c$, maximisers for $J$?

In fact, the answer to both questions is yes. It is stated in Theorem \ref{Existence_optimal_sol_KD} (Kantorovich Duality Theorem Part 2. - Existence of optimal solutions). In order to write Theorem \ref{Existence_optimal_sol_KD} formally, we need to introduce the notions of $c$-concavity, $c$-transform, and conjugate $c$-transform functions.

\begin{defn}[c-transform] \label{c-transform def}
Let $\XX$ and $\YY$ be two nonempty sets and let $c(x,y)$ be defined on $\XX \times \YY$ with values in $\RR \cup\{+\infty\}$.

For any function $\varphi : \XX \rightarrow \RR \cup \{-\infty\}$, $\varphi \not \equiv -\infty$, one can define its $c$-transform $\varphi^c: \YY \rightarrow \RR \cup\{-\infty\}$ by 
\begin{equation*}
\varphi^c(y) = \inf_{x\in\XX}[c(x,y)-\varphi(x)].
\end{equation*}

The functions $\varphi$ and $\varphi^c$ are said to be $c$-conjugate.
\end{defn}

\begin{prop}\label{varphi^ccc=varphi^c}
Let $\XX$ and $\YY$ be two nonempty sets and let $c(x,y)$ be defined on $\XX \times \YY$ with values in $\RR \cup\{+\infty\}$. For any function $\varphi : \XX \rightarrow \RR \cup \{-\infty\}$, $\varphi \not \equiv -\infty$, one has the identity $\varphi^{ccc} = \varphi^c$ where $\varphi^{cc} = (\varphi^c)^c$.
\end{prop}

Note that both $\varphi^{ccc}$ and $\varphi^c$ are functions defined on $\XX$ while $\varphi$ and $\varphi^{cc}$ are defined on $\YY$.

\begin{proof}[Proposition \ref{varphi^ccc=varphi^c}]
Consider a function $\varphi : \XX \rightarrow \RR \cup \{-\infty\}$ and its $c$-transform $\varphi^c$. Then, we have
\begin{align*}
\varphi^{cc} & = \inf_{y \in \YY} [c(x,y) - \varphi^c(y)] \\
 &= \inf_{y \in \YY} \Big[ c(x,y) - \inf_{\tilde{x} \in \XX} [ c(\tilde{x},y) - \varphi(\tilde{x}) ] \Big] \\
 & = \inf_{y \in \YY} \Big[ c(x,y) + \sup_{\tilde{x} \in \XX} [ \varphi(\tilde{x}) - c(\tilde{x},y) ] \Big] \\
 & =  \inf_{y \in \YY}  \sup_{\tilde{x} \in \XX} [ \varphi(\tilde{x}) + c(x,y) - c(\tilde{x},y) ].
\end{align*}
and therefore,
\begin{align*}
\varphi^{ccc}(y) & = \inf_{x \in \XX} [c(x,y) - \varphi^{cc}(x)] \\
&=  \inf_{x \in \XX} \Big[ c(x,y) - \inf_{\tilde{y} \in \YY}  \sup_{\tilde{x} \in \XX} [ \varphi(\tilde{x}) + c(x,\tilde{y}) - c(\tilde{x},\tilde{y}) ] \Big] \\
&=  \inf_{x \in \XX} \Big[ c(x,y) + \sup_{\tilde{y} \in \YY} [- \sup_{\tilde{x} \in \XX} [ \varphi(\tilde{x}) + c(x,\tilde{y}) - c(\tilde{x},\tilde{y}) ] ] \Big] \\
&=  \inf_{x \in \XX} \Big[ c(x,y) + \sup_{\tilde{y} \in \YY}  \inf_{\tilde{x} \in \XX} [ c(\tilde{x},\tilde{y}) - c(x,\tilde{y}) - \varphi(\tilde{x}) ] \Big] \\
&=  \inf_{x \in \XX} \sup_{\tilde{y} \in \YY}  \inf_{\tilde{x} \in \XX} [ c(x,y) + c(\tilde{x},\tilde{y}) - c(x,\tilde{y}) - \varphi(\tilde{x}) ].
\end{align*}
If we set $\tilde{x} = x$ we obtain 
\begin{equation*}
\varphi^{ccc}(y) \leq \inf_{x\in\XX} \sup_{\tilde{y}\in\YY} [c(x,y) - \varphi(x)]. \mbox{ That is } \varphi^{ccc}(y) \leq \inf_{x\in\XX} [c(x,y) - \varphi(x)] = \varphi^c(y).
\end{equation*}
On the other hand, by setting $\tilde{y} = y$ we obtain 
\begin{equation*}
\varphi^{ccc}(y) \geq \inf_{x\in\XX} \inf_{\tilde{x}\in\XX} [c(\tilde{x},y) - \varphi(\tilde{x})]. \mbox{ That is }  \varphi^{ccc}(y) \geq \inf_{\tilde{x}\in\XX} [c(\tilde{x},y) - \varphi(\tilde{x})] = \varphi^c(y).
\end{equation*}
\end{proof}

\begin{defn}\label{c_concavity_definition}[c-concavity]
Let $\XX$ and $\YY$ be two nonempty sets and let $c(x,y)$ be defined on $\XX \times \YY$ with values in $\RR \cup\{+\infty\}$.

A function $\rho : \XX \rightarrow \RR \cup \{-\infty\}$ is said to be $c$-concave if there exists \\
$\psi: \YY \rightarrow \RR \cup \{-\infty\}$, $\psi \not \equiv -\infty$, such that $\rho(x) = \psi^c(x)$, where 
\begin{equation*}
\psi^c(x) = \inf_{y\in\YY} [c(x,y) - \psi(y)].
\end{equation*} 
That is $\rho$ is $c$-concave if there exists $\psi$ such that $\rho$ is the $c$-transform of $\psi$.
\end{defn}
 
\begin{prop}[Alternative characterization of c-concavity]\label{c-concavity_alternative_characterization}
 Let $\XX$ and $\YY$ be two nonempty sets and let $c(x,y)$ be defined on $\XX \times \YY$ with values in $\RR \cup\{+\infty\}$.
Consider a function $\rho : \XX \rightarrow \RR \cup \{-\infty\}$, $\rho \not \equiv -\infty$. \\
Then $\rho$ is $c$-concave if and only if $\rho^{cc} = \rho$, where $\rho^{cc} = (\rho^c)^c$.
\end{prop}
 
\begin{proof}[Proposition \ref{c-concavity_alternative_characterization}]
Suppose that $\rho$ is $c$-concave. Then, there exists $\psi: \YY \rightarrow \RR \cup \{-\infty\}$, $\psi \not \equiv -\infty$, such that $\rho(x) = \psi^c(x)$, for all $x \in \XX$. By Proposition \ref{varphi^ccc=varphi^c},  $\psi^c = \psi^{ccc}$. Thus $\rho = (\psi^c)^{cc} = \rho^{cc}$.

Conversely, suppose that $\rho(x)=\rho^{cc}(x)$ for all $x\in \XX$. Let $\rho^c(y) = \psi(y)$. Then, $\psi^c(x) = \rho^{cc}(x) = \rho(x)$. Hence, there exists $\psi$ (namely $\rho^c$) such that $\rho$ is the \\
$c$-transform of $\psi^c$ (namely $\rho^{cc}$).
 \end{proof}

\begin{rmk}
Note that for any not necessarily $c$-concave function $\varphi$, $\varphi \not \equiv -\infty$, its  $c$-transform $\varphi^c$ is always $c$-concave. Indeed, by Proposition \ref{varphi^ccc=varphi^c}, we have $\varphi^c = \varphi^{ccc} = (\varphi^c)^{cc}$. Thus, by Proposition \ref{c-concavity_alternative_characterization}, $\varphi^c$ is $c$-concave.
\end{rmk}

We can now state the Kantorovich Duality Theorem Part 2 - Existence of optimal solutions:

\begin{theo}[Kantorovich Duality Theorem Part 2 - Existence of optimal solutions]\label{Existence_optimal_sol_KD}
Let $\mu$ and  $\nu$ be probability measures on the Polish spaces $\XX$ and $\YY$ respectively. Let $c \,:\, \XX \times \YY \rightarrow \RR \cup \{+\infty\}$ be a nonnegative lower semi-continuous cost function.

Let $\Phi_c$ be the set of all measurable functions $(\varphi,\psi)\in L_1(\mu)\times L_1(\nu)$ satisfying
\begin{equation}
\varphi(x) + \psi(y) \leq c(x,y)
\end{equation}
for $\mu$-almost all $x\in\XX$, $\nu$-almost all $y\in\YY$. 

\noindent As defined in Theorem \ref{M-K_duality_thm}, let
\begin{equation*}
J(\varphi,\psi) = \int_\XX \varphi(x) \diff\mu(x) + \int_\YY \psi(y) \diff\nu(y) \quad \mbox{and} \quad I[\vartheta] = \int_{\XX\times \YY} c(x,y) \diff\vartheta(x,y). 
\end{equation*}

\begin{enumerate} 
\item[(i)] Then, the infimum of $I[\vartheta]$ over all $\vartheta \in \NN(\mu,\nu)$ is attained. That is, there exists $\vartheta^* \in \NN(\mu,\nu)$ such that
\begin{equation*}
\inf_{\vartheta\in\NN(\mu,\nu)} I[\vartheta] = I[\vartheta^*].
\end{equation*}

\item[(ii)]  Assume, moreover, that there exist nonnegative measurable functions \\
$c_{\XX} \in \mathrm{L}_1(\mu)$ and $c_{\YY} \in \mathrm{L}_1(\nu)$ such that $\,\forall (x,y) \in \XX \times \YY, \quad c(x,y) \leq c_{\XX}(x) + c_{\YY}(y).$ 

Then, the supremum of $J(\varphi,\psi)$ over all $(\varphi,\psi) \in \Phi_c$  is attained. Indeed, the dual Kantorovich problem admits a maximiser in the form of a pair of conjugate $c$-concave functions $(\rho, \rho^c)$. Hence,
\begin{equation*}
\sup_{\Phi_c} J(\varphi,\psi) = \max_{\rho \in L_1(\mu)} \Big( \int_{\XX} \rho(x) \diff\mu(x) + \int_{\YY} \rho^c(y) \diff \nu(y) \Big).
\end{equation*}
 \end{enumerate}
\end{theo}

The proof of existence of an optimal coupling (Theorem \ref{Existence_optimal_sol_KD}, Part $(i)$) requires the notion of tightness of measures, basic results associated to the notion of tightness and two theorems: Prokhorov theorem and Portmanteau theorem.

For the proof of the existence of a maximiser for the dual Kantorovic h problem (Theorem \ref{Existence_optimal_sol_KD}, Part $(ii)$), one can consult p.86 of Villani's \textit{Optimal Transport: Old and New} \cite{VillaniOptimalTrans}.

To give the definition of the tightness of a measure, one needs to define a topology (and thus a convergence) on the space of Borel probability measures $P(\XX)$.

\begin{defn}\label{def_weak_convergence_non_Polish}[Weak Convergence]
Let $(\XX,d)$ be a metric space and $P(\XX)$ be its space of Borel probability measures. A sequence of probability measures $(\mu_k)_{k\in\NNN} \in P(\XX)$ is said to converge weakly to $\mu \in P(\XX)$ (denoted by $\mu_k \longrightarrow \mu$) if, for every bounded continuous function $\varphi: \XX \rightarrow \RR$, 
\begin{equation*}
\int \varphi \diff \mu_k \rightarrow \int \varphi \diff \mu.
\end{equation*}
The topology induced by the weak convergence on $M(\XX)$, the space of probability measures on $\XX$, is called the weak topology.
\end{defn}

The following theorem gives useful equivalent definitions of weak convergence:

\begin{theo}[Portmanteau Theorem]\label{Portmanteau}
Let $(\XX,d)$ be a metric space and $P(\XX)$ be its space of Borel probability measures. Then, these five conditions are equivalent to the definition \ref{def_weak_convergence_non_Polish} of weak convergence:
\begin{enumerate}[topsep=10pt, itemsep=15pt]
\item[(i)] $\dps\int \varphi \diff\mu_k \rightarrow \int \varphi \diff\mu$ for any bounded, real function $\phi$, continuous $\mu$-a.e. ;
\item[(ii)] $\dps\int \varphi \diff\mu_k \rightarrow \int \varphi \diff\mu$ for any bounded, uniformly continuous function $f$;
\item[(iii)] $\dps \limsup_{k \rightarrow \infty} \mu_k(F) \leq \mu(F)$ for all closed $F$;
\item[(iv)] $\dps \mu(G) \leq \liminf_{k \rightarrow \infty} \mu_k(G)$ for all open $G$;
\item[(v)] $ \dps \mu_k(A) \rightarrow \mu(A)$ for all Borel sets $A$ for which $\mu(\partial A) = 0$,  where $\partial A := \bar{A} \cap \bar{A^\mathsf{c}}$.
\end{enumerate}
\end{theo}

\begin{proof}
For a proof of Theorem \ref{Portmanteau}, one can consult Billingsley (\cite{Billingsley_Convergence_of_Prob_measures}, Theorem 2.1).
\end{proof}

\noindent We recall the definition of \textit{tightness} for measures and family of measures (see for example Billingsley \cite{Billingsley_Convergence_of_Prob_measures}):

\begin{defn}
Let $(\XX,d)$ be a metric space and $P(\XX)$ be its space of Borel probability measures.\\
A probability measure $\mu\in P(\XX)$  is tight if, for any $\epsilon>0$, there is a compact set $K_{\epsilon}$ such that $\mu(\mathcal{X}\setminus K_{\epsilon}) \leq \epsilon$.\\
A subset $S \subset P(\XX)$ of probability measures is tight if, for any $\epsilon>0$, there is a compact set $K_{\epsilon}\subset\mathcal{X}$ such that $\mu(\mathcal{X}\setminus K_{\epsilon}) \leq \epsilon$, for all $\mu\in S$.  
\end{defn}
 
\begin{theo}\label{Tightness_Theorem}
Let $\XX$ be a Polish space and $P(\XX)$ be its space of Borel probability measures.\\
Then, any probability measure $\mu\in P(\XX)$ is tight.
\end{theo}

\begin{proof}[Theorem \ref{Tightness_Theorem}]
For a proof of Theorem \ref{Tightness_Theorem}, one can consult Billingsley (\cite{Billingsley_Convergence_of_Prob_measures}, Theorem 1.3).
\end{proof}

\begin{lem}[Tightness of Couplings]\label{Tightness of transference plans}
Let $(\XX,d_{\XX})$ and $(\YY,d_{\YY})$ be two metric spaces. Let $S_x$ and $S_y $ be tight subsets of $P(\XX)$ and $P(\YY)$, respectively. Then the set $\NN(S_x, S_y)$ of all couplings whose marginals lie in $S_x$ and $S_y$ respectively, is itself tight in $P(\XX \times\YY)$.
\end{lem}

\begin{proof}[Lemma \ref{Tightness of transference plans}]
Let $\mu \in S_x, \nu\in S_y$. Since $S_x$ is tight, for any $\epsilon > 0$, there is a compact set $K_{\epsilon}\subset \XX$, independent of the choice of $\mu$ in $S_x$, such that $\mu(\XX\setminus K_{\epsilon}) \leq \epsilon$. Similarly, since $S_y$ is tight, there is a compact set $L_{\epsilon}\subset \YY$, independent of the choice of $\nu$ in $S_y$, such that $\nu(\YY\setminus L_{\epsilon}). \leq \epsilon$. Then for any coupling $\vartheta\in\NN(\mu,\nu)$, one have
\begin{equation*}
\NN[(\XX\times\YY) \setminus (K_{\epsilon} \times L_{\epsilon})] \leq \NN[(\XX\setminus K_{\epsilon}) \times\YY] + \NN[\XX \times (\YY \setminus L_{\epsilon})] \mu[\XX \setminus K_{\epsilon}] + \nu[\YY \setminus L_{\epsilon}] \leq 2\epsilon
\end{equation*}
Now, since this bound does not depend on the choice of the coupling $\vartheta$ and since $K_{\epsilon}\times L_{\epsilon}$ is compact in $\XX\times\YY$, the proof is complete.
\end{proof}

\begin{theo}[Prokhorov Theorem]\label{Prokhorov}
Let $(\XX,d)$ be a metric space and $P(\XX)$ be its space of Borel probability measures.\\
If $S \subset P(\XX)$ is tight, then $S$ is relatively compact (ie. $\bar{S}$ is compact in $P(\XX)$).\\
Let $\XX$ be a Polish space and $P(\XX)$ be its space of Borel probability measures.\\
For any subset S $\subseteq P(\XX)$, the following two statements are equivalent:
\begin{enumerate}
\item[(i)] $\bar{S}$ is relatively compact 
\item[(ii)] $S$ is tight.
\end{enumerate}
\end{theo}

\begin{proof}
For a proof of Theorem \ref{Prokhorov}, one can consult Billingsley (\cite{Billingsley_Convergence_of_Prob_measures}, Theorem 5.1 and Theorem 5.2).
\end{proof}

We can now prove the existence of an optimal coupling (Theorem \ref{Existence_optimal_sol_KD}).

\begin{proof}[Theorem \ref{Existence_optimal_sol_KD} - part $(i)$]
Since $\XX$ and $\YY$ are Polish spaces, both $\{{\mu}\} \subset P(\XX)$ ad $\{{\nu}\} \subset P(\YY)$ are tight. By Lemma \ref{Tightness of transference plans}, $\NN(\mu,\nu)$ is a tight subset of $P(\XX\times\YY)$ and by by Prokhorov's theorem (\ref{Prokhorov}), it is relatively compact for the weak topology.

Let $I^*$ denote the infimum of $I(\vartheta) = \int_{\XX \times \YY} c \diff \vartheta$, for $\vartheta \in \NN(\mu,\nu)$, and let $(\vartheta_k)_{k \geq 1}$ be a minmizing sequence for $I$. let us show that if $\vartheta^*$ is any weak unit point of $\vartheta \in \NN(\mu,\nu)$, then $I(\vartheta^*) = I^*$. Since $c$ is a nonnegative lower semi-continuous function, it is the pointwise limit of an increasing sequence $(c_n)_{n \geq 1}$ of continuous bounded sequence. Then, by the monotone convergence theorem, we have for all $k \geq 1$, 
\begin{equation*}
\lim_{n \rightarrow \infty} \int c_n \diff \vartheta_k = \int c \diff \vartheta_k \quad \mbox{and} \quad \lim_{n \rightarrow \infty}  \int c_n \diff \vartheta^* = \int c \diff \vartheta^*.
\end{equation*}
then, by definition of $\vartheta^*$, we have 
\begin{equation*}
I(\vartheta*) = \lim_{n \rightarrow \infty} \int c_n \diff \vartheta^* \leq \lim_{n \rightarrow \infty} \limsup_{k \rightarrow\infty} \int c_n \diff \vartheta_k \leq \limsup_{k \rightarrow \infty} \int c \diff \vartheta_k = I.
\end{equation*}
As $\vartheta^* \in \NN(\mu, \nu)$, $I(\vartheta^*) = I$.
\end{proof}

\section{Kantorovich-Rubinstein Duality Theorem}\label{KR_Duality_Theorem}

When the cost function $c$ is in fact a metric $d$ on some Polish space $\XX$, and both $\mu$ and $\nu$ are probability Borel measures on $\XX$, we obtain a particular case of the Kantorovich Duality Theorem \ref{M-K_duality_thm}, the so-called Kantorovich-Rubinstein theorem. It first appeared in a paper by Kantorovich in 1942. In Kantorovich original paper, the result is proved for a compact metric space equipped with Borel probability measures. 

The version of the Kantorovich-Rubinstein Theorem \ref{K-R_Duality_thm_Villani_version}, given below is found in \textit{Topics in Optimal Transportation} by Villani \cite{VillaniTopicsOptimalTransportation}. 

\begin{ntn}\label{notation_for_I_d}  For clarity, we first introduce the following notations:
\begin{enumerate}
\item[(i)] Let $\DD_{\XX}$ denote the set of all lower semi-continuous metrics $d$ on $\XX$. 
\item[(ii)] Let $\XX$ be a Polish space and $d \in \DD_{\XX}$. Let $\mu_1$ and $\mu_2$ be two Borel probability measures on $\XX$. Let $I_d: \XX \times \XX \rightarrow \RR$ be the function defined by
 \begin{equation*}
I_d(\vartheta) = \inf_{\vartheta \in \NN(\mu_1,\mu_2)} \int_{\XX\times\XX} d(x_1,x_2) \diff \vartheta(x_1,x_2).
\end{equation*}
\end{enumerate}
\end{ntn}

\noindent We can now state the Kantorovich-Rubinstein theorem:
 
\begin{theo}[Kantorovich-Rubinstein theorem]\label{K-R_Duality_thm_Villani_version}
Let $\XX$ be a Polish space and $d \in \DD_{\XX}$. Let $\mu_1$, $\mu_2$ be two Borel probability measures on $\XX$ and $I_d$ be defined as in Notation \ref{notation_for_I_d}. \\
Let $\lip(\XX)$ denote the space of all Lipschitz functions $g$ on $\XX$, and define 
 \begin{equation*}
 ||g||_{\lip} \equiv \sup_{x\not=y} \frac{|g(x_1) - g(x_2)|}{d(x_1,x_2)}.
 \end{equation*}
Then 
\begin{equation*}
I_d(\vartheta) = \sup \Big\{ \int_{\XX} g(x) \diff (\mu_1-\mu_2)(x); \, g \in L_1(|\mu_1-\mu_2|), \, ||g||_{\lip} \leq 1 \Big\}.
\end{equation*}
Moreover, it does not change the value of the supremum above to impose the additional condition that $g$ be bounded.
\end{theo}
 
\begin{rmk}\label{distance_doesnt_define_topology} There are two facts worthy of attention:
\begin{enumerate}
\item[(i)] The metric $d \in \DD_{\XX}$ is not necessarily the distance defining the topology on $\XX$. More often than not it is the same distance, but it need not be.  
\item[(ii)] If the distance $d \in \DD_{\XX}$ and the distance defining the topology on $\XX$ are different, the Lipschitz property for functions on $\XX$ is defined with respect to the distance $d$. 
\end{enumerate}
\end{rmk} 

\begin{ntn}\label{ntn_for_d_defining_topology} We introduce these two notations that we will use for the remaining of the thesis:
\begin{enumerate}
\item[(i)] Let $\XX$ be a Polish space. Any metric in $\DD_{\XX}$ that defines the topology on $\XX$ will be denoted $d^*$. Hence,
we denote by $(\XX,d^*)$ a Polish space $\XX$ whose topology is endowed by the metric $d^* \in \DD_{\XX}$.
\item[(ii)] Let $\DD^*_{\XX}$ denote the space of all metrics $d^* \in \DD_{\XX}$ defining the topology on $\XX$. Hence, $\DD^*_{\XX} \subset \DD_{\XX}$.
\end{enumerate}
\end{ntn}

\begin{cor}[Kantorovich-Rubinstein theorem for atomic measures]\label{K-R_Duality_thm_Villani_version_atomic_measures}
Let $\XX$ be a Polish space and $d \in \DD_{\XX}$. Let us define two atomic probability measures
\begin{equation*}
\mu_1 = \sum_{i = 1}^{n} \mu_i^{(1)} \delta_{x_i} \quad \mbox{and} \quad \mu_2 = \sum_{i = 1}^{m} \mu_i^{(2)} \delta_{x_i},
\end{equation*}
supported on the same finite number of ordered points $\{x_1, \ldots , x_n\} \in \XX$. 

Let $I_d$ be given by
 \begin{equation*}
 I_d[\vartheta] = \min_{\vartheta \in \NN(\mu_1,\mu_2)} \Big\{ \sum_{i=1}^n \sum_{j=1}^n d(x_i,x_j) \vartheta_{ij} \Big\}
\end{equation*}
where 
\begin{equation*}
\NN(\mu, \nu) = \Big\{ \sum_{i=1}^n \sum_{j=1}^m \vartheta_{ij} \delta_{(x_i,y_j)} : \, \forall i, \sum_{j=1}^m \vartheta_{ij} = \mu_i^{(1)} \quad \mbox{and} \quad \forall j, \sum_{i=1}^n \vartheta_{ij} = \mu_j^{(2)} \quad \mbox{with} \quad \vartheta_{ij} \geq 0 \Big\}.
\end{equation*}
Let $\lip(\XX)$ denote the space of all Lipschitz functions on $\XX$, and $||g||_{\lip}$ be defined as in Theorem \ref{K-R_Duality_thm_Villani_version}. Then,
 \begin{equation*}
I_d[\vartheta] =  \sup \Big\{ \sum_{i=1}^n g(x_i) (\mu_i^{(1)} - \mu_i^{(2)}) ; \, ||g||_{\lip} \leq 1  \Big\}.
\end{equation*}
\end{cor}
 
\begin{proof}[Corollary \ref{K-R_Duality_thm_Villani_version_atomic_measures}]
A direct application of Theorem \ref{K-R_Duality_thm_Villani_version} for two atomic measures $\mu_1$ and $\mu_2$ supported on an ordered set $\{x_1, \ldots , x_n\} \in \XX$ yields the following dual representation formula:
\begin{equation*}
I_d[\vartheta] = \sup \Big\{ \sum_{i=1}^n g(x_i) (\mu_i^{(1)} - \mu_i^{(2)}) : \, ||g||_{\lip} \leq 1 \Big\}.
\end{equation*}
Since $g$ is only evaluated on the finite set $\{ x_1, \ldots, x_n \}$, it is bounded. Thus, as shown in the proof of Lemma \ref{equality_with_pair_d-concave_fct} (given below), the condition $g \in L_1(|\mu_1-\mu_2|)$ is not necessary.
\end{proof}

The following Lemma \ref{equality_with_pair_d-concave_fct} is required for the proof of the Kantorovich-Rubinstein Theorem \ref{K-R_Duality_thm_Villani_version}.

\begin{prop}\label{equality_with_pair_d-concave_fct}
Let $\XX$ and $\YY$ be Polish spaces and let $c \in \DD_{\XX\times\YY}$ be a bounded lower semi-continuous cost function on $\XX \times \YY$. Let $\mu$,$\nu$ be two probability measures on $\XX$ and $\YY$. \\
Let $\Phi_c^{''}$ be the set of all measurable and bounded (bnd) functions $(\varphi^{''},\psi^{''})\in L_1(\mu)\times L_1(\nu)$ satisfying
\begin{equation*}
\varphi^{''}(x) + \psi^{''}(y) \leq c(x,y),
\end{equation*}
for $\mu$-almost all $x\in\XX$, $\nu$-almost all $y\in\YY$.\\
Let $J: L_1(\mu)\times L_1(\nu) \rightarrow \RR$ be defined by
\begin{equation*}
J(\varphi, \psi) = \int \varphi(x) \diff \mu(x) + \int \psi(y) \diff \nu(y).
\end{equation*}
Then,
\begin{equation*}
 \sup_{\Phi^{''}_c} J(\varphi,\psi)= \sup_{\rho\,bnd} J(\rho^{cc}, \rho^c),
 \end{equation*}
 \begin{equation*}
\mbox{where} \quad \rho^c(y) = \inf_{x\in\XX} [c(x,y)-\rho(x)] \quad \mbox{and} \quad \rho^{cc}(x) = \inf_{y\in\YY} [c(x,y) - \rho^c(y)].
 \end{equation*}
\end{prop}

Note that, in order to be consistent with the notations in Proposition \ref{Prop_for_M-K_duality_with_C_b_functions}, one should write  $(\varphi^{''}, \psi^{''}) \in \Phi_c^{''}$ and not $(\varphi, \psi) \in \Phi_c^{''}$. We have chosen to use the latter instead of the former in order to avoid writing $(\varphi^{''})^c$ and $(\varphi^{''})^{cc}$.\\

\begin{proof}[Proposition \ref{equality_with_pair_d-concave_fct}]
Let us first note that $\rho^c$ and similarly $\rho^{cc}$ are measurable. As recalled in Villani \cite{VillaniTopicsOptimalTransportation}, p.26, $c$ is a point wise limit of an increasing sequence $(c_l)_{l \leq 1}$ of bounded uniformly continuous functions (to see this, it suffice to write $c_l(w) = \inf [c(z) + l d(w,z)]$, where the infimum is taken over all $z \in \XX \times \YY$). Then, 
\begin{equation*}
\rho^c = \lim_n \rho_n^c, \mbox{ where } \, \rho_n^c (y) = \inf_{x \in \XX} [c_n(x,y) - \rho(x)].
\end{equation*}
The functions $\rho_n^c$ are uniformly continuous and therefore $\rho^c$ is measurable. The same argument shows that $\rho^{cc}$ is measurable. Moreover, $\rho^c$ and $\rho^{cc}$ are bounded since $\rho$ and $c$ are bounded.

\noindent By construction, $\rho^{cc}(x) + \rho^c(y) \leq c(x,y)$ and therefore $\{(\rho^{cc},\rho^{c}); \, \rho \mbox{ bnd} \} \subset \Phi^{''}_c$. Hence, 
\begin{equation*}
\sup_{\rho\, bnd} J(\rho^{cc}, \rho^c) \leq \sup_{\Phi_c} J(\varphi, \psi^{''})
\end{equation*}
To prove the converse, we first observe that if $(\varphi,\psi) \in \Phi_c^{''}$, then $\psi(y) \leq c(x,y) - \varphi(x)$ for all $(x,y) \in \XX \times \YY$. Hence, $\psi \leq \varphi^c$ since $\varphi^c(y) = \inf_x [c(x,y) - \varphi(x)]$.

Likewise, for any pair of functions $(\varphi, \varphi^c) \in \Phi_c^{''}$, we have $\varphi \leq \varphi^{cc}$ since the function $\varphi^{cc}(x) = \inf_y [c(x,y) - \varphi^{c}(y)]$.

As $\rho^{cc}(x) = \inf_y [c(x,y) - \rho^c(y)]$, we obtain, for any $x_0 \in \XX$, that
\begin{align*}
\rho^{cc}(x_0) &= \inf_{y \in \YY} \big( c(x_0,y) - \inf_{x \in \XX} (c(x,y) - \rho(x)) \big) \\
& \geq \inf_y \big( c(x_0,y) - (c(x_0,y) - \rho(x_0)) \Big) = \rho(x_0).
\end{align*}
Hence,
\begin{equation*}
J(\varphi, \psi) \leq J(\varphi, \varphi^c) \leq J(\varphi^{cc}, \varphi^c) \mbox{ and } \sup_{\Phi_c^{''}} J(\varphi,\psi) \leq \sup_{\rho\, bnd} J(\rho^{cc}, \rho^c).
\end{equation*}
\end{proof}
 
\begin{rmk}\label{if_lip_then_L1}
If $\diam(\XX)<\infty$, then a 1-Lipschitz function $g$ is such that $g$ is bounded and therefore is $L^1(\mu)$. Indeed, for any $x_{\circ} \in \XX$ and $x \in \XX$, 
\begin{equation*}
|g(x)| \leq |g(x_{\circ})| + |g(x) - g(x_{\circ})| \leq |g(x_{\circ})| + d(x,x_{\circ}) \leq |g(x_{\circ})| + \diam (\XX).
\end{equation*}
\end{rmk}

\begin{ntn}
For simplicity, we will use the following notations in the proof Theorem \ref{K-R_Duality_thm_Villani_version}:
\begin{enumerate}
\item[(i)] $E_d = \sup\big\{J(\rho^{dd}, \rho^d) ;\, \rho \in L_1(\mu) \big\}$ and $E^{'}_d = \sup\big\{J(\rho^{dd}, \rho^d) ;\, \rho \in C_b(\XX) \big\}$. 
\item[(ii)] $L_d = \sup\big\{J(g,-g) ;\, g \in L_1(|\mu_1-\mu_2|), \, ||g||_{\lip} \leq 1 \big\}$ and $L^{'}_d = \sup\big\{J(h,-h) ;\, h \in C_b(\XX), \, ||h||_{\lip} \leq 1 \big\}$.
\end{enumerate}
\end{ntn}

\noindent We can now give the proof of the Kantorovich-Rubinstein Theorem \ref{K-R_Duality_thm_Villani_version}: \\

\begin{proof}[Theorem \ref{K-R_Duality_thm_Villani_version}]
We define $d_n = d/(1+n^{-1} d)$ for $n \in \mathbb{N}$. For each $n$, $d_n$ is a bounded distance satisfying $d_n \leq d$ and $d_n(x_1,x_2)$ converges monotonically to $d(x_1,x_2)$, for all $(x_1,x_2)$. 

The proof is separated in 2 sections. For the first section, we assume that the equality $S_{d_n} = L_{d_n}$ holds for any $n \in \mathbb{N}$. Given that assumption, we show that $S_d = L_d$. 
The second section focuses on proving that, $S_{d_n} = L_{d_n}$ for any bounded distance $d_n$.

1. We need to prove that $S_d \leq L_d$. We write $d =  \sup d_n$, where $d_n$ is a nondecreasing sequence of nonnegative, uniformly continuous functions. We assume that $S_{d_n} = L_{d_n}$ holds for any $n \in \mathbb{N}$. Also, it is clear that $L_{d_n} \leq L_d$, for all $n \geq 1$ since $\Phi_{d_n} \subset \Phi_d$. Hence, we have the inequality: $S_{d_n} \leq L_d$, for all $n$ and by the property of the supremum, we obtain $\sup_{n} S_{d_n} \leq L_d$. 

It is left to show that $\sup_n S_{d_n} = S_d$. In the proof of the Monge-Kantorovich duality theorem \ref{M-K_duality_thm} presented in \textit{Topics in Optimal Transportation} \cite{VillaniOptimalTrans}, Villani proves that $\sup I_{d_n} = I_d$, where  
\begin{equation*}
I_d = \inf \left\{ \int_{\XX\times \YY} d(x,y) \diff\vartheta(x,y); \, \vartheta\in\NN(\mu,\nu) \right\}.
\end{equation*}
 Moreover, by Kantorovich Duality Theorem \ref{M-K_duality_thm}, we know that $I_d = S_d$. Therefore, we have $\sup_n S_{d_n} = S_d$. Since $\sup_{n} S_{d_n} \leq L_d$ and $\sup_n S_{d_n} = S_d$, we obtain the inequality $S_d \leq L_d$.\\

2. We shall now prove that $S_{d_n} = L_{d_n}$ for any bounded distance $d_n$. For clarity, we shall drop the index $n$ thus writing $d$ instead of $d_n$.
First of all, note that, by Remark \ref{if_lip_then_L1}, if $\rho$ is a 1-Lipschitz function and $d$ is bounded, then $\rho$ is bounded, and therefore is in $L_1(|\mu - \nu|)$, for any pair of probability measures $\mu$ and $\nu$. Hence $(\rho,-\rho) \in \Phi_d$.
Thus, it is clear that $L_d \leq S_d$. 

To prove that $S_d \leq L_d$, it is enough to show, by Remark \ref{if_lip_then_L1} that 
\begin{equation}\label{J_equal_sup_over_lip}
\sup_{\Phi_d} J(\varphi,\psi) = \sup \Big \{ \int_{\XX} \rho \diff(\mu-\nu) ; \, ||\rho||_{Lip} \leq 1 \Big\}
\end{equation}
Consider $S_d = \sup_{\Phi_d} J(\varphi,\psi)$. By Proposition \ref{equality_with_pair_d-concave_fct}, we know that $S_d = E_d$ where 
\begin{equation*}
E_d = \sup_{\rho \in L_1(\mu)} J(\rho^{dd}, \rho^d) \quad \mbox{with} \quad \rho^d(x_2) = \inf_{x_1\in\XX} [d(x_1,x_2) - \rho(x_1)].
\end{equation*}

Since $\rho^d$ is the infimum of two 1-Lipschitz and bounded functions, $\rho^d$ is bounded and 1-Lipschitz. Hence, $\rho^d(x_2)$ is finite for any $x_2 \in \XX$. 

Therefore, we obtain two inequalities: First, $\rho^{dd} \leq - \rho^d$. That is because $\inf_{x_2\in\XX} [d(x_1,x_2)-\rho^d(x_2)] \leq d(x_1,x_1)-\rho^d(x_1)$. 
Secondly, $-\rho^d \leq \rho^{dd}$. Indeed, since $\big\lvert \rho^d(x_1) - \rho^d(x_2) \big\rvert \leq d(x_1,x_2)$, we have $-\rho^d(x_1) \leq d(x_1,x_2) - \rho^d(x_2)$, for all $x_1,x_2\in\XX$. The infimum being the largest lower bound, we get $-\rho^d(x_1) \leq \rho^{dd}(x_1)$, for all $x_1\in\XX$, with $\rho^{dd}(x_1) = \inf_{x_2\in\XX} [d(x_1,x_2) - \rho^d(x_2)]$.

Together, these two inequalities yield that $\rho^{dd} = -\rho^d$.

Finally, one notes that, since $\rho^d$ is bounded and 1-Lipschitz, $\big\{(-\rho^d,\rho^d) ; \, \rho \, \mbox{bounded} \big\} \subset \big\{(-\rho,\rho) ; \, ||\rho||_{Lip} \leq 1 \big\}$. Thus,
\begin{equation*} 
\sup_{\rho\in L_1(\mu)} J(-\rho^d,\rho^d) \leq L^{'}_d \quad \mbox{where} \quad L^{'}_d = \sup_{||\rho||_{Lip} \leq 1} J(-\rho^d,\rho^d).
\end{equation*}

Therefore we have, for $\rho$ bounded,
\begin{equation*}
 S_d = S^{'}_d = E_d = \sup_{\rho\in L_1(\mu)} J(-\rho^d,\rho^d) \leq L_d^{'} = L_d.
\end{equation*}

Since $L_d \leq S_d$ and $S_d \leq L_d$, we have shown that $S_d$ and $L_d$ are equal.
\end{proof}

\subsection{Existence of Optimal measures and functions}

Since the Kantorovich-Rubinstein theorem is a particular case of the Kantorovich Duality theorem \ref{M-K_duality_thm}, it is natural to suppose that both the infimum over all couplings is realized and that the dual formulation achieves its supremum over all 1-Lipschitz functions.

\begin{defn}\label{Def_optimal_fct_and_measure_KR_thm}
Let $\XX$ be a Polish space and $d \in \DD_{\XX}$. Let $\mu_1$, $\mu_2$ be two probability measures on $\XX$.
\begin{itemize}
\item[(i)] By an optimal measure on $\XX \times \XX$ we shall understand a measure $\rho \in \NN(\mu_1, \mu_2)$ such that 
\begin{equation*}
\int_{\XX} d(x_1,x_2) \diff \rho(x_1,x_2) =  \inf_{\vartheta \in \NN(\mu_1,\mu_2)} \int_{\XX\times\XX} d(x_1,x_2) \diff\vartheta(x_1,x_2).
\end{equation*}

\item[(ii)] By an optimal function on $\XX$, we shall understand a 1-Lipschitz function $f$, $f \in L_1(|\mu_1-\mu_2|)$ such that 
\begin{equation*}
\int_{\XX} f(x) \diff(\mu_1-\mu_2)(x) = \sup \Big\{ \int_{\XX} g \diff (\mu_1-\mu_2); \, g \in L_1(|\mu_1-\mu_2|), \, ||g||_{\lip} \leq 1 \Big\}.
\end{equation*}
\end{itemize}
\end{defn}

\begin{rmk}\label{optimalfunctionsf+c}
If $f$ is an optimal function, then $f+c$, for all $c \in \RR$ is also an optimal function. Indeed, as any constant function is 1-Lipschitz, and belongs to $L_1(|\mu_1-\mu_2|)$, we only have to note that 
\begin{align*}
 \int_{\XX} (f+c)(x) \diff (\mu_1-\mu_2)(x) & = \int_{\XX} f(x) \diff (\mu_1-\mu_2)(x) + c \int_{\XX}\diff (\mu_1-\mu_2)(x) \\
  & = \int_{\XX} f(x) \diff (\mu_1-\mu_2)(x).
 \end{align*}
 \end{rmk}

\begin{theo}[Existence of optimal solutions] \label{existence_optimal_sol_KR_thm}
Let $\XX$ be a Polish space and $d \in D_{\XX}$ be a bounded lower semi-continuous metric on $\XX$. Let $\mu_1$, $\mu_2$ be two probability measures on $\XX$. 
\begin{enumerate}
 \item[(i)] Then, there exists an optimal measure $\rho$ on $\XX \times \XX$. That is, the infimum of $I[\vartheta]$ over all $\vartheta \in \NN(\mu_1,\mu_2)$ is attained. 
 \item[(ii)]  Then, there exists an optimal 1-Lipschitz function $f$ on $\XX$. That is, the supremum of $J(g,-g)$ over all 1-Lipschitz functions $g$ is attained by a 1-Lipschitz function. 
\end{enumerate}
\end{theo}

To write the proof of Theorem \ref{existence_optimal_sol_KR_thm}, we first need Proposition \ref{d_concave_equivalent_inequality} and Proposition \ref{1_lip_is_d_concave}. We start by recalling the definitions and basic results of $c$-concavity seen in subsection \ref{subsection_Existence_of_Optimal_Solutions}:

\begin{defn}\label{d_concavity_definition}[d-concavity]
Let $(\XX,d)$ be a metric space. A function $\rho:\XX \rightarrow \RR \cup \{-\infty\}$ is said to be $d$-concave if there exists 
$\psi: \XX \rightarrow \RR \cup \{-\infty\}$, \\
$\psi \not \equiv -\infty$, such that $\rho(x) = \psi^d(x)$, where 
\begin{equation*}
\psi^d(x_1) = \inf_{x_2 \in \XX} [d(x_1,x_2) - \psi(x_2)].
\end{equation*} 
That is $\rho$ is $d$-concave if there exists $\psi$ such that $\rho$ is the $d$-transform of $\psi$.
\end{defn}
 
\begin{prop}[Alternative characterization of d-concavity]\label{d-concavity_alternative_characterization}
Let $(\XX,d)$ be a metric space. Consider a function $\rho : \XX \rightarrow \RR \cup \{-\infty\}$, $\rho \not \equiv -\infty$. \\
Then, $\rho$ is $d$-concave if and only if $\rho^{dd} = \rho$, where $\rho^{dd} = (\rho^d)^d$.
\end{prop}

We can now give propositions \ref{d_concave_equivalent_inequality} and \ref{1_lip_is_d_concave}:

\begin{prop}\label{d_concave_equivalent_inequality}
 Let $(\XX,d)$ be a metric space and let $\rho: \XX \rightarrow \RR$ a real function defined on
$\XX$.

The function $\rho$ is $d$-concave if and only if 
\begin{equation*}
\rho(x_2) - \rho(x_1) \leq  d(x_1,x_2), \,  \forall x_1, x_2  \in \XX.
\end{equation*}
\end{prop}
 
\begin{proof}[Proposition \ref{d_concave_equivalent_inequality}]
Suppose that $\rho(x_2) - \rho(x_1) \leq d(x_1,x_2)$, $\forall x_1,x_2  \in \XX$. To show that $\rho$ is $d$-concave, one needs to find $\psi$ such that 
\begin{equation*}
\rho(x_2) = \inf_{x_1 \in \XX} \big( d(x_1,x_2) - \psi(x_1) \big).
\end{equation*}
Let $\psi(x_1) = -\rho(x_1)$. We have $\rho(x_2) + \psi(x_1)  \leq d(x_1,x_2)$ and thus \\
$\rho(x_2) \leq d(x_1,x_2) - \psi(x_1), \,\forall x_1,x_2  \in \XX$. Since the infimum is the greatest lower bound, 
 \begin{equation*}
 \rho(x_2) \leq \inf_{x_1 \in \XX} \big( d(x_1,x_2) - \psi(x_1) \big), \,\forall  x_2 \in \XX.
\end{equation*}
Now, for $x_1 = x_2$, we have  $\rho(x_2) = d(x_1,x_2) - \psi(x_1)$. Hence 
\begin{equation*}
\rho(x_2) \geq \inf_{x_1 \in \XX} \big( d(x_1,x_2) - \psi(x_1) \big) .
\end{equation*}
Thus, $\displaystyle \rho(x_2) = \inf_{x_1 \in \XX} \big( d(x_1,x_2) - (-\rho(x_1)) \big)$, for all $x_2 \in \XX$, and therefore $\rho$ is $d$-concave.
 
Conversely, suppose that $\rho$ is $d$-concave. Hence, there exists $\psi$ such that, for all $x_2 \in \XX$,
 \begin{align*}
 \rho(x_2) &= \inf_{x_3 \in \XX} \big(d(x_3, x_2) - \psi(x_3\big) \\
 &= \inf_{x_3 \in \XX} \big( d(x_3, x_1) - \psi(x_3) - d(x_3,x_1) + d(x_3, x_2) \big).
 \end{align*}
 The triangle inequality gives $d(x_3,x_2) - d(x_3,x_1) \leq d(x_1,x_2)$, hence
 \begin{equation*}
 \rho(x_2) \leq \inf_{x_3 \in \XX} \big( d(x_3,x_1) - \psi(x_3) + d(x_1,x_2) \big) \mbox{ for all } x_2 \in \XX.
 \end{equation*}
 For $x_3 = x_1$, we obtain $\rho(x_2) \leq d(x_1,x_2) - \psi(x_1)$. As $\rho(x_1) = \inf_{x_3 \in \XX} \big( d(x_3,x_1) - \psi(x_3)\big)$, we get $\rho(x_1) \geq d(x_1,x_1) - \psi(x_1) = - \psi(x_1)$. Hence $\rho(x_2) - \rho(x_1) \leq d(x_1, x_2)$.
\end{proof}

\begin{prop} \label{1_lip_is_d_concave}
Let $(\XX,d)$ and $\rho$ be defined as in Proposition \ref{d_concave_equivalent_inequality}. \\
Then, the function $\rho$ is $d$-concave if and only if $\rho$ is 1-Lipschitz.
Also, $-\rho$ is the $d$-transform of $\rho$.
\end{prop}

\begin{proof}[Proposition \ref{1_lip_is_d_concave}]
If $\rho$ is $d$-concave, then, by Proposition \ref{d_concave_equivalent_inequality},  $\rho(x_2) - \rho(x_1) \leq d(x_1,x_2), \, \forall x_1,x_2 \in \XX$. As $d$ is symmetric, we get $ |\rho(x_2) - \rho(x_1)| \leq d(x_1,x_2), \, \forall x_1,x_2 \in \XX$. Hence $\rho$ is 1-Lipschitz.

Similarly, if $\rho$ is 1-Lipschitz, then $\rho(x_1) - \rho(x_2) \leq d(x_1,x_2),  \, \forall x_1,x_2 \in \XX$, and by Proposition \ref{d_concave_equivalent_inequality}, $\rho$ is $d$-concave.

Regarding the $d$-transform, we need to show that $\rho^d \leq \rho$ and $\rho \leq \rho^d$. Recall that $\rho^d$ of $\rho$ is given by 
\begin{equation*}
\rho^d(x_2) = \inf_{x_1 \in \XX} [d(x_1,x_2) - \rho(x_1)]. 
\end{equation*}
For $x_1 = x_2$, we have $\rho^d(x_2) \leq - \rho(x_2)$.

Now, since $\rho$ is 1-Lipschitz, $ \inf_{x_1 \in \XX} [ -(\rho(x_2) - \rho(x_1)) - \rho(x_1)] \leq \rho^d(x_2)$. That is $ \rho \leq  \rho^d$.
\end{proof}

We can now give the proof of Theorem \ref{existence_optimal_sol_KR_thm}:\\

\begin{proof}[Theorem \ref{existence_optimal_sol_KR_thm}]
\begin{enumerate}
\item[(i)] Since $d$ is a nonnegative lower semi-continuous metric, a direct application of Theorem \ref{Existence_optimal_sol_KD} (Kantorovich Duality Theorem - Part 2.) yields the conclusion. \\
\item[(ii)] Since $d$ is bounded, the conditions the Kantorovich Duality Theorem \ref{Existence_optimal_sol_KD} are clearly satisfied. Therefore, we know that the supremum of $J(\varphi,\psi)$ over all $(\varphi,\psi) \in \Phi_d$  is realized by a pair of conjugate $d$-concave functions $(\rho, \rho^d)$. 
By Proposition \ref{1_lip_is_d_concave}, a function is $d$-concave if and only if it is 1-Lipschitz. Moreover, $\rho^d = -\rho$. Hence, the supremum of $J(\varphi,\psi)$ over all $(\varphi,\psi) \in \Phi_d$ is equivalent to the supremum over all 1-Lipschitz functions $\rho$.
\end{enumerate}
\end{proof}

\begin{rmk}\label{rmk_on_existence_optimal_sol_KR_thm}
More general versions of Theorem \ref{existence_optimal_sol_KR_thm} have been proved. For example, let $(\XX, d)$ be a Polish space and $\mu_1$, $\mu_2$ be two probability measures in $P_{\XX}$ (ie. for any fixed $x_{\circ} \in \XX, \, d( . \, , x_{\circ}) \in L_1(\mu_i)$), then there exists a function $f: \XX \rightarrow \RR$, 1-Lipschitz and in $L_1(|\mu_1-\mu_2|)$ which is optimal.
\end{rmk}

We end this section with the following important characterization of an optimal function (see, for example, \cite{Edwards_KR_theorem}, Theorem $8.1$):

\begin{prop}[Kantorovich Optimality Criterion] \label{OptimalityKR}
Let $(\XX,d)$ be a metric space whose diameter is bounded and $\mu_1$, $\mu_2$ be two probability measures on $\XX$. For $\vartheta \in \NN(\mu_1, \mu_2)$ and $f: \XX \rightarrow \RR$ a 1-Lipschitz function, the following statements are equivalent:
\begin{itemize}
\item[(i)] $\vartheta$ is an optimal measure and $f$ is an optimal function.
\item[(ii)] $d(x_1,x_2) = f(x_1) - f(x_2), \qquad \forall (x_1,x_2) \in \supprt \vartheta$.
\end{itemize}
\end{prop}

\begin{proof}[Proposition \ref{OptimalityKR}]
First of all, notice that, as $d$ is bounded, $\dps I(\vartheta) = \int_{\XX \times \XX} d(x,y) \diff \vartheta(x,y) < \infty$ and that $f \in L_1(\mu_i)$ for $i = 1,2$. Hence, 
\begin{equation*}
f \in L_1(|\mu_1 - \mu_2|), \mbox{ as } \, \int_{\XX} |f(x)| \diff |\mu_1-\mu_2|(x) \leq \int |f(x)| (\diff \mu_1(x) + \diff \mu_2(x)) <\infty.
\end{equation*}

Suppose that $\rho$ is an optimal measure on $\XX\times\XX$ and $f$ is an optimal function on $\XX$. Then,
\begin{align*}
\mathcal{T}_{d}(\mu_1,\mu_2) &= \int_{\XX\times\XX} d(x_1,x_2) \diff \rho(x_1,x_2) \qquad \mbox{by the optimality of $\rho$.} \\
& \geq \int_{\XX\times\XX} |f(x_1)-f(x_2)| \diff \rho(x_1,x_2) \qquad \mbox{since $f$ is 1-Lipschitz}. \\
& \geq \int_{\XX\times\XX} \big( f(x_1)-f(x_2) \big) \diff \rho(x_1,x_2) \\
& = \int_{\XX} f(x_1)\diff \rho(x_1,x_2) - \int_{\XX} f(x_2)\diff\rho(x_1,x_2) \\ 
&= \int_{\XX} f(x_1)\diff \mu_1(x_1) - \int_{\XX} f(x_2)\diff\mu_2(x_2) \qquad \mbox{since $\vartheta \in \NN(\mu_1,\mu_2)$}, \\
&= \int_{\XX} f(x) \diff(\mu_1-\mu_2)(x) \\
&= \mathcal{T}_{d}(\mu_1,\mu_2) \qquad\mbox{since f is optimal.}
\end{align*}
\noindent Therefore we obtain:
\begin{align*}
& \int_{\XX\times\XX} d(x_1,x_2) \diff \rho(x_1,x_2) = \int_{\XX\times\XX} \big(f(x_1)-f(x_2) \big) \diff \rho(x_1,x_2),\\
\mbox{hence} \qquad & \int_{\XX} \Big( d(x_1,x_2)- \big( f(x_1)-f(x_2) \big)\Big) \diff \rho(x_1,x_2) = 0.
\end{align*}
Since $f$ is a 1-Lipschitz function, $d(x_1,x_2)- \big( f(x_1)-f(x_2) \big) \geq 0$. Thus $d(x_1,x_2) = f(x_1) - f(x_2)$ for all $(x_1,x_2)$, $\rho$-a.e. 
In fact, $d(x_1,x_2) = f(x_1) - f(x_2)$ for all $(x_1,x_2) \in \supprt \rho$. Indeed suppose there exists $(x^*_1,x^*_2) \in \supprt \rho$ such that $g(x^*_1,x^*_2) = d(x_1,x_2)- \big( f(x_1)-f(x_2) > 0$. Then, there exists an open neighbourhood $\UU$ of $(x_1^*,x_2^*)$ such that $g(x_1,x_2) > g(x^*_1,x^*_2) / 2, \, \forall (x_1,x_2) \in \UU$. Thus, 
\begin{equation*}
\int_{\UU} g(x_1,x_2) \diff \rho(x_1,x_2) > \frac{g(x^*_1,x^*_2)}{2} \rho(\UU) > 0,
\end{equation*}
which is a contradiction.

Conversely, suppose that $f$ is a 1-Lipschitz function such that, \\
for all $(x_1,x_2) \in \supprt \rho$, $d(x_1,x_2) = f(x_1)-f(x_2)$ . Then, for $\vartheta \in \NN(\mu_1, \mu_2)$,
\begin{align*}
\mathcal{T}_{d}(\mu_1,\mu_2) & \leq \int_{\XX\times\XX} d(x_1,x_2) \diff\vartheta(x_1,x_2) \qquad \mbox{ by definition of $\mathcal{T}_d$,}\\
&= \int_{\XX\times\XX} \big(f(x_1) - f(x_2) \big) \diff \vartheta(x_1,x_2) \\
&= \int_{\XX\times\XX} f(x_1) \diff \vartheta(x_1,x_2) - \int_{\XX\times\XX} f(x_2) \diff \vartheta(x_1,x_2) \\
&= \int_{\XX\times\XX} f(x_1) \diff \mu_1(x_1) - \int_{\XX\times\XX} f(x_2) \diff \mu_2(x_2) \\
&= \int_{\XX} f(x) \diff (\mu_1-\mu_2)(x) \\
& \leq \mathcal{T}_{d}(\mu_1,\mu_2) \qquad \mbox{ by the K-R Theorem \ref{K-R_Duality_thm_Villani_version}. }
\end{align*}

Hence we obtain:
\begin{equation*}
\mathcal{T}_{d}(\mu_1,\mu_2) = \int_{\XX\times\XX} d(x_1,x_2) \diff\vartheta(x_1,x_2) \quad \mbox{and} \quad \mathcal{T}_{d}(\mu_1,\mu_2) = \int_{\XX} f(x) d(\mu_1-\mu_2)(x).
\end{equation*}

Therefore, $\vartheta$ is an optimal measure on $\XX\times\XX$ and $f$ is an optimal function on $\XX$.
\end{proof}

\begin{rmk}
Proposition \ref{OptimalityKR} is still valid if $d$ is not a bounded distance, but if $\mu_1, \mu_2 \in P_{\XX}$, ie. for some $x_0 \in \XX$ (and therefore for any $x_0$), $\int d(x,x_0) \diff \mu_i(x) < \infty$. 

Indeed, keeping the notation of \ref{OptimalityKR}, we have in this case:
\begin{equation*}
\int d(x,y) \diff \vartheta(x,y) \leq \int d(x,x_0) \diff \mu_1(x) + \int d(x_0,y) \diff \mu_2(x) < \infty
\end{equation*}
Since $f$ is 1-Lipschitz,  $f$ therefore belongs to $L_1(\mu_i)$ as 
\begin{equation*}
\int |f(x)| \diff \mu_i(x) \leq \int |f(x) - f(x_0)| \diff \mu_i(x) + |f(x_0)| \leq \int d(x,x_0) \diff \mu_i(x) + |f(x_0)| < \infty.
\end{equation*}
Then, $f \in L_1(|\mu_1 - \mu_2|)$.                 
\end{rmk}


\cleardoublepage

\chapter{The Kantorovich-Rubinstein Distance} \label{Chapitre_KR_distance}

In the previous chapter, we introduced the functional $\mathrm{I}_c$ defined by
\begin{equation*}
\mathrm{I}_c (\vartheta)= \int_{\XX\times\YY} c(x,y) \diff \vartheta(x,y),
\end{equation*}
where $c$ is the cost function. When the cost function is defined in terms of a distance on $\XX$, it allows to define a family of distances between the measures $\mu$ and $\nu$, called the Wasserstein distances. We briefly introduce the Wasserstein distances and their convergence and topological properties before focusing on a particular case called the Kantorovich-Rubinstein distance. Using the Kantorovich-Rubinstein theorem \ref{K-R_Duality_thm_Villani_version}, it is possible for some particular pairs of Polish spaces and distances on that space to construct explicit formulas for the Kantorovich-Rubinstein distance. In this chapter, we construct the explicit formula for the Kantorovich-Rubinstein distance for three different pairs of Polish space and distance: any Polish space equipped with the discrete distance, the real line equipped with the Euclidean distance and the circle $\SSS_1$.

\section{The Wasserstein Distance}

\begin{defn}[Wasserstein space]\label{Def_Wasserstein_space}
Let $\XX$ be a Polish space with $d \in \DD_{\XX}$ and $P(\XX)$ be the space of Borel probability measures on $\XX$. The Wasserstein space of order $p$ ($p>1$) is defined as 
\begin{equation}
P_p(\XX) := \big\{ \mu \in P(\mathcal{X});\, \int_{\mathcal{X}}d(x_0,x)^p \diff\mu < +\infty \big\}, 
\end{equation} 
where $x_0\in\mathcal{X}$ is arbitrary. This space does not depend on the choice of the point $x_0$.
\end{defn}

\begin{defn}\label{wasserstein_distance}
Let $\XX$ be a Polish space with $d \in \DD_{\XX}$ and let $P_p(\XX)$ be the Wasserstein space associated to $\XX$. For any two probability measures $\mu_1, \mu_2 \in P_p(\XX)$, the Wasserstein function of order $p$ ($p>1$) between $\mu_1$ and $\mu_2$ is defined by the formula
\begin{equation*}
W_p(\mu_1,\mu_2) = \Big(\inf_{\nu\in \NN(\mu_1,\mu_2)} \int_{\XX} d(x,y)^p \diff\nu(x,y)\Big)^{\frac1p}.
\end{equation*} 

In probability theory, one would write:
\begin{equation*}
W_p(\mu_1,\mu_2) = \inf \big\{ [\mathbb{E} \, d(X,Y)^p]^{\frac1p},\, \law(X)=\mu_1, \law(Y)=\mu_2\big\}.
\end{equation*}
where the infimum is taken over all possible coupling $(X,Y)$ of $\mu$ and $\nu$.
\end{defn}

\begin{prop}[Wasserstein distance]\label{prop_wasserstein_distance}
Let $\XX$ be a Polish space with $d \in \DD_{\XX}$ and $P(\XX)$ be the space of Borel probability measures on $\XX$. Then, the Wasserstein function $W_p$ defines a finite distance on the Wasserstein space $P_p(\XX)$. The distance $W_p$ is called the \textit{Wasserstein distance}.
\end{prop}

\begin{proof}[Proposition \ref{prop_wasserstein_distance}] 
Let $\mu_1, \mu_2 \in P_p(\XX)$ and $\nu \in \NN(\mu_1,\mu_2)$. We need to show that $W_p$ satisfies the axioms of a distance and that $W_p$ is finite on $P_p(\XX)$. To prove that $W_p$ satisfies the axioms of a distance, one can consult Villani (\cite{VillaniOptimalTrans}, p.106). We give the proof that $W_p$ is finite on $P_p(\XX)$.

\noindent For $p \in [1,\infty)$, the function $f(x)=x^p$ is convex. By the definition of convexity, the inequality $\big( t x_1+(1-t)x_2\big)^t \leq t x_1^p + (1-t) x_2^p$ holds for $x_1,x_2 \in \RR$ and $t \in [0,1]$. Thus, for $x_0 \in \RR$, and $t = 1/2$, we obtain:
\begin{equation*}
d(x,y)^p \leq 2^{p-1}[d(x,x_0)^p + d(x_0,y)^p].
\end{equation*}
Since $\mu_1, \mu_2 \in P_p(\XX)$, we therefore have
\begin{equation*}
\int_{\XX} d(x,y)^p \diff\nu(x,y) \leq 2^{p-1} \int_{\XX} d(x,x_0)^p \diff\mu_1(x) + \int_{\XX} d(x_0,y)^p \diff\mu_2(y) < \infty,
\end{equation*}
which completes the proof.
\end{proof}

\subsection{Convergence in Wasserstein sense}

Now we shall define a type of convergence on the Wasserstein space. \\
Recall from Definition \ref{def_weak_convergence_non_Polish} that, for Borel probabililty measures, $\mu_k \rightarrow \mu$ means that $\mu_k$ converges weakly to $\mu$, i.e. $\dps \int \varphi \diff \mu_k \rightarrow \int \varphi \diff \mu$ for any bounded continuous $\varphi$ on $\XX$. 

\begin{defn}\label{def_weak_convergence_Polish}[Weak Convergence on $P_p(\XX)$]
Let $\XX$ be a Polish space with $d \in \DD_{\XX}$ and $P_p(\XX)$ its Wasserstein space of order $p$. A sequence of probability measures $(\mu_k)_{k\in\NNN} \in P_p(\XX)$, is said to converge weakly to $\mu \in P_p(\XX)$ (denoted by $\mu_k \Rightarrow \mu$) if 
\begin{equation*}
\mu_k \longrightarrow \mu \quad \mbox{and} \quad \int d(x_0,x)^p \diff \mu_k \rightarrow \int d(x_0,x)^p \diff \mu,
\end{equation*}
for some (and then any) $x_0 \in \XX$.
\end{defn}

\begin{theo}\label{equiv_for_weak_convergence_in_P(X)}
Let $\XX$ be a Polish space with $d \in \DD_{\XX}$ and $p\in [1,\infty)$. Let $(\mu_k)_{k\in\mathbb{N}}$ be a sequence of probability measures in $P_p(\XX)$ and let $\mu$ be another element of $P_p(\XX)$. Then, for any $x_0 \in \XX$, the following three statements are equivalent to the definition of weak convergence \ref{def_weak_convergence_Polish}: 
\begin{enumerate}
\item[(i)] $\mu_k\longrightarrow \mu$ and $\dps \limsup_{k\rightarrow \infty} \int d(x_0,x)^p \diff\mu_k(x) \leq \int d(x_0,x)^p \diff\mu(x)$;
\item[(ii)] $\mu_k\longrightarrow \mu$ and $\dps \lim_{R\rightarrow\infty} \limsup_{k\rightarrow\infty} \int_{d(x_0,x)\geq R}d(x_0,x)^p \diff\mu_k(x) = 0$;
\item[(iii)] For all continuous functions $\varphi$ with $|\varphi(x)|\leq C(1+d(x_0,x)^p)$, $C\in \RR$, one has
\begin{equation*}
\int \varphi(x) \diff \mu_k(x) \rightarrow \int \varphi(x) \diff \mu(x).
\end{equation*}
\end{enumerate}
\end{theo}

\begin{theo}[$W_p$ metrizes $P_p$]
Let $\XX$ be a Polish space with $d \in \DD_{\XX}$ and $p\in [1,\infty)$. Then the Wasserstein distance $W_p$ metrizes the weak convergence in $P_p(\XX)$. In other words, if $(\mu_k)_{k\in\mathbb{N}}$ is a  sequence of measures in $P_p(\mathcal{X})$ and $\mu$ is another measure in $P(\mathcal{X})$, then the following two statements are equivalent:
\begin{enumerate}
\item[(i)] $\mu_k$ converges weakly in $P_p(\XX)$ to $\mu$;
\item[(ii)] $W_p(\mu_k,\mu) \rightarrow 0$.
\end{enumerate}
\end{theo}

\begin{proof}
For a proof of the $W_p$ \textit{metrizes} $P_p$ theorem, one can consult Villani (\cite{VillaniOptimalTrans}, p.113).
\end{proof}

\begin{cor}[Continuity of $W_p$]
Let $\XX$ be a Polish space with $d \in \DD_{\XX}$ and $p\in [1,\infty)$. Then the Wasserstein distance $W_p$ is continuous on $P_p(\XX)$. More explicitly, if the sequence $(\mu_k)_{k\in\NNN}$ (respectively $(\nu_k)_{k\in\NNN}$) converges weakly to $\mu$ (respectively $\nu$) in $P_p(\XX)$ then $W_p(\mu_k,\nu_k)\rightarrow W_p(\mu,\nu)$.
\end{cor}

\subsection{Topological properties of the Wasserstein space}

The Wasserstein space $P_p(\XX)$ inherits several properties of the base space $\XX$ . Here is a first illustration:

\begin{theo}[Topology of the Wasserstein space]\label{Topology of the Wasserstein space}
Let $\mathcal{X}$ be a Polish space and $p\in[1,\infty)$. Then the Wasserstein space $P_p(\mathcal{X})$, equipped with the Wasserstein distance $W_p$, is also a Polish space. In short, the Wasserstein space over a Polish space is itself a Polish space. Moreover, any probability measure can be approximated by a sequence of probability measures with finite support.
\end{theo}

\begin{rmk}
If $\mathcal{X}$ is compact, then $P_p(\mathcal{X})$ is also compact.
\end{rmk}

\begin{proof}[Theorem \ref{Topology of the Wasserstein space}]
For a proof of the \textit{Topology of the Wasserstein space} theorem, one can consult Villani (\cite{VillaniOptimalTrans}, p.117).
\end{proof}

\section{The Kantorovich-Rubinstein distance}\label{section_on_KR_distance}


For the particular case with $p=1$, the Wasserstein distance is commonly called the Kantorovich-Rubinstein distance. In this thesis, we will use the Kantorovich-Rubinstein distance extensively. To make it easier to read, the notation $W_{\XX}$ will be used instead of $W_1$ for the Kantorovich-Rubinstein distance while the notation $P_{\XX}$ will be used instead of $P_1(\XX)$ for the Kantorovich-Rubinstein space. Recall that $P(\XX)$ represents the set of all Borel measures on $\XX$. We obtain the following definitions for the Kantorovich-Rubinstein distance and space:

\begin{defn}[Kantorovich-Rubinstein space]\label{Def_Kantorovich-Rubinstein_space}
Let $\XX$ be a Polish space with $d \in \DD_{\XX}$ and $P(\XX)$ be the space of Borel probability measures on $\XX$. The Kantorovich-Rubinstein space is defined as 
\begin{equation*}
P_{\XX} = \big\{ \mu \in P(\XX);\, \int_{\XX} d(x_0,x) \diff\mu(x) < +\infty \big\}, 
\end{equation*}
where $x_0 \in \XX$ is arbitrary. This space does not depend on the choice of the point $x_0$.
\end{defn}

\begin{defn}[Kantorovich-Rubinstein Distance]\label{Kantorovich-Rubinstein_distance}
Let $\XX$ be a Polish space with $d \in \DD_{\XX}$ and let $P_{\XX}$ be the Kantorovich-Rubinstein space associated to $\XX$. For any two probability measures $\mu_1, \mu_2 \in P_{\XX}$, the Kantorovich-Rubinstein distance between $\mu_1$ and $\mu_2$ is defined by the formula
\begin{equation*}
W_{\XX}(\mu_1,\mu_2) = \inf_{\vartheta \in \NN(\mu_1,\mu_2)} \int_{\mathcal{X}} d(x,y) \diff\vartheta(x,y),
\end{equation*}
where $\NN(\mu_1,\mu_2)$ denotes the set of couplings of $\mu_1$ and $\mu_2$ (see Definition \ref{Coupling}).
\end{defn}

\begin{rmk}\label{remark_on_defintion_of_KR_distance} 
There are two remarks worth making:
\begin{enumerate}
\item[(i)]The Kantorovich-Rubinstein theorem leads to the following useful duality formula for the Kantorovich-Rubinstein distance:\\
For any $\mu_1, \mu_2 \in P_{\XX}$, we have 
\begin{equation}
W_{\XX}(\mu_1,\mu_2) = \sup \left\{\int_{\XX} \psi\diff(\mu_1-\mu_2); \,\psi \in L_1(|\mu_1-\mu_2|), \, \psi \mbox{ 1-Lipschitz }\right\}.
\end{equation}
\item[(ii)]Using the classical probability definition of the Monge-Kantorovich minimization problem, we can write the Kantorovich-Rubinstein distance as follow:
\begin{equation*}
W_{\XX}(\mu_1,\mu_2) = \inf \big\{ \mathbb{E} \, d(X,Y),\, \law(X)=\mu_1, \law(Y)=\mu_2 \big\}.
\end{equation*}
where the infimum is taken over all possible coupling $(X,Y)$ of $\mu$ and $\nu$.
\item[(iii)] Using the same notations as in Definition \ref{Def_Kantorovich-Rubinstein_space}, we define
\begin{equation*}
M_{\XX} = \big\{ \mu \in M_+(\XX);\, \int_{\XX} d(x_0,x) \diff\mu(x) < +\infty \big\},
\end{equation*}
where $M_+(\XX)$ is the space of finite positive Borel measures on $\XX$. \\
The Definition \ref{Kantorovich-Rubinstein_distance} of $W_{\XX}$ can be generalized to the case of $\mu, \nu \in M_{\XX}$ such that $\mu(\XX) = \nu(\XX)$, since, as mentioned in Remark \ref{generalization_of_couplings}, the definition of a coupling is generalizable to two bounded measures.
\end{enumerate}
\end{rmk}

An important property of the Kantorovich-Rubinstein distance is its invariance under mass subtraction. This property appears in Corollary 1.16 of Villani \cite{VillaniTopicsOptimalTransportation} and is given below. 

\begin{prop}[Invariance of Kantorovich-Rubinstein distance under mass subtraction]\label{Invariance_of_KR_under_sub}
Let $\XX$ be a Polish space with $d \in \DD_{\XX}$. Let $\mu_1$, $\mu_2$ and $\nu$ be three measures in $M_{\XX}$ such that $\mu_1(\XX) = \mu_2(\XX)$. Then, 
\begin{equation*}
W_{\XX}(\mu_1+\nu, \mu_2+\nu) = W_{\XX}(\mu_1,\mu_2).
\end{equation*}
\end{prop}

A Corollary of the invariance of the Kantorovich-Rubinstein distance under mass subtraction is presented below:  

\begin{cor}\label{removing_middle_of_measurable_space_thm}
Let $(\XX, d^*_{\XX})$ be a Polish space such that $\XX = \XX_1 \amalg \XX_2$. Let $\mu_1$ and $\mu_2$ be two probability measures on $\XX$ such that $\mu_1(\XX_i) = \mu_2(\XX_i)$, for $i = 1,2$. Define by $\mu^{(1)}_i$ and $\mu^{(2)}_i$ the respective measures of $\mu_1$ and $\mu_2$ supported on $\XX_i$, for $i = 1,2$. That is, $\mu_i^{(1)} = \mu_1 \, _{ \big| \XX_i}$ and $\mu_i^{(2)} = \mu_2 \, _{ \big| \XX_i}$. Then,
\begin{equation*}
|W_{\XX}(\mu_1,\mu_2) - W_{\XX}(\mu_1^{(1)}, \mu_1^{(2)})| \leq W_{\XX}(\mu_2^{(1)}, \mu_2^{(2)}).
\end{equation*}
\end{cor}

To prove Corollary \ref{removing_middle_of_measurable_space_thm}, we will need the following Lemma \ref{KR_of sum_smaller_than_sum_of_KR}:

\begin{lem}\label{KR_of sum_smaller_than_sum_of_KR}
Let $(\XX, d^*_{\XX})$ be a Polish space such that $\XX = \amalg_{i=1}^n \XX_i$. Let $\mu_1$ and $\mu_2$ be two probability measures on $\XX$ such that $\mu_1(\XX_i) = \mu_2(\XX_i)$, for $i = 1, \ldots, n$. Define by $\mu^{(1)}_i$ and $\mu^{(2)}_i$ the respective measures of $\mu_1$ and $\mu_2$ supported on $\XX_i$, for $i = 1, \ldots, n$. That is, $\mu_i^{(1)} = \mu_1 \, _{ \big| \XX_i}$ and $\mu_i^{(2)} = \mu_2 \, _{ \big| \XX_i}$.  Then,
\begin{equation*}
W_{\XX}(\mu_1, \mu_2) \leq \sum_{i = 1}^{n} W_{\XX} \big(\mu_i^{(1)}, \mu_i^{(2)} \big).
\end{equation*}
\end{lem}

\begin{proof}[Lemma \ref{KR_of sum_smaller_than_sum_of_KR}]
The proof is done by induction. For $1 \leq k \leq n$ fixed, 
\begin{align}
&W_{\XX} \Big(\sum_{i = 1}^{k-1} \mu_i^{(i)} + \mu_k^{(1)}, \sum_{i=1}^{k-1} \mu_i^{(2)} + \mu_k^{(2)} \Big) \nonumber\\
&\leq W_{\XX} \Big(\sum_{i = 1}^{k-1} \mu_i^{(1)} + \mu_k^{(1)}, \sum_{i=1}^{k-1} \mu_i^{(2)} + \mu_k^{(1)} \Big) + W_{\XX} \Big(\sum_{i = 1}^{k-1} \mu_i^{(2)} + \mu_k^{(1)}, \sum_{i=1}^{k-1} \mu_i^{(2)} + \mu_k^{(2)} \Big) \label{equation1}\\
& \leq W_{\XX} \Big(\sum_{i = 1}^{k-1} \mu_i^{(1)}, \sum_{i=1}^{k-1} \mu_i^{(2)} \Big) + W_{\XX} \big( \mu_k^{(1)}, \mu_k^{(2)}\big) \label{equation2}.
\end{align}
Inequality (\ref{equation1}) is obtained by the triangle inequality and inequality (\ref{equation2}) is obtained by direct application of Corollary \ref{Invariance_of_KR_under_sub}.
\end{proof}

We can now write the proof of Corollary \ref{removing_middle_of_measurable_space_thm}

\begin{proof}[Corollary \ref{removing_middle_of_measurable_space_thm}]
Applying Lemma \ref{KR_of sum_smaller_than_sum_of_KR}, for $n=2$, we obtain 
\begin{equation*}
W_{\XX}(\mu_1, \mu_2) \leq W_{\XX}\big(\mu_1^{(1)}, \mu_1^{(2)} \big) + W_{\XX}\big(\mu_2^{(1)}, \mu_2^{(2)} \big)
\end{equation*}
and thus $\dps W_{\XX}(\mu_1, \mu_2) - W_{\XX}\big(\mu_1^{(1)}, \mu_1^{(2)} \big) \leq W_{\XX}\big(\mu_2^{(1)}, \mu_2^{(2)} \big).$ \\

\noindent It is therefore left to show that $\dps W_{\XX}\big(\mu_1^{(1)}, \mu_1^{(2)} \big) - W_{\XX}(\mu_1, \mu_2) \leq W_{\XX}\big(\mu_2^{(1)}, \mu_2^{(2)} \big).$

\noindent By Proposition \ref{Invariance_of_KR_under_sub}, we obtain:
\begin{align*}
2 W_{\XX}\big(\mu_1^{(1)}, \mu_1^{(2)} \big) &= W_{\XX} \big(\mu_1^{(1)} + \mu_2^{(1)}, \mu_1^{(2)} + \mu_2^{(1)}\big) + W_{\XX} \big(\mu_1^{(1)} + \mu_2^{(2)}, \mu_1^{(2)} + \mu_2^{(2)}\big) \\
&= W_{\XX}\big(\mu^{(1)}, \mu_1^{(2)} + \mu_2^{(1)}\big) + W_{\XX} \big( \mu_1^{(1)} + \mu_2^{(2)}, \mu^{(2)}\big).
\end{align*}
Using the triangle inequality, we have the following two inequalities:
\begin{align*}
W_{\XX}\big( \mu^{(1)}, \mu_1^{(2)} + \mu_2^{(1)} \big) &\leq W_{\XX}\big(\mu^{(1)}, \mu^{(2)}\big) + W_{\XX}\big( \mu_1^{(2)} + \mu_2^{(2)}, \mu_1^{(2)} + \mu_2^{(1)} \big) \\
W_{\XX}\big( \mu_1^{(1)} + \mu_2^{(2)},\mu^{(2)} \big) &\leq W_{\XX}\big( \mu_1^{(1)} + \mu_2^{(2)}, \mu_1^{(1)}+\mu_2^{(1)} \big) + W_{\XX}\big( \mu^{(1)}, \mu^{(2)} \big).
\end{align*}
By Proposition \ref{Invariance_of_KR_under_sub}, we can write:
\begin{align*}
W_{\XX}\big( \mu_1^{(2)} + \mu_2^{(2)}, \mu_1^{(2)} + \mu_2^{(1)} \big) &= W_{\XX}\big( \mu_2^{(2)}, \mu_2^{(1)} \big) \\
\mbox{and } \, W_{\XX}\big( \mu_1^{(1)} + \mu_2^{(2)}, \mu_1^{(1)} + \mu_2^{(1)} \big) &= W_{\XX}\big( \mu_2^{(2)}, \mu_2^{(1)} \big).
\end{align*}
Therefore we obtain 
\begin{equation*}
W_{\XX}\big( \mu_1^{(1)}, \mu_1^{(2)} \big) \leq W_{\XX}\big( \mu^{(1)}, \mu^{(2)} \big) + W_{\XX}\big( \mu_2^{(2)}, \mu_2^{(1)} \big).
\end{equation*}
\end{proof}
 
\subsection{The Kantorovich-Rubinstein distance between product measures}

In this section, we study the additivity of the Kantorovich-Rubinstein distance for product measures.

\begin{theo}\label{thm_mesure_produit}
Let $\XX$ and $\YY$ be two Polish spaces and $d_{\XX} \in \DD_{\XX}$ and $d_{\YY} \in \DD_{\YY}$ be bounded lower semi-continuous distances. Let $\XX \times \YY$ be the Polish cartesian product space, equipped with the distance $d$, given by 
\begin{equation*}
d((x_1,y_1),(x_2,y_2)) = d_{\XX}(x_1,x_2) + d_{\YY}(y_1,y_2), \, \mbox{ for }\, (x_1,y_1),(x_2,y_2) \in \XX \times \YY.
\end{equation*}
Let $\mu_1,\mu_2$ be two probability measures in $P_{\XX}$ and $\nu_1,\nu_2$ two probability measures in $P_{\YY}$. Then, $\mu_1 \otimes \nu_1$, $\mu_2 \otimes \nu_2$ are probability measures in $P_{\XX \times \YY}$ and 
\begin{equation*}
W_{\XX\times\YY}(\mu_1 \otimes \nu_1, \mu_2 \otimes \nu_2) = W_{\XX} (\mu_1,\mu_2) + W_{\YY}(\nu_1,\nu_2).
\end{equation*}
\end{theo}
 
\begin{proof}[Theorem \ref{thm_mesure_produit}]
Let $(x_{\circ}, y_{\circ}) \in \XX\times\YY$ be fixed. Then, by definition of $d$, 
\begin{align*}
\int_{\XX\times\YY} d((x_{\circ},y_{\circ}),(x,y)) \diff(\mu_i\otimes\nu_i)(x,y) &= \int_{\XX\times\YY} (d_{\XX}(x_{\circ},x) + d_{\YY}(y_{\circ},y)) \diff\mu_i(x) \diff\nu_i(y)\\
&= \int_{\XX} d(x_{\circ},x) \diff \mu_i(x) + \int_{\YY} (y_{\circ},y) \diff \nu_i(y) < \infty, 
\end{align*}
as $\mu_i \in P_{\XX}, \, \nu_i \in P_{\YY}$. Hence, $\mu_i \otimes \nu_i \in P_{\XX \times \YY}$, for $i =1,2$.\\

\noindent Now, let us show that $W_{\XX\times\YY}(\mu_1 \otimes \nu_1, \mu_2 \otimes \nu_2) \geq W_{\XX} (\mu_1,\mu_2) + W_{\YY}(\nu_1,\nu_2)$. There exist, by the Kantorovich-Rubinstein theorem \ref{K-R_Duality_thm_Villani_version}, two bounded measurable, 1-Lipschitz functions $f: \XX \rightarrow \RR$ and $g: \YY \rightarrow \RR$ such that 
\begin{equation*}
W_{\XX}(\mu_1,\mu_2)-\frac{\epsilon}{2} < \int_{\XX} f \diff(\mu_1-\mu_2) \, \mbox{ and } \, W_{\YY}(\nu_1,\nu_2)-\frac{\epsilon}{2} < \int_{\XX} g \diff(\nu_1-\nu_2).
\end{equation*}
Denote by $k: \XX \times \YY \rightarrow \RR$ the function defined by $k(x,y) = f(x) + g(y)$. The function $k$ is bounded  and 1-Lipschitz as 
\begin{align*}
|k(x_1,y_1) - k(x_2,y_2)| &\leq |f(x_1) - f(x_2)| + |g(y_1) - g(y_2)| \\
&\leq d_{\XX}(x_1,x_2) + d_{\YY}(y_1,y_2) = d((x_1,y_1),(x_2,y_2)), \, \mbox{ for } (x_i,y_i) \in \XX\times\YY.
\end{align*}
By the Kantorovich-Rubinstein theorem \ref{K-R_Duality_thm_Villani_version},
\begin{equation*}
\int_{\XX\times\YY} k(x,y) \diff(\mu_1\otimes\nu_1 - \mu_2\otimes\nu_2)(x,y) \leq W_{\XX\times\YY}(\mu_1\otimes\nu_1, \mu_2\otimes\nu_2).
\end{equation*}
By the definition of $k$, we have:
\begin{align*} \int_{\XX\times\YY} k(x,y) \diff(\mu_1\otimes\nu_1 - \mu_2\otimes\nu_2)(x,y)
&= \int_{\XX\times\YY} (f(x) + g(y)) \diff\mu_1(x)\diff\nu_1(y) - \diff\mu_2(x) \diff\nu_2(y) \\
&= \int_{\XX} f(x) (\diff\mu_1 - \diff \mu_2)(x) + \int_{\YY} g(y) (\diff \nu_1 - \diff\nu_2)(y) \\
& \geq W_{\XX}(\mu_1,\mu_2) + W_{\YY}(\nu_1, \nu_2) - \epsilon.
\end{align*}
As $\epsilon$ is arbitrary, we have $\,W_{\XX}(\mu_1,\mu_2) + W_{\YY}(\nu_1,\nu_2) \leq W_{\XX\times\YY}(\mu_1\otimes\nu_1, \mu_2\otimes\nu_2).$\\

\noindent For the converse inequality, we show first that $W_{\XX\times\YY}(\mu_1\otimes\nu_1, \mu_2\otimes\nu_1) \geq W_{\XX}(\mu_1,\mu_2).$
For every $\epsilon>0$, there exists $h:\XX\times\YY \rightarrow \RR$, a bounded 1-Lipschitz function such that 
\begin{equation*}
W_{\XX\times\YY}(\mu_1\otimes\nu_1, \mu_2\otimes\nu_1) - \epsilon < \int_{\XX\times\YY} h \diff(\mu_1\otimes\nu_1 -\mu_2\otimes\nu_1).
\end{equation*}
Using Fubini Theorem, we have
\begin{align*}
\int_{\XX\times\YY} h \diff(\mu_1\otimes\nu_1 -\mu_2\otimes\nu_1)  =& \int_{\XX\times\YY} h(x,y) (\diff\mu_1(x) - \diff\mu_2(x)) \diff\nu_1(y)\\
=& \int_{\XX} \left( \int_{\YY} h(x,y) \diff\nu_1(y)\right) \diff(\mu_1(x) - \mu_2(x)).
\end{align*}
Set $\dps  t(x) = \int_{\YY} h(x,y) \diff\nu_1(y)$. Then, $t$ is bounded and 1-Lipschitz as, for all $x_1, x_2 \in \XX$, 
\begin{equation*}
|t(x_1)- t(x_2)| \leq \int_{\YY} |h(x_1,y) - h(x_2,y)| \diff \nu_1(y) \leq d((x_1,y),(x_2,y)) = d_{\XX}(x_1,x_2).
\end{equation*}

\noindent Thus, by Kantorovich-Rubinstein theorem \ref{K-R_Duality_thm_Villani_version}: $\,W_{\XX\times\YY}(\mu_1\otimes\nu_1, \mu_2\otimes\nu_1) \leq W_{\XX}(\mu_1,\mu_2).$


\noindent Similarly, we obtain $W_{\XX\times\YY}(\mu_2\otimes\nu_1, \mu_2\otimes\nu_2)  \leq W_{\YY}(\nu_1,\nu_2)$. Hence, 
\begin{align*}
W_{\XX\times\YY}(\mu_1 \otimes\nu_1, \mu_2\otimes\nu_2) &\leq W_{\XX\times\YY}(\mu_1 \otimes\nu_1, \mu_2\otimes\nu_1) + W_{\XX\times\YY}(\mu_2 \otimes\nu_1, \mu_2\otimes\nu_2) \\
&\leq W_{\XX}(\mu_1,\mu_2) + W_{\YY}(\nu_1,\nu_2).
\end{align*}
Thus completing the proof.
\end{proof}

\subsection{The Kantorovich-Rubinstein distance for atomic measures}

In the case of atomic measures supported on a finite number of points, we can use the Kantorovich-Rubinstein Theorem \ref{K-R_Duality_thm_Villani_version_atomic_measures} for atomic measures. We obtain the following definition:

\begin{defn}[Kantorovich-Rubinstein distance for atomic measures]\label{Kantorovich-Rubinstein_distance_atomic_measures}
Let $\XX$ be a Polish space with $d \in \DD_{\XX}$. Let us define two atomic probability measures
\begin{equation*}
\mu_1 = \sum_{i = 1}^{n} \mu_i^{(1)} \delta_{x_i} \quad \mbox{and} \quad \mu_2 = \sum_{i = 1}^{n} \mu_i^{(2)} \delta_{x_i},
\end{equation*}
supported on the same finite number of points $\{x_1, \ldots , x_n\} \in \XX$. 

The Kantorovich-Rubinstein distance between $\mu_1$ and $\mu_2$ is then given by 
\begin{equation*}
W_{\XX}(\mu_1,\mu_2) =  \sup \Big\{ \sum_{i=1}^n g(x_i) (\mu_i^{(1)} - \mu_i^{(2)}) :|g(x_i) - g(x_j)| \leq d(x_i ,x_j), 1 \leq i, j \leq n \Big\}.
\end{equation*}
\end{defn}

\section{The Kantorovich-Rubinstein distance for a bounded distance on $\XX$}

Let $\XX$ be a Polish space and $d \in \DD_{\XX}$ be a bounded lower semi-continuous distance on $\XX$. We will denote by $\Delta = \sup \{ d(x,y); x,y \in \XX\}$ the $d$-diameter of $\XX$.\\
We state and, for completeness, give the proof of Remark 1.15 in \cite{VillaniTopicsOptimalTransportation}:

\begin{theo}\label{KR_is_sup_over_F_delta}
Let $\XX$ be a Polish space and $d \in \DD_{\XX}$ be a bounded lower semi-continuous distance on $\XX$. Let $\mu_1$ and $\mu_2$ be two probability measures on $\XX$. Let $\Delta$ be the $d$-diameter of $\XX$ and define $\FF_{\Delta}$ by
\begin{equation*}
\FF_{\Delta} = \left\{ f: \XX \rightarrow [0,\Delta] ;\, f \mbox{ measurable and 1-Lipschitz} \right\}.
\end{equation*}
Then, the Kantorovich-Rubinstein distance can be written as 
\begin{equation*}
W_{\XX}(\mu_1,\mu_2) = \sup \left\{ \int f \diff (\mu_1-\mu_2) ;\, f \in \FF_{\Delta} \right\}.
\end{equation*}
\end{theo}

\begin{rmk} There are two facts worth mentionning:
\begin{enumerate}
\item[(i)] Recall from Remark \ref{distance_doesnt_define_topology} that the distance $d$ on $\XX$ need not be the distance defining the topology on $\XX$. In Theorem \ref{KR_dist_equal_TV_for_discrete_dist}, the $1$-discrete distance does not define the topology on $\XX$. If it was, $\XX$ would not be a Polish space as it is not separable with the $1$-discrete distance.
\item[(ii)] The Lipschitz property for functions on $\XX$ is defined with respect to the $k$-discrete distance. 
\end{enumerate}
\end{rmk}

Before giving the proof of Theorem \ref{KR_is_sup_over_F_delta}, let us state the following short result we will use often in the rest of the thesis.

\begin{lem}\label{Image_of_lip_fct_is_smaller_diamX}
Let $(\XX,d_{\XX})$ and $(\YY,d_{\YY})$ be two metric spaces and $\varphi: \XX \rightarrow \YY$ be a 1-Lipschitz map.
Then, $\diam(\varphi(\XX)) \leq \diam(\XX)$.
\end{lem}

\begin{proof}[Lemma \ref{Image_of_lip_fct_is_smaller_diamX}]
Without loss of generality, we can assume that $\diam(\XX) < \infty$. Then, \\
$\dps \diam(\varphi(\XX)) = \sup_{x_1,x_2 \in \XX} d_{\YY}(\varphi(x_1), \varphi(x_2)) \leq \sup_{x_1,x_2 \in \XX} d_{\XX}(x_1,x_2) = \diam\XX.$
\end{proof}

\begin{proof}[Theorem \ref{KR_is_sup_over_F_delta}]
Let us denote by $\FF$ the set of all measurable, 1-Lipschitz functions $f: \XX \rightarrow \RR$.\\
Since $\Delta < \infty$, by Lemma \ref{Image_of_lip_fct_is_smaller_diamX}, if $f \in \FF$, then $f$ is bounded and therefore belongs to $L_1(|\mu_1-\mu_2|)$.\\
By Theorem \ref{K-R_Duality_thm_Villani_version}, we have 
\begin{equation*}
W_{\XX}(\mu_1,\mu_2) = \sup \left\{ \int f \diff (\mu_1-\mu_2);\, f \in \FF \right\}
\end{equation*}
To finish the proof, it is therefore enough to show that for any $f \in \FF$, there exists $g \in \FF_{\Delta}$ such that 
\begin{equation*}
\int f \diff (\mu_1-\mu_2) = \int g \diff (\mu_1-\mu_2).
\end{equation*}
Let $f \in \FF$. Since $f$ is 1-Lipschitz, and $\Delta < \infty$, then 
\begin{equation*}
-\infty < a = \inf (f(\XX)) \leq \sup f(\XX) = b < \infty,
\end{equation*} 
and $b-a \leq \Delta$. Set $g = f - a$, then $g \in \FF_{\Delta}$ and as 
\begin{equation*}
\int g \diff (\mu_1-\mu_2) = \int (f - a) \diff(\mu_1-\mu_2) = \int f \diff (\mu_1-\mu_2),
\end{equation*}
the proof is complete.
\end{proof}

\noindent Keeping the assumptions and notations of Theorem \ref{KR_is_sup_over_F_delta}, there is, by Theorem \ref{existence_optimal_sol_KR_thm}, an optimal function $g \in \FF$ such that 
\begin{equation*}
W_{\XX}(\mu_1,\mu_2) = \int_{\XX} g \diff (\mu_1- \mu_2).
\end{equation*}
 As seen in the proof of Theorem \ref{KR_is_sup_over_F_delta}, the optimal function $g$ is invariant by translation. Hence we obtain the following result:

\begin{theo}\label{existence_of_positive_bounded_optimal_function}
Let $\XX$ be a Polish space and $d \in \DD_{\XX}$ be a bounded lower semi-continuous distance on $\XX$. Let $\mu_1$ and $\mu_2$ be two probability measures on $\XX$. If $\Delta$ denotes the $d$-diameter of $\XX$, then there exists a function $f: \XX \rightarrow [0, \Delta]$, measurable and 1-Lipschitz such that
\begin{equation*}
W_{\XX}(\mu_1, \mu_2) = \int_{\XX} f \diff (\mu_1-\mu_2).
\end{equation*}
\end{theo}

\noindent Theorem \ref{existence_of_positive_bounded_optimal_function} yields the following corollary:

\begin{cor}\label{existence_of_2_conjugate_positive_bounded_optimal_function}
Let $\XX$ be a Polish space and $d \in \DD_{\XX}$ be a bounded lower semi-continuous distance on $\XX$. Let $\mu_1$ and $\mu_2$ be two probability measures on $\XX$. If $\Delta$ denotes the $d$-diameter of $\XX$, then there exist measurable and 1-Lipschitz functions $f$ and $g$ from $\XX$ to $[0, \Delta]$ such that 
\begin{equation*}
W_{\XX}(\mu_1,\mu_2) = \int_{\XX} f \diff(\mu_1 - \mu_2) = \int_{\XX} g \diff (\mu_2 - \mu_1).
\end{equation*}
\end{cor} 

\begin{proof}[Corollary \ref{existence_of_2_conjugate_positive_bounded_optimal_function}]
By Theorem \ref{existence_of_positive_bounded_optimal_function}, there exists $f: \XX \rightarrow [0, \Delta]$, measurable and 1-Lipschitz such that 
\begin{equation*}
W_{\XX}(\mu_1, \mu_2) = \int_{\XX} f \diff (\mu_1-\mu_2).
\end{equation*}
Then, let $g: \XX \rightarrow [0, \Delta]$ be defined as $g = \Delta - f$. The function $g$ is measurable, 1-Lipschitz and such that 
\begin{equation*}
\int_{\XX} g \diff (\mu_2-\mu_1) = \int_{\XX} (\Delta - f) \diff (\mu_2-\mu_1) = \int_{\XX} (-f) \diff (\mu_2 - \mu_1) = \int_{\XX} f \diff (\mu_1 - \mu_2).
\end{equation*}
\end{proof}

\begin{cor}\label{writing_KR_dist_with_dist_funct}
Let $\XX$ be a Polish space and $d \in \DD_{\XX}$ be a bounded lower semi-continuous distance on $\XX$. Let $\mu_1$ and $\mu_2$ be two probability measures on $\XX$. If $\Delta$ denotes the $d$-diameter of $\XX$, then there exist measurable and 1-Lipschitz functions $f$ and $g$ from $\XX$ to $[0, \Delta]$ such that
\begin{equation*}
W_{\XX}(\mu_1,\mu_2) = \int_0^{\Delta} F_2(y) - F_1(y) \diff(y) = \int_0^{\Delta} G_1(y) - G_2(y) \diff(y),
 \end{equation*}
where $F_i$ and $G_i$ are the respective distribution functions of $f(\mu_i)$ and $g(\mu_i)$.
\end{cor}

\begin{proof}[Corollary \ref{writing_KR_dist_with_dist_funct}]
By Theorem \ref{existence_of_positive_bounded_optimal_function}, there exist a function $f:\XX \rightarrow [0,\Delta]$ such that 
\begin{equation*}
W_{\XX}(\mu_1,\mu_2) = \int_{\XX} f \diff (\mu_1-\mu_2).
\end{equation*}
As written on p.162 of \textit{Measure Theory} \cite{measure_theory_Cohn}, a direct application of Fubini Theorem allows to write (with $\lambda$ denoting the Lebesgue measure on $\RR$ and $E= \{(x,y) \in \XX\times\RR: \, 0\leq y <f(x)\}$):
\begin{align*}
(\mu_i \times \lambda)(E) &= \int_{\XX} f(x) \diff \mu_i(x)\\
&=\int_0^{\infty} \mu_i(\{x \in \XX; f(x)>y\}) \diff y \\
&= \int_0^{\infty} (1-F_i(y)) \diff y.
\end{align*}
Since $f(\XX)\in[0, \Delta]$ and $F_i(y) = 1$ for $y\geq \Delta$, the integrals of $1-F_i(y)$ over $[0,\infty]$ and $[0,\Delta]$ are equal. Therefore, we can write:
 \begin{align*}
W_{\XX}(\mu_1,\mu_2) &= \int_0^{\infty} (1-F_1(y)) - (1-F_2(y)) \diff y \\
 &= \int_0^{\infty} (F_2(y) - F_1(y)) \diff y \\
 &= \int_0^{\Delta} (F_2(y) - F_1(y)) \diff y.
\end{align*}

\noindent Likewise, using Corollary \ref{existence_of_2_conjugate_positive_bounded_optimal_function}, we have a function $g:\XX \rightarrow [0,\Delta]$ such that 
\begin{equation*}
W_{\XX}(\mu_1,\mu_2) = \int_{\XX} g \diff (\mu_2-\mu_1).
\end{equation*}
Repeating the same steps as above, we obtain $\dps \int_0^{\Delta} G_1(y) - G_2(y) \diff(y)$.
\end{proof}

\section{The Kantorovich-Rubinstein distance for the discrete distance on $\XX$}

In this section, we show that for two probability measures the Kantorovich-Rubinstein distance associated to the discrete distance is equal to the total variation distance of these measures. 

Let us first recall the definition of the total variation distance, induced by the total variation norm. The definition and properties of the total variation norm are in the Appendix (see \ref{total_variation_norm_def}). 

\begin{defn}\label{Def_total_variation_distance}[Total Variation distance]
Let $(\XX, \FF)$ be a probability space and $\FF$ be a $\sigma$-algebra of subsets of $\XX$. \\
For any two probability measures $\mu_1$ and $\mu_2$ on $\XX$, the total variation distance between $\mu_1$ and $\mu_2$ is defined by the formula:
\begin{equation*}
||\mu_1-\mu_2||_{TV} = \frac12 |\mu_1-\mu_2|(\XX) = \frac12 (\mu_1-\mu_2)_+(\XX) + \frac12 (\mu_1-\mu_2)_-(\XX),
\end{equation*}
where $(\mu_1-\mu_2)_+$ and $(\mu_1-\mu_2)_-$ are the positive and negative part, respectively, of the signed measure $\mu_1-\mu_2$.
\end{defn}

By definition of the total variation norm of a signed measure (see Definition \ref{total_variation_norm_def}) we have that for two probability measures $\mu_1$ and $\mu_2$,
\begin{equation*}
||\mu_1-\mu_2||_{TV} = \frac12 ||\mu_1-\mu_2||_1.
\end{equation*}
Therefore, by Proposition \ref{prop_of_tv_norm}, we obtain: 

\begin{prop}\label{existence_optimal_function_for_TV}
Let $(\XX, \BB)$ be a measurable space, and $\mu_1,\mu_2$ be two probability measures on $\XX$. Then
\begin{equation*}
||\mu_1-\mu_2||_{TV} = \sup \left\{ \left| \int f \diff (\mu_1-\mu_2) \right|; \, f: \XX \rightarrow [-\frac12; \frac12], measurable \right\}
\end{equation*}
Moreover, there exists $M\in\BB$ such that
\begin{equation*}
||\mu_1-\mu_2||_{TV} = \int (\frac12 \chi_M - \frac12 \chi_{M^{\comp}}) \diff(\mu_1-\mu_2).
\end{equation*}
\end{prop}

Let us now recall that the discrete distance $1_{x \not = y}$ is such that $1_{x \not = y} = 0$ if $x = y$ and $1$ otherwise. Formally, we have:

\begin{defn}[discrete distance]
The discrete distance $1_{x \not = y} : \XX \times \XX \rightarrow \RR$  is the function defined by
 \begin{equation*}
1_{x \not = y}(x,y) = \begin{cases}
0 &\quad\textnormal{ if $x = y$}\\
1 &\quad\textnormal{ if $x \not= y.$}
\end{cases}
\end{equation*}
\end{defn} 

To consider the Kantorovich-Rubinstein distance with respect to the discrete distance, notice the following:

\begin{lem}\label{discrete_distance_lower_semi_continuous}
Let $\XX$ be a Polish space. The $1$-discrete distance $ 1_{x \not = y}$ on $\XX$ is lower semi-continuous.
\end{lem}

\noindent To prove Lemma \ref{discrete_distance_lower_semi_continuous}, we first recall two topological results: 
\begin{enumerate}
\item[(i)] On a topological space $\YY$, a function $f: \YY \rightarrow \RR$ is lower semi-continuous if the set $\{ y \in \YY; \, f(y) \leq \alpha\}$ is closed in $\YY$, for all $\alpha \in \RR$.  
\item[(ii)] The topological space $\YY$ is separable if and only if the diagonal $\Delta$, defined by $\Delta = \{ (x,y) \in \YY \times \YY; \, x = y \}$, is closed in $\YY \times \YY$. See the proof in \cite{Topologie_St_Raymond}.
\end{enumerate}

\begin{proof}[Lemma \ref{discrete_distance_lower_semi_continuous}]
Let $\alpha \in \RR$ and define the set $F_{\alpha}$ by $F_{\alpha} = \{(x,y) \in \XX \times \XX; \, 1_{x\not=y}(x,y) \leq \alpha\}$.\\
If $\alpha < 0$, $F_{\alpha} = \emptyset$. For $0 \leq \alpha < 1$, $F_{\alpha} = \{(x,y) \in \XX \times \XX;  1_{x\not=y}(x,y) = 0\}$. Finally, for $\alpha \geq 1$, $F_{\alpha} = \XX \times \XX$. \\
Since $F_{\alpha}$ is closed $\forall \alpha$, $1_{x \not= y}$ is a lower semi-continuous distance.
\end{proof}

\begin{theo}\label{KR_dist_equal_TV_for_discrete_dist}
Let $\XX$ be a Polish space and let $\mu_1$ and $\mu_2$ be two Borel probability measures on $\XX$. If $W_{\XX}(\mu_1,\mu_2)$ denotes the Kantorovich-Rubinstein with respect to the discrete distance $1_{x \not= y}$ on $\XX$, then 
\begin{equation*}
W_{\XX}(\mu_1,\mu_2) = ||\mu_1-\mu_2||_{TV}.
\end{equation*}
\end{theo}

\begin{proof}[Theorem \ref{KR_dist_equal_TV_for_discrete_dist}]
We first note that any function $f: \XX \rightarrow [0,1]$ is 1-Lipschitz with respect to the discrete distance $1_{x \not= y}$. Therefore, as shown in the proof of Theorem \ref{KR_is_sup_over_F_delta}, 
\begin{align*}
W_{\XX}(\mu_1,\mu_2) &= \sup \left\{ \int f \diff(\mu_1-\mu_2); \, f: \XX \rightarrow [0,1] \mbox{ measurable }\right\} \\
&= \sup \left\{ \int f \diff(\mu_1-\mu_2);\, f:\XX \rightarrow [-\frac12, \frac12]  \mbox{ measurable }\right\}.
\end{align*}
Using Proposition \ref{existence_optimal_function_for_TV} we then obtain $W_{\XX}(\mu_1,\mu_2) = ||\mu_1-\mu_2||_{TV}$.
\end{proof}

\begin{rmk}
For $k >0$, let $k_{x \not = y} = k . 1_{x \not = y}$ denote the $k$-discrete distance on $\XX$. If $\XX$ is a Polish space and $\mu_1, \mu_2$ are two Borel probability measures on $\XX$, then the Kantorovich-Rubinstein distance (with respect to $k_{x \not = y}$) of $\mu_1, \mu_2$ is equal to $k ||\mu_1 - \mu_2||_{TV}$. 
\end{rmk}

\subsection{The Kantorovich-Rubinstein distance of atomic measures for the discrete distance on $\XX$}

In the case of atomic measures supported on a finite number of points, the total variation distance can be written as an analytical and computationally friendly expression
To obtain this analytical expression for the total variation distance, we first need the following result:

By direct application of Theorem \ref{KR_dist_equal_TV_for_discrete_dist} and Proposition \ref{equality_tv_sum_over_support_of_measures} we obtain the theorem:

\begin{theo}\label{formula_MK_with_discrete_distance_atomic_measures}
Let $(\XX, \FF)$ be a probability space and $\FF$ be a $\sigma$-algebra of subsets of $\XX$. Let us define two atomic probability measures on $S \in \FF$:
\begin{equation*} 
\mu_1 = \sum_{x \in S} \mu_x^{(1)} \delta_{x} \quad \mbox{and} \quad \mu_2 = \sum_{x \in S} \mu_x^{(2)} \delta_{x},
\end{equation*}
Let $S_1 = \{ x \in S; \, \mu_x^{(1)} \geq \mu_x^{(2)} \} $ and $S_2 = \{ x \in S; \, \mu_x^{(2)} > \mu_x^{(1)}\}$ partition $S$. That is, $S = S_1 \amalg S_2$. Then, 
\begin{equation*}
||\mu_1-\mu_2||_{TV} =  \sum_{i=1}^n |\mu_1(x_i) - \mu_2(x_i)|.
\end{equation*}
\end{theo} 

\begin{rmk}\label{rmk_total_variation} There are two remarks worth making:
\begin{enumerate}
\item[(i)] It is interesting to note that, informally, the total variation distance between two probability measures can be seen as the maximum difference between the two probabilities assigned to a single event by the two distributions. \\

\item[(ii)] As said above, total variation is a classical notion of distance between probability measures. There is also a classical probabilistic representation formula of the total variation: 

For two given probability measure $\mu$ and $\nu$ on a measurable space $\XX$, the total variation formula can be defined as 
\begin{equation*}
||\mu-\nu||_{TV} = 2 \inf \mathbb{P}[X\not = Y], 
\end{equation*}
where the infimum is over all couplings $(X, Y)$ of $(\mu,\nu)$; this identity can be seen as a very particular case of duality for the cost function $c(x,y) = 1_{x \not = y}$. For a proof of this result, please see Lindvall (\cite{Lindvall_coupling_method}, Theorem 5.2).
\end{enumerate}
\end{rmk}

\noindent From Theorem \ref{formula_MK_with_discrete_distance_atomic_measures}, we can deduce a very useful corollary:

\begin{lem}\label{formula_MK_with_equidistant_points_atomic_measures}
Let $(\XX,d)$ be a Polish metric space and let $S = \{x_1, \ldots , x_n\}$ be a finite set of mutually equidistant points in $\XX$. Let us define two atomic probability measures
\begin{equation*}
\mu_1 = \sum_{i = 1}^{n} \mu_i^{(1)} \delta_{x_i} \quad \mbox{and} \quad \mu_2 = \sum_{i = 1}^{n} \mu_i^{(2)} \delta_{x_i},
\end{equation*}
supported on the set $S$. Then, 
\begin{equation*}
W_{\XX}(\mu_1,\mu_2) = k \sum_{i=1}^n |\mu_1(x_i) - \mu_2(x_i)|,
\end{equation*}
where $k$ is the distance between each points.
\end{lem}

\begin{proof}[Lemma \ref{formula_MK_with_equidistant_points_atomic_measures}]
Since the $n$ points in $S$ are mututally equidistant, it is clear that the distance $d$ restricted to $S$ is equal to the $k$-discrete distance. Then, a direct application of theorem \ref{formula_MK_with_discrete_distance_atomic_measures} finishes the proof.
\end{proof}

\subsection{The Kantorovich-Rubinstein distance of atomic measures on the product space $\XX^n$ where each $\XX$ is equipped with the discrete distance}\label{KR_dist_atomic_measures_on_product_space}

\begin{defn}\label{definition_metric_on_cartesian_space_recall}
Let $(\XX_i,d_i)$ be $m$ metric spaces and let $\YY = \XX_1 \times \ldots \times \XX_m$ be the Cartesian product of these $m$ metric spaces. For $p\in [1,+\infty)$, the $p$ product metric $d_p$ is defined as the $p$ norm of the $m$-vector of the distances $d_m$. That is :
\begin{equation*}
d_p(x,y) = \big(  \sum_{i = 1}^m d_i(x_i,y_i)^p \big)^{1/p},\, \mbox{ for } x,y \in \YY,
\end{equation*}
where $x = (x_1, \ldots, x_m)$ and $y = (y_1, \ldots, y_m)$. \\
For $p = \infty$, the $p$ product metric is also called the $\sup$ metric and is defined as 
\begin{equation*}
d_{\infty}(x,y) = \max_{i \leq m} d_i(x_i,y_i).
\end{equation*}
\end{defn}

\begin{prop}\label{discrete_distance_on_cartesian_product}
Let $(\XX_i,k_{x \not= y})$ be $m$ metric spaces each equipped with the $k$-discrete distance and let $\YY = \XX_1 \times \ldots \times \XX_m$ be the Cartesian product of these $m$ metric spaces. \\
Then, the $\sup$ metric $d_{\infty}$ on $\YY$ is the $k$-discrete distance on $\YY$. That is, \\
$d_{\infty}(x,y) = k_{x \not= y}(x,y)$,  for all $(x,y)\in \YY \times \YY$.
\end{prop}

\begin{proof}[Proposition \ref{discrete_distance_on_cartesian_product}]
 As written in Definition \ref{definition_metric_on_cartesian_space_recall}, $d_{\infty}(x,y) = \max_{i \leq m} d_i(x_i,y_i)$ for $x = (x_1, \ldots, x_m)$ and $y = (y_1, \ldots, y_m)$.When $d_i(x,y) = k_{x \not= y}(x,y)$, $\forall i$, it is clear that
 \begin{equation*}
d_{\infty}(x,y) = \begin{cases}
0 &\quad\textnormal{ if $x_i = y_i, \, \forall i$}\\
k &\quad\textnormal{ otherwise}.
\end{cases}
\end{equation*}
Thus, we can deduce that $d_{\infty} = k_{x \not= y}$ on $\YY$. Indeed, if $x = y$ we have $x_i = y_i, \, \forall i$ and therefore $d_{\infty}(x,y) = 0$. On the contrary, if $x \not= y$, then $x_{i_r} = y_{i_r}$ for some $r \in \{1, \ldots, s\}, \, s \leq m$. Thus, $d_{\infty}(x,y) = k$.
\end{proof}

\noindent A direct application of Theorem \ref{formula_MK_with_discrete_distance_atomic_measures} on the cartesian product $\YY$ yields the following result:

\begin{theo}\label{formula_MK_with_discrete_distance_atomic_measures_on_cartesian_product}[l.636]
Let $(\XX_j, k_{x \not= y})$ be $m$ metric spaces each equipped with the $1$-discrete distance and let $\YY = \XX_1 \times \ldots \times \XX_m$ be the Cartesian product of these $m$ metric spaces. \\
For each $\XX_j$, let us define two atomic probability measures,
\begin{equation*} 
\mu_{1,j} = \sum_{i = 1}^{n} \mu_{i,j}^{(1)} \delta_{x_{i,j}} \quad \mbox{and} \quad \mu_{2,j} = \sum_{i = 1}^{n} \mu_{i,j}^{(2)} \delta_{x_{i,j}},
\end{equation*}
both supported on $n_j$ finite number of points $\{x_{1j}, \ldots , x_{nj}\} \in \XX_j$. \\
On $\YY$ equipped with the $\sup$ metric, the two atomic probability measures are given by
\begin{equation*}
\mu_1 = \sum_{i = 1}^{n^m} \mu_i^{(1)} \delta_{y_i} \quad \mbox{and} \quad \mu_2 = \sum_{i = 1}^{n^m} \mu_i^{(2)} \delta_{y_i},
\end{equation*}
supported on the same finite number of points $\{y_1, \ldots , y_{n^m}\} \in \YY$. Then, \\
\begin{equation*}
W_{\YY}(\mu_1,\mu_2) =  \sum_{i=1}^{n^m} |\mu_1(y_i) - \mu_2(y_i)|.
\end{equation*}
\end{theo}

\section{The Kantorovich-Rubinstein distance on the line equipped with the Euclidean metric}\label{section_KR_distance_on_R}

There is a useful way to describe the collection of all finite measures on $\RR$. If $\mu$ is such a measure, one defines the real function $F$ by 
\begin{equation*}
F(x) = \mu \big((-\infty, x] \big).
\end{equation*}
Then, $F$ is non-decreasing, right-continuous and satisfies $\lim_{x \to -\infty} F(t) = 0$ and $\lim_{x \to \infty} F(t) = \mu(\RR)$. Finally, for any bounded interval $(a,b]$, the following equality holds: 
\begin{equation}\label{from_mu_to_F}
\mu \big((a,b]\big) = F(b) - F(a).
\end{equation}

If $\mu$ is a probability measure, the function $F$ is called the distribution function of $\mu$. It is also often called the cumulative distribution function of $\mu$.\\
 
The measure $\mu$ is completely determined by its distribution function $F$. Indeed, the following theorem (see\cite{Billingsley_Prob_and_Measure} p.for a proof) ensures that to each $F$, there exist a $\mu$:

\begin{theo}
 Let $F$ be a non-decreasing, right-continuous, real function on $\RR$. Then there exists on the Borel $\sigma$-algebra of $\RR$ a unique measure $\mu$ satisfying equation \ref{from_mu_to_F}, for all $a,b \in \RR$.
\end{theo}

An immediate consequence of the previous theorem is that such an $F$ is the distribution function of some random variable:

\begin{lem}
If $F$ is non-decreasing, right-continuous, real function on $\RR$, satisfying $\lim_{x \to -\infty} F(t) = 0$ and $\lim_{x \to \infty} F(t) = \mu(\RR)$, then there exists on $\RR$ a Borel random variable $\XX$ such that $F(x) = \mu[X \leq x]$.
\end{lem}

The two propositions \ref{generalized_inverse} and \ref{Prob_Integral_Transform} are well known and necessary to prove Theorem \ref{KR_distance_on_R}:

\begin{prop}\label{generalized_inverse}
Let $F: \RR \rightarrow [0,1]$ be non-decreasing, right-continuous, real function on $\RR$, satisfying $\lim_{x \to -\infty} F(t) = 0$ and $\lim_{x \to \infty} F(t) = 1$. Then, there exist a measurable, left-continuous function $G: [0,1] \rightarrow \RR \cup \{ \pm \infty \}$ defined in the following way:
\begin{enumerate}
\item[(i)] for a given $y \in [0,1]$, if there is an $x$ such that $F(x)=y$, \\
$G(y) = \inf \{x;\,F(x)=y \}$ 
\item[(ii)] for a given $y \in [0,1]$, if there is no $x$ such that $F(x)=y$, \\
$G(y) = \inf_{z>y} G(z)$ where all $z$ are such that $G(x)$ exists.
\end{enumerate} 
The function $G$ is called the generalized inverse of $F$ and is denoted by $F^{-1}$.
\end{prop}

\begin{prop}[Probability Integral Transformation] \label{Prob_Integral_Transform} 
Let $(\Omega, \mu)$ be a probability measured space and let $X: \Omega \rightarrow [0,1]$ be a uniformly distributed random variable.\\
Consider $F: \RR \rightarrow [0,1]$, a non-decreasing, right-continuous function satisfying $\lim_{x \to -\infty} F(t) = 0$ and $\lim_{x \to \infty} F(t) = 1$. \\
Then, $F^{-1}(X)$ is a random variable with distribution function $F$.
\end{prop}

\begin{proof}[Proposition \ref{Prob_Integral_Transform}]
The distribution function of the random variable $F^{-1}(X)$ will be denoted by $G$ while the distribution function of $X$ will be denoted $F_X$. Hence, we need to show that $G(t) = F(t),\, \forall t$. \\
By definition, $G(t) = \mu (\{\omega \in \Omega;\, F^{-1}(X)(\omega) \leq t \})$. Since $F$ is non-decreasing, we obtain $G(t) = \mu (\{\omega \in \Omega;\, X(\omega) \leq F(t) \})$. But $\mu (\{\omega \in \Omega;\, X(\omega) \leq F(t) \}) = F_X (F(t))$ and $X$ is uniformly distributed hence $\mu (\{\omega \in \Omega;\, X(\omega) \leq F(t) \}) = F(t)$.
\end{proof}

We can now state a very important result. A proof was published by Vallender \cite{Vallender_Wasserstein_on_R} in 1974. The proof has been revisited for clarity and to use the concept of couplings.

\begin{theo}\label{KR_distance_on_R}
Let us consider $\RR$ equipped with the usual Euclidean metric. Let $\mu_1$,$\mu_2$ be two probability measures on $\RR$ and $F_{\mu_1}$, $F_{\mu_2}$ be their respective cumulative distribution functions. Then,
\begin{equation*}
W_{\RR}(\mu_1,\mu_2) = \int_{-\infty}^{+\infty} |F_{\mu_1}(x) - F_{\mu_2}(x)| \diff x.
\end{equation*}
\end{theo}

\begin{proof}[Theorem \ref{KR_distance_on_R}]
Suppose that $\dps \int_{-\infty}^{+\infty} |F_{\mu_1}(x) - F_{\mu_2}(x)| \diff x < \infty$.
Let  $\vartheta$ be a coupling of $\mu_1$ and $\mu_2$ on $\RR \times \RR$ such that $\displaystyle\int |x-y|  \diff \vartheta(x,y) \leq \infty$. Then, let us first show that
\begin{equation*}
\int_{-\infty}^{+\infty} |F_{\mu_1}(x) - F_{\mu_2}(x)| \diff x \leq \int |x-y| \diff \vartheta(x,y).
\end{equation*}

\noindent Let us set $A = \{(x,y) \in \RR^2;\, x > y\}$ and $B = \{(x,y)\in \RR^2;\, y \geq x \}$. Then we can write
\begin{equation*}
\int |x-y| \diff \vartheta(x,y) = \int_A (x-y) \diff \vartheta(x,y) + \int_B (y-x) \diff \vartheta(x,y).
\end{equation*}

\noindent Now, we need to prove that
\begin{equation*}
\int_A (x-y) \diff \vartheta(x,y) = \int_{\RR} \vartheta(D_s) \diff s \quad \mbox{ and } \quad \int_B (y-x) \diff \vartheta(x,y) = \int_{\RR} \vartheta(C_s) \diff s,
\end{equation*}
where $C_s = \{(x,y); \, x \leq s \mbox{ and } y > s\}$ and $D_s = \{(x,y); \, x > s \mbox{ and } y \leq s\}$.\\

We note that, by the construction of $B$,
\begin{equation*}
\int_B (y-x) \diff \vartheta(x,y) = \int_0^{+\infty} \vartheta (\{(x,y); y-x \geq t \}) \diff t 
\end{equation*} 
By the Disintegration Theorem \ref{Disintegration_thm}, we obtain 
\begin{equation*}
\vartheta(\{ (x,y);\, y-x \geq t\}) = \int_{\RR} \vartheta_{x} (\{ (x,y);\, y \geq x + t \}) \diff \mu_1(x)
\end{equation*} 
Then, by Fubini Theorem, we obtain 
\begin{align*}
\int_B (y-x) \diff \vartheta(x,y) &= \int_0^{+\infty} \int_{\RR} \vartheta_{x} (\{ (x,y);\, y \geq x + t \}) \diff \mu_1(x) \diff t \\
&= \int_{\RR} \int_0^{+\infty} \vartheta_{x} (\{ (x,y);\, y \geq x + t \}) \diff t \diff \mu_1(x)
\end{align*}
But with the change of variable $s = x + t$, we have 
\begin{align*}
\int_0^{+\infty} \vartheta_{x} (\{ (x,y);\, y \geq x + t \}) \diff t &= \int_{x}^{+\infty} \vartheta(\{(x,y);\, y \geq s \}) \diff s \\
&= \int_{\RR} \chi_{[x,+\infty)} (s) \vartheta(\{(x,y); y \geq s\}) \diff s
\end{align*}
For $s$ fixed,  we have $\{x \in \RR; \chi_{[x,+\infty)} (s) = 1\} = \{x \in \RR; \, x \leq s\}$ and \\
$\{x; \chi_{[x,+\infty)} (s) = 0\} = \{x \in \RR; \, x > s\}$. That is, $\chi_{[x, +\infty)}(s) = \chi_{(-\infty,s]}(x)$. Hence,
\begin{align*}
\int_B (y-x) \diff \vartheta(x,y) &= \int_{\RR} \int_{\RR} \chi_{[x,+\infty)} (s) \, \vartheta_x(\{(x,y); \, y \geq s\}) \diff s \diff \mu_1(x) \\
 &=  \int_{\RR} \int_{\RR} \chi_{(-\infty,s]} (x) \, \vartheta_x(\{(x,y);\, y \geq s\}) \diff \mu_1(x) \diff s
\end{align*}
For a given $s$, $\displaystyle\int_{\RR} \chi_{(-\infty,s]} (x) \, \vartheta_s(\{(x,y);\, y \geq s\}) \diff \mu_1(x) = \vartheta(\{(x,y);\, x \leq s \mbox{ and } y \geq s\})$.\\
Hence, we have $\displaystyle \int_B (y-x) \diff \vartheta(x,y) = \int_{\RR} \vartheta(C_s) \diff s$.\\

Using a similar argument, we obtain $\displaystyle \int_A (x-y) \diff \vartheta(x,y) = \int_{\RR} \vartheta(D_s) \diff s$.  Therefore,  
\begin{equation*}
\displaystyle \int |x-y| \diff \vartheta(x,y) = \int_{\RR} \big(\vartheta(C_s) + \vartheta(D_s)\big) \diff s.
\end{equation*}
Geometrically, it is easy to see that 
\begin{align*}
C_s = \{(x,y); \, x \leq s\} \setminus \{(x,y);\, x \leq s \mbox{ and } y \leq s \} \\
D_s = \{(x,y); y \leq s\} \setminus \{(x,y);\, x \leq s \mbox{ and } y \leq s \}.
\end{align*}
To simplify notation for the rest of the proof we will denote the set $\{(x,y);\, x \leq s \mbox{ and } y \leq s \}$ by $E_s$. We thus obtain:
\begin{align*}
\vartheta(C_s) + \vartheta(D_s) = & \, \vartheta(\{(x,y); \, x \leq s\}) + \vartheta(\{(x,y); y \leq s\}) - 2 \vartheta(E_s)  \\
 \geq & \, \mu_1(x \leq s) + \mu_2(y \leq s) \\
 & - 2 \min \{ \vartheta(\{(x,y);\, x \leq s \}), \, \vartheta(\{(x,y);\, y \leq s \} ) \} \\
\geq & F_{\mu_1} (s) + F_{\mu_2} (s) - 2 \min \{ F_{\mu_1}(s), \, F_{\mu_2}(s) \} 
 \end{align*}
For $F_{\mu_1}(s) \leq F_{\mu_2}(s)$, we have $F_{\mu_1}(s) + F_{\mu_2}(s) - 2\min(F_{\mu_1}(s), F_{\mu_2} (s)) = F_{\mu_2}(s) - F_{\mu_1}(s)$. 
For $F_{\mu_1}(s) \geq F_{\mu_2}(s)$, we have $F_{\mu_1}(s) + F_{\mu_2}(s) - 2\min(F_{\mu_1}(s), F_{\mu_2}(s)) = F_{\mu_1}(s) - F_{\mu_2}(s)$. 
Thus, $F_{\mu_1}(s) + F_{\mu_2}(s) - 2\min(F_{\mu_1}(s), F_{\mu_2}(s)) =  |F_{\mu_1}(s) - F_{\mu_2}(s)|$. Hence we have
\begin{equation*}
\int |x-y| \diff \vartheta(x,y) \geq \int_{\RR} |F_{\mu_1}(s) - F_{\mu_2}(s)| ds.
\end{equation*}
As $W_{\RR}(\mu_1,\mu_2)$ is the infimum over all couplings of $\mu_1$ and $\mu_2$, it is the greatest lower bound and thus 
\begin{equation*}
W_{\RR}(\mu_1,\mu_2) \geq \int_{\RR} |F_{\mu_1}(s) - F_{\mu_2}(s)| ds.
\end{equation*}

\noindent Now, it is left to show that there exist a coupling $\vartheta^*$ such that 
\begin{equation}\label{egalite_MK_et_int_des_distributions}
\int_{\RR} |F_{\mu_1}(x) - F_{\mu_2}(x)| \diff x = \int |x-y| \diff \vartheta^*(x,y).
\end{equation}
Let $X: \RR \rightarrow [0,1]$ be a uniformly distributed random variable. Then, by Proposition \ref{Prob_Integral_Transform}, the real random variable $F_{\mu_1}^{-1}(X)$, respectively $F_{\mu_2}^{-1}(X)$ has distribution function $F_{\mu_1}$, respectively $F_{\mu_2}$.
We choose, as a candidate for $\vartheta^*$, the product measure on $\RR^2$ of the push-forward measures $F_{\mu_1}^{-1}(X)(\mu_1)$ and $F_{\mu_2}^{-1}(X)(\mu_1)$ of $\mu_1$ from $\RR$. 

\noindent First, we verify that the candidate product measure for $\vartheta^*$ is a coupling of the measures $\mu_1$ and $\mu_2$. That is, one needs to show that, for $i=1,2$, $\pi_i(\vartheta^*) = \mu_i$, where $\pi_i$ is the natural projection from $\RR^2$ to $\RR$. If $A_s$ denotes $\{(x_1,x_2) \in \RR^2 ;\, x_i \in (-\infty,s] \}$, we have:  
\begin{align*}
\pi_i(\vartheta^*)((-\infty,s]) &= \vartheta^* (\{(x_1,x_2) \in \RR^2 ;\, \pi_i(x_1,x_2) \in (-\infty,s] \}) \\
&= (F_{\mu_1}^{-1} \times F_{\mu_2}^{-1}) (X)(\mu_1)(A_s) \quad \mbox{where }  \\
&= \mu_1 (\{x \in \RR; \, (F_{\mu_1}^{-1} \times F_{\mu_2}^{-1}) (X)(x) \in A_s \}) \\
&= \mu_1 (\{x \in \RR; \, (F_{\mu_1}^{-1}(X)(x) ;  F_{\mu_2}^{-1}(X)(x)) \in A_s \} \\
&= \mu_1 (\{x \in \RR; \, F_{\mu_i}^{-1}(X)(x) \in (-\infty,s] \}) \\
&= \mu_1 (\{x \in \RR; \, X(x) \in (-\infty, F_{\mu_i}(s)] ) \} \\
&= F_{\mu_i}(s) \qquad \mbox{(since $X$ is uniformly distributed),} \\
&= \mu_i((-\infty,s]).
\end{align*}
It is worth noting that one could have used the push-forward measures of $\mu_2$ or of any other probability measure on $\RR$.

\noindent To complete the proof, we show that the equality (\ref{egalite_MK_et_int_des_distributions}) holds for our choice of $\vartheta^*$. We know that:
\begin{equation*}
 \int_{\RR} \vartheta^*(C_s) + \vartheta^*(D_s) \diff s = \,\vartheta^*(\{(x,y); \, x \leq s\}) + \vartheta^*(\{(x,y); y \leq s\}) - 2 \vartheta^*(E_s) \\
\end{equation*}
where $\vartheta^*(\{(x,y); \, x \leq s\}) = F_{\mu_1}(s)$ and $\vartheta^*(\{(x,y); y \leq s\}) = F_{\mu_2}(s)$. Let us now focus on $\vartheta^*(E_s)$:
\begin{align*}
\vartheta^*(E_s) &= (F_{\mu_1}^{-1} \times F_{\mu_2}^{-1}) (X)(\mu_1)(E_s) \\
&= \mu_1 (\{x \in \RR; \, (F_{\mu_1}^{-1}(X)(x) ;  F_{\mu_2}^{-1}(X)(x)) \in E_s \} \\
&= \mu_1 (\{x \in \RR; \, X(x) \in (-\infty, F_{\mu_1}(s)] ) \mbox{ and }  X(x) \in (-\infty, F_{\mu_2}(s)] )\} \\
&= \min \{ \mu_1 (\{ x \in \RR ;\, X(x) \leq F_{\mu_1}(s) \}), \, \mu_1 (\{x \in \RR ;\, X(x) \leq F_{\mu_2}(s) \} ) \} \\
&= \min \{ F_{\mu_1}(s), \, F_{\mu_2}(s) \} \qquad \mbox{(since $X$ is uniformly distributed).} \\
\end{align*}
As shown above, we therefore obtain $\vartheta^*(C_s) + \vartheta^*(D_s) = |F_{\mu_1} - F_{\mu_2}|$ and thus Equation (\ref{egalite_MK_et_int_des_distributions}) is true and the proof is complete.
\end{proof}

\subsection{The Kantorovich-Rubinstein distance for atomic measures on the line}

\begin{prop}\label{KR_distance_on_R_atomic_measures}
Let us consider $\RR$ equipped with the usual euclidean metric. Let us define, on $\RR$, two atomic probability measures
\begin{equation*}
\mu_1 = \sum_{i = 1}^{n} \mu_i^{(1)} \delta_{x_i} \quad \mbox{and} \quad \mu_2 = \sum_{i = 1}^{n} \mu_i^{(2)} \delta_{x_i},
\end{equation*}
supported on the same finite number of ordered points $\{x_1, \ldots , x_n\} \in \RR$. 

The Kantorovich-Rubinstein distance between $\mu_1$ and $\mu_2$ is then given by 
\begin{equation*}
W_{\RR}(\mu_1,\mu_2) = \sum_{i=1}^{n-1} |F_i| (x_{i+1} - x_i),
\end{equation*}
where $\displaystyle F_i = \sum_{k=1}^i \mu_k^{(1)} - \mu_k^{(2)}$.
\end{prop}

\begin{proof}[Proposition \ref{KR_distance_on_R_atomic_measures}]
By construction, the distribution functions $F_{\mu_{1}}$ and $F_{\mu_{2}}$ of the atomic measures $\mu_1$ and $\mu_2$ supported on the ordered set $\{x_1, \ldots , x_n\} \in \RR$ are defined as
\begin{equation*}
F_{\mu_{1}}(t) = \sum_{i=1}^n \mu_i^{(1)} 1_{x_i \leq t} \quad \mbox{and} \quad F_{\mu_{2}}(t) = \sum_{i=1}^n \mu_i^{(2)} 1_{x_i \leq t}.
\end{equation*}

Since the distribution function of an atomic measure is a simple function, $|F_{\mu_{1}} - F_{\mu_{2}}|$ is also a simple function. Integrals of simple functions are well known (see Billingsley \cite{Billingsley_Convergence_of_Prob_measures}). Thus, a direct application of Theorem \ref{KR_distance_on_R} for $\mu_1$ and $\mu_2$ yields the following Kantorovich-Rubinstein distance:
\begin{equation*}
W_{\XX}(\mu_1,\mu_2) = \sum_{i=1}^{n-1} |F_i| (x_{i+1} - x_i),
\end{equation*}
where $\displaystyle F_i = \sum_{k=1}^i \mu_k^{(1)} - \mu_k^{(2)}$.
\end{proof}

\section{The Kantorovich-Rubinstein distance on the circle $\SSS_1$}
In this section we consider the Kantorovich-Rubinstein distance where the underlying Polish space $\XX$ is the unit circle. 
 We use the following notations from the paper on the Wasserstein distance on the circle by Cabrelli and Molter\cite{Cabrelli_Molter_K_R_circle}: 
 
\noindent We identify the circle $K = \{z \in \mathbb{C}: |z| = 1 \}$ with $T = [0, 1)$ as a fundamental domain for $\RR/\mathbb{Z}$, via the transformation $t \mapsto e^{i2\pi t}$, and use the natural metric on $T$ given by $\rho(s_1,s_2)= min(|s_1- s_2|, 1 - |s_1 - s_2|)$. It corresponds to the minimum arc length on the circle. 

\noindent We identify the functions on $T$ with the periodic functions on $\RR$ of period 1; using this identification we will use $f(t)$, with $t\in\RR$, for functions on $T$. \\
We denote by $(\XX, |.|)$ the unit interval $[0,1]$ on the line with the Euclidean distance. \\
For measures with bounded support on $\XX$, the distribution function of $\mu$ is defined by $F_{\mu}(t) = \mu( \{ x \in [0,1] : x \in [0, t] \} ) = \mu([0,t])$. \\

An analytic expression for the Kantorovich-Rubinstein distance is between probability measures on the circle is given in Theorem \ref{KR_on_circle}. We first need the following definition \ref{balanced_function} and three lemmas from \cite{Cabrelli_Molter_K_R_circle}. 

\begin{defn}\label{balanced_function}
For an arbitrary measurable function $\gamma : \XX \rightarrow \RR$, consider the following three measurable sets: 
\begin{equation*}
A^+(\gamma)=\left\{ x \in \XX : \gamma(x) > 0\right\},\,  A^-(\gamma)=\left\{ x \in \XX : \gamma(x) < 0\right\}, \,A^0(\gamma)=\left\{ x \in \XX : \gamma(x) = 0\right\}.
\end{equation*}

Then, a function $\gamma : \XX \rightarrow \RR$ is said to be \textit{balanced} if 
\begin{equation*}
| \lambda(A^+(\gamma)) - \lambda(A^-(\gamma)) | \leq \lambda(A^0(\gamma)).
\end{equation*}
We also denote by $D$ the set defined by $D(\gamma)=\left\{ c: A^0(\gamma)\rightarrow\RR, \, \mbox{measurable:}\, \vert c(t) \rvert \leq 1 \mbox{ a.e.}\right\}$.
\end{defn}

We now relate the distance between measures on the circle with the distance between measures on $\XX$ obtained from the former by "cutting" the circle. 
First, we identify the measures on $T$ with the appropriate measures on $\XX$.

For $\mu \in M(T)$, we consider the function $G_{\mu}: \RR \rightarrow \RR$ defined by 
\begin{equation*}
G_{\mu}(x) = \mu([0,x]), \, \mbox{ for }\, 0\leq x < 1,
\end{equation*}
and extended to $\RR$ by the equation $G_{\mu}(x+1) = G_{\mu}(x) + 1$. By construction, $G_{\mu}$ is right-continuous.

For each $s \in T$, we associate to any measure $\mu \in M(T)$ the pair $(\mu_s^D, \mu_s^T)$ of measures on $T$, determined by their respective distribution functions:
\begin{align*}
D_{\mu}^s: \XX \rightarrow \RR, \quad &\mbox{given by}  \quad D_{\mu}^s(x) = G_{\mu}(x+s) - G_{\mu}(s), \\
I_{\mu}^s: \XX \rightarrow \RR, \quad &\mbox{given by}  \quad I_{\mu}^s(x) = G_{\mu}(x+s) - G_{\mu}(s-).
\end{align*}

As distribution functions, $D_{\mu}^s$ and $I_{\mu}^s$ are defined on $\RR$. Note that: 
\begin{enumerate}
\item[(i)] $D_{\mu}^s(x) = 0$ for $x \in (-\infty,0]$, $D_{\mu}^s(x) = 1$ for $x \in [1,+\infty)$ and $D_{\mu}^s$ has a jump of height $G_{\mu}(s)-G_{\mu}(s^-)$ at $x=1$ while $I_{\mu}^s(x) = 0$ for $ x \in (-\infty,0)$, $I_{\mu}^s(x) = 1$ for $x \in [1,+\infty)$ and has a jump of height $G_{\mu}(s)-G_{\mu}(s^-)$ at $x=0$;
\item[(ii)] if $\mu({s}) = 0$ then $G_{\mu}$ is continuous at $s$ and thus $\mu_s^D = \mu_s^I$ and $D_{\mu}^s = I_{\mu}^s$.
Moreover, note that $\mu_s^D(\{0\}) = 0$ and that $\mu_s^I(\{1\}) = 0$. 
\item[(iii)] $\mu_s^D(\{0\}) = 0$ and $\mu_s^I(\{1\}) = 0$.
\end{enumerate}

Informally, one refers to the value $s \in T$ as a cut of the circle. One can picture $\mu_s^D$, $\mu_s^I$, as representing the measures on  the line obtained by "cutting" the circle at $s$ and taking $(s,s+1]$ and $[s,s+1)$ as fundamental domains in $\RR / \ZZZ$. 
Now that we have defined balanced functions, we can state the lemmas:

\begin{lem}\label{distance_on_circle_distance_on_X}
Let $\mu, \nu \in M(T)$ and $\eta = \mu - \nu$. For $r \in T$ fixed, we consider the function $D^r = D_{\mu}^r - D_{\nu}^r$, the difference of the distribution functions associated to $\mu_r^D$ and $\nu_r^D$. Then, 
\begin{equation}\label{wasserstein_distance_circle_equal_inf_on_[0,1]}
W_T(\mu, \nu) \leq \inf_{r \in T} W_{\XX}(\mu_r^D, \nu_r^D).
\end{equation}

Moreover if, for $s \in T$, the function $D^s$ is balanced, the infimum on the right-hand side of the equation (\ref{wasserstein_distance_circle_equal_inf_on_[0,1]}) is attained at $s$. Thus , for $D^s$ balanced we can write
\begin{equation*}
W_T(\mu, \nu) = W_{\XX}(\mu_s^D, \nu_s^D).
\end{equation*} 

Analogously, if for $s \in T$, the function $I^s = I_{\mu}^s - I_{\nu}^s$ is balanced, we have 
\begin{equation*}
W_T(\mu, \nu) = W_{\XX}(\mu_s^I, \nu_s^I) =  \inf_{r \in T} W_{\XX}(\mu_r^I, \nu_r^I).
\end{equation*}
\end{lem}

The next lemma show that there always exists a so-called optimal cut $s \in T$. that is, an $s$ such that $D^s$ or $I^s$ is balanced:

\begin{lem}\label{existence_of_optimal_cut}
Let $\mu$ and $\nu$ be two measures on $T$. Define $G = G_{\mu} - G_{\nu}$ on $\RR$ and let $\alpha$ be the restriction of $G$ on $\XX$ (i.e. $\alpha(x) = G_{\mu}(x) - G_{\nu}(x)$, for $x \in [0,1]$). Then, there exists an optimal cut $s \in T$ such that either $D^s$ or $I^s$ is balanced.
\end{lem}

Now, we need one more lemma to prove Theorem \ref{KR_on_circle}, on the existence of an analytic expression for the Kantorovich distance between probability measures on the circle. 

\begin{lem}\label{a_gamma_in_closure_of_gamma(X)}
Let $\gamma: \XX \rightarrow \RR$ be a right continuous function and $\lambda$ be the Lebesgue measure on $\RR$. Consider the function $m_{\gamma} : \RR \rightarrow \XX$ defined by 
\begin{equation*}
m_{\gamma}(t) = \lambda (\{ x \in \XX : \, \alpha(x) \geq t\}).
\end{equation*}
Define the constant $\mathnormal{a}_{\gamma}$ by 
\begin{equation*}
\mathnormal{a}_{\gamma} = \sup \Big\{ t \in \RR : \, m_{\gamma}(t) >\frac12  \Big\}.
\end{equation*}
Then, for any neighbourhood $V_{\mathnormal{a}_{\gamma}}$ of $\mathnormal{a}_{\gamma}$, $\lambda(V_{\mathnormal{a}_{\gamma}} \cap \gamma(\XX)) >0$. In particular, $\mathnormal{a}_{\gamma}$ is in the closure of $\gamma(\XX)$. \\
Likewise, define the constant $\mathnormal{b}_{\gamma}$ by 
\begin{equation*}
\mathnormal{b}_{\gamma} = \inf \Big\{ t \in \RR : \, m_{\gamma}(t) < \frac12  \Big\}.
\end{equation*}
Then, for any neighbourhood $V_{\mathnormal{b}_{\gamma}}$ of $\mathnormal{b}_{\gamma}$, $\lambda(V_{\mathnormal{b}_{\gamma}} \cap \gamma(\XX)) >0$. In particular, $\mathnormal{b}_{\gamma}$ is in the closure of $\gamma(\XX)$. \\
\end{lem}

\noindent We can now prove the theorem:

\begin{theo}\label{KR_on_circle}
Let $T$ be the unit circle equipped with the minimum arc length metric $\rho(s_1,s_2)= min(|s_1- s_2|, 1 - |s_1 - s_2|)$. Define two probability measures $\mu$ and $\nu$ on $T$. Then,
\begin{equation*}
W_T(\mu, \nu) = \int_T |\alpha(x) - \mathnormal{a}_{\alpha}| \diff x,
\end{equation*}
where $\alpha: \XX \rightarrow \RR$ is defined on $x \in [0,1)$ by 
\begin{equation*}
\alpha(x) = \mu([0,x]) -\nu([0,x]) \mbox{ with }  \alpha(1) = 0,
\end{equation*}
\begin{equation*}
\mbox{and } \mathnormal{a}_{\alpha} = \sup \Big\{ t \in \RR : \, \lambda (\{ x \in \XX : \, \alpha(x) \geq t\}) >\frac12  \Big\}
\end{equation*} 
is a translation constant that depends on $\mu$ and $\nu$.
\end{theo}

\begin{proof}[Theorem \ref{KR_on_circle}]
Recall that, by definition, $\alpha$ is the restriction to $\XX$ of the function $G$ defined in Lemma \ref{existence_of_optimal_cut}. Moreover, by the same lemma, we know that there exists $s \in T$ such that either $D^s$ or $I^s$ is balanced.

Assume first that $D^s = D^s_{\mu} - D^s_{\nu}$ is balanced, where $D^s_{\mu}$ and $D^s_{\nu}$ are the distribution functions of $\mu_s^D$ and $\nu_s^D$. By Lemma \ref{distance_on_circle_distance_on_X}, 
 $W_T(\mu,\nu) = W_{\XX}(\mu_s^D,\nu_s^D)$. 

\noindent Now, by Theorem \ref{KR_distance_on_R}, we know that 
\begin{equation*}
W_{\XX}(\mu_s^D,\nu_s^D) = \int_0^1 |D^s(x)| \diff x \quad \mbox{since} \quad D^s = D^s_{\mu} - D^s_{\nu}.
\end{equation*}

Recall also that $D^s(x) = G(x+s) - G(s)$. But $s \in T$ hence $G(s) = \alpha(s)$. Also, the function $\alpha(s)$ is right-continuous. Hence, by Lemma \ref{a_gamma_in_closure_of_gamma(X)}, there exists a sequence $\{s_n\}$ in $\XX$ such that $\alpha(s_n) \rightarrow \mathnormal{a}_{\alpha}$. Let $s \in \XX$ be the limit of point of $\{s_n\}$. Since $\{s_n\} \subset \XX$, we can extract a decreasing sub-sequence $\{s_{n_k}\}$ such that $\{s_{n_k}\} \rightarrow s$. Since $\alpha$ is right-continuous, $\alpha(s_{n_k}) \rightarrow \alpha(s)$ and thus $\alpha(s) = \mathnormal{a}_{\alpha}$. Therefore we have $D^s(x) = G(x+s) - \mathnormal{a}_{\alpha}$.

By a basic change of variables we obtain:
\begin{equation*}
\int_0^1 |G(x+s) -  \mathnormal{a}_{\alpha}| \diff x = \int_s^1 |G(x) - \mathnormal{a}_{\alpha}| \diff x + \int_1^{1+s} |G(x) - \mathnormal{a}_{\alpha}| \diff x.
\end{equation*}
For $x \in [1,1+s]$, we can write $x = 1+r$ for $r \in [0,s]$ and thus we have 
\begin{equation*}
G(x) = G(1+r) =  G_{\mu}(1+r) - G_{\nu}(1+r) = (1 + G_{\mu}(r)) - (1 + G_{\nu}(r)) = G(r) = \alpha(r).
\end{equation*}
Therefore,
\begin{equation*}
\int_0^1 |G(x+s) -  \mathnormal{a}_{\alpha}| \diff x = \int_s^1 |G(x) -  \mathnormal{a}_{\alpha}| \diff x +  \int_0^s |G(x) -  \mathnormal{a}_{\alpha}| \diff x = \int_0^1 |\alpha(x) -  \mathnormal{a}_{\alpha}| \diff x,
\end{equation*}
which completes the proof.

If, on the other hand, $s$ is such that $I^s$ is balanced, the same steps allow to show that 
\begin{equation*}
W_T(\mu,\nu) = \int_0^1 |I^s(x)| \diff x =  \int_0^1 |\alpha(x) -  \mathnormal{a}_{\alpha}| \diff x.
\end{equation*}
\end{proof}

We end this section, with the following corollary of Theorem \ref{KR_on_circle} that we will use in the next section:

\begin{cor}\label{min_btwn_2_mk_expressions}
Under the same hypothesis as Theorem \ref{KR_on_circle}, we have 
\begin{equation*}
W_T(\mu,\nu) = \min \left( \inf_{s\in T} \int_0^1 |\alpha(x)-\alpha(s)| \diff x, \, \inf_{s \in T} \int_0^1 |\alpha(x) - \alpha(s-)| \diff x \right).
\end{equation*}
\end{cor}

\begin{proof}[Corollary \ref{min_btwn_2_mk_expressions}]
By Lemma \ref{distance_on_circle_distance_on_X}, we know that 
\begin{equation*}
W_T(\mu,\nu) \leq  \inf_{r \in T} W_{\XX}(\mu_r^D, \nu_r^D) \quad \mbox{and} \quad W_T(\mu,\nu) \leq\inf_{r \in T} W_{\XX}(\mu_r^I, \nu_r^I).
\end{equation*}
Therefore, 
\begin{equation*}
W_T(\mu,\nu) \leq \min \left( \inf_{r \in T} W_{\XX}(\mu_r^I, \nu_r^I) , \, \inf_{r \in T} d_X(\mu_r^D, \nu_r^D) \right).
\end{equation*}
\noindent Now, by Theorem \ref{KR_distance_on_R}, we know that 
\begin{equation*}
W_{\XX}(\mu_s^D,\nu_s^D) = \int_0^1 |D^s(x)| \diff x \quad \mbox{since} \quad D^s = D^s_{\mu} - D^s_{\nu},
\end{equation*}
where $D^s_{\mu}$ and $D^s_{\nu}$ are the distribution functions of $\mu_s^D$ and $\nu_s^D$, respectively.

Moreover, by Lemma \ref{existence_of_optimal_cut} and Lemma \ref{distance_on_circle_distance_on_X}, we know that  there exists an $s \in T$ such that 
\begin{equation*}
W_T(\mu,\nu) = W_{\XX}(\mu_s^D, \nu_s^D) \quad \mbox{or} \quad W_T(\mu,\nu) = W_{\XX}(\mu_s^I, \nu_s^I).
\end{equation*}
 Therefore, since $D^s (x) = G(x+s) - \alpha(s)$ and $I^s(x) = G(x+s) - \alpha(s-)$ (see the proof of Theorem \ref{KR_on_circle}),
\begin{equation*}
d_T(\mu,\nu) = \min \left( \inf_{s\in T} \int_0^1 |\alpha(x)-\alpha(s)| \diff x, \, \inf_{s \in T} \int_0^1 |\alpha(x) - \alpha(s-)| \diff x\right).
 \end{equation*}
\end{proof}

\subsection{The Kantorovich-Rubinstein distance on the circle $\SSS^1$ for atomic mesures}

Let us consider the case of atomic measures on the circle $T$. We have the following proposition:

\begin{prop}\label{KR_on_circle_atomic_measures}

Let the unit circle $T$ be equipped with the minimum arc length metric $\rho(s_1,s_2)= min(|s_1- s_2|, 1 - |s_1 - s_2|)$.
 Let us define, on $T$, two atomic probability measures 
\begin{equation*}
\mu_1 = \sum_{i = 1}^{n} \mu_i^{(1)} \delta_{s_i} \quad \mbox{and} \quad \mu_2 = \sum_{i = 1}^{m} \mu_i^{(2)} \delta_{s_i},
\end{equation*}
supported on the same finite number of ordered points $\{s_1, \ldots , s_n\} \in T$.

The Kantorovich-Rubinstein distance between $\mu_1$ and $\mu_2$ is then given by
\begin{equation*}
W_{T}(\mu_1,\mu_2) = \min_{1 \leq s \leq n} \sum_{i=1}^n \rho(s_{i+1},s_i) \,|\alpha_i - \alpha_s|,
\end{equation*}
where $\displaystyle \alpha_j = \sum_{k=1}^j \mu_k^{(1)} - \mu_k^{(2)}$ for $1 \leq j \leq n$, and $s_{n+1} \equiv s_1$.
\end{prop}

\begin{proof}[Proposition \ref{KR_on_circle_atomic_measures}]
By Corollary \ref{min_btwn_2_mk_expressions}, we know that 
\begin{equation*}
W_T(\mu,\nu) = \min \left( \inf_{s\in T} \int_0^1 |\alpha(x)-\alpha(s)| \diff x, \, \inf_{s \in T} \int_0^1 |\alpha(x) - \alpha(s-)| \diff x \right).
\end{equation*}
where $\alpha(x) = F_{\mu_1}(x) - F_{\mu_2}(x)$ with $F_{\mu_1}(x) = \mu_1([0,x])$ and $F_{\mu_2}(x) = \mu_2([0,x])$, respectively. 

Since both $\mu_1$ and $\mu_2$ are atomic measures, $\alpha$ is a simple function. Evaluating $\alpha$ at $s$ and $s^-$, we have 
\begin{equation*}
\alpha(s) = \sum_{k=1}^i \mu_k^{(1)} - \mu_k^{(2)} \, \mbox{ for } \, s_i \leq s < s_{i+1} \quad \mbox{and} \quad  \alpha(s^-) = \sum_{k=1}^i \mu_k^{(1)} - \mu_k^{(2)} \, \mbox{ for } \, s_i < s \leq s_{i+1},
\end{equation*}
with $\alpha(s) = 0$ for $s_n \leq s < s_1$ and $\alpha(s^-) = 0$ for $s_n <s \leq s_1$.

Note that the functions $\alpha(x) - \alpha(s)$ and  $\alpha(x) - \alpha(s^-)$ remain simple since $\alpha(s)$ and $\alpha(s^-)$ are  constants. Integrals of simple functions are well known (see Billingsley \cite{Billingsley_Convergence_of_Prob_measures}) thus a direct application of Corollary \ref{min_btwn_2_mk_expressions} for $\mu_1$ and $\mu_2$ yields the following Kantorovich-Rubinstein distance:
\begin{equation*}
W_T(\mu,\nu) = \min \left( \min_{1 \leq s \leq n} \sum_{i =1}^n \rho(s_{i+1},s_i) |\alpha(x)-\alpha(s)| , \, \min_{1 \leq s \leq n}\sum_{i =1}^n \rho(s_{i+1},s_i) |\alpha(x)-\alpha(s^-)| \right).
\end{equation*}
where $\displaystyle \alpha_j = \sum_{k=1}^j \mu_k^{(1)} - \mu_k^{(2)}$ for $1 \leq j \leq n$, and $s_{n+1} \equiv s_1$.

\noindent By construction of $\alpha(s)$ and $\alpha(s^-)$,  $\dps\min_{s_1 \leq s \leq s_n} \alpha(x)- \alpha(s) = \min_{s_1 \leq s \leq s_n} \alpha(x)- \alpha(s^-)$. Hence, we obtain 
\begin{equation*}
W_T(\mu_1,\mu_2) = \min_{1 \leq s \leq n} \sum_{i=1}^n \rho(s_{i+1},s_i) \,|\alpha_i - \alpha_s|.
\end{equation*}
\end{proof}

\cleardoublepage

\chapter{The Kantorovich-Rubinstein Distance and Statistical Trend Tests}\label{chapter_on _KR_distance_and_STT}

The most common design to test for association between a genetic marker and a disease is the case-control study. Basic test statistics for GWAS using case-control sample are reviewed by Balding \cite{tutorial_on_stat_method_for_gwas}. In this chapter, we start by studying two of the most commonly used test statistics for genetic association: Pearson's Chi-square test and the Chochran-Armitage trend test which assumes a dose-response effect between the genotype and the disease: the risk of disease increases with the number of risk allele (Sasieni \cite{Sasieni_From_genotypes_to_genes}).

In 2009, Zheng \textit{et al.} \cite{Pearson_and_CATT} showed that, in fact, Pearson's test is a trend test with unrestricted data-driven scores. By re-writting both test statistics in a more general form, we are able to significantly simplify the proofs of Zheng \textit{et al.} and build an upper and lower bound for the Pearson statistic that depends on the Kantorovich-Rubinstein distance. 

The chapter is organised as follows: the first section introduces the two test statistics. In the second section, we define a particular application $T$ and proceed to give the interesting properties of $T$. In section 3, we show how both the Pearson and the Cochran-Armitage statistic can be defined as multiples of $T$. Finally, in section 4, we construct an upper and lower bound for the Pearson statistic as functions of the Kantorovich-Rubinstein distance.

\section{The Cochran-Armitage Trend Test and the Pearson's Chi-square Test of Homogeneity}
The purpose of this section is to introduce two single variable tests of association which are well-known and commonly used in genome-wide association studies (GWAS): the Pearson chi-square test of independence and the Cochrane-Armitage test for trend.

\subsection{The Pearson's Chi-square Test of Homogeneity}

In the case of a multinomial distribution with joint probabilities $\pi_{ij}$, for $i = 1, \ldots, k$ and $j = 1, \ldots, l$, the sampling can be summarize by an $k \times l$ contigency table. The null hypothesis is that the joint probabilities $\pi_{ij}$ are equal to the product of their marginals $\pi_{i+}$ and $\pi_{+j}$. Hence the null hypothesis of the statistical test is $H_{\circ}: \pi_{ij} = \pi_{i+} \pi_{+j}$, for all $(i,j)$, where $\sum_i \sum_j \pi_{ij} = 1$. Since the marginal distributions $ \pi_{i+}$ and $\pi_{+j}$ are unknown, the sample marginal proportions $\hat{\pi}_{i+} = n_{i+}/n$ and $\hat{\pi}_{+j} = n_{+j}/n$  are used as estimates. Under the null hypothesis, one obtains the following statistic:
\begin{equation}\label{chi_square_general}
T_{\chi^2} = \sum_i^k \sum_j^l \frac{(o_{ij} - \hat{\varepsilon}_{ij})^2}{\hat{\varepsilon}_{ij}},
\end{equation}
where $o_{ij}$ is the number of observations of type $(i,j)$ and $\hat{\varepsilon}_{ij}$ is the estimated expected frequency of type $(i,j)$ under the null hypothesis (hence $\hat{\varepsilon}_{ij} = n \hat{\pi}_{i+} \hat{\pi}_{+j}$). \\
Replacing $\varepsilon_{ij}$ by the estimates $\hat{\varepsilon}_{ij}$ affects the distribution of the Pearson chi-squared statistic. Indeed, $\hat{\varepsilon}_{ij}$ require estimating both marginal distributions $ \pi_{i+}$ and $\pi_{+j}$ thus the degrees of freedom of the statistic is $(k-1)(l-1)$.

\subsubsection{The Pearson's Chi-square Test of Homogeneity Applied to Genetics}

To summarize the results of case-control samples for a single marker (in our case a single snp), one can build a $2 \times 3$ contingency table. The two rows summarize the outcome of controls and cases while the columns summarize the outcome of the genotype of homozygous major (AA), heterozygous(Aa) and homozygous minor (aa), respectively. Hence, we obtain categorical data with 2 variables. The row variable has two categories while the ordinal column variable has 3 categories. To build this contingency table, we will follow the notation used by Zheng \textit{et al.} in \cite{Pearson_and_CATT}. 
Denote the genotype counts for cases by $(r_0, r_1, r_2)$ and for controls by $(s_0,s_1,s_2)$. Let $r = r_0 + r_1 +r_2$ and $s = s_0 + s_1 + s_2$ be the number of cases and controls while $n_i = r_i + s_i$ (for $i = 0,1,2$) is the number of patients with each genotype. Let $n = r+s$ be the total number of patients and $p = r/n$, $q= s/n$ be the sample marginal proportions. Recall that $q = 1-p$.
With such notations, the estimated expected frequencies can be written as $\hat{\varepsilon}_{ij} = n \hat{\pi}_{r+} \hat{\pi}_{+j} = n \frac{r}{n} \frac{n_i}{n} = n_i \frac{r}{n} = n_i p$.

\begin{table}
\begin{center}
\begin{tabular}{cc|c|c|c|}
\cline{2-4}
 & \multicolumn{1}{ |c| }{aa}  & aA & AA  \\ 
\cline{1-5} 
\multicolumn{1}{ |c| }{Case} & $r_0$ & $r_1$ & $r_2$ & \multicolumn{1}{ c| }{$r$} \\
\cline{1-5}
\multicolumn{1}{ |c| }{Control} & $s_0$ & $s_1$ & $s_2$ & \multicolumn{1}{ c| }{$s$}  \\ 
\cline{1-5}
& \multicolumn{1}{ |c| }{$n_0$} & $n_1$ & $n_2$ & \multicolumn{1}{ c| }{$n$} \\
\cline{2-5}
\end{tabular}
\caption{Nomenclature for Contingency Table with one SNP}
\end{center}
\end{table}
\vspace{1cm}

The Pearson-Chi square test of homogeneity works under the null hypothesis that the probability distribution of any genotype is the same between case and control. Hence, the null hypothesis is $H_{\circ}: \pi_{ri} = \pi_{si}$ for $i = 0,1,2$, where $\sum_i \pi_{ri} = \sum_i \pi_{si} = 1$. Under the  null hypothesis, the statistic \ref{chi_square_general} is written in the following form :

\begin{equation}\label{Pearson_independence_statistic}
T_{\chi^2} = \sum_{i=0}^2 \frac{ \Big( \displaystyle r_i - n_i p \Big)^2 }{ \displaystyle n_i p} + \sum_{i=0}^2 \frac{ \Big( \displaystyle s_i- n_i q \Big)^2 }{ \displaystyle n_i q}. 
\end{equation}

The statistic \ref{Pearson_independence_statistic} can equivalently be written using the \textit{Brandt-Snedecor formula}:
\begin{equation} \label{Brandt-Snedecor-Pearson}
T_{\chi^2} = \frac{1}{pq}\sum_{i=0}^3 n_i (p_i - p)^2, \quad \mbox{where} p_i = \frac{r_i}{n_i}.
\end{equation}
Indeed,
\begin{align*}
T_{\chi^2} &= \sum_{i=0}^2 \frac{ \Big( r_i - n_i p \Big)^2 }{n_i p} + \sum_{i=0}^2 \frac{ \Big( s_i- n_i q \Big)^2 }{n_i q}  \\
&= \sum_{i=0}^2 \frac{ n_i^2 \Big( \displaystyle \frac{r_i}{n_i} - p \Big)^2 }{n_i p} + \sum_{i=0}^2 \frac{ n_i^2 \Big( \displaystyle \frac{s_i}{n_i} - q \Big)^2 }{n_i q} \\
&=  \sum_{i=0}^2 \frac{n_i}{p} \Big( p_i - p \Big)^2  + \sum_{i=0}^2 \frac{n_i}{q} \Big( q_i - q \Big)^2 \quad \mbox{where } \displaystyle q_i = \frac{s_i}{n_i},\\
&=  \sum_{i=0}^2 \frac{n_i}{p} \Big( p_i - p \Big)^2  + \sum_{i=0}^2 \frac{n_i}{q} \Big( p - p_i \Big)^2 \quad (\mbox{note that } p_i + q_i = 1), \\
&= \sum_{i=0}^2 n_i \Big( p_i - p \Big)^2 \Big( \displaystyle \frac{1}{p} + \frac{1}{q} \Big) \\
&= \frac{1}{pq} \sum_{i=0}^2 n_i \Big( p_i - p \Big)^2. \\
\end{align*}


\subsection{The Cochran-Armitage Trend Test} \label{CATT_section}

The Cochran-Armitage Trend Test (CATT) is a method of directing the Pearson's Chi-square Test towards narrow alternatives. Recall that for the Pearson's Chi-square Test, Pearson assumes that an $2 \times l$ contingency table represents the sampling of a $2l$-multinomial distribution. In the case of the CATT, Cochrane and Armitage assume that a $2 \times l$ contingency table represents the sampling of $J$ independent binomial random variables of parameter $(\rho_i, n_i)$ where the probability of success $\rho_i$ is the probability of either categories as a function of the ordinal categories.

As seen above, in the in the case of a case-control setting for a single snp, the ordinal column variable has 3 categories: the 3 genotypes. In genetics $\rho_i$ is the probability of being a case given the genotype. It is called the penetrance.
 CATT works under the null hypothesis that the probability of success $\rho_i$ of each random variable is the same. Hence the null hypothesis is $H_{\circ}: \rho_i = \rho_{\circ}$ for $i = 0,1,2$. Under the null  hypothesis, Cochrane and Armitage proposed the following trend statistic:
\begin{equation}\label{CAstatistic}
T_{CA}(c) = \frac{\bigg( \displaystyle\sum_{i=0}^{2}c_i ( q r_i - p s_i) \bigg)^2}{n p q \bigg( \displaystyle\sum_{i=0}^2 c_i^2 \frac{n_i}{n} - \Big( \sum_{i=0}^2 c_i \frac{n_i}{n} \Big)^2 \bigg)}
\end{equation}
where $c = (c_0, c_1, c_2)$ is a chosen triple of increasing scores such that $c_0 \leq c_1 \leq c_2$.

For large samples, the statistic $T_{CA}$ is asymptotically equal to $T_{\chi^2}$ with one degree of freedom.
 
The Cochran-Armitage statistic (\ref{CAstatistic}) can be rearranged to obtain the following expression:

\begin{equation}\label{Agresti_CA_stat}
T_{CA}(c) = \frac{b^2}{pq} \sum_{i=0}^2 n_i (c_i - \bar{c}). 
\end{equation}
where 
\begin{equation*} 
b = \frac{\sum n_i (p_i - p) (c_i - \bar{c})}{\sum n_i (c_i - \bar{c})^2} \, \mbox{ and } \, \bar{c} = \frac1n \sum_{i=0}^2 n_i c_i. 
\end{equation*}

In order to obtain the expression (\ref{Agresti_CA_stat}), we first show that 
\begin{equation*}
T_{CA}(c) = \frac{\bigg( \displaystyle\sum_{i=0}^{2} r_i ( c_i - \bar{c} ) \bigg)^2}{ p q  \displaystyle\sum_{i=0}^2 n_i ( c_i - \bar{c})^2 }.
\end{equation*}

To do so, we show that their respective numerators and denominators are equal. For the numerator we obtain:
\begin{align*}
\sum_{i=0}^{2} r_i ( c_i - \bar{c} ) &= \sum_{i=0}^{2} r_i \bigg( c_i - \sum_{i=0}^{2} c_i \frac{n_i}{n} \bigg) \\
&= \sum_{i=0}^{2} r_i c_i - \sum_{i=0}^{2} r_i \sum_{i=0}^{2} c_i \frac{n_i}{n} \\
&= \sum_{i=0}^{2} r_i c_i - \frac{1}{n} \sum_{i=0}^{2} r_i \sum_{i=0}^{2} c_i n_i \\ 
&= \sum_{i=0}^{2} r_i c_i - p \sum_{i=0}^{2} c_i (r_i + s_i) \\
&= (1 - p) \sum_{i=0}^{2} r_i c_i - p \sum_{i=0}^{2} c_i s_i \\
&= \sum_{i=0}^{2} c_i ( q r_i - p s_i ). 
\end{align*}

\noindent For the denominator we have: 
\begin{align*}
\sum_{i=0}^{2} n_i ( c_i - \bar{c} )^2 &= \sum_{i=0}^{2} n_i \big(c_i^2 - 2c_i\bar{c} + \bar{c}^2 \big) \\
&= \sum_{i=0}^{2} n_i c_i^2 - 2\bar{c} \sum_{i=0}^{2} c_i n_i + \bar{c}^2 \sum_{i=0}^{2}n_i \\
&= \sum_{i=0}^{2} n_i c_i^2 - 2 n \bar{c}^2 + \bar{c}^2 n \qquad (\mbox{recall that } n \bar{c}= \sum_{i=0}^{2} c_i n_i )\\
&= n \Big( \sum_{i=0}^{2} n_i c_i^2 - \bar{c}^2 \Big) \\
&= n \bigg( \sum_{i=0}^{2} c_i^2 \frac{n_i}{n} - \Big( \sum_{i=0}^2 c_i \frac{n_i}{n} \Big)^2 \bigg).
\end{align*}

\noindent Now, it is left to show that 
\begin{equation*}
\frac{\bigg( \displaystyle\sum_{i=0}^{2} r_i ( c_i - \bar{c} ) \bigg)^2}{pq  \displaystyle \sum_{i=0}^2 n_i ( c_i - \bar{c})^2 }  = \frac{b^2}{pq} \sum_{i=0}^2 n_i (c_i - \bar{c})^2.
\end{equation*}
Since $\displaystyle b = \frac{\sum n_i (p_i - p) (c_i - \bar{c})}{\sum n_i (c_i - \bar{c})^2}$, it is easy to see that  
\begin{equation*}
 \frac{b^2}{pq} \displaystyle \sum_{i=0}^2 n_i (c_i - \bar{c})^2 = \frac{ \Big( \displaystyle\sum_{i=0}^2 n_i (p_i - p) (c_i - \bar{c})\Big)^2 }{ pq  \displaystyle \sum_{i=0}^2 n_i (c_i - \bar{c})^2 }.
\end{equation*}

\noindent Thus to finish the proof we show that $\displaystyle\sum_{i=0}^2 n_i (p_i - p) (c_i - \bar{c}) = \displaystyle\sum_{i=0}^2 r_i ( c_i - \bar{c} )$:
\begin{align*}
\sum_{i=0}^2 n_i (p_i - p) (c_i - \bar{c}) &= \sum_{i=0}^2 (r_i - n_i p) (c_i - \bar{c}) \qquad (\mbox{recall that } r_i = n_i p_i) \\
&= \sum_{i=0}^2 r_i c_i - \bar{c} \sum_{i=0}^2 r_i  - p \sum_{i=0}^2 n_i c_i  + p \bar{c} \sum_{i=0}^2 n_i  \\
&= \sum_{i=0}^2 r_i c_i - \bar{c} r - p n \bar{c} + p \bar{c} n \qquad (\mbox{recall that } n \bar{c}= \sum_{i=0}^{2} c_i n_i )\\
&= \sum_{i=0}^2 r_i ( c_i - \bar{c} ).
\end{align*}

\begin{rmk}[On the Choice of Scores for the CATT]

To test $H_{\circ}$ using a CATT, the score $c = (c_1, c_2, c_3)$ is assigned to the genotypes $(aa, aA, AA)$ such that $ 0 \leq c_0 \leq c_1 \leq c_2$ or $c_2 \leq c_1 \leq c_0 \leq 0$. In a 1997 paper \cite{Sasieni_From_genotypes_to_genes}, Sasieni assigned $c = (0,0,1)$ to the recessive model, $c = (0, 1/2, 1)$ to the additive model, and $c = (0,1,1)$ to the dominant model. The intuition underlying the scores is the following: for the recessive model, the relative risks of the genotypes $aa$ and $aA$ are the same, so the same score is assigned to $aa$ and $aA$. Likewise, for the dominant model, the same score is assigned to $aA$ and $AA$. For the additive model, the effect of $aA$ should be the average of the effect of $aa$ and $AA$.
\end{rmk}

\subsection{Relationship between the Pearson and Cochran-Armitage statistics}  

Recall that, in the case of CATT applied to genetics, one assumes that a $2 \times 3$ contingency table represents the sampling of 3 independent binomial random variables of parameters $(\rho_i, n_i)$ where $\rho_i$ is the penetrance (the probability of being a case given the genotype).\\
CATT works by looking for the presence of a linear trend among penetrance, across the 3 genotypes. The trend test may give strong evidence of positive or increasing linear trends, of constant or stable trends over time, or of negative or decreasing trends. Cochrane and Armitage used a linear model of the form $\rho_i = \alpha + \beta c_i$, fitted by the ordinary least squares method.
Recall that $p_i = r_i/n_i$, $p = r/n$ and $\bar{c} =  \sum_i n_i c_i /n$. The prediction equation for $\rho_i$, given by Agresti (section 5.3.5 \cite{Agresti_Cat_data_analysis} is:
\begin{equation*}
\hat{\rho}_i = p + b(c_i - \bar{c}),
\end{equation*} 
where 
\begin{equation*}
b = \frac{\sum_i n_i (p_i - p) (c_i - \bar{c})}{\sum_i n_i (c_i - \bar{c})^2}. 
\end{equation*}

To build the Cochran-Armitage statistic, Cochran partitioned the Brandt-Snedecor formula (\ref{Brandt-Snedecor-Pearson}) of the Pearson Chi-square statistic. Indeed, in 1954, Cochran noted that the \textit{Brandt-Snedecor formula} (\ref{Brandt-Snedecor-Pearson}) decomposes into 
\begin{equation}\label{decomposition_Pearson_formula}
 T_{\chi^2} =T_{CA}(x) + T_{fit},
\end{equation}
where
$\displaystyle T_{fit} = \frac{1}{p(1-p)} \sum_{i=0}^2 n_i (p_i - \hat{\rho}_i)^2.$ \vspace{0.5cm}

\noindent To verify decomposition (\ref{decomposition_Pearson_formula}), we first note that
\begin{align*}
\sum_{i=0}^2 n_i (p_i - \hat{\rho}_i)^2 &= \sum_{i=0}^2 n_i \big( (p_i - p) - b(c_i -\bar{c}) \big)^2 \\
   &= \sum_{i=0}^2 n_i \big( (p_i - p)^2 - 2 (p_i-p) b (c_i - \bar{c}) - b^2 (c_i -\bar{c})^2 \big) \\
   &= \sum_{i=0}^2 n_i (p_i - p)^2 - 2b \sum_{i=0}^2 n_i (p_i - p) (c_i -\bar{c}) + b^2 \sum_{i=0}^2 n_i (c_i - \bar{c})^2.
\end{align*}

\noindent Hence, we have the following equality:
\begin{equation*}
\displaystyle T_{CA}(x) + T_{fit} = T_{\chi^2} - \frac{2b^2}{p(1-p)} \sum_{i=0}^2 n_i (p_i - p) (c_i -\bar{c}) + \frac{2b^2}{p(1-p)} \sum_{i=0}^2 n_i (c_i - \bar{c})^2.
\end{equation*}

Thus, it is left to show that $ \sum n_i (p_i - p) (c_i -\bar{c}) = b \sum n_i (c_i - \bar{c})^2$. But this is true by construction since the denominator of b is $\sum n_i (c_i - \bar{c})^2$ and the numerator of b is $\sum n_i (p_i - p) (c_i -\bar{c})$. 

When the linear probability model holds, $T_{fit}$ is asymptotically chi-squared with $df = 1$ and the statistic $z^2$ , based on df = 1, tests for a linear trend in the proportions. 
Note that an equivalent way to express the null hypothesis is to say that the slope $\beta$ of the linear model is zero. That is, $H_{\circ}: \beta = 0$.

\section{The function $T$}
In this section, we compare the Kantorovich-Rubinstein distance with both the Cochran-Armitage statistic and the Pearson statistic. To do so, we write these two statistics as follows. 

\begin{defn}\label{defn_T}
Consider $(\XX, \mu)$, a probability space and define $L_2(\mu)_*$ and $L_1^+(\mu)$ as follow:
\begin{align*}
L_2(\mu)_{*} &= \{ f \in L_2(\mu): f \not= constant \}, \\
L_1^+(\mu) &= \Big\{ f \in L_1(\mu): f(x) \geq 0 \, \mu\mbox{-a.e.} , \mbox{ and } 0 \leq ||f||_1 \leq 1 \Big\}.
\end{align*}

We now define the application $T: L_2(\mu)_{*} \times L_1^+(\mu) \longrightarrow \RR$ by:
\begin{equation}\label{def_T}
T(c, \alpha) = \frac{\Big( \displaystyle \int c (\alpha - m) \diff \mu \Big)^2}{V_{\mu}(c)},
\end{equation}
where $\displaystyle V_{\mu}(c) = \int c^2 \diff \mu - \Big( \int c \diff \mu \Big)^2$ and $m = \displaystyle \int \alpha \diff \mu$. 
\end{defn}

\begin{rmk}\label{equality_of_T}
If the function $\alpha \in L_1^+(\mu)$ is such that $0 \leq \alpha \leq 1$, then $T(c, \alpha) = T(c, 1- \alpha)$.
\end{rmk}

\begin{proof}[Remark \ref{equality_of_T}]
It is enough to show that $\Big( \displaystyle \int c (\alpha - m_{\alpha}) \diff \mu \Big)^2 = \Big( \displaystyle \int c \big((1-\alpha) - m_{1-\alpha} \big) \diff \mu \Big)^2$. \\
Since $m_{1-\alpha} = \displaystyle \int (1- \alpha) \diff \mu = \int 1 \diff \mu - \int \alpha \diff \mu = 1- m_{\alpha}$, we have:
\begin{align*}
\Big( \displaystyle \int c \big((1-\alpha) - m_{1-\alpha} \bigg) \diff \mu \Big)^2 &= \Big( \displaystyle \int c \big(1-\alpha - (1-m_{\alpha}) \big) \diff \mu \Big)^2 \\
= \Big( \displaystyle \int c (-\alpha + m_{\alpha}) \diff \mu \Big)^2 &= \Big( \displaystyle \int c (\alpha - m_{\alpha}) \diff \mu \Big)^2.
\end{align*}
\end{proof}

\begin{rmk}\label{rmk_on_V_mu} Consider $\displaystyle V_{\mu}(c) = \int c^2 \diff \mu - \Big( \int c \diff \mu \Big)^2$. Then,
\begin{enumerate}
\item[(i)] $V_{\mu}(c) \geq 0$ for all $c \in L_2(\mu)$,
\item[(ii)] For $c \in L_2(\mu)$, $V_{\mu}(c) = 0$ if and only if the $c=const, \,\mu$-a.e. 
\item[(iii)] $\displaystyle \int c (\alpha - m) \diff \mu = 0$  if $c=const, \,\mu$-a.e. 
\end{enumerate}
\end{rmk}

\begin{proof}[Remark \ref{rmk_on_V_mu}]
Recall the Cauchy-Schwarz  inequality: 
\begin{equation*}
\left| \int f g \diff \mu \right|^2 \leq  \int f^2 \diff \mu  \int g^2 \diff \mu.
\end{equation*}
By the Cauchy-Schwarz inequality applied to $c \in L_2(\mu)$, we obtain:
\begin{enumerate}
\item[(i)] $\dps  \int c^2 \diff \mu  \int 1 \diff \mu - \Big( \int c \diff \mu \Big)^2 \geq 0$ 
\item[(ii)] Moreover, since $\displaystyle \int 1 \diff \mu = \mu(\XX) = 1$, then $\displaystyle \int c^2 \diff \mu  \int 1 \diff \mu - \Big( \int c \diff \mu \Big)^2 = 0$ if and only if $\displaystyle \int c^2 \diff \mu = \Big( \int c \diff \mu \Big)^2$. That is, if and only if $V_{\mu}(c) = 0$.
\item[(iii)] Since $\mu(\XX)=1$ and  $\dps \int \alpha \diff \mu = m$, then$\dps \int (\alpha - m) \diff \mu = \int \alpha \diff \mu - \int m \diff \mu = 0.$ Thus, $\dps \int c (\alpha - m) \diff \mu = c \int (\alpha - m) \diff \mu = 0$.
\end{enumerate}
\end{proof}

To better understand the behaviour of the function $T$, we now prove important properties:

Let $G \subset GL_2(\RR)$ be the group defined by
\begin{equation*}
G = \left\{ \left( 
\begin{array}{cc}
a & b \\
0 & 1 \\
\end{array} \right)
 ; a \not = 0 , b \in \RR \right\}.
\end{equation*}
and its action on $L_2(\mu)_*$ given, for $g \in G$ and  $c \in L_2(\mu)_*$, by
\begin{equation*}
(gc)(x) = a c(x) + b, \mbox{ for all } x \in \XX. 
\end{equation*}
Note that if, for every $g \in G$, $c \in L_2(\mu)_*$, then $gc \in L_2(\mu)_*$. \\
Let $\sim_G$ denote the equivalence relation on $L_2(\mu_*)$ given by
\begin{equation*}
c_1 \sim_G c_2 \mbox{ if and only if } \exists \,g \in G \mbox{ such that } gc_1 = c_2.
\end{equation*}

\begin{lem}\label{equalT_for_equiv_c}
Let $c_1,c_2 \in L_2(\mu)_*$ be such that $c_1 \sim_G c_2$. Then the two functions $T(c_1, \cdot)$ and $T(c_2, \cdot)$, from $ L_1^+(\mu)$ to $\RR_+$ are equal.
\end{lem}

\begin{proof}[Lemma \ref{equalT_for_equiv_c}]
Let $c_1,c_2 \in L_2(\mu)_*$ and $g = \begin{psmallmatrix}a&b\\0&1\end{psmallmatrix} \in G$ be such that $gc_1 = c_2$. We show separately that $V_{\mu}(c_2) = a^2 V_{\mu}(c_1)$ and that$\dps \int c_2(\alpha-m) \diff \mu = a \int c_1(\alpha-m) \diff \mu$:
\begin{align*}
V_{\mu}(c_2) &= \int (ac_1+b)^2 \diff \mu - \Big( \int (ac_1+b)  \diff \mu \Big)^2 \\
&= \int (ac_1)^2 \diff \mu + b^2 + 2b \int c_1 \diff \mu - \Big( \int c_1 \diff \mu + b \Big)^2 \\
&= a^2 V_{\mu}(ac_1) + b^2 + 2b \int c_1 \diff \mu - b^2 - 2b \int c_1 \diff \mu = a^2 V_{\mu} (c_1).
\end{align*}
$\dps \int c_2(\alpha-m) \diff \mu =  \int (ac_1 + b) (\alpha-m) \diff \mu = a \int c_1(\alpha-m) \diff \mu + b \int (\alpha-m) \diff \mu$. Since we know by Remark \ref{rmk_on_V_mu} that $\int (\alpha-m) \diff \mu = 0$, we obtain the desired result.
\end{proof}

To prove the converse (Property \ref{c1_equiv_c2_ssi_T(c1)_equivT(c2)}), we will need the following two lemmas:

\begin{lem}\label{value_of_T_wlog}
Let $c \in L_2(\mu)_*$. Then, there exists $\tilde{c} \in  L_2(\mu)_*$ such that 
\begin{equation*}
c \sim_G \tilde{c} \quad \mbox{and} \quad  T(\tilde{c}, \alpha) = \Big( \displaystyle \int \tilde{c} \alpha \diff \mu \Big)^2,
\end{equation*}
for all $\alpha \in L_1^+(\mu)$.
\end{lem}

\begin{proof}[Lemma \ref{value_of_T_wlog}]
Define $g_1,g_2 \in G$ by 
\begin{equation*}
g_1 = \left( 
\begin{array}{cc}
1 & -\int c \diff \mu \\
0 & 1 \\
\end{array} \right),
\quad
g_2 = \left( 
\begin{array}{cc}
a_2 & 0 \\
0 & 1 \\
\end{array} \right)
\quad \mbox{ where }
a_2 =\bigg( \displaystyle \int (g_1c)^2 \diff \mu \bigg)^{-1/2},
\end{equation*} 
and set  $\tilde{c} = g_2g_1 c$. Then $\displaystyle \int \tilde{c} \diff \mu = 0$ and $\displaystyle \int (\tilde{c})^2 \diff \mu = 1$. Indeed,
\begin{equation*}
\int g_1 c \diff \mu = \int c \diff \mu - \mu (\XX) \int c \diff \mu = 0. 
\end{equation*}
Therefore,
\begin{equation*}
 \int \tilde{c} \diff \mu = \int g_2g_1 c \diff \mu = \frac{1}{a_2} \int g_1c \diff \mu = 0, \mbox{ and }  \int\tilde{c}^2 \diff \mu = \frac{1}{a_2^2} \int (g_1c)^2 \diff \mu = 1. \mbox{ Thus } V_{\mu}(\tilde{c}) = 1.
\end{equation*}
Since $V_{\mu}(\tilde{c}) = 1$ and $\displaystyle \int \tilde{c} \diff \mu =0 $, we have $T(\tilde{c}, \alpha) = \Big( \displaystyle \int \tilde{c} \alpha \diff \mu \Big)^2$.
\end{proof}

\begin{lem}\label{c1=c2_or_c1=-c2}
Let $(\XX, \FF, \mu)$ be a finite probability space and let $c_0$  and $c_1$ be two measurable functions from $\XX$ to $\RR$. Given $p \in (0,1)$, define the set $L_{\infty}^+(\mu)_p$  by 
\begin{equation*}
L_{\infty}^+(\mu)_p = \Big\{ f \in L_{\infty}(\mu): \int f \diff \mu  = p, \, f(x) \geq 0, \, \mu\mbox{-a.e.} \Big\}.
\end{equation*}
If, for all $f \in L_{\infty}^+ (\mu)_p$,
\begin{equation*} 
\bigg( \displaystyle \int c_1 f \diff \mu \bigg)^2 =\bigg( \displaystyle \int c_0 f \diff \mu \bigg)^2, \mbox{ then either } c_1(x) = c_0(x) \mbox{ or } c_1(x) = -c_0(x), \, \mu \mbox{-a.e.}.
\end{equation*}
As for any $B \in \FF, \, \mu(B)>0$, the function $f = \frac{p}{\mu(B)} \indicator_B \in L_{\infty}^+(\mu)_p$. 
\end{lem}
\noindent Lemma \ref{c1=c2_or_c1=-c2} is a direct consequence of the following lemma:

\begin{lem}\label{lemma_for_lemma_c1=c2_or_c1=-c2}
Let $(\XX, \FF, \mu)$ be a probability space and let $c_0$  and $c_1$ be two measurable functions from $\XX$ to $\RR$.
If, for all $A \in \FF$,
\begin{equation*}
\bigg( \displaystyle \int c_1 \indicator_A \diff \mu \bigg)^2 =\bigg( \displaystyle \int c_0 \indicator_A \diff \mu \bigg)^2, \mbox{ then either } c_1(x) = c_0(x) \mbox{ or } c_1(x) = -c_0(x), \, \mu \mbox{-a.e.}.
\end{equation*}
\end{lem}

\begin{proof}[Lemma \ref{lemma_for_lemma_c1=c2_or_c1=-c2}]
Let us define $Y = \{x \in \XX; \, |c_1(x)| > |d(x)|\}$, $Z = \{x \in \XX; \, |c(x)| < |d(x)|\}$ and $\XX_e = \XX \setminus (Y \cup Z)$. 
We will show that $\mu(Y) = 0$. Note that, by interchanging the functions $c_0$ and $c_1$ in $Y$ and $Z$, we can show that $\mu(Z) = 0$.

\noindent Let us partition $Y$ as $\dps \coprod_{i,j \in \{0,1\}} Y_{i,j}$ where
\begin{equation*}
Y_{i,j} = \left\{ x \in Y; \sgn \big( c(x) \big) = (-1)^i \mbox{ and } \sgn \big( d(x) \big) = (-1)^j \right\}
\end{equation*}
and show that $\mu(Y_{i,j}) = 0$, for all $i,j$.

Suppose that there exist $i_{\circ},j_{\circ}$ such that $\mu(Y_{i_{\circ},j_{\circ}}) >0$ and denote, to simplify notation, $Y_{i_{\circ},j_{\circ}}$ by $Y_{\circ}$. For $x \in Y_{\circ}$, $0 \leq (-1)^{j_{\circ}}d(x) < (-1)^{i_{\circ}} c(x)$. Therefore, 
\begin{equation*}
\int_{Y_{\circ}} (-1)^{j_{\circ}} d \diff \mu < \int _{Y_{\circ}} (-1)^{i_{\circ}} c \diff \mu \quad \mbox{and} \quad \left| \int _{Y_{\circ}} d \diff \mu \right| < \left| \int_{Y_{\circ}} c \diff \mu \right|
\end{equation*}
which is impossible since, by assumption,
$ \dps \left( \int _{Y_{\circ}} d \diff \mu \right)^2 = \left( \int_{Y_{\circ}}  c \diff \mu \right)^2$.

\noindent For $i,j \in \{0,1\}$, set $\XX_e^{ij} = \{ x \in \XX; \, (-1)^i c(x) = (-1)^j d(x) >0\}$ and $\tilde{\XX}_e = \{x \in \XX; c(x) = d(x) = 0\}$.  For $i \in \{0,1\}$, let $\XX_e^i = \{x \in \XX; \, d(x) = (-1)^i c(x) \not= 0\}$. Then $\XX_e^i = \XX_e^{i,i+1} \coprod \XX_e^{i+1,i}$, where the indices are computed modulo 2. 

Claim 1 below completes the proof of the Lemma while Claim 2 is used in the proof of Claim 1.

\noindent \underline{Claim 1}: If $\mu(\XX_e^i) > 0$, then $\mu(\XX_e^{i+1}) = 0$.

\noindent \underline{Claim 2}: Let $k \in \{i,j\}$. If $\mu(\XX_e^{i,j}) > 0$ (respectively, $\mu(\XX_e^{k,k}) > 0$), then
\begin{equation*}
\int_{E_{i,j}} c_1 \diff \mu = \epsilon \int_{E_{i,j}} c_0 \diff \mu,
\end{equation*}
where $E_{i,j} = \XX_e^{i,j} \coprod \XX_e^{k,k}$ and $\epsilon = (-1)^{i-j} = -1$ (respectively $\epsilon = (-1)^{k-k} = 1$).\\

\noindent \underline{Proof of Claim 2}: By the assumption of the lemma,
\begin{equation*}
\int_{E_{i,j}} c_1 \diff \mu = \epsilon \int_{E_{i,j}} c_0 \diff \mu, \quad \mbox{ with } \epsilon \in \{\pm 1\}.
\end{equation*}
Then, as $ \dps \int_{E_{i,j}} c_1 \diff \mu = - \int_{\XX_e^{i,j}} c_0 \diff \mu + \int_{\XX_e^{k,k}} c_0 \diff \mu $, we obtain 
\begin{equation*}
(\epsilon + 1) \int_{\XX_e^{i,j}} c_0 \diff \mu = (1 - \epsilon) \int_{\XX_e^{k,k}} c_0 \diff \mu.
\end{equation*}

Then, if $\mu(\XX_e^{i,j}) > 0$ (respectively $\mu(\XX_e^{k,k}) > 0$), then $\epsilon = -1$ (respectively $\epsilon = 1$). Indeed, if not, we would have $\dps \int_{\XX_e^{i,j}} c_0 \diff \mu = 0 \,\, (\mbox{respectively } \, \int_{\XX_e^{k,k}} c_0 \diff \mu = 0)$, which is impossible.\\

\noindent \underline{Proof of Claim 1}: As $\XX_e^i = \XX_e^{i,0} \amalg \XX_e^{i+1,1}$, if $\mu(\XX_e^i) > 0$, then either $\mu(\XX_e^{i,0}) > 0$ or $\mu(\XX_e^{i+1,1}) > 0$. To prove the claim, we then need to show that $\mu(\XX_e^{i+1,0}) = \mu(\XX_e^{i,1}) = 0$, as $\XX_e^{i+1} = \XX_e^{i+1,0} \amalg \XX_e^{i,1}$.

We start with the case $i =0$. If $\mu(\XX_e^{k,k}) > 0$, for $k=0$ or $k=1$, then by Claim 2 applied to $\XX_e^{j,j+1} \amalg \XX_e^{k,k}$, for $j=0$ or $j=1$, we obtain 
\begin{equation*}
\int_{E_{j,j+1}} c_1 \diff \mu = \int_{E_{j,j+1}} c_0 \diff \mu \quad \mbox{with }\, E_{j,j+1} = \XX_e^{j,j+1} \amalg \XX_e^{k,k}.
\end{equation*}
\begin{equation*}
\mbox{As }\, \int_{X_e^{j,j+1}} c_1 \diff \mu = - \int_{X_e^{j,j+1}} c_0 \diff \mu, \, \mbox{ then }\, \int_{E_e^{j,j+1}} c_0 \diff\mu = 0,
\end{equation*}
and therefore $\mu(\XX_e^{j,j+1}) = 0$, for $j=0$ or 1 ($\sgn(c)$ is constant on $\XX_e^{j,j+1}$).\\

\noindent For the case $i=1$, ie. $\mu(\XX_e^1)>0$, then $\mu(\XX_e^{1,0})>0$ or $\mu(\XX_e^{0,1})$. Then, to show that $\mu(\XX_e^{0,0}) = \mu(\XX_e^{1,1}) = 0$, we apply, as above, Claim 2 to $\XX_e^{j,j+1} \amalg \XX^{k,k}$, for $j=0$ or 1 and $k=0$ or 1.
\end{proof}

To prove Property \ref{c1_equiv_c2_ssi_T(c1)_equivT(c2)}, we define, for $p \in(0,1)$ fixed, the set $L_1^+(\mu)_p$ defined by
\begin{equation*}
L_1^+(\mu)_p =\{ f \in L_1(\mu); \, \int f \diff \mu = p, \, f(x) \leq 0, \mu\mbox{-a.e.}\}.
\end{equation*}
Consider, for $c \in L_2(\mu)_*$ fixed, the function $T_c: L_1^+(\mu) \rightarrow \RR$ defined by $T_c(\alpha) = T(c,\alpha)$.
Let us denote by $T_{p,c}$ the restriction of  $T_c$ to $L_1^+(\mu)_p$.

\begin{prpty}\label{c1_equiv_c2_ssi_T(c1)_equivT(c2)}
Let $c_0, c_1 \in L_2(\mu)_*$ and $p \in (0,1)$ be fixed. If $T_{p,c_0} = T_{p,c_1} $, then, $c_0 \sim_G c_1$. 
\end{prpty}

\begin{proof}[Property \ref{c1_equiv_c2_ssi_T(c1)_equivT(c2)}]
Consider $c_0, c_1 \in L_2(\mu)_*$. By Lemma \ref{value_of_T_wlog}, we can suppose, without loss of generality, that 
\begin{equation}\label{restriction_for_T}
T(c_0, \alpha) = \bigg(\int c_0 \alpha \diff \mu \bigg)^2 \quad \mbox{and} \quad T(c_1, \alpha) = \bigg( \int c_1 \alpha  \diff \mu \bigg)^2, \mbox{ for all } \alpha \in L_1^+(\mu)_p.
\end{equation}
Since both equalities in (\ref{restriction_for_T}) are true for all $\alpha \in L_1^+(\mu)_p$, they remain true for all $\alpha \in L_{\infty}^+ (\mu)_{\circ}$. \\
As well, as seen in the proof of  lemma \ref{value_of_T_wlog}, $\displaystyle \int c_i \diff \mu = 0$ and $\displaystyle \int c_i^2 \diff \mu = 1$ for $i = 0,1$.

Now, by assumption, $T_{p,c_0}( \alpha) = T_{p,c_1}(\alpha)$, for all $\alpha \in L_1^+(\mu)_p$. Hence, we obtain that 
\begin{equation*}
 \bigg(\int c_0 \alpha \diff \mu \bigg)^2 = \bigg( \int c_1 \alpha  \diff \mu \bigg)^2, \mbox{ for all } \alpha \in L_{\infty}^+ (\mu)_{\circ}.
\end{equation*}
Therefore, by Lemma \ref{c1=c2_or_c1=-c2}, we know that either $c_1 = c_2$ or $c_1 = -c_2$. Thus $c_0 \sim_G c_1$.
\end{proof}

From now on, for any $\alpha \in L_1^+(\mu)_p$ fixed, we study the function $T_{\alpha}: L^2(\mu)_* \rightarrow \RR$ defined by $T_{\alpha}(c) = T(c, \alpha)$.

In Corollary (\ref{sup_of_T}), we show that, $\sup T_{\alpha} = T(\alpha, \alpha)$, where the supremum is taken over all $c \in L_2(\mu)_*$. Corollary (\ref{sup_of_T}) will be a consequence of propositions (\ref{proposition_of_b}) and (\ref{proposition_of_S}).

\begin{prop}\label{proposition_of_b}
Let $p \in (0,1)$ be fixed and $\alpha \in L_1^+(\mu)_p$. Consider the function $b: L_2(\mu)_* \rightarrow \RR$ defined by 
\begin{equation}\label{def_de_b}
b(c) = \frac{\displaystyle \int \big( \alpha - p \big) \big(c -\int c \diff \mu\big) \diff \mu}{\displaystyle \int \big(c - \int c \diff \mu\big)^2 \diff \mu}.
\end{equation}
where $p = \dps \int \alpha \diff \mu$. Then we have:
\begin{enumerate}
\item[(i)] If $d \in \RR$ and $e \in \RR^*$, then $b(c+d) = b(c)$ and $b(ec) = \frac1e b(c)$;
\item[(ii)] The denominator of $b(c)$ is equal to $V_{\mu}(c)$. That is, 
\begin{equation*}
\int \big(c - \int c \diff \mu\big)^2 \diff \mu = \int c^2 \diff \mu - \Big( \int c \diff \mu \Big)^2;
\end{equation*}
\item[(iii)] The numerator of $b(c)$ can be simplified to $\displaystyle \int c(\alpha-p) \diff \mu$;
\item[(iv)] $T(c, \alpha) = b^2(c) V_{\mu}(c)$.
\end{enumerate}
\end{prop}

\begin{proof}[Proposition \ref{proposition_of_b}]
\begin{enumerate}
\item[(i)] The linearity of integrals makes this result clear. 
\item[(ii)]  \begin{align*}
\int \big( c - \int c \diff \mu \big)^2 \diff \mu &= \int \Big[ c^2 + \big( \int c \diff \mu \big)^2 - 2c \big( \int c \diff \mu \big) \Big] \diff \mu \\
&= \int c^2 \diff \mu + \big( \int c \diff \mu \big)^2 - 2 \big( \int c \diff \mu \big)^2 \\
&= \int c^2 \diff \mu - \big( \int c \diff \mu \big)^2 = V_{\mu}(c).
\end{align*} 
\item[(iii)]  $\displaystyle \int \big( \alpha - p \big) \big(c -\int c \diff \mu\big) \diff \mu = \int c(\alpha-p) \diff \mu - \int c \diff \mu \,\, \int (\alpha-p) \diff \mu$. \\
Since $\displaystyle \int \alpha \diff \mu = p$, we have $\displaystyle \int (\alpha-p) \diff \mu = 0$. 
\item[(iv)] A direct application of (ii) and (iii) yields the first equation of (iv). 
\end{enumerate}
\end{proof}

\begin{prop}\label{proposition_of_S}
Let $p \in (0,1)$ be fixed and $\alpha \in L_1^+(\mu)_p \cap L_2^+(\mu)$. Consider the function $S_{\alpha}: L_2(\mu)_* \rightarrow \RR$ defined by
\begin{equation}\label{def_de_S}
S_{\alpha}(c) = \int \Big( \alpha - p - b(c) \big( c - \int c \diff \mu \big) \Big)^2 \diff \mu.
\end{equation}
Then we have:
\begin{enumerate}
\item[(i)] $S_{\alpha}(c) \geq 0$;
\item[(ii)] If $c_0 \sim_G c_1$, then $S_{\alpha}(c_0) = S_{\alpha}(c_1)$;
\item[(iii)] $T_{\alpha}(\alpha) = T_{\alpha}(c) + S_{\alpha}(c)$.
\end{enumerate}
\end{prop}

\begin{proof}[Proposition \ref{proposition_of_S}]
\begin{enumerate}
\item[(i)] $S_{\alpha} \geq 0$ since it is defined as an integral of a positive function with respect to a probability measure.
\item[(ii)] If $c_0 \sim_G c_1$, then $b(c_1) = \frac1p b(c_0)$ and $\, c_1 - \displaystyle \int c_1 \diff \mu  = p \big( c - \int c \diff \mu \big)$. Hence $S_{\alpha}(c_0) = S_{\alpha}(c_1)$.
\item[(iii)] 
\begin{align}
S_{\alpha}(c) & + T_{\alpha}(c) \nonumber\\
&=  \int \Big( \alpha - p - b(c) \big( c - \int c \diff \mu \big) \Big)^2 \diff \mu + b^2(c) \displaystyle \int \big( c -  \int c \diff \mu \big)^2 \diff \mu \nonumber \\
&= \int \bigg( (\alpha - p)^2 + b^2(c) \big( c - \displaystyle \int c \diff \mu \big)^2 - 2 b(c) (\alpha - p) \big( c - \int c \diff \mu \big) \bigg) \diff \mu \nonumber \\
& \hspace{7.5cm} + b^2(c) \displaystyle \int \big( c -  \int c \diff \mu \big)^2 \diff \mu \nonumber \\
&= \int (\alpha - p)^2 \diff \mu + 2 b^2(c)  \int \big( c -  \int c \diff \mu \big)^2 \diff \mu \nonumber \\
& \hspace{6.5cm}- 2 b(c) \int (\alpha - p) \big( c - \int c \diff \mu \big) \diff \mu \nonumber \\
& = \int (\alpha - p)^2 \diff \mu + 2 b(c) \bigg( b(c)  \int \big( c -  \int c \diff \mu \big)^2 \diff \mu - \int c (\alpha - p) \diff \mu \bigg).\label{last_equation}
\end{align}
Now, recall from Proposition \ref{proposition_of_b} that the function $b$ can be written as 
\begin{equation*}
b(c) = \frac{1}{V_{\mu}(c)} \int c (\alpha - p) \diff \mu.
\end{equation*}
Hence Equation (\ref{last_equation}) becomes: 
\begin{equation*}
 \int (\alpha - p)^2 \diff \mu + 2 b(c) \Big( b(c) V_{\mu}(c) - V_{\mu}(c) b(c) \Big) =  \int (\alpha - p)^2 \diff \mu.
\end{equation*}

Using formula (\ref{def_T}) that defines $T$, we obtain the following:
\begin{equation*}
T(\alpha, \alpha) = \frac{\Big( \displaystyle \int \alpha (\alpha - p) \diff \mu \Big)^2}{V_{\mu}(\alpha)} = \frac{\Big( \displaystyle \int \alpha^2 \diff \mu - p \int \alpha \diff \mu \Big)^2}{ \displaystyle \int \alpha^2 \diff \mu - \Big( \int \alpha \diff \mu \Big)^2} =  \displaystyle \int \alpha^2 \diff \mu - p^2,
\end{equation*}
To finish the proof, one needs to show that $T(\alpha, \alpha) =  \displaystyle \int (\alpha - p)^2 \diff \mu$. By expanding $(\alpha - p)^2$, one obtains
\begin{equation*}
 \int (\alpha - p)^2 \diff \mu = \int \alpha^2  \diff \mu - 2 p \int \alpha \diff \mu + p^2 \int \diff \mu = \int \alpha^2 \diff \mu - p^2.
\end{equation*}
\end{enumerate}
\end{proof}

\begin{cor}\label{sup_of_T}
Let $p \in (0,1)$ be fixed and $\alpha \in L_1^+(\mu)_p \cap L_2^+(\mu)$. Then,
\begin{equation*}
\sup\{ T_{\alpha}(c); \,c \in L_2(\mu)_* \} = T(\alpha, \alpha), \quad \mbox{where} \quad  T(\alpha, \alpha) = ||\alpha - p||_2^2.
\end{equation*}
\end{cor}

\begin{proof}[Corollary \ref{sup_of_T}]
Recall the function $S_{\alpha}: L_2(\mu)_* \rightarrow \RR$ defined in Equation (\ref{def_de_S}) in Proposition \ref{proposition_of_S} and the function $b: L_2(\mu)_* \rightarrow \RR$ defined in Equation (\ref{def_de_b}) in Proposition \ref{proposition_of_b}. Evaluated for $c = \alpha$, we obtain $b(\alpha) = 1$ and  thus $S_{\alpha}(\alpha) = 0$. Since, by Proposition \ref{proposition_of_S}, we know that $T_{\alpha}(c) = T_{\alpha} (\alpha) - S_{\alpha}(c)$ with $S(c, \alpha) \geq 0$, $T_{\alpha}$ attains its supremum at $\alpha$.

Now we show that $T_{\alpha}(\alpha) = ||\alpha - p||_2^2$. As shown in the proof of Proposition \ref{proposition_of_S} part $(iii)$, $T_{\alpha}(\alpha) = \dps \int (\alpha - p)^2 \diff \mu$. 
\end{proof}

In the last part of this section, we study the function $T_{\alpha}$ in the following special case:
\begin{enumerate}
\item[(i)] $(\XX, \mu)=(\{x_0, x_1, x_2\}, \{\mu_0, \mu_1, \mu_2\})$ where $0 \leq \mu_i \leq 1$ and $\dps \sum_{i = 0}^2 \mu_i =1$.
\item[(ii)] $p \in (0,1)$ and $\alpha = (\alpha_0, \alpha_1, \alpha_2)$ with $\sum_{i=0}^2 \alpha_i \mu_i = p$, $\alpha_i \geq 0$ and $\alpha \not= constant$. If not, then $\alpha_0 = \alpha_1 = \alpha_2 = p$ and thus $T_{\alpha}(c) = 0, \, \forall c$.
\end{enumerate}
The remark below shows that, without loss of generality, we can assume that $0 \leq \alpha_0 \leq \alpha_1 \leq \alpha_2$:

\begin{rmk}\label{V_with_permutation_of_c}
Let $\sigma \in \mathcal{S}_3$ be a permutation and $\nu$ be the probility measure $\nu = \mu \! \circ \! \sigma$ i.e. $\nu(j) = \mu(\sigma(j))$. Then,
\begin{enumerate}
\item[(i)] $V_{\nu}(c \circ \sigma) = V_{\mu \circ \sigma}(c \circ \sigma) = V_{\mu}(c)$ where $V_{\mu}$ is as defined in Equation (\ref{def_T}).

Here is the proof of $(i)$:
\begin{align*}
V_{\mu \circ \sigma} (c \circ \sigma) &= \int (c \circ \sigma)^2(x) \diff \mu \circ \sigma (x) - \Big( \int c \circ (x) \diff \mu \circ \sigma (x)\Big)^2 \\
&= \int c^2 (\sigma(x)) \diff \mu(\sigma(x)) - \Big( \int c(\sigma(x)) \diff \mu (\sigma(x)) \Big)^2 \\
&= \int c^2 (y) \diff \mu(y) - \Big( \int c(y) \diff \mu (y) \Big)^2 = V_{\mu}(c)
\end{align*}

\item[(ii)] $T(c \circ \sigma, \alpha \circ \sigma) = T(c, \alpha)$.

Indeed, 
\begin{align*}
\int c \circ \sigma (\alpha \circ \sigma - p) \diff \mu \circ \sigma &= \int c (\sigma(x)) (\alpha (\sigma(x))- p) \diff \mu (\sigma(x)) \\
&= \int c(y) (\alpha(y) - p) \diff \mu (y).
\end{align*}
\end{enumerate}
\end{rmk}

\noindent\textbf{\underline{Case 1}:} $ | \supprt \mu | = 3$.

\underline{Subcase (a)}: $c_{\circ} \not = c_2$.

In this case, there is 3 facts of interest: 
\begin{enumerate}
\item[(i)] $T(c, \alpha) = T\left((0,x,1), \alpha \right)$
\item[(ii)] $\max_c T(c, \alpha) = T\left((0, x^* , 1), \alpha \right)$ where $x^* =  \frac{\alpha_1 - \alpha_{\circ}}{\alpha_2 - \alpha_{\circ}} \in [0,1]$.
\item[(iii)] If $\alpha_{\circ} = \alpha_1$, then $x^* = 0$, and $\max_c T(c, \alpha) = T\left( (0,0,1), \alpha \right)$.\\
If $\alpha_{\circ} = \alpha_1$, then $x^* = 1$, and $\max_c T(c, \alpha) = T\left( (0,1,1), \alpha \right)$.
\end{enumerate}

\underline{Subcase (b)}: $c_{\circ} = c_2$.

In this case, there is 3 facts of interest:
\begin{enumerate}
\item[(i)] $T(c, \alpha) = T\left((0,x,0), \alpha \right)$ where $x \not = 0$ since $c$ cannot be constant.\\
After computation, one obtains
\begin{equation*} 
T\left((0,x,0), \alpha \right) = \frac{(\alpha_1 - p)^2 \mu_1^2}{\mu_1 - \mu_1^2} = \frac{(\alpha_1 - p)^2 \mu_1}{1 - \mu_1}.
\end{equation*}
\item[(ii)] $\lim_{|x| \to \infty} T\left((0,x,1), \alpha \right) = T\left((0,x,0), \alpha \right).$ \\
Indeed, when $x \to \infty$,
\begin{equation*}
T\left((0,x,1), \alpha \right) = \dps \frac{\Big( x(\alpha_1 - p)\mu_1 + (\alpha_2 - p)\mu_2 \Big)^2}{x^2(\mu_1 + \mu_2) - (x\mu_1 + \mu_2)^2}  \longrightarrow \frac{(\alpha_1 - p)^2 \mu_1^2}{\mu_1 - \mu_1^2}. 
\end{equation*}
\end{enumerate}

\noindent\textbf{\underline{Case 2}:} $| \supprt \mu | = 2$

Without loss of generality, we can assume that $\supprt \mu = \{0,1\}$ and that $c_{\circ} < c_1$. Moreover, by Lemma \ref{equalT_for_equiv_c}, we can assume that $c_{\circ} = 0$ and $c_1 = 1$. Under these assumptions, we have
\begin{equation*}
T(c, \alpha) = \frac{(\alpha_1 - p)^2 \mu_1}{1 - \mu_1}.
\end{equation*}
Since $\supprt \mu = \{0,1\}$ we can write $\mu_{\circ} = 1- \mu_1$ and therefore $p = \alpha_{\circ} \mu_{\circ} + \alpha_1 (1-\mu_{\circ}) = \alpha_1 + (\alpha_{\circ} - \alpha_1) \mu_{\circ}$. Replacing $\mu_{\circ}$ and $p$ in $T$, yields
\begin{equation}\label{T_Pour_espace_a_3_points}
T(c, \alpha) = (\alpha_1 - \alpha_{\circ})^2 \mu_{\circ} \mu_1.
\end{equation}

\section{Generalization of the Pearson and Cochran-Armitage Statistics}

The purpose of this section is to construct the Generalized Cochran-Armitage statistic and the Generalized Pearson statistic using the function $T$ constructed in Definition \ref{defn_T}.
 
\begin{defn}[Generalized Cochrane-Armitage Statistic]\label{definition_generalized_catt}
Let $\XX$ be a Borel space and $\tilde{\mu}_r$, $\tilde{\mu}_s$ be two positive measures on $\XX$. \\ 
Consider the two positive measures $\mu_r, \mu_s$ defined by $\mu_r = \displaystyle\frac{\tilde{\mu}_r}{n}$ and $\mu_s = \displaystyle\frac{\tilde{\mu}_s}{n}$, where $n = \tilde{\mu}_r(\XX) + \tilde{\mu}_s(\XX)$. \\
Note that the measure $\mu = \mu_r + \mu_s$ is a probability measure. \\
Let $\alpha_r$, $\alpha_s$ be the respective Radon-Nikodym derivative of $\mu_r$ and $\mu_s$ with respect to $\mu$.
Then, we can define the generalized Cochran-Armitage statistic $\tilde{T}_{CA}: L_2(\mu)_* \rightarrow \RR$ by
\begin{equation*}
\tilde{T}_{CA}(c) = \frac{n}{p_r p_s} \, T(c, \alpha_r),
\end{equation*}
where $p_r = \displaystyle \int \alpha_r \diff \mu$ and $p_s = \displaystyle \int \alpha_s \diff \mu$.
\end{defn}

Note that, by definition of the Radon-Nikodym derivative, and as $\mu$ is a probability measure, $p_r = \mu_r(\XX)$ and $p_s = \mu_s(\XX)$. Hence $p_r + p_s = 1$. As well, $\alpha_r + \alpha_s = 1$ since $\alpha_r + \alpha_s = \frac{\diff (\mu_r + \mu_s)}{\diff \mu}$.

\begin{rmk}
Using the same assumptions as in Definition \ref{definition_generalized_catt}, we notice that 
\begin{equation*}
\int c ( \alpha_r - p_r) \diff \mu = \int c (p_s \alpha_r - p_r \alpha_s) \diff \mu.
\end{equation*}
Indeed, 
\begin{align*}
\int c (p_s \alpha_r - p_r \alpha_s) \diff \mu &= \int c \big( (1-p_r) \alpha_r - p_r (1 - \alpha_r) \big) \diff \mu \\
&= \int c (\alpha_r - p_r\alpha_r - p_r + p_r \alpha_r) \diff \mu \\
&= \int c (\alpha_r - p_r) \diff \mu.\\
\end{align*} 
\end{rmk}

\begin{defn}[Generalized Pearson Statistic] \label{definition_generalized_Pearson}
Using the same assumptions as in Definition \ref{definition_generalized_catt}, we can define the Generalized Pearson Statistic $\tilde{T}_{\chi^2}$ by 
\begin{equation*}
\tilde{T}_{\chi^2} = \frac{n}{p_r p_s} \, T(\alpha_r, \alpha_r ).
\end{equation*}
\end{defn}

\subsection{Standard Statistics: a Special Case of the Generalized Statistics}

The purpose of this subsection is to study the Generalized Cochran-Armitage statistic and the Generalized Pearson statistic in the case of a three point space equipped with the discrete distance $1_{x \not= y}$. The generalized Cochran-Armitage statistic and the generalized Pearson statistic simplify to the well known Cochran-Armitage statistic and Pearson statistic, respectively. 

\begin{pc}[Cochran-Armitage] \label{cas_particulier_Cochran-Armitage}
Let $\XX$ be a Borel space. Let us define two atomic probability measures
\begin{equation*}
\mu_1 = \sum_{i = 1}^{n} \mu_i^{(1)} \delta_{x_i} \quad \mbox{and} \quad \mu_2 = \sum_{i = 1}^{n} \mu_i^{(2)} \delta_{x_i},
\end{equation*}
supported on the same countable subset $S$ of $\XX$, and such that $\mu = \mu_1 + \mu_2$ is a probability measure. The respective Radon-Nikodym derivative of $\mu_1$ and $\mu_2$ are defined by 
\begin{equation*}
\alpha_r(x_i) = \frac{\mu_r(x_i)}{\mu(x_i)} \quad \mbox{and} \quad \alpha_s(x_i) = \frac{\mu_s(x_i)}{\mu(x_i)}, \quad x_i \in S.
\end{equation*}
Then, the Generalized Cochran-Armitage statistic, defined in \ref{definition_generalized_catt}, can be written as
\begin{equation*}
\tilde{T}_{CA}(c) = \frac{n}{p_rp_s} \, \frac{ \bigg( \displaystyle \sum_{x_i\in S} c_i \big( p_s \mu_r(x_i) - p_r \mu_s(x_i) \big) \bigg)^2 }{ \displaystyle\sum_{x_i\in S} c^2_i \mu(x_i) - \Big( \displaystyle \sum_{x_i\in S} c(x_i) \mu(x_i) \Big)^2},
\end{equation*}
where $p_r = \displaystyle \sum_{x_i \in S} \alpha_r(x_i) \mu(x_i) = \displaystyle \sum_{x_i \in S} \mu_r(x_i) = \mu_r(S)$ and, likewise, $p_s = \mu_s(S).$ 
\vspace{1cm}

\noindent If $\mu_r$ and $\mu_s$ are supported on a three-point set $S=\{x_0, x_1, x_2\}$, we set $\mu_r(x_i) = r_i /n$, $\mu_r(S) = r/n$, $\mu_s(x_i) = s_i/n$, $\mu_s(S) = s/n$, $\mu(x_i) = n_i/n$ and $\mu(S) = 1$, and we observe that 
\begin{equation*}
\frac{n}{p_rp_s} = \frac{1}{n\frac{p_r}{n}\frac{p_s}{n}} = \frac{1}{npq}.
\end{equation*}
By doing so, we obtain the "classic" Cochran-Armitage statistic given in equation (\ref{CAstatistic}).
\end{pc}

\begin{pc}[Pearson]\label{cas_particulier_Pearson}
Using the same assumptions as in the particular case \ref{cas_particulier_Cochran-Armitage}, and recalling the expression of $T(\alpha, \alpha)$ obtained in the proof of Corollary \ref{sup_of_T}, the Generalized Pearson Statistic $\tilde{T}_{\chi^2}$, defined in \ref{definition_generalized_Pearson}, can be written as  
\begin{equation*}
\tilde{T}_{\chi^2} = \frac{n}{p_r p_s} \sum_{x_i\in S} \Big( \frac{\mu_r(x_i)}{\mu(x_i)} - p_r \Big)^2 \mu(x_i).
\end{equation*}
Moreover, if $\mu_r$ and $\mu_s$ are supported on a three-point set $S=\{x_0, x_1, x_2\}$, and we set $\mu_r(x_i)$, $\mu_r(S)$, $\mu_s(x_i)$, $\mu_s(S)$, $\mu(x_i)$ and $\mu(S)$ as in Particular Case \ref{cas_particulier_Cochran-Armitage}, we obtain the  Brandt-Snedecor formula (\ref{Brandt-Snedecor-Pearson}) of the "classic" Pearson statistic.

If $\mu_r$ and $\mu_s$ are supported on a two point set $S = \{x_0, x_1\}$, we recall the equation (\ref{T_Pour_espace_a_3_points}) of $T(c,\alpha)$ and set $\mu_r(x_i)$, $\mu_r(S)$, $\mu_s(x_i)$, $\mu_s(S)$, $\mu(x_i)$ and $\mu(S)$ as above to obtain 
\begin{equation*}
T(c,\alpha) = \left( \frac{r_1/n}{n_1/n}- \frac{r_0/n}{n_0/n} \right)^2 \frac{n_0}{n} \frac{n_1}{n} = \frac{(s_0 r_1 - r_0 s_1)^2}{n_1n_0} \frac{1}{n^2}
\end{equation*}
This is interesting as it means that $T$ is equivalent to the odds ratio (OR) measure of association commonly used in case-control studies with two-by-two frequency tables. the OR represents  between an exposure and an outcome. The OR represents the odds that an outcome will occur given a particular exposure, compared to the odds of the outcome occurring in the absence of that exposure. The OR measure is given by $r_{\circ}s_1/r_1s_{\circ}$. If OR = 1, the exposure does not affect odds of outcome and if $\mbox{OR}\not=1$, the exposure is associated with the odds of the outcome. 
\end{pc}

\section{Relationship between the Kantorovich-Rubinstein Distance and Trend Tests}

In the case of a metric space equipped with the $k$-discrete distance, we have already proved the following (see Theorem \ref{KR_dist_equal_TV_for_discrete_dist}): 

\begin{theo}
Let $\XX$ be a Polish space equipped with the $k$-discrete distance $k_{x \not= y}$. Let $\mu_1$ and $\mu_2$ be two Borel probability measures on $\XX$. Then,
\begin{equation*}
W_{\XX}(\mu_1,\mu_2) = \frac{k}{2}  ||\mu_1-\mu_2||_{TV}.
\end{equation*}
\end{theo} 

Recall that the Total Variation norm is related to the $L_1$-norm: 

\begin{lem}\label{relation_TV_L_1_norm}
Let $\mu_1$ and $\mu_2$  be two Borel probability measures on a Borel space $\XX$ absolutely continuous with respect to a Borel probability measure $\mu$ . Let $\alpha_1 = \frac{\diff \mu_1}{\diff \nu}$ and $\alpha_2 = \frac{\diff \mu_2}{\diff \nu}$ be the respective Radon-Nikodym derivatives of $\mu_1$ and $\mu_2$. Then,
\begin{equation*}
||\mu_1 - \mu_2||_{TV} = \frac{1}{2} ||\alpha_1 - \alpha_2||_1.
\end{equation*}
\end{lem}
\noindent Note that $\nu = \mu_1 + \mu_2$ satisfies the condition of the lemma.\\
As a direct consequence of Theorem \ref{KR_dist_equal_TV_for_discrete_dist} and Lemma \ref{relation_TV_L_1_norm}, we have the following lemma:

\begin{lem}\label{link_wasserstein_L1}
Let $\XX$ be a Polish space equipped with the $k$-discrete distance $k_{x \not= y}$. Let $\mu_1$ and $\mu_2$ be two probability measures on $\XX$ and $\alpha_1$ and $\alpha_2$ be their respective Radon-Nikodym derivatives, with respect to some Borel measure $\nu$. Then,
\begin{equation*}
W_{\XX}(\mu_1, \mu_2) = \frac{k}{4} ||\alpha_1 - \alpha_2||_1.
\end{equation*}
\end{lem}

We now have all the necessary results to compare the Kantorovich-Rubinstein distance with both the Cochran-Armitage statistic and the Pearson statistic:

\begin{prop}\label{lien_Pearson_KR}
Let $\XX$ be a Polish space equipped with the $k$-discrete distance $k_{x \not= y}$ and $\tilde{\mu}_r$, $\tilde{\mu}_s$ two finite positive measures on $\XX$. \\
Consider the two positive measures $\mu_r, \mu_s$ defined by $\mu_r = \displaystyle\frac{\tilde{\mu}_r}{n}$ and $\mu_s = \displaystyle\frac{\tilde{\mu}_s}{n}$, where $n = \tilde{\mu}_r(\XX) + \tilde{\mu}_s(\XX)$. \\
Let $\alpha_r$, $\alpha_s$ be the respective Radon-Nikodym derivative of $\mu_r$, $\mu_s$ with respect to the measure $\mu = \mu_r + \mu_s$ and $p_r = \int \alpha_r \diff \mu$, $p_s = \int \alpha_s \diff \mu$.
Then, the following inequalities hold:
\begin{equation*}
n p_r p_s  \frac{4^2}{k^2} W_{\XX}^2 \big(\frac{1}{p_r} \mu_r, \frac{1}{p_s} \mu_s\big) \leq \tilde{T}_{\chi^2} \leq \frac{4n}{k} ||\alpha_r - p_r||_{\infty} W_{\XX}\big(\frac{1}{p_r} \mu_r, \frac{1}{p_s}\mu_s\big),
\end{equation*}
where $\tilde{T}_{\chi^2}$ denotes the Generalized Pearson Statistic as defined in Definition \ref{definition_generalized_Pearson}.
\end{prop}

\begin{proof}[Proposition \ref{lien_Pearson_KR}]
By Lemma \ref{link_wasserstein_L1}, we know that 
\begin{equation*}
W_{\XX}\big(\frac{1}{p_r} \mu_r, \frac{1}{p_s}\mu_s\big)= \frac{k}{4} ||\frac{1}{p_r} \alpha_r - \frac{1}{p_s}\alpha_s||_1.
\end{equation*}
Recall that $p_r + p_s = 1$ and $\alpha_r + \alpha_s = 1, \, \mu$-a.e. Then, 
\begin{equation*}
||\frac{1}{p_r} \alpha_r - \frac{1}{p_s}\alpha_s||_1 = \frac{1}{p_r p_s} || (1 - p_r) \alpha_r - p_r (1-\alpha_r)||_1 = \frac{1}{p_r p_s} ||\alpha_r - p_r ||_1,
\end{equation*}
and therefore
\begin{equation}\label{equation_lien_KR_L1}
 \frac{4n}{k} W_{\XX}\big(\frac{1}{p_r} \mu_r, \frac{1}{p_s}\mu_s\big) = \frac{n}{p_r p_s} ||\alpha_r - p_r||_1.
\end{equation}

\noindent Now, we first show that 
\begin{equation}\label{borne_sup_de_Pearson}
\tilde{T}_{\chi^2} \leq \frac{4n}{k}||\alpha_r - p_r||_{\infty} W_{\XX}\big(\frac{1}{p_r} \mu_r, \frac{1}{p_s}\mu_s\big).
\end{equation}

\noindent Basic properties of $L_p$ distances inequality, yield:
\begin{equation*}
||\alpha_r - p_r||_2^2 = \int |\alpha_r - p_r|^2 \diff \mu \leq ||\alpha_r - p_r||_{\infty} \int |\alpha_r - p_r| \diff \mu =  ||\alpha_r - p_r||_{\infty} ||\alpha_r - p_r||_1.
\end{equation*}
By multiplying on both sides of the inequality by $\displaystyle\frac{n}{p_r p_s}$ and using equation (\ref{equation_lien_KR_L1}), one obtains:
\begin{equation*}
\frac{n}{p_r p_s} ||\alpha_r - p_r||_2^2 \leq \frac{4n}{k}||\alpha_r - p_r||_{\infty} W_{\XX}\big(\frac{1}{p_r} \mu_r, \frac{1}{p_s}\mu_s\big).
\end{equation*}
By Corollary \ref{sup_of_T}, we know that $\displaystyle\tilde{T}_{\chi^2} = \frac{n}{p_r p_s} ||\alpha_r - p_r||_2^2$. Hence we obtain inequality (\ref{borne_sup_de_Pearson}). 
It is now left to show that 
\begin{equation*}
n p_r p_s  \frac{4^2}{k^2} W_{\XX}^2 \big(\frac{1}{p} \mu_r, \frac{1}{q} \mu_s\big) \leq \tilde{T}_{\chi^2}.
\end{equation*}
Using H\"older inequality, we have: $\displaystyle \frac{n}{p_r p_s} ||\alpha_r - p_r||_1^2 \leq \frac{n}{p_r p_s} ||\alpha_r - p_r||_2^2 = T_{\chi^2}$.
Using equation (\ref{equation_lien_KR_L1}) on the lower bound gives:
\begin{align*}
 \frac{n}{p_r p_s} ||\alpha_r - p_r||_1^2 &=  \frac{p_r p_s}{n} \frac{n^2}{p_r^2 p_s^2} ||\alpha_r - p_r||_1^2 \\[3mm]
 &=  \frac{p_r p_s}{n} \frac{(4n)^2}{k^2}  W_{\XX}^2 \big(\frac{1}{p_r} \mu_r, \frac{1}{p_s}\mu_s\big) \\[3mm]
 &= n p_r p_s  \frac{4^2}{k^2}  W_{\XX}^2 \big(\frac{1}{p_r} \mu_r, \frac{1}{p_s}\mu_s\big),
\end{align*}
which completes the proof.
\end{proof}

\cleardoublepage

 \chapter{K-Lipschitz Maps and the Kantorovich-Rubinstein distance}\label{chapter_K-Lipschitz_mappings_and_KR_distance}
 
This chapter is short and technical but very important as it is a building bloc for Chapter \ref{KR_score_appplied_to_gwas}. All the results presented are new. We study the functorial properties of the Kantorovich-Rubinstein distance. The goal is to compare the Kantorovich-Rubinstein distances between two measures $\mu_1, \mu_2 \in P_{\XX}$ and their respective push-forward measures $f(\mu_1), f(\mu_2) \in P_{\YY}$, for a mapping $f: \XX \rightarrow \YY$ with particular properties. 

In the first section, we obtain interesting results when $f$ is a $k$-Lipschitz function and when $f$ is an isometry. These results are given in Theorem \ref{inequality of KR distances with Lipschitz function} and Proposition \ref{isometric_maps}.

The second section considers the cases of canonical projection $\pi_a: \XX \rightarrow \XX_A$, where $A$ is a Cartesian product of Polish spaces and $\XX_A$ is a canonical subset of $\XX$. We compare the Kantorovich-Rubinstein distances between two measures $\mu_1, \mu_2 \in P_{\XX}$ and their respective push-forward measures $\pi_A(\mu_1), \pi_A(\mu_2) \in P_{\XX_A}$. It is a particular case of the first section as canonical projections are Lipschitz functions.

\section{Functorial properties of the Kantorovich-Rubinstein distance}\label{functorial_properties}

Following the notation introduced in Section \ref{section_on_KR_distance}, if $\XX$ is a Polish space and $d$ is a lower semi-continuous distance on $\XX$, we denote by $P_{\XX}$ the (convex) set of all Borel probability measures $\nu$ on $\XX$ such that 
\begin{equation*}
\int_{\XX} d(x,x_0) \diff \nu(x) < \infty,
\end{equation*}
for some (and therefore for any) $x_0 \in \XX$.\\
Recall also that $W_{\XX}$ denotes the Kantorovich-Rubinstein distance on $P_{\XX}$ and that $\DD_{\XX}$ denotes the set of all lower semi-continuous metrics $d$ on $\XX$. A metric $d \in \DD_{\XX}$ does not necessarily define the topology on $\XX$.

\begin{theo}\label{inequality of KR distances with Lipschitz function}
Let $\XX$ and $\YY$ be two Polish spaces, equipped with $d_{\XX} \in \DD_{\XX}$ and $d_{\YY} \in \DD_{\YY}$, respectively. Let $k \in \RR_+^*$ and $\varphi : \XX \rightarrow \YY$ be a $k$-Lipschitz, measurable mapping.\\
Let $\mu_1, \mu_2$ be two probability measures in $P_{\XX}$, and $\varphi(\mu_1)$, $\varphi(\mu_2)$ their respective push-forward measures on $\YY$. Then, for $i = 1,2$, $\varphi(\mu_i) \in P_{\YY}$ and
\begin{equation*} 
W_\mathcal{Y}(\varphi(\mu_1), \varphi(\mu_2)) \leq k\, W_\mathcal{X}(\mu_1, \mu_2).
\end{equation*}
\end{theo}

\noindent To prove Theorem \ref{inequality of KR distances with Lipschitz function}, we will need Lemma \ref{tau-nu_coupling}. Let us first recall a notation:
\begin{ntn}
For any two sets $Z_1$ and $Z_2$, let $\pi_j$, $j=1,2$, denote the canonical projection from the Cartesian product $Z_1 \times Z_2$ onto $Z_j$.     
\end{ntn}

\begin{lem}\label{tau-nu_coupling}
Let  $\XX_i$ and $\YY_i$, for $i =1,2$, be Borel spaces and $\varphi_i: \XX_i \rightarrow \YY_i$ be Borel maps. Let $\varphi_1 \times \varphi_2 : \XX_1 \times \XX_2 \rightarrow \YY_1 \times \YY_2$ denote the Borel map defined by $(\varphi_1 \times \varphi_2)(x_1,x_2) = \big(\varphi_1(x_1), \varphi_2(x_2)\big)$.\\
For $i = 1,2$, let $\mu_i$ be a probability measure on $\XX_i$ and $\varphi_i(\mu_i)$ be its push-forward measure on $\YY_i$. \\
If $\vartheta \in \NN(\mu_1,\mu_2)$ is a coupling of $\mu_1$ and $\mu_2$, then its push-forward measure $(\varphi_1 \times \varphi_2) \, (\vartheta)$ belongs to $\NN(\varphi_1(\mu_1), \varphi_2(\mu_2))$.
 \end{lem}
 
\begin{proof}[Lemma \ref{tau-nu_coupling}]
 Since, $\pi_i \circ (\varphi_1 \times \varphi_2) = \varphi_i \circ \pi_i$, for $i = 1,2$, we have 
 \begin{equation*}
 \pi_i  \Big((\varphi_1 \times \varphi_2)(\vartheta)\Big) = \pi_i \circ \big(\varphi_1 \times \varphi_2\big)(\vartheta) = \varphi_i \circ \pi_i \, (\vartheta),
 \end{equation*}
 for any probability measure $\vartheta$ on $\XX_1 \times \XX_2$.
 
\noindent Then, if $\vartheta \in \NN(\mu_1,\mu_2)$, we have, for $i = 1,2$, 
\begin{equation*}
\pi_i  \Big((\varphi_1 \times \varphi_2)(\vartheta)\Big) = \varphi_i \circ \pi_i(\vartheta) = \varphi(\mu_i).
\end{equation*}
Hence $(\varphi_1 \times \varphi_2)(\vartheta) \in \NN \big(\varphi_1(\mu_1), \varphi_2(\mu_2)\big)$.
 \end{proof} 
 
\noindent We can now prove Theorem \ref{inequality of KR distances with Lipschitz function}:\\

\begin{proof}[Theorem \ref{inequality of KR distances with Lipschitz function}]
Let $\mu_1$ and $\mu_2$ be two probability measures in $P_{\XX}$ and let $\vartheta \in \NN(\mu_1,\mu_2)$. As $d_{\XX}$ and $d_{\YY}$ are lower semi-continuous, 
\begin{align*}
\int_{\YY} d_{\YY}(\varphi(x_0),y) \diff \varphi(\mu_i)(y) &= \int_{\XX} d_{\YY}(\varphi(x_0),\varphi(x)) \diff \mu_i(x) \\
&\leq k \int_{\XX}d_{\XX}(x_0,x) \diff \mu_i(x) < \infty,
\end{align*}
for any $x_0 \in \XX$ fixed. Therefore, $\varphi(\mu_i) \in P_{\YY}$.\\
Since $\vartheta \in \NN(\mu_1,\mu_2)$, 
\begin{align*}
\int_{\XX \times \XX} d_{\XX}(x,y) \diff \vartheta(x,y) &\leq \int_{\XX \times \XX} d_{\XX}(x,x_0) \diff \vartheta(x,y) + \int_{\XX \times \XX} d_{\XX}(x_0,y) \diff \vartheta(x,y) \\
&\leq \int_{\XX}d_{\XX}(x,x_0) \diff \mu_1(x) + \int_{\XX} d_{\XX}(x_0,y) \diff \mu_2(y).
\end{align*}
Thus $d \in L_1(\vartheta)$. By definition of push-forward measures, we get:
\begin{equation*}
\int_{\YY \times \YY} d_{\YY}(y_1,y_2) \diff (\varphi\times\varphi)(\vartheta)(y_1,y_2) = \int_{\XX \times \XX} d_{\YY}(\varphi(x_1),\varphi(x_2)) \diff\vartheta (x_1,x_2).
\end{equation*}
Since $\varphi$ is $k$-Lipschitz, we obtain:
\begin{equation*}
\int_{\XX \times \XX} d_{\YY}(\varphi(x_1),\varphi(x_2)) \diff\vartheta(x_1,x_2) \leq k \int _{\XX \times \XX} d_{\XX}(x_1,x_2) \diff \vartheta(x_1,x_2).
\end{equation*}

\noindent By Lemma \ref{tau-nu_coupling},  $(\varphi \times\varphi)(\vartheta)$ belongs to $\NN (\varphi(\mu_1), \varphi(\mu_2))$ and therefore 
\begin{align*}
W_{\YY}(\varphi(\mu_1), \varphi(\mu_2)) & \leq \int_{\YY \times \YY} d_{\YY}(y_1,y_2) \diff (\varphi\times\varphi)(\vartheta)(y_1,y_2) \\
& \leq k \int_{\XX \times \XX} d_{\XX}(x_1,x_2) \diff \vartheta (x_1,x_2). 
\end{align*}

\noindent By definition of the Kantorovich-Rubinstein distance, $W_{\YY} (\varphi(\mu_1), \varphi(\mu_2)) \leq k W_{\XX} (\mu_1, \mu_2)$.
\end{proof}

Let $\XX$ be a measurable space. Recall from Remark \ref{rmk_total_variation} that for any probability measures $\mu_1$ and $\mu_2$ on $\XX$, the total variation distance $||\mu_1-\mu_2||_{TV}$ has a coupling characterization given by
\begin{equation*}
||\mu_1-\mu_2||_{TV} = 2 \inf \mathbb{P}[X\not = Y], 
\end{equation*}
where the infimum is over all couplings $(X, Y)$ of $(\mu_1,\mu_2)$. We get the following useful result:

\begin{theo}\label{borne_sup_KR}
Let $\XX$ be a Polish space and $d \in D_{\XX}$ be a bounded distance. Let $\mu_1$ and $\mu_2$ be two Borel probability measures on $\XX$. Then,
\begin{equation*}
W_{\XX}(\mu_1,\mu_2) \leq \Delta ||\mu_1-\mu_2||_{TV}, \mbox{ where } \Delta = \diam(\XX).
\end{equation*} 
\end{theo} 

\begin{proof}[Theorem \ref{borne_sup_KR}]
Let $1_{x \not = y}$ be the discrete distance on $\XX$. Then, for all $x,y \in \XX, d_{\XX}(x,y)  \leq \Delta 1_{x\not=y}$ and therefore 
\begin{align*}
W_{\XX}(\mu_1,\mu_2) &= \inf_{\vartheta \in \NN(\mu_1,\mu_2)} \int d(x,y) \diff \vartheta (x,y) \\
& \leq \Delta \inf_{\vartheta \in \NN(\mu_1,\mu_2)} \int 1_{x\not=y}(x,y) \diff \vartheta (x,y) \\
& \leq \Delta ||\mu_1 - \mu_2||_{TV}.
\end{align*}
\end{proof}

\noindent In Theorem \ref{inequality of KR distances with Lipschitz function}, we have compared the Kantorovich-Rubinstein distance between two measures and between their respective push-forward measures in the case of a k-Lipschitz map $\varphi$ from $\XX$ to $\YY$. Now, we study two particular cases: when $\varphi$ is an optimal function and when $\varphi$ is an isometric map. \\

\noindent Let $(\XX, \BB)$ be a measurable space and $\mu_1$ and $\mu_2$ be probability measures on $\XX$. We recall that an optimal function $\varphi: \XX \rightarrow \RR$, is a 1-Lipschitz function that realizes the Kantorovich-Rubinstein distance. Hence, $\dps \int_{\XX} \varphi \diff (\mu_1 - \mu_2)$.

\begin{prop}\label{equality_MK_dist_for_1-Lip_real_functions}
Let $\XX$ be a Polish space and $d \in \DD_{\XX}$ be a bounded lower semi-continuous distance on $\XX$. Let $\mu_1$ and $\mu_2$ be two probability measures on $\XX$ and $f: \XX \rightarrow \RR$ be an optimal 1-Lipschitz function. Then,
\begin{equation*}
 W_{\XX}(\mu_1,\mu_2) = W_{\RR}(f(\mu_1),f(\mu_2)),
\end{equation*} 
where $f(\mu_1)$ and $f(\mu_2)$ are the respective push-forward measures of $\mu_1$ and $\mu_2$ on $\RR$.
\end{prop}

\noindent For the proof of Proposition \ref{equality_MK_dist_for_1-Lip_real_functions}, we first need the following lemma:

\begin{lem}\label{lemme_technique_pour_prop_equality_MK_dist_for_1-Lip_real_functions}
Let $\XX$ be a Polish space and $d \in \DD_{\XX}$ be a bounded lower semi-continuous distance on $\XX$. Let $\mu_1$ and $\mu_2$ be two probability measures on $\XX$ and $f: \XX \rightarrow \RR$ be a measurable, 1-Lipschitz function. Then, the identity function $i: \RR \rightarrow \RR$, defined by $i(t)= t$, is in $L_1(|f(\mu_1)-f(\mu_2)|)$.
\end{lem}

\begin{proof}[Lemma \ref{lemme_technique_pour_prop_equality_MK_dist_for_1-Lip_real_functions}]
Let $\Delta$ be the $d$-diameter of $\XX$. By Lemma \ref{Image_of_lip_fct_is_smaller_diamX}, $\diam(f(\XX)) < \infty$.  Therefore there exist $-\infty < a < b < \infty$ such that $f(\XX) \subset [a,b]$. Hence, for any Borel set $A \subset [a,b]^{\comp}$, $f(\mu_1)(A) = f(\mu_2)(A) = 0$ and thus $|f(\mu_1)-f(\mu_2)|([a,b]^{\comp}) = 0$. Then,
Therefore,
\begin{equation*}
\int_{\RR} |t| \diff \, |f(\mu_1)-f(\mu_2)|(t) = \int_a^b |t| \diff \,| f(\mu_1) - f(\mu_2)|(t) < \infty,
\end{equation*}
which completes the proof.
\end{proof}

\noindent We can now give the proof of Proposition \ref{equality_MK_dist_for_1-Lip_real_functions}:\\

\begin{proof}[Proposition \ref{equality_MK_dist_for_1-Lip_real_functions}]
As $f$ is 1-Lipschitz and measurable, then, by Theorem \ref{inequality of KR distances with Lipschitz function}, we have 
\begin{equation*}
W_{\RR}(f(\mu_1),f(\mu_2)) \leq W_{\XX}(\mu_1,\mu_2). 
\end{equation*}
To prove the converse inequality, first recall that, by the Kantorovich-Rubinstein Duality Theorem \ref{K-R_Duality_thm_Villani_version},
\begin{align*}
W_{\RR}(f(\mu_1),f(\mu_2)) = \sup \big\{ \int_{\RR} \psi(t) \diff (f(\mu_1) - f(\mu_2))(t), \, \psi \in L_1(|f(\mu_1)-f(\mu_2)|), \, ||\psi||_{Lip} \leq 1 \big\}.
\end{align*}
Now, by the optimality of $f$ and with a change of variable, we obtain:
\begin{equation*}
W_{\XX}(\mu_1,\mu_2) = \int_{\XX} f(x) \diff (\mu_1-\mu_2)(x) = \int_{\RR} t \diff(f(\mu_1)-f(\mu_2))(t).
\end{equation*}
As the identity function $i: \RR \rightarrow \RR$ is 1-Lipschitz and, by Lemma \ref{lemme_technique_pour_prop_equality_MK_dist_for_1-Lip_real_functions}, is in $L_1(|f(\mu_1) - f(\mu_2)|)$, the proof is complete.
\end{proof}

Let $(\XX, \BB)$ be a measurable space and $\mu$ be a positive measure on $\XX$. Recall (see for example Definition 6.7 in \cite{Real_and_complex_analysis_Rudin}) that $\mu$ is concentrated on $C \in \BB$ if $\mu(B) = \mu(B \cap C)$, for all $B \in \BB$, or equivalently, $\mu(B) = 0$, if $B \cap C = \emptyset$.

\begin{prop}\label{concentration_of_coupling}
Let $(\XX, \BB)$ be a measurable space. For $i = 1,2$, let $\mu_i$ be a probability measure on $(\XX,\BB)$ concentrated on $C_i \in \BB$. If $\vartheta \in \NN(\mu_1,\mu_2)$ is a coupling of $\mu_1$ and $\mu_2$, then $\vartheta$ is concentrated on $C_1 \times C_2$.
\end{prop}

\begin{proof}[Proposition \ref{concentration_of_coupling}]
Let $A \in \BB \times \BB$ be contained in $(C_1 \times C_2)^{\comp}$. Let us show that $\vartheta(A) = 0$. Set $A_1 = A \cap (C_1^{\comp} \times \XX)$ and $A_2 = A \cap (C_2^{\comp} \times \XX)$. Therefore, $A_i \in \BB \times \BB$, $\pi_i(A_i) = \pi_i(A) \cap C_i^{\comp}$ and $\vartheta(A_i) \leq \vartheta\left(\pi^{-1}_i(\pi_i(A_i))\right) = \mu_i \left( \pi_i(A_i)\right) = \mu_i \left( \pi_i(A_i) \cap C_i^{\comp}\right) = 0$. Hence $\varphi(A) \leq \varphi(A_1) + \varphi(A_2) = 0$.
\end{proof}

\noindent The following proposition is the key technical tool of Theorem \ref{isometric_maps}. To facilitate the understanding of the proof, we chose to denote by $\NN_{\XX}$ the set of all couplings of two measures both supported on $\XX$. Likewise, $\NN_{\YY}$ is the set of all couplings of two measures both supported on $\YY$.

\begin{prop}\label{image_of_coupling_equal_coupling}
Let $\XX$ and $\YY$ be two Polish spaces and let $\varphi: \XX \rightarrow \YY$ be an injective measurable map. \\
Let $\mu_1, \mu_2$ be two probability measures in $P_{\XX}$. Then, 
\begin{equation*}
(\varphi \times \varphi) \left(\NN_{\XX}(\mu_1, \mu_2) \right) = \NN_{\YY}(\varphi(\mu_1), \varphi(\mu_2)).
\end{equation*}
\end{prop}

\begin{proof}[Proposition \ref{image_of_coupling_equal_coupling}]
By Lemma \ref{tau-nu_coupling}, we have $(\varphi \times \varphi) \left(\NN_{\XX}(\mu_1, \mu_2) \right) \subset \NN_{\YY}(\varphi(\mu_1), \varphi(\mu_2))$.\\
To prove the other inclusion, let us first note that by Theorem 8.3.6 in \cite{measure_theory_Cohn}, $\varphi(\XX) \subset \YY$ is Borel and $\varphi: \XX \rightarrow \varphi(\XX) \subset \YY$ is a Borel isomorphism. If $\vartheta \in \NN_{\YY}(\varphi(\mu_1), \varphi(\mu_2))$, then by Proposition \ref{concentration_of_coupling}, $\vartheta$ is concentrated on $\varphi(\XX) \times \varphi(\XX)$. Then, $(\varphi^{-1} \times \varphi^{-1})(\vartheta)$ is a probability measure on $\XX \times \XX$, belonging to $\NN_{\XX}(\mu_1,\mu_2)$. Indeed, denoting by $\pi_i^{\XX}: \XX \times \XX \rightarrow \XX$ (respectively, $\pi_i^{\YY}: \YY \times \YY \rightarrow \YY$) the canonical projections, we have by Lemma \ref{tau-nu_coupling}, $\pi^{\XX}_i(\varphi^{-1} \times \varphi^{-1})(\vartheta) = \varphi^{-1}\left( \pi_i^{\YY}(\vartheta)\right) = \varphi^{-1}(\varphi(\mu_i)) = \mu_i$. Then $(\varphi \times \varphi) ((\varphi^{-1} \times \varphi^{-1})(\vartheta)) = \vartheta$.
\end{proof}

\noindent We can now give the proof of Theorem \ref{isometric_maps}:

\begin{theo}[l.176]\label{isometric_maps}
Let $\XX$ and $\YY$ be two Polish spaces equipped with corresponding lower semi-continuous distances $d_{\XX}$ and $d_{\YY}$. Let $k \in \RR_+^*$ and $\varphi : \XX \rightarrow \YY$ be a measurable isometry.\\
Let $\mu_1, \mu_2$ be two probability measures in $P_{\XX}$, and $\varphi(\mu_1)$, $\varphi(\mu_2)$ their respective push-forward measures on $\YY$. Then, $\varphi(\mu_i) \in P_{\YY}$, for $i = 1,2$, and $W_{\XX}(\mu_1,\mu_2) = W_{\YY}(\varphi(\mu_1),\varphi(\mu_2))$.
\end{theo}

\begin{proof}[Theorem \ref{isometric_maps}]
As in Theorem \ref{inequality of KR distances with Lipschitz function}, $\varphi(\mu_i) \in P_{\YY},$ for $i = 1,2$. By Proposition \ref{concentration_of_coupling} and Proposition \ref{image_of_coupling_equal_coupling}, $\vartheta \in \NN_{\XX}(\mu_1,\mu_2)$ if and only if $(\varphi \times \varphi) \vartheta \in \NN_{\YY}(\varphi(\mu_1), \varphi(\mu_2))$ and 
\begin{align*}
\int_{\YY \times \YY} d_{\YY} (y_1,y_2) \diff(\varphi \times \varphi) \vartheta (y_1,y_2) &= \int_{\varphi(\XX) \times \varphi(\XX)} d_{\YY} (y_1,y_2) \diff (\varphi \times \varphi) \vartheta(y_1,y_2) \\
&= \int_{\XX \times\XX} d_{\XX} (\varphi(x_1), \varphi(x_2)) \diff \vartheta(x_1,x_2) \\
&= \int_{\XX \times\XX} d_{\XX}(x_1,x_2) \diff \vartheta(x_1,x_2).
\end{align*}
Thus, by definition of the Kantorovich-Rubinstein distance, the proof is complete.
\end{proof}

\section{Projections and the Kantorovich-Rubinstein distance}

After recalling the definition of the $\ell_p$-product metric (for $1 \leq p \leq \infty$) given in Section \ref{KR_dist_atomic_measures_on_product_space}, we introduce notations of canonical projections and state results needed in the following chapters.
\begin{defn}\label{definition_metric_on_cartesian_space_recall} 
Let $(\XX_k,d_k)$ be $r$ metric spaces and let $\XX = \XX_1 \times \ldots \times \XX_r$ be the Cartesian product of these $r$ metric spaces. For $p\in [1,+\infty)$, the $p$-product metric $\tilde{d}_p$ is defined as the $\ell_p$-norm of the $r$-vector of the distances $d_i$. That is :
\begin{equation*}
\tilde{d}_p(x,y) = \big(  \sum_{i = 1}^r d_i(x_i,y_i)^p \big)^{1/p} \quad \mbox{ with} \quad x,y \in \XX.
\end{equation*}
For $p = \infty$, the $\ell_{\infty}$-product metric is also called the $\sup$ metric and is defined as 
\begin{equation*}
\tilde{d}_{\infty}(x,y) = \max_{1\leq i \leq r} d_i(x_i,y_i).
\end{equation*}
\end{defn}

\begin{lem}\label{equivalence_product_metric}
Let $(\XX_k,d_k)$ be $r$ metric spaces and let $\XX = \XX_1 \times \ldots \times \XX_r$ be the Cartesian product of these $r$ metric spaces. For $1 \leq q < q' \leq \infty$, let $\tilde{d}_q$ and $\tilde{d}_{q'}$ be two product metrics on the Cartesian product
$\XX = \XX_1 \times \ldots \times \XX_r$. \\
Then, the metrics $\tilde{d}_q$, $\tilde{d}_{q'}$ and $\tilde{d}_{\infty}$ are equivalent.
\end{lem}

\begin{proof}[Lemma \ref{equivalence_product_metric}]
The fact that a product metric $\tilde{d}_p$ ($1 \leq p \leq \infty$) is a metric is straight forward. \\
Now, recall that two metrics $d$ and $d'$ are said to be equivalent if there exist two positive constants $c_1,c_2$ such that, for all $(x,y) \in \XX \times \XX$, we have the inequalities $c_1d(x,y) \leq d'(x,y) \leq c_2d(x,y)$. \\
Clearly, $\tilde{d}_{\infty} \leq \tilde{d}_q$ and $\tilde{d}_{\infty} \leq \tilde{d}_{q'}$. Now, since $d_k(x_k,y_k) \leq d_{\infty}(x,y)$, for all $k = 1,\ldots,r$, we have 
\begin{equation*}
\tilde{d}_q(x,y) \leq n^{1/q} \tilde{d}_{\infty}(x,y) \quad \mbox{and} \quad \tilde{d}_{q'}(x,y) \leq n^{1/q'} \tilde{d}_{\infty}(x_i,y_i).
\end{equation*}
Thus $\tilde{d}_{\infty}$ is equivalent to $\tilde{d}_q$ and $\tilde{d}_{q'}$, respectively. \\
We can then write $(1/r^{1/q}) \tilde{d}_q(x,y) \leq \tilde{d}_{q'}(x,y) \leq r^{1/q'} \tilde{d}_q(x,y)$, showing that $\tilde{d}_q$ and $\tilde{d}_{q'}$ are equivalent.
\end{proof}

\noindent We now give the definition of a canonical projection as it is used in the rest of the chapter.

\begin{defn}\label{canonical_projection_and_distances_on projections}
Let $(\XX_k,d_k)$ be $r$ metric spaces and let $\XX = \XX_1 \times \ldots \times \XX_r$ be the Cartesian product of these $r$ metric spaces. For any non-empty subset $A = \{i_1,i_2, \ldots, i_a\} \subset \{1,\ldots, r\}$, consider the Cartesian product $\XX_A = \XX_{i_1} \times \ldots \times \XX_{i_a}$. 
\begin{enumerate}
\item[(i)] The canonical projection $\pi_A :\XX \rightarrow \XX_A$ is defined by $\pi_A(x) = (x_{i_1},x_{i_2},\ldots ,x_{i_a})$.
\item[(ii)] The $\ell_p$-distance $\tilde{d}_{p,a}$ on $\XX_A$ is defined by 
\begin{align*}
\tilde{d}_{p,a}(x,y) &= \Big( \sum_{k= 1}^a d_{i_k}(x_{i_k},y_{i_k})^p \Big)^{\frac{1}{p}}, \, \mbox{ for } \, 1 \leq p < \infty, \mbox{ and} \\
 \tilde{d}_{\infty,a}(x,y) &= \max_{1 \leq k \leq a} d_{i_k}(x_{i_k}, y_{i_k}).
\end{align*} 
\end{enumerate}
\end{defn}

\noindent The following result will be necessary in Theorem \ref{inequality of KR distances with projection}:

\begin{lem}\label{lsc_of_product_distance}
Let $r$ be a positive integer. For $1 \leq k \leq r$, let $\XX_k$ be a Polish space and $d_k \in \DD_{\XX_k}$. Consider the product space $\XX = \XX_1 \times \ldots \times \XX_r$ endowed with the product topology, and $\tilde{d}_p$, ($1 \leq p \leq \infty$), the $\ell_p$-distance on $\XX$. Then, $\tilde{d}_p$ is lower semi-continuous distance on $\XX$.
\end{lem}

\begin{proof}[Lemma \ref{lsc_of_product_distance}]
For $1 \leq k \leq r$, let $\pi_k: \XX \rightarrow \XX_k$ be the canonical projection and $d^k = d_k \circ \pi_k$. For any $\alpha \geq 0$, 
\begin{equation*}
F_{\alpha} = (d^k)^{-1}([0, \alpha]) = \XX_1 \times \ldots \times \XX_{k-1} \times F^k_{\alpha} \times \XX_{k+1} \times \ldots \times \XX_r,
\end{equation*}
where $F^k_{\alpha} = d_k^{-1}([0,\alpha])$.\\
As $d_k$ is lower semi-continuous, $F^k_{\alpha}$ is a closed subset of $\XX_k$ and therefore $F_{\alpha}$ is closed in $\XX$. Hence, $d^k$ is lower semi-continuous. As the sum of lower semi-continuous functions and the product of positive lower semi-continuous functions are lower semi-continuous (see \cite{Dieudonne_element_danalyse}), $\tilde{d}_p$ is lower semi-continuous for $1 \leq p \leq \infty$. As the supremum of lower semi-continuous functions is lower semi-continuous (see \cite{Dieudonne_element_danalyse}), $\tilde{d}_{\infty}$ is also lower semi-continuous.
\end{proof}

\begin{theo}\label{inequality of KR distances with projection}
Let $r$ be a positive integer. For $1 \leq k \leq r$, let $\XX_k$ be a Polish space and $d_k \in \DD_{\XX_k}$. Consider the Polish Cartesian product space $\XX = \XX_1 \times \ldots \times \XX_r$,and its subset $\XX_A = \XX_{i_1} \times \ldots \times \XX_{i_a}$, where $S = \{i_1,i_2, \ldots, i_a\} \subset \{1,\ldots, r\}$ is non-empty. \\
Let $\mu_1$ and $\mu_2$ be two probability measures in $P_{\XX}$, then the push-forward measures $\pi_A(\mu_1)$ and $\pi_A(\mu_2)$ are in $P_{\XX_A}$ and 
\begin{equation} \label{equation_inegalite_KR_for_projection}
W_{\XX_A}(\pi_A(\mu_1), \pi_A(\mu_2)) \leq W_{\XX}(\mu_1,\mu_2),
\end{equation}
where $W_{\XX}$ and $W_{\XX_A}$ are the respective Wasserstein distances associated to $(\XX, \tilde{d}_p)$ and $(\XX_A, \tilde{d}_{p,a})$.
\end{theo}

\begin{proof}[Theorem \ref{inequality of KR distances with projection}]
As each $\XX_k$ is a Polish space and $d_k \in \DD_{\XX_k}$, it follows from Lemma \ref{lsc_of_product_distance} that $\tilde{d}_p \in \DD_{\XX}$ and $\tilde{d}_{p,a} \in \DD_{\XX_A}$, for $1 \leq p \leq \infty$.\\
Now, the projection $\pi_A: \XX \rightarrow \XX_A$ is open and therefore Borel measurable. Moreover, by definition of $\pi_A$ and of the distances $\tilde{d}_p$ and $\tilde{d}_{p,a}$, it is clear that $||\pi_A||_{\lip} \leq 1$. Indeed, we have $\tilde{d}_{p,a}(x,y) \leq \tilde{d}_p(x,y)$ since 
\begin{equation*}
\Big( \sum_{k=1}^a d_{i_k}(x_{i_k},y_{i_k})^p \Big)^{\frac{1}{p}} \leq \Big( \sum_{i=1}^r d_i(x_i,y_i)^p \Big)^{\frac{1}{p}}.
\end{equation*}

\noindent Thus, by Theorem \ref{inequality of KR distances with Lipschitz function},  $\pi_A(\mu_i) \in P_{\XX_A}$, for $i = 1,2$, and Equation (\ref{equation_inegalite_KR_for_projection}) is satisfied. 
\end{proof}

\begin{cor}\label{inequality of KR distances with projection on Y_s and Y_s+r}
Let $r$ be a positive integer. For $1 \leq k \leq r$, let $\XX_k$ be a Polish space and $d_k \in \DD_{\XX_k}$. Let $S=\{i_1,\ldots, i_a\} \subset B= \{j_1, \ldots, j_b\}$ be non-empty subsets of $\{1, \ldots, r\}$. Keeping the notation of Theorem \ref{inequality of KR distances with projection}, we denote by $\XX_A = \XX_{i_1} \times \ldots \times \XX_{i_a}$, $\XX_B = \XX_{j_1} \times \ldots \times \XX_{j_b}$ and $\XX = \XX_1 \times \ldots \times \XX_r$, the respective Polish Cartesian product spaces. \\
If $\mu_1$ and $\mu_2$ are two probability measures on $P_{\XX}$, then the push-forward measures $\pi_A(\mu_i)$ and $\pi_B(\mu_i)$ are respectively in $P_{\XX_A}$ and $P_{\XX_B}$, and 
\begin{equation*}
W_{\XX_A}(\pi_A(\mu_1), \pi_A(\mu_2)) \leq W_{\XX_B}(\pi_B(\mu_1), \pi_B(\mu_2)),
\end{equation*}
where $W_{\XX_A}$ and $W_{\XX_B}$ are the respective Kantorovich-Rubinstein distances associated to $(\XX_A, \tilde{d}_{p,a})$ and $(\XX_B, \tilde{d}_{p,b})$.
\end{cor}

\begin{proof}[Corollary \ref{inequality of KR distances with projection on Y_s and Y_s+r}]
Let $\pi_{AB}: \XX_B \rightarrow \XX_A$ denote the canonical projection from $\XX_B$ to $\XX_A$. Then, by Lemma \ref{lsc_of_product_distance}, $\tilde{d}_{p,a} \in \DD_{\XX_A}$ and $\tilde{d}_{p,b} \in \DD_{\XX_B}$. As in Theorem \ref{inequality of KR distances with projection}, $\pi_{AB}$ is measurable and $(\tilde{d}_{p,b},\tilde{d}_{p,a})$-Lipschitz. Since $\pi_A = \pi_{AB} \circ \pi_B$, Theorem \ref{inequality of KR distances with projection} yields
\begin{align*}
W_{\XX_A}(\pi_A(\mu_1), \pi_A(\mu_2)) &=  W_{\XX_A}(\pi_{AB}(\pi_B(\mu_1)), \pi_{AB}(\pi_B(\mu_2)))\\
&\leq W_{\XX_B}(\pi_B(\mu_1), \pi_B(\mu_2)),
\end{align*}
which completes the proof.
\end{proof}

\cleardoublepage

\chapter{An Introduction to Classification problems}\label{Classification_problem_background}

\section{Statistical Machine Learning}\label{section_def_Stat_ML}

\textit{Machine learning} is the field encompassing the study, the conception and the implementation of computer algorithms that can learn. In his book \cite{Mitchell_ML}, T. M. Mitchell provides the following definition of \textit{learning} : \\
``A entity is said to learn from experience \textit{E} with respect to some class of tasks \textit{T} and performance measure \textit{P}, if its performance at tasks in \textit{T}, as measured by \textit{P}, improves with experience \textit{E}." 

\noindent In order to develop, compare and ultimately improve learning algorithms, one needs a framework to characterise the mathematical representation of the experiences \textit{E}, and to mathematically define the class of tasks \textit{T} and the notion of performance measures \textit{P}. \textit{Statistical learning theory} provides such a framework. Note that statistical learning theory is a framework built for the scenario of \textit{supervised learning}. According to Vapnik (one of the key contributors of the domain), the idea is to consider learning problems through the statistical framework of minimising the expected value of a carefully chosen loss function using labeled data. Hence, statistical learning theory considers machine learning tasks from the perspective of both statistical inference and optimization theory. The following description of statistical learning is found in Vapnik's \textit{The Nature of Statistical Learning Theory} \cite{Vapnik_livre_jaune}.
It requires three assumptions:
\begin{enumerate}[noitemsep]
\item a measure space $\XX$ of observations (also called an input space) and the existence of a generator of observations, drawn independently from the same unknown probability measure $P$.
\item a (measurable) output space $\YY$ and the existence of a supervisor that returns an output $y$ for every input $x$, according to a fixed but unknown conditional distribution $P(\, . \, | \, x)$.
\item the existence of a learning machine capable of implementing any function $g \in \GG$ where $\GG$ is a fixed subset of all measurable functions mapping $\XX$ to $\YY$. The function $g$ is called a \textit{learning function} and the set $\GG$ is called a \textit{model}. 
\end{enumerate}
Learning is to choose, from a given model $\GG$, the function $g$ which best predicts the supervisor's response. In order to choose the best available approximation to the supervisor's response, one measures the loss (or discrepancy) $L(y, g(x))$ between the response $y$ of the supervisor to a given input $x$ and the response $g(x)$ provided by the learning machine. Consider the expected value of the loss, given by the \textit{risk functional}:
\begin{equation}\label{risk_functionnal}
R(g) = \int L(y, g(x)) \diff P(x,y).
\end{equation}
The goal is to find the function $g_{\circ}$ which minimises the risk functional $R$ (over the model $\GG$) in the situation where the joint probability distribution $P(x,y)$ is unknown and the only available information is contained in a sample of $n$ independent and identically distributed observations on $\XX \times \YY$.

\section{Classification} \label{section_definition_classification}

\textit{Classification} is the assignment of any given observation $x \in \XX$ to one of $m$ classes ($m \in \NNN^*$). An observation that has been assigned to a particular class $c$ is said to be labeled. Hence, to classify is to create a measurable function $g: \XX \rightarrow \{1, \ldots , m\}$, where $\XX$ is the space of observations and $\YY = \{1, \ldots , m\}$ is the output space (or label space). Such a measurable function $g$ is called a \textit{classifier}. Given a classifier $g$, an observation $x \in \XX$ is well-classified if $g(x)$ matches the label $y \in \YY$ associated to $x$. A \textit{classification error} (also known as a \textit{misclassification}) occurs if $g(x) \not = y$. The performance of a classifier $g$ is measured by its accuracy or, conversely, its probability of misclassification. The accuracy of a classifier is the probability that it assigns a given observation $x \in \XX$ to the correct class $c$. Conversely, the classification error is the probability that the classifier labels an observation incorrectly.

To formalize the definition of a classification problem, we consider a probability measure $P$ on the product $\XX \times \YY$ of the input space $\XX$ and the output space $\YY$ and a family $\GG$ of classifying functions mapping $\XX$ to $\YY$.\\
The problem of classification is to find a classifier $g_{\circ} \in \GG$ which minimises the probability of misclassification $P(g(x) \not= y)$ in the situation where $P$ is unknown but an independent and identically distributed sample $S_n = \left( (x_1,y_1), \ldots, (x_n,y_n)\right)$ of $n$ labeled observations is given.

\noindent For the purpose of this thesis, we limit ourselves to the case of binary classification. That is, we consider a label space $\YY$ of cardinality 2. By convention, we identify $\YY$ with the set $\{-1, 1\}$. Hence, for  A good reference covering the theoretical material of multi-class classification is the chapter 8 of \textit{Foundations of Machine Learning} \cite{MIT_Foundation_of_ML}.

\section{Bayes Error and the Estimation Error}\label{Bayes_error_subsection}

Over all measurable classifiers, the infimum of the classification errors is defined as the Bayes error and is denoted by $R^*$. One can construct the unique measurable classifier that attains the Bayes error. It is called the Bayes classifier and is denoted by $g^*$. One can show that the Bayes classifier $g^*$ is defined as follow:
\begin{equation*}
g^*(x) = \begin{cases}
1 &\quad\textnormal{ if } \eta(x)\geq \frac12\\
-1 &\quad\textnormal{ otherwise }
\end{cases}
\end{equation*}
where $\eta(x)$ is the conditional probability that the label $y$ is equal to 1, given an observation $x$.\\
Since a given model $\GG$ is by construction a subset of all measurable functions, \textit{the excess error} of a classifier $g \in \GG$ is defined as the discrepancy between the misclassification error of $g \in \GG$ and the Bayes error. Thus, the problem of classification defined in section \ref{section_definition_classification} is equivalent to the minimization of the excess error over all $g \in \GG$. For $g \in \GG$, the difference $R(g) - R^*$ can be decomposed as follow:
\begin{equation*}
R(g) - R^* = R(g) - \inf_{g \in\GG} R(g) \, + \, \inf_{g \in\GG} R(g) - R^*.
\end{equation*}
The first difference on the right hand side of the equal sign is referred to as the estimation error while the second difference is known as the approximation error. The estimation error measures the quality of the classifier $g$ with respect to the optimal misclassification error of $\GG$ while the approximation error measures how well the optimal classification error of $\GG$ can approximate the Bayes error. The decomposition of the excess error in terms of the estimation and approximation errors shows that the choice of the model $\GG$ is subject to a trade-off: a rich model is more likely to have a small approximation error but at the price of a larger estimation error, and vice-versa. Note that since both the joint probability distribution $P(x,y)$ and the conditional distribution $P(\, . \, | \, x)$ are unknown, both $g^*$ and $R(g)$ are unknown. Even with various noise assumptions, estimating the approximation error is difficult \cite{MIT_Foundation_of_ML}. On the other hand, the estimation error can be bounded.

\section{Representational Capacity of a Model and Generalization Ability of a Classifier}\label{Capacity_and_Generalization}

As written in the subsection \ref{Bayes_error_subsection}, the focus of statistical learning theory is to bound the estimation error $R(g) - \inf_{g \in\GG} R(g)$. The difficulty of this task stems from the fact that the only information we have is contained in the training set of $n$ unlabeled observations drawn independently according to the joint probability measure $P$ on $\XX \times \YY$. Thus, as said above, for a classifier $g \in \GG$, one cannot obtain $R(g)$ and thus cannot obtain $\inf_{g\in\GG} R(g)$ either. One can only measure the agreement of $g$ with the $n$ points in the training set $S_n$. The standard measurement of agreement is the empirical probability of misclassification:
\begin{equation*}
\hat{R}_S(g) = \frac1n \sum_{i=1}^n \indicator_{g(X_i) \not = Y_i}.
\end{equation*} 

Once we compute the empirical probability of misclassification $\hat{R}_s(g)$ we need a guarantee that the ``general" probability of misclassification $R(g)$ is close to $\hat{R}_s(g)$. A classifier with a small discrepancy between its empirical error and its general misclassification error $R(g)$ is said to have a high generalization ability. The problem is that the generalization ability of a particular classifier $g \in \GG$ is unknown since the general misclassification error is unknown. Therefore, a lot of effort has been spent to construct quantitative predictions on the discrepancy between $\hat{R}_S(g)$ and $R(g)$ over all $g \in \GG$. These predictions are based on the concept of \textit{representational capacity} (often just called \textit{capacity}). Informally, the representational capacity of a model $\GG$ tries to capture the idea that $\GG$ contains classifiers nimble enough to mimic well the unknown classification. Said differently, the capacity of a model $\GG$ tries to capture the model's ability to fit a wide variety of classifications. It is by properly quantifying the capacity of the model $\GG$ that statistical learning theory was able to construct, over all $g \in \GG$, uniform upper bounds for the difference $R(g) - \hat{R}_S(g)$. These upper bounds are functions of both the cardinality of the training set $S_n$ of the training set and values that measure the capacity of the model $\GG$. The general form of most generalisation bounds is composed of three different terms and has the following form:\\
\textit{With probability at least }$1-\delta$,
\begin{equation*}
R(g) \leq \hat{R}_S(g) + \textit{capacity}\,(\GG) + \textit{confidence}\,(n,\delta).
\end{equation*}
The central results in statistical learning theory show that for a given model, the upper bound decreases as the cardinality $n$ of the training set increases while for a given training set, the upper bound is larger for a model with greater capacity. Hence, a given model $\GG$ of high capacity increases the risk that a classifier $g \in \GG$ has a low generalization ability. On the other hand, a model $\GG$ of low capacity increases the risk that a classifier $g \in \GG$ has a high generalization ability.\\
For binary classification problems, the most well known means of quantifying representational capacity is the Vapnik-Chervonenskis dimension (VC dimension). The VC dimension is defined as being the largest possible integer $l$ for which there exists a training set of $l$ different points in $\XX$ that the model $\GG$ can label arbitrarily. An important aspect of the VC dimension is that it is independent of the distribution $P$ and thus the same upper bound holds for any distribution. The drawback is that the bound may be loose for most distributions. In the early 2000's, several authors considered alternative notions of capacity such as maximum discrepancy, Rademacher averages and Gaussian averages. These new(er) notions of capacity are dependent on the distribution $P$ but give sharper upper bounds and have properties that make their computation possible from the training set only. We study Rademacher averages in more details later in the thesis. 

\section{Underfitting, Overfitting and Regularisation}

As explained in the previous sections, the goal of classification is to minimise the probability of misclassification $R(g)$. For a given model $\GG$, we focus on minimising the estimation error even though we cannot compute $R(g)$ as we do not know the underlying measure of probability. We thus try to infer a classifier $g_s\in\GG$, based on the training sample $S_n$, whose probability of error $R(g_s)$ is close to $\inf_{g\in\GG}R(g)$. The most intuitive way to find such a classifier $g_s$ is by:
\setlist{nolistsep}
\begin{enumerate}[noitemsep]
\item \noindent replacing the ``general" probability of misclassification $R$ by the empirical probability of misclassification $\hat{R}_S$, constructed on the basis of the training set $S_n$.
\item approximating the function $g_{\circ}$ that minimises $R$ over $\GG$ by $g_s$ that minimises $\hat{R}_S(g)$ over $G$. 
\end{enumerate}
Using the empirically optimal classifier $g_s$ to approximate $g_{\circ}$ is known as the \textit{empirical risk minimization} inductive principle (ERM).\\
Once we compute $g_s$, there are two factors determining how close the misclassification error $R(g_s)$ is from $\inf_{g \in\GG} R_S(g)$:
\setlist{nolistsep}
\begin{enumerate}[noitemsep]
\item The empirical misclassification error $\hat{R}_S$ of $g_s$.
\item The generalization ability of $g_s$.
\end{enumerate}
These two factors are linked to the two central concerns in the field of Machine Learning: underfitting and overfitting. Underfitting occurs when the classifier $g_s \in \GG$ is such that $\hat{R}_S(g_s)$ is not sufficiently small. Overfitting occurs when the difference between the empirical error $\hat{R}_s(g)$ and the misclassification error $R(g)$ is big. The fondamental results of Statistical Learning Theory show that one can control how likely it will be that $g_s$ underfits or overfits by altering the capacity of the model $\GG$. As seen in the last section, if the model is of large capacity, it is more likely that the empirically optimal classifier $g_s$ overfits while for a small capacity model it is less likely that $g_s$ overfits. On the other hand, a small capacity model increases the risk of underfitting while a model with a large capacity $C$ decreases the risk of underfitting. Thus, there is a tradeoff to make: a larger capacity decreases the risk of underfitting but increases the risk of overfitting while a smaller capacity decreases the risk of overfitting but increases the risk of underfitting. \\
In order to strike the right balance for the model's capacity, Vapnik and Chervonenkis developed the concept of \textit{Structural Risk Minimization} (SRM) \cite{Vapnik_Theory_of_Pattern_Recognition_book}. SRM is intended to minimise the classification error with respect to both the empirical error $\hat{R}_s$ and the capacity of the model used. The algorithm is defined as follow: Consider a nested sequence $\{\GG_1, \GG_2, \GG_3, \ldots\}$ of models with respective capacities $C_1, C_2, C_3, \ldots$, such that the countable union of all $\GG_i$'s is dense in $\GG$. For each $i$, consider $g_s^{(i)}$, the empirically optimal classifier over the classe $\GG_i$. We select the classifier $g_s^{\circ}$ minimising the capacity-penalized empirical error:
\begin{equation*}
g_s^{\circ} = \argmin_{i \in \NNN} \left( g_s^{(i)} + r(n,C_i) \right),
\end{equation*}
where the penalty $r(n,C_i)$ can be understood as an estimate of the overfitting magnitude of $g_s^{(i)}$. As such, $r$ depends on the capacity $C_i$ of $\GG_i$. Thus, SRM identifies an optimal model $\GG_{i_{\circ}}$ and returns the classifier $g_s^{(i_{\circ})}$ that minimises its empirical error. \\
While SRM has a strong theoretical footing \cite{Vapnik_Overview_Stat_Learning_theory, Prob_theory_of_Pattern}, it is often computationally very expensive as it requires determining the solution of multiple empirical risk minimization problems. It is therefore rarely used. 

Today, rather than SRM, the approach used to minimise the classification error is \textit{regularisation}. Regularisation based algorithms are inspired by SRM based algorithms but are more general. A regularisation algorithm is defined as follow: Consider an uncountable union of nested models $\GG_{j}, \, j \in \JJ$ with respective capacities $C_{j}$, such that the union $\GG$ of all $\GG_{j}$'s is dense in the space of continuous functions over $\XX$. As for SRM, we select the classifier $g_s^{\circ}$ minimising the constrained optimization problem of the form:
\begin{equation}\label{constrained_quadratic_programming}
\min_{j \in \JJ,\, g \in \GG_j} \left( g^{(j)}_s + r(n,C_{\gamma}) \right),
\end{equation}
where the penalty $r(n,C_j)$ has the same signification than in the case of SRM.

Under quite general assumptions, there exist a function $\Gamma: \,\GG \rightarrow \RR$ such that the constrained optimization problem (\ref{constrained_quadratic_programming}) can be equivalently written as the unconstrained optimization problem:
\begin{equation*}
g_s^{\circ} = \argmin_{g\in\GG} \left( \hat{R}(g) + \lambda \Gamma(g)\right),
\end{equation*}
where $\lambda \in [0, +\infty)$ is the regularisation parameter. The regularisation parameter $\lambda$ is treated as an hyperparameter since its optimal value is not known. Note that the value of $\lambda$ is set prior to solving the constrained optimization problem. The function $\Gamma$ is called a \textit{regularizer}. A regularizer (also known as a regularisation term) is a quantity that penalizes one classifier over another in the model $\GG$. Hence, for two classifiers in $\GG$, both are eligible as solutions of the optimization problem but one classifier is preferred. The penalized function is chosen only if it classifies the training data significantly better than the preferred function. Often, the regularizer $\Gamma$ is designed to express a generic preference for a simpler model class in order to promote generalization. As noted by Bengio \textit{et al.} in \textit{Deep Learning} \cite{Deep_Learning_book}, regularisation is a more general way of controlling a model's capacity than SRM. Penalizing one classifier over another is more general than excluding subsets of classifiers from the model $\GG$. One can think of removing a classifier from the model $\GG$ as expressing an infinitely strong penalty against that classifier. The hyperparameter $\lambda$ weights the relative contribution of the regularisation term $\Gamma(g)$ with respect to the empirical misclassification error. Setting $\lambda = 0$ means no regularisation and increasing $\lambda$ corresponds to more regularisation.

\section{Loss Functions and Confidence Margins}

A quick introduction to \textit{the problem of learning} was given in section \ref{section_def_Stat_ML}. As stated in that section, we use the definition of the learning problem given by Vapnik in \textit{The Nature of Statistical Learning Theory} \cite{Vapnik_livre_jaune}. In his book, Vapnik presents learning problems as particular cases of minimising the risk functional on the basis of empirical data:\\
Let $(\ZZ, \BB)$ be a measurable space and $\mu$ be a probability measure on $\ZZ$. Consider a set $\Lambda$ of $\mu$-integrable random variables $Q:\ZZ \rightarrow \RR_+$. For a given random variable $Q \in \Lambda$, the risk functional $R$ is the expected value of $Q$. The goal is to minimise the risk functional $R$ over all $Q\in \Lambda$, where the probability measure $\mu$ is unknown but a independent and identically distributed sample $z_1,\ldots,z_n$ is given.\\
To obtain the particular case of learning problems from the minimising of risk functional on the basis of empirical data, one considers a family $\HH$ of measurable functions $h:\XX\rightarrow\YY$ and one loss function $L:\YY \times \YY \rightarrow \RR_+$. Then, $\ZZ = \XX \times \YY$ and the $\mu$-integrable random variables $Q:\XX\times\YY \rightarrow \RR$ are defined by $Q(x,y) = L(h(x), y)$. The goal of a learning problem is to find the learning function $h_{\circ}$ that minimises the risk functional
\begin{equation*}
R(h) = \int_{\XX\times\YY} L(h(x),y) \diff \mu(x,y)
\end{equation*}
where the probability measure $\mu$ on $\XX\times\YY$ is unknown but a independent and identically distributed sample $(x_1,y_1), \ldots,(x_n,y_n)$ is given.\\

\noindent Let us now formalise the definition of a loss function $L$.
\begin{defn}[Loss function]\label{def_loss_fct}
Let $(\XX\times\YY, \BB)$  be a measurable space, $\mu$ be a probability measure on $\XX\times\YY$ and let $\HH$ be a family of measurable functions $h:\XX\rightarrow\YY$. A map $L:\YY\times\YY' \rightarrow \RR_+$, with $\YY' \subset \YY$, is a loss function for $\HH$ if:
\begin{enumerate} 
\item[(i)]the composed function
\begin{equation*}
\setlength\arraycolsep{0pt}
L_{h}\colon \begin{array}[t]{ >{\displaystyle}r >{{}}c<{{}}  >{\displaystyle}l } 
          \XX \times \YY' &\rightarrow& \RR_+ \\ 
          (x,y') &\mapsto& L(h(x),y) 
\end{array}
\end{equation*}
is $\mu$-integrable for all $h\in\HH$;\\
\item[(ii)] For all $h \in \HH$, for all $(x,y)\in D_h \times I_h$, h(x) = y implies $L(h(x),y) = 0$.
\end{enumerate}
\end{defn}

\begin{rmk}
Definition \ref{def_loss_fct} of a loss function is found in \textit{Learning with Kernels} by Sch\"olkopf and Smola \cite{livre_noir}. The authors specify that it is possible to relax the non-negativity of $L$. Indeed, it is enough for the image of $L$ to be bounded from below. In that case, an appropriate translation would recover the non-negativity.\\
Likewise, it is possible to relax condition $(ii)$ because an appropriate translation would recover the condition that exact predictions have a loss of zero.
\end{rmk}

\vspace{2mm}
\noindent The definition of \textit{classification} given in section \ref{section_definition_classification} can be expressed as a particular case of a \textit{learning problem}:\\
Consider a model $\GG$ of measurable binary classifiers $g:\XX\rightarrow\{-1,1\}$ and the zero-one loss function $L_{\text{z-o}}:\{-1,1\} \times \{-1,1\} \rightarrow \RR_+$ defined by:
\begin{equation*}
L_{\text{z-o}}(y_1, y_2) = \begin{cases}
0 &\quad\textnormal{ if } y_1 = y_2\\
1 &\quad\textnormal{ if } y_1 \not= y_2
\end{cases}
\end{equation*}
Note that the zero-one loss function can also be written as  $L_{z-o}(y_1, y_2) = \indicator_{y_1\not= y_2}$.\\
The goal is thus to find a classifier $g_{\circ}$ that minimises the risk functional $R(g)$. But $R(g)$ can be simplified:
\begin{align}
R(g) &= \int_{\XX\times\{-1,1\}} L_{\text{z-o}}(g(x),y) \diff \mu(x,y) \notag \\
&= \int_{\XX\times\{-1,1\}} \indicator_{g(x)\not=y} \diff \mu(x,y) \notag \\
&= \mu \left(\{(x,y)\in \XX\times\{-1,1\}: \, g(x)\not=y \}\right). \label{functional_for_zero_one}
\end{align}
We thus see from the equalities above that if we incur a loss of 1 to any couple $(x,y)\in \XX\times\{-1,1\}$ such that $g(x)\not=y$, and $0$ otherwise, the risk functional $R$ is the probability $\mu$ of misclassification.

\begin{rmk}
There are many equivalent formulas to express the risk functional obtained with the zero-one loss: \\
\begin{enumerate}
\item[(i)] $\dps R(g) = \int_{\XX\times\{-1,1\}} \frac12 |g(x) - y| \diff \mu(x,y)$;\\
\item[(ii)] $\dps R(g) = \int_{\XX\times\{-1,1\}} \frac12 (1 - g(x)y) \diff \mu(x,y)$.
\end{enumerate}
\end{rmk}

\vspace{4mm}

\noindent For a binary classifier $g:\XX \to \{-1,1\}$, it is often useful to consider its associated \textit{classification function} $f: \XX\to\RR$:

\begin{defn}
Let $(\XX,\BB)$ be a measurable space and $f: \XX \to \RR$ be a real-valued measurable function. The sign $\sgn(f)$ of $f$ defines a binary classifier $g:\XX \to \{-1,1\}$ defined by:
\begin{equation*}
g(x) = \begin{cases}
1 &\quad\textnormal{ if } f(x)\geq 0 \\
-1 &\quad\textnormal{ if } f(x)< 0.
\end{cases}
\end{equation*}
\end{defn}

\begin{term}
One says that the function $f$ is the \textit{classification function} of $g$ and one writes $g_f$. 
\end{term}
\vspace{3mm}
For a family of classifiers $g_f:\XX \to \{-1,1\}$, it is often useful to consider both its associated family $\FF$ of classification functions $f: \XX\to\RR$ and loss functions of the form:
\begin{equation*}
\setlength\arraycolsep{0pt}
\L \colon \begin{array}[t]{ >{\displaystyle}r >{{}}c<{{}}  >{\displaystyle}l } 
          \XX \times \{-1,1\}&\rightarrow& \RR_+ \\ 
          (y_1,y_2) &\mapsto& \Phi(y_1y_2),
\end{array}
\end{equation*}
where $\Phi: \RR \rightarrow \RR_+$ is a bounded real function.\\
We can now express, in the spirit of Equation \ref{functional_for_zero_one}, the risk functional $R(g_f)$ for the classifier $g_f$ linked to a classification function $f$. To do so, we apply the Disintegration Theorem \ref{Disintegration_thm}:\\
For $(\XX,\BB)$ and $(\{-1,1\}, \PPP\{-1,1\})$ two Borel spaces, we consider $(\XX\times\{-1,1\}, \BB\otimes\PPP\{-1,1\})$ the product space with the natural product measure, a probability measure $\mu$ on $\XX\times\{-1,1\}$ and the canonical function $\pi:\XX\times\{-1,1\}\to\{-1,1\}$. Then, there exists a $\pi(\mu)$-almost everywhere determined pair of measures $\mu_1$ and $\mu_{-1}$ on $\XX\times\{-1,1\}$ such that
\begin{enumerate}
\item[(i)] $\mu_i(\pi^{-1}(\{i\})) = 1$, for $i=-1,1$. Hence, for $A \subset \XX\times\{-1,1\}$, $\mu_i(A) = \mu(A \cap \pi^{-1}(\{i\})) = \mu(A \cap \XX\times\{i\})$, 
\item[(ii)] for every Borel measure function $h: \XX\times\{-1,1\} \to \RR_+$,
\begin{align*}
\int_{\XX\times\{-1,1\}} h \diff \mu(x,y) &= \int_{\{-1,1\}} \int_{\pi^{-1}(\{i\})} h \diff \mu_i(x,y) \diff \pi(\mu)(\{i\}) \\[1.8mm]
&= \sum_{i=1}^2 \pi(\mu)(\{i\}) \int_{\XX\times\{i\}} h \diff \mu_i(x,y).
\end{align*}
\end{enumerate}
Thus, in the case where $h(x,y) = \Phi(yf(x))$, we have 
\begin{equation*}
R(g_f) = \sum_{i=1}^2 \pi(\mu)(\{i\}) \int_{\XX\times\{i\}} \Phi(if(x)) \diff \mu_i(x,y).
\end{equation*}
We define two new measures $\mu_+$ and $\mu_-$ on $\XX$ such that, for any $B\in\BB$:
\begin{equation}\label{def_mu_+_et_mu_-_theorique}
\mu_+(B) = \frac{1}{\mu_1(\XX)}\,\mu_1(\pi_x^{-1}(B)) \, \mbox{ and } \, \mu_-(B) = \frac{1}{\mu_{-1}(\XX)}\,\mu_{-1}(\pi_x^{-1}(B)),
\end{equation}
\noindent With the measures $\mu_+$ and $\mu_-$, we can then write the risk functional as:
\begin{equation}\label{eqn_risk_avec_mu+_mu-}
R(g_f) = \int_{\XX} \Phi(f(x)) \diff \mu_+ + \int_{\XX} \Phi(-f(x)) \diff \mu_-.
\end{equation}

We now give an interpretation of the quantities $\Phi(yf(x))$ and $yf(x)$ of a given point $(x,y) \in \XX\times\{-1,1\}$. The interpretation is based on the notion of \textit{confidence of prediction}. Consider a real-valued function $f\in\FF$. By construction, $f$ defines the classifier $g_f = \sgn(f)$. One interprets the absolute value $|f(x)|$ as the confidence of the prediction $g_f(x) = \sgn(f(x))$ made by the classifier $g_f$. Given a point $(x,y)\in\XX\times\{-1,1\}$, the confidence margin of the prediction $\sgn(f(x))$ is the quantity $yf(x)$. Hence, when the product $yf(x)$ is positive, one can conclude that the classifier $\sgn(f(x))$ classifies the point $x \in\XX$ correctly with a confidence $f(x)$. The function $\Phi: \RR \rightarrow[a,b]$ represents the loss that one wants to incur on a classification $\sgn(f(x))$ with confidence $|f(x)|$.\\
It is interesting to note that the zero-one loss defined earlier in this subsection can be written in the $\Phi(yf(x))$ form: 

\begin{prop}
Let $(\XX\times\{-1,1\}, \BB)$  be a measurable space, $\mu$ be a probability measure on $\XX\times\{-1,1\}$. Let $f:\XX\to\RR$ be a real-valued function and $g_f$ be the classifier linked to $f$. Then, for any $(x,y) \in \XX\times\{-1,1\}$, 
\begin{equation*}
L_{\text{z-o}}(g_f(x),y) =  1 - H(yf(x)) =  \begin{cases}
1 &\quad\textnormal{ if } yf(x) \leq 0\\
0 &\quad\textnormal{ if } yf(x) > 0
\end{cases}
\end{equation*} 
where $H$ is the Heaviside function. 
\end{prop}
We observe that for $\Phi = 1- H$, the value of the composed function $(x,y)\mapsto 1-H(yf(x))$ is only dependent on the sign of $f$ and not on its value. Said differently, a misclassified point $(x,y)$ incurs a loss of 1 no matter the value of $|f(x)|$. Likewise, a well classified point incurs no loss no matter the value $|f(x)|$. If one wants to incur a loss that takes into consideration the value $|f(x)|$ of a point $(x,y) \in \XX\times \{-1,1\}$, one needs to find an appropriate function $\Phi$.\\
We give two examples of loss functions of the form $\Phi(y_1y_2)$. First we consider the \textit{$\alpha$-translated zero-one loss}. It penalizes both misclassified points and correctly classified points with confidence smaller than $\alpha$. The penalty is 1. The $\alpha$-translated zero-one loss function can be written in this form:

\begin{defn}[$\alpha$-translated zero-one loss]\label{alpha_translated_loss}
For any $\alpha >0$, the $\alpha$-translated zero-one function $L_{\alpha}: \RR \times \{-1,1\} \rightarrow \RR_+$ is defined by 
\begin{equation*}
L_{\alpha}(y_1,y_2) = \indicator_{y_1y_2 \leq \alpha} = \begin{cases}
1 &\quad\textnormal{ if } y_1y_2 \leq \alpha\\
0 &\quad\textnormal{ if } y_1y_2 > \alpha
\end{cases}
\end{equation*}
\end{defn}

Next, we consider the \textit{$\alpha$-margin loss function} defined in \textit{Foundations of Machine Learning, $2^{nd}$ edition} \cite{MIT_Foundation_of_ML}. It penalizes misclassified points with the cost of 1 but also penalizes points correctly classified with a confidence smaller or equal to $\alpha$ with a linear penalty of slope $-1/\alpha$. The $\alpha$-margin loss function can be written in the following format:

\begin{defn}[$\alpha$-margin loss function]
For any $\alpha >0$, the $\alpha$-margin loss function $L_{\Phi_{\alpha}}: \RR \times \RR \rightarrow \RR_+$ is defined by $L_{\Phi_{\alpha}}(y_1,y_2) = \Phi_{\alpha}(y_1y_2)$ where
\begin{equation*}
\Phi_{\alpha}(u) =  \min \left(1, \max\ \left(0, 1- \frac{u}{\alpha}\right)\right) = \begin{cases}
1 &\quad\textnormal{ if } u \leq 0\\
1-\frac{x}{\alpha} &\quad\textnormal{ if } 0 \leq u \leq \alpha \\
0 & \quad\textnormal{ if } y \geq \alpha
\end{cases}
\end{equation*}
\end{defn}

Note that, for any $u \in \RR$, $\Phi_{\alpha}(u) \leq \indicator_{u \leq \alpha}$. Thus, for all $(x,y) \in \XX\times\{-1,1\}$, the $\alpha$-margin loss function is smaller or equal to the $\alpha$-translated zero-one loss function. We therefore have the following inequality between their respective risk functionals:
\begin{equation}\label{margin_loss_smaller_than_alpha_translated_0-1_loss}
\int_{\XX\times\{-1,1\}} \Phi_{\alpha}(yf(x)) \diff\mu(x,y) \leq \int_{\XX\times\{-1,1\}} \indicator_{yf(x)\leq\alpha} \diff\mu(x,y).
\end{equation}
The $\alpha$-translated zero-one loss admits an interpretation: it is a measure of the set of points that has been misclassified or classified with a confidence smaller than $\alpha$. The limitation of the $\alpha$-translated zero-one loss is, like for the zero-one loss, the lack of proportionality between the loss incurred by a couple $(x,y)$ and the confidence of the prediction $f(x)$. The margin loss function addresses this short-coming: the penalty for a well classified point with a confidence smaller or equal to $\alpha$ decreases by a slope of $1/\alpha$ with respect to its confidence. If the confidence is greater than $\alpha$, the penalty is zero, signalling that the impact of points with very large confidence should be limited. 

\section{Rademacher Averages and Generalization Bounds}

We now present strong theoretical justifications for the use of loss functions that take into consideration the confidence of predictions. These justifications are based on generalization bounds and the concept of representational capacity, briefly presented in subsection \ref{Capacity_and_Generalization}. As mentioned in that subsection, a properly defined representational capacity allows to construct generalization bounds. In this subsection, we present the definition of, and the important results linked to, \textit{Rademacher averages}. The notion of Rademacher average quantifies the representational capacity of a model $\GG$. We rely heavily on chapter 3, 4 and 5 of \textit{Foundations of Machine Learning, $2^{nd}$ edition} \cite{MIT_Foundation_of_ML} and on the paper \textit{Rademacher and Gaussian Complexities: Risk Bounds and Structural Results} by Bartlett and Mendelson \cite{Bartlett_et_Mendelson}. \\
Note that many papers use the term \textit{Rademacher complexity} instead of \textit{Rademacher averages}. We do not use the term \textit{Rademacher complexity} as we think that it can be confused for a descriptor of classification problem complexity that we introduce in Chapter \ref{KR_score_appplied_to_gwas}. Hence, in this thesis, we use the term \textit{Rademacher averages}, as Bousquet \textit{et al.} \cite{Intro_to_Stat_Learning_theory_Bousquet}. 

\begin{defn}[Empirical Rademacher Average]
Let $(\ZZ, \BB)$ be a measurable space and let $\mu$ be a probability measure on $\ZZ$. Let $S_n = \{z_1, \ldots, z_n\}$ be an independent and identically distributed sample with respect to $\mu$. For a given family $\HH$ of functions $h: \ZZ\rightarrow \RR$, we define its empirical Rademacher average $\hat{\RRR}_S(\HH)$ by
\begin{equation*}
\hat{\RRR}_S(\HH) = \EE_{\sigma} \left( \sup_{h \in \HH} \left(\frac{1}{n} \sum_{i = 1}^n \sigma_i h(z_i)\right)\right),
\end{equation*}
where $\sigma = (\sigma_1, \ldots, \sigma_m)$, with $\sigma_i$'s independent uniform random variables taking values in $\{-1, 1\}$. The random variables $\sigma_i$ are called \mbox{Rademacher variables}.
\end{defn}

\begin{defn}[Rademacher average]
Let $(\ZZ, \BB)$ be a measurable space and let $\mu$ be a probability measure on $\ZZ$. Let $S_n = \{z_1, \ldots, z_n\}$ be an independent and identically distributed sample with respect to $\mu$. Let $\HH$ be a family of functions $h: \ZZ \rightarrow \RR$ and $\hat{\RRR}_S(\HH)$ be its the empirical Rademacher average. The Rademacher average of $\HH$ is the expected value of the empirical Rademacher averages $\hat{\RRR}_S(\HH)$ over all samples $S_n$ of size $n$:
\begin{equation*}
{\RRR}_n(\HH) = \EE \left( \hat{\RRR}_s(\HH) \right).
\end{equation*}
\end{defn}

According to Bartlett and Mendelson \cite{Bartlett_et_Mendelson}, Rademacher empirical averages are reasonable as measures of representational capacity as they quantify the extend to which the functions of $\HH$ can be correlated with a noise sequence of length $n$. The noise sequence is represented by the sequence of Rademacher variables. If a family of functions $\HH$ has a large Rademacher empirical average, there is, on average, more chance that one can find a function $h \in \HH$, such that $\dps \sum \sigma_i h(z_i)$ is large. 

\noindent We can now present the first generalization bounds theorem that uses the Rademacher averages. 

\begin{theo}\label{thm_generalization_bound_for_family_of_classifiers_conjugated_with_loss_function}
Let $(\ZZ, \BB)$ be a measurable space and let $\mu$ be a probability measure on $\ZZ$. Let $S_n = \{z_1, \ldots, z_n\}$ be an independent and identically distributed sample with respect to $\mu$. Let $\HH$ be a family of functions $h: \ZZ \rightarrow [0,1]$. Then, for any $\delta >0$, with probability at least $1- \delta$, each of the two following inequalities holds for all $h \in \HH$:
\begin{align*}
\EE[h(z)] &\leq \frac1n \sum_{i=1}^n h(z_i) + 2 \RRR_n(\HH) + \sqrt{\frac{\ln\frac{1}{\delta}}{2n}}, \\
\EE[h(z)] &\leq \frac1n \sum_{i=1}^n h(z_i) + 2 \hat{\RRR}_S(\HH) + 3\sqrt{\frac{\ln\frac{2}{\delta}}{2n}}.
\end{align*}
\end{theo}

\begin{proof}[Theorem \ref{thm_generalization_bound_for_family_of_classifiers_conjugated_with_loss_function}]
The proof can be found in page 31 of \textit{Foundations of Machine Learning, $2^{nd}$ edition} \cite{MIT_Foundation_of_ML}.
\end{proof}

The Rademacher averages are well defined for any family of functions $\HH$ mapping from an arbitrary measurable space $\ZZ$ to $\RR$. In statistical learning theory, given a model of binary classifier $\GG$ and a loss function $L: \YY \times \YY' \rightarrow [0,1]$, we construct $\LL_{\GG}$, the family of loss functions associated to $\GG$, mapping from $\ZZ= \XX \times \YY$ to $[0,1]$ and denoted by $\LL_{\GG}$:
\begin{equation*}
\LL_{\GG} = \{L_g:(x,y)\mapsto L(g(x),y):\, g \in \GG\}.
\end{equation*}

Now, in order to make use of Theorem  \ref{thm_generalization_bound_for_family_of_classifiers_conjugated_with_loss_function} to construct a generalization bound for a model $\GG$ of binary classifiers, one needs to relate the Rademacher averages of the family $\LL_{\GG}$ of loss functions to the Rademacher averages of the model $\GG$. Recall that in the case of binary classification, $\YY' = \{-1,1\} \subset \RR$. \\
We first consider the case of the zero-one loss: $L(y_1,y_2) = \indicator_{y_1\not= y_2}$:

\begin{lem}\label{link_btwn_Rademacher_concept_class_and_loss_class}
Let $(\XX\times\{-1,1\}, \BB)$ be a measurable space and let $\mu$ be a probability measure on $\XX\times\{-1,1\}$. Let $S_n = \{(x_1,y_1), \ldots, (x_n,y_n)\}$ be an independent and identically distributed sample with respect to $\mu$ and $S_x$ be the projection of $S_n$ on $\XX$. Let $\GG$ be a model of binary classifiers $g: \XX \rightarrow \{-1, 1\}$ and $\LL_{\GG}$ be the family of loss functions associated to $\GG$ for the zero-one loss. Then, $\hat{\RRR}_S(\LL_{\GG}) = \frac12 \hat{\RRR}_{S_x}(\GG)$.
\end{lem}

The proof of Lemma \ref{link_btwn_Rademacher_concept_class_and_loss_class} is very short and uses a clever trick that will be used again in the proof of Proposition \ref{egalite_Rademacher_pour_f(x)_et_yf(x)}. It can be found in \textit{Foundations of Machine Learning, $2^{nd}$ edition} \cite{MIT_Foundation_of_ML}.\\

\noindent Using both Theorem \ref{thm_generalization_bound_for_family_of_classifiers_conjugated_with_loss_function} and Lemma \ref{link_btwn_Rademacher_concept_class_and_loss_class}, we obtain the following result:

\begin{theo}[Rademacher averages bounds - binary classification]
Let $(\XX\times\{-1,1\}, \BB)$ be a measurable space and let $\mu$ be a probability measure on $\XX\times\{-1,1\}$. Let $S_n = \{(x_1,y_1), \ldots, (x_n,y_n)\}$ be an independent and identically distributed sample with respect to $\mu$. Let $\GG$ be a model of binary classifiers $g: \XX\rightarrow \{-1, 1\}$. Then, for any $\delta >0$, with probability at least $1- \delta$, each of the two following inequalities holds for all $g \in \GG$:
\begin{align*}
R(g) &\leq \hat{R}_S(g) +  \RRR_n(\GG) + \sqrt{\frac{\ln\frac{1}{\delta}}{2n}}, \\
R(g) &\leq \hat{R}_S(g) +  \hat{\RRR}_{S_x}(\GG) + 3\sqrt{\frac{\ln\frac{2}{\delta}}{2n}}.
\end{align*}
\end{theo} 

\begin{proof}
We apply Theorem \ref{thm_generalization_bound_for_family_of_classifiers_conjugated_with_loss_function} to the family $\LL_{\GG}$ defined by $\LL_{\GG} = \{L_g:(x,y)\mapsto \indicator_{g(x)\not=y}:\, g \in \GG\}$. Hence we obtain the inequation:
\begin{align*}
\EE[\indicator_{g(x)\not=y}] &\leq \frac1n \sum_{i=1}^n \indicator_{g(x_i)\not=y_i} + 2 \RRR_n(\LL_{\GG}) + \sqrt{\frac{\ln\frac{1}{\delta}}{2n}}, \\
\EE[\indicator_{g(x)\not=y}] &\leq \frac1n \sum_{i=1}^n \indicator_{g(x_i)\not=y_i} + 2 \hat{\RRR}_{S_x}(\LL_{\GG}) + 3\sqrt{\frac{\ln\frac{2}{\delta}}{2n}}.
\end{align*}
By definition of the risk $R$, $\EE[\indicator_{g(x)\not=y}] = R(g)$ and $\frac1n \sum_{i=1}^n \indicator_{g(x_i)\not=y_i} = \hat{R}_S(g)$. \\
By Lemma \ref{link_btwn_Rademacher_concept_class_and_loss_class}, $\hat{\RRR}_S(\LL_{\GG}) = \frac12 \hat{\RRR}_{S_x}(\GG)$ and by taking the expectation we get  $\RRR_n(\LL_{\GG}) = \frac12 \RRR_n(\GG)$.
\end{proof}

Now, we consider the case of loss functions of the form $L(y_1,y_2) = \Phi(y_1y_2)$, where $\Phi:\RR \rightarrow [0,1]$ is a k-Lipschitz function. We start with the following simple but useful proposition that is unrelated to the Lipschitz property:

\begin{prop}\label{egalite_Rademacher_pour_f(x)_et_yf(x)}
Let $(\XX\times\{-1,1\}, \BB)$ be a measurable space and let $\mu$ be a probability measure on $\XX\times\{-1,1\}$. Let $S_n = \{(x_1,y_1), \ldots, (x_n,y_n)\}$ be an independent and identically distributed sample with respect to $\mu$ and $S_x$ be the projection of $S_n$ on $\XX$. Let $\FF$ be a family of real-valued functions $f: \XX \rightarrow \RR$ and define $\tilde{\FF} = \{\tilde{f}: (x,y) \mapsto yf(x) : \, f \in \FF\}$. Then, $\hat{\RRR}_S(\tilde{\FF}) = \hat{\RRR}_{S_x}(\FF)$.
\end{prop}

\begin{proof}[Proposition \ref{egalite_Rademacher_pour_f(x)_et_yf(x)}]
The proof uses the fact that, for a fixed $y \in \{-1,1\}$, a Rademacher variable $\sigma$ and the variable $y\sigma$ are both distributed in the same way and take the same values. That is, they are both uniform random variables taking their values in $\{-1,1\}$. Thus:
\begin{align*}
\hat{\RRR}_S(\tilde{\FF}) &= \EE_{\sigma} \left( \sup_{\tilde{f} \in \tilde{\FF}} \left(\frac{1}{n} \sum_{i = 1}^n \sigma_i \tilde{f}(x_i,y_i)\right)\right) \\
&= \EE_{\sigma} \left( \sup_{f \in \FF} \left(\frac{1}{n} \sum_{i = 1}^n \sigma_i y_i f(x_i)\right)\right) \\
&= \EE_{\sigma} \left( \sup_{f \in \FF} \left(\frac{1}{n} \sum_{i = 1}^n \sigma_i f(x_i)\right)\right) \\
&= \hat{\RRR}_{S_x}(\FF).
\end{align*}
\end{proof}
The Lipschitz property of $\Phi$ is necessary to apply \textit{Talagrand's Lemma}. This lemma bounds the empirical Rademacher averages of $\LL_{\FF}$ (for the loss functions of the form $L(y_1,y_2) = \Phi(y_1y_2)$ in terms of the Rademacher averages of the family $\FF$ of real-valued functions $f:\XX \rightarrow \RR$:

\begin{lem}[Talagrand's lemma]\label{Talagrand}
Let $\Phi_1, \ldots, \Phi_m$ be k-Lipschitz functions from $\RR$ to $\RR$ and $\sigma_1, \ldots, \sigma_m$ be Rademacher random variables. Then, for any family $\FF$ of real-valued functions, the following inequality holds:
\begin{equation*}
\frac1m \, \EE_{\sigma} \left( \sup_{f \in \FF} \left(\sum_{i = 1}^m \sigma_i (\Phi_i \circ f) (x_i)\right)\right) \leq \frac{k}{m}\, \EE_{\sigma} \left( \sup_{f \in \FF} \left(\sum_{i = 1}^m \sigma_i f(x_i)\right)\right).
\end{equation*}
In particular, if $\Phi_i = \Phi$ for all $i$, then the following holds:
\begin{equation*}
\hat{\RRR}_S(\Phi \circ \FF) \leq k\, \hat{\RRR}_S(\FF).
\end{equation*}
\end{lem} 

\begin{proof}[Talagrand's Lemma \ref{Talagrand}]
The version of Talagrand's Lemma above is given in \textit{Foundations of Machine Learning, $2^{nd}$ edition} \cite{MIT_Foundation_of_ML}. It is a simpler and more concise version of a more general version given by Ledoux and Talagrand in \cite{livre_Talagrand_Ledoux}.
\end{proof}
\noindent Using both Proposition \ref{egalite_Rademacher_pour_f(x)_et_yf(x)}, and Talagrand's Lemma \ref{Talagrand}, we obtain a particular case of Theorem \ref{thm_generalization_bound_for_family_of_classifiers_conjugated_with_loss_function} for the case of loss functions of the form $L(y_1,y_2) = \Phi(y_1y_2)$:

\begin{theo}\label{thm_generalization_bound_for_classifiers_conjugated_with_loss_function}
Let $(\XX\times\{-1,1\}, \BB)$ be a measurable space and let $\mu$ be a probability measure on $\XX\times\{-1,1\}$. Let $S_n = \{(x_1,y_1), \ldots, (x_n,y_n)\}$ be an independent and identically distributed sample with respect to $\mu$ and $S_x$ be the projection of $S_n$ on $\XX$. Let $\FF$ be a family of real-valued functions $f: \XX \rightarrow \RR$ and $\Phi: \RR \rightarrow [0,1]$ be a k-Lipschitz function. Denote by $\LL_{\FF}$ the family of loss functions associated with $\FF$ for the loss function $L(y_1,y_2) = \Phi(y_1y_2)$. Then, for any $\delta >0$, with probability at least $1- \delta$, each of the two following inequalities holds for all $g \in \FF$: 
\begin{align*}
\EE[\Phi(yf(x))] &\leq \frac1n \sum_{i=1}^n \Phi(y_if(x_i)) + 2k \RRR_n(\FF) + \sqrt{\frac{\ln\frac{1}{\delta}}{2n}}, \\
\EE[\Phi(yf(x))] &\leq \frac1n \sum_{i=1}^n \Phi(y_if(x_i)) + 2k \hat{\RRR}_S(\FF) + 3\sqrt{\frac{\ln\frac{2}{\delta}}{2n}}.
\end{align*}
\end{theo}

\begin{proof}[Theorem \ref{thm_generalization_bound_for_classifiers_conjugated_with_loss_function}]
We apply Theorem \ref{thm_generalization_bound_for_family_of_classifiers_conjugated_with_loss_function} to the family $\LL_{\FF}$ of loss functions where $L(y_1,y_2) = \Phi(y_1y_2)$. We obtain:
\begin{equation*}
\EE[\Phi(\tilde{f})] \leq \frac1n \sum_{i=1}^n \Phi(\tilde{f}(x_i)) + 2 \RRR_n(\Phi \circ \tilde{\FF}) + \sqrt{\frac{\ln\frac{1}{\delta}}{2n}} 
\end{equation*} 
where $\tilde{\FF} = \{\tilde{f}:(x,y)\mapsto yf(x):\, g \in \FF\}$.
By Talagrand's lemma \ref{Talagrand}, we have $\RRR_S(\Phi \circ \tilde{\FF}) = k\,\RRR_S(\tilde{\FF})$ and by Proposition \ref{egalite_Rademacher_pour_f(x)_et_yf(x)} we obtain $\RRR_S(\tilde{\FF}) = \RRR_{S_{\XX}}(\FF)$.
\end{proof}

Let us now consider the $\alpha$-margin loss function $L_{\Phi_{\alpha}}(y_1,y_2) = \Phi_{\alpha}(y_1y_2)$ with $\Phi_{\alpha}(u) =  \min \left(1, \max \left(0,1- \frac{u}{\alpha}\right)\right)$. We obtain the following corollary of Theorem \ref{thm_generalization_bound_for_classifiers_conjugated_with_loss_function}:

\begin{cor}\label{generalisation_inequality_with_margin}
Let $(\XX\times\{-1,1\}, \BB)$ be a measurable space and let $\mu$ be a probability measure on $\XX\times\{-1,1\}$. Let $S_n = \{(x_1,y_1), \ldots, (x_n,y_n)\}$ be an independent and identically distributed sample with respect to $\mu$ and $S_x$ be the projection of $S_n$ on $\XX$. Let $\FF$ be a family of real-valued functions $f: \XX \rightarrow \RR$ and denote by $\LL_{\FF}$ the family of loss functions associated with $\FF$ for the loss function $L(y_1,y_2) = \Phi_{\alpha}(y_1y_2)$ with $\Phi_{\alpha}(u) =  \min \left(1, \max \left(0,1- \frac{u}{\alpha}\right)\right)$. Then, for any $\delta >0$, with probability at least $1- \delta$, each of the two following inequalities holds for all $f \in \FF$: 
\begin{align*}
R(f) &\leq \hat{R}_{S,\Phi_{\alpha}}(f) + \frac{2}{\alpha} \RRR_n(\FF) + \sqrt{\frac{\ln\frac{1}{\delta}}{2n}}, \\
R(f) &\leq \hat{R}_{S,\Phi_{\alpha}}(f) +  \frac{2}{\alpha}\hat{\RRR}_{S_x}(\FF) + 3\sqrt{\frac{\ln\frac{2}{\delta}}{2n}},
\end{align*}
where $\hat{R}_{S,\Phi_{\alpha}}(f) = \frac1n \sum_{i=1}^n \Phi_{\alpha}(y_if(x_i))$.
\end{cor}

While the Rademacher average is a good theoretical quantifier for the representational capacity, it suffers from practical limitations: for most models $\GG$, computing the empirical Rademacher average is an NP-hard problem \cite{MIT_Foundation_of_ML}. For some particular family of classification functions, the Rademacher empirical average can be upper bounded. We consider two family of classification functions: the set of linear functions with bounded weight vectors: $\FF_{w} = \{x \mapsto w.x : \, ||w|| \leq \Lambda\}$ and the set of 1-Lipschitz real-functions $\FF_{Lip} = \{f:\XX\to\RR:\, ||f||_{Lip}\leq 1\}$. In both cases, the metric space ($\XX$,d) is totally bounded. Recall that a metric space $\XX$ is totally bounded if and only if, for every $\varepsilon >0$, there exists a finite collection of open balls of radius $\varepsilon$ whose centres lie in $\XX$ and whose union contains $\XX$. That is, For $\FF_{w}$, we have the following theorem:

\begin{theo}\label{borne_rademacher_separateur_lineaire}
Let $(\XX\times\{-1,1\}, \BB)$ be a measurable space and let $\mu$ be a probability measure on $\XX\times\{-1,1\}$. For $r>0$, et $S_n = \{(x_1,y_1), \ldots, (x_n,y_n):\, ||x_i||\leq r\}$ be an independent and identically distributed sample with respect to $\mu$ and $S_x$ be the projection of $S_n$ on $\XX$. Let $\FF_{w} = \{\XX \ni x \mapsto w.x : \, ||w|| \leq \Lambda\}$, with $\Lambda>0$. The, the empirical Rademacher complexity of $\FF_w$ can be bounded as follows:
\begin{equation*}
\hat{\RRR}_S(\FF_w) \leq \sqrt{\frac{r^2\Lambda^2}{n}}.
\end{equation*}
\end{theo}

\begin{proof}
The proof can be found on page 97 of \textit{Foundations of Machine Learning, $2^{nd}$ edition} \cite{MIT_Foundation_of_ML}.
\end{proof}

Combining Corollary \ref{generalisation_inequality_with_margin} and Theorem \ref{borne_rademacher_separateur_lineaire}, we obtain the following general margin bound for $\FF_w$. Recall that $\Phi_{\alpha}(u) =  \min \left(1, \max\ \left(0,1- \frac{u}{\alpha}\right)\right)$.

\begin{cor}\label{generalisation_separateur_lin}
Let $(\XX\times\{-1,1\}, \BB)$ be a measurable space and let $\mu$ be a probability measure on $\XX\times\{-1,1\}$. For $r>0$, let $S_n = \{(x_1,y_1), \ldots, (x_n,y_n):\, ||x_i||\leq r\}$ be an independent and identically distributed sample with respect to $\mu$ and $S_x$ be the projection of $S_n$ on $\XX$. Let $\FF_{w} = \{\XX \ni x \mapsto w.x : \, ||w|| \leq \Lambda\}$, with $\Lambda>0$ and denote by $\LL_{\FF_w}$ the family of loss functions associated with $\FF_w$ for the loss function $L(y_1,y_2) = \Phi_{\alpha}(y_1y_2)$. Then, for any $\delta >0$, with probability at least $1- \delta$, the following inequality holds for all $f \in \FF_w$:
\begin{equation*}
R(f) \leq \hat{R}_{S,\Phi_{\alpha}}(f) + 2 \sqrt{\frac{r^2\Lambda^2/\alpha^2}{n}} + \sqrt{\frac{\ln\frac{1}{\delta}}{2n}}.
\end{equation*}
\end{cor}

\noindent We can construct a similar bound for the model of 1-Lipschitz functions $\FF_{Lip}$. To do so, we use the notion of covering numbers:
 
\begin{defn}
Let $(\XX,d)$ be a totally bounded space. The covering number $N(\XX,\epsilon,d)$ of $\XX$ is the smallest number of balls or radius $\epsilon$ with centers in $\XX$ which can cover $\XX$ completely.
\end{defn}

\noindent Using the covering number of a totally bounded space $(\XX,d)$, we obtain an upper bound on the Rademacher average of the unit ball of $Lip(\XX)$. It is important to note that for the theorem below , the norm associated with the space $Lip(\XX)$ is 
\begin{equation*}
||f||_{Lip} = \max \left\{L(f), \frac{||f||_{\infty}}{\diam(\XX)}\right\}
\end{equation*}
This norm was suggested by Bousquet and von Luxbourg in \cite{Luxburg_et_Bousquet}.  

\begin{theo}\label{thm_18_Lux_Bousquet}
Let $(\XX,d)$ be a totally bounded space and $\Delta$ denote the $d$-diameter of $\XX$. If $\RRR_n(B)$ denotes the unit ball in $Lip(\XX)$, then, for any $\epsilon >0$,
\begin{equation*}
\RRR_n(B) \leq 2\epsilon + \frac{4\sqrt{2}}{\sqrt{n}} \int_{\epsilon/4}^{2\Delta} \sqrt{N(\XX,\frac{u}{4},d) \ln{\left(2\left\lceil \frac{2\Delta}{u} \right\rceil + 1\right)}} \diff u.
\end{equation*}
\end{theo}

\begin{proof}
The proof can be found on page 684 of \textit{Distance-Based Classification with Lipschitz Functions} \cite{Luxburg_et_Bousquet}.
\end{proof}

Using Theorem \ref{thm_18_Lux_Bousquet}, we can write a general margin bound for $\FF_{Lip}$:

\begin{cor}\label{generalisation_separateur_1Lip}
Let $(\XX,d)$ be a totally bounded space and $\Delta$ denote the $d$-diameter of $\XX$. Let $(\XX\times\{-1,1\}, \BB)$ be a measurable space and let $\mu$ be a probability measure on $\XX\times\{-1,1\}$. Let $S_n = \{(x_1,y_1), \ldots, (x_n,y_n)\}$ be an independent and identically distributed sample with respect to $\mu$ and $S_x$ be the projection of $S_n$ on $\XX$. Let $\FF_{Lip} = \{f:\XX\to\RR:\, ||f||_{Lip}\leq 1\}$ and denote by $\LL_{\FF_{Lip}}$ the family of loss functions associated with $\FF_{Lip}$ for the loss function $L(y_1,y_2) = \Phi_{\alpha}(y_1y_2)$. Then, for any $\delta >0$, with probability at least $1- \delta$, the following inequality holds for all $f \in \FF_{Lip}$:
\begin{equation*}
R(f) \leq \hat{R}_{S,\Phi_{\alpha}}(f) + \frac{2}{\alpha} \inf_{\epsilon>0} \left(2\epsilon + \frac{4\sqrt{2}}{\sqrt{n}} \int_{\epsilon/4}^{2\Delta} \sqrt{N(\XX,\frac{u}{4},d) \ln{\left(2\left\lceil \frac{2\Delta}{u} \right\rceil + 1\right)}} \diff u \right) + \sqrt{\frac{\ln\frac{1}{\delta}}{2n}}.
\end{equation*}
\end{cor}

The two generalisation bounds obtained in Corollary \ref{generalisation_separateur_lin} and Corollary \ref{generalisation_separateur_1Lip} do not depend directly on the dimension of the space $\XX$. It depends only on the margin $\alpha$. It suggests that a small generalisation error can be achieved when $\alpha$ is large while the empirical margin loss $\hat{R}_{S,\Phi_{\alpha}}(f)$ remains relatively small. The latter occurs when few points are either misclassified or well classified but with confidence smaller than $\alpha$. A favourable margin situation depends on the probability measure $\mu$: eventhough the generalisation bound is independent of $\mu$, the existence of a large margin is $\mu$-dependent.

\cleardoublepage

\chapter{Classifiers and the Kantorovich-Rubinstein Distance}\label{KR_distance_upper_bound_of_classifiers}

Chapter \ref{KR_distance_upper_bound_of_classifiers} is the central theoretical chapter of this thesis. All the results presented are new. We study thoroughly the association between misclassification errors and the Kantorovich-Rubinstein distance. \\
The first section starts with measurable classifiers and how their error is linked to the total variation distance. We also show that the definition of classifiers and error of classification used in classification theory (see \cite{Prob_theory_of_Pattern} and \cite{pattern_recognition_review}) are particular cases of our definition of measurable classifiers. In the second part of section 1, we express the risk function of a classifier with two new measures that we construct.  \\
The second section studies Lipschitz classifiers on a metric space. On a space equipped with a distance, we can then define $(\epsilon,\delta)$-Lipschitz classifiers and then prove the two most important results of the chapter given in Theorem \ref{given_classifier} and Theorem \ref{theo_upper_bound_for_delta}. In these two theorems, we bound from above the classification error of $(\epsilon,\delta)$-Lipschitz classifiers with the Kantorovich-Rubinstein distance. \\

\section{Definition of measurable classifiers}\label{section_Def_of_measurable_classifiers}

\begin{defn}\label{def_binary_classifier_c_f}
Let $(\XX, \BB)$ be a measurable space. To any $B\in\BB$, the measurable function $c_B: \XX \rightarrow \{-1,1\}$ given by $c_B = \indicator_{B} -  \indicator_{B^{\comp}}$ is called a binary classifier.
\end{defn}

\begin{ntn}
For any measurable function $f: (\XX, \BB) \rightarrow (\RR, \BB(\RR))$ and constant $\gamma\in\RR$, we denote by $c_{f,\gamma}$ the family of binary classifiers defined by
\begin{equation*}
c_{f,\gamma} = \indicator_{\{x \in \XX; f(x) > \gamma\}} - \indicator_{\{x \in \XX; f(x) \leq \gamma\}}.
\end{equation*}
To make the notations lighter, when $\gamma=0$, we write $c_f$ instead of $c_{f,0}$. 
\end{ntn}

\begin{rmk} We make two remarks:
\begin{enumerate}
\item[(i)]  On $f^{-1}(\RR_*)$, we have $c_{f,0} = \sgn(f)$. 
\item[(ii)] For any $\gamma \in \RR$, we have $\dps c_{f,\gamma} = c_{f-\gamma,0}$.
\end{enumerate}
\end{rmk}

\begin{defn}\label{error_quantity}
Let $(\XX, \BB)$ be a measurable space, $\mu_1$ and $\mu_2$ be two finite measures on $\XX$ and $B\in\BB$. To the classifier $c_B$ on $\XX$, we associate the quantity $\varepsilon$ of $c$, defined by
\begin{equation*}
\varepsilon(c_B ; \mu_1,\mu_2) = \mu_1(B^{\comp}) + \mu_2(B).
\end{equation*}
\end{defn}

\begin{rmk}\label{properties_of_classifier_error}
For $\mu_1,\mu_2$ and $c$ as in Definition \ref{error_quantity} we have:
\begin{enumerate}
\item[(i)] $\varepsilon(c;\mu_1,\mu_2) = \varepsilon(-c;\mu_2,\mu_1)$,
\item[(ii)] $\varepsilon(c;\mu_1,\mu_2) + \varepsilon(c;\mu_2,\mu_1) = \mu_1(\XX) + \mu_2(\XX)$.
\item[(iii)] For any measurable function $f: (\XX, \BB) \rightarrow (\RR, \BB(\RR))$ and constant $\gamma\in\RR$, we have $c_{f,\gamma} = c_B$ when $B=\{x \in \XX; f(x) > \gamma\}$. Therefore, 
\begin{equation*}
\varepsilon(c_{f,\gamma}; \mu_1,\mu_2) = \mu_1(\{x \in \XX; f(x) \leq \gamma\}) + \mu_2(\{x \in \XX; f(x) > \gamma\}).
\end{equation*}
\end{enumerate}
\end{rmk}

\noindent Recall (see Definition \ref{total_variation_norm_def}) that the total variation norm of a signed measure $\mu$ on $(\XX,\BB)$ is given by $||\mu|| = |\mu|(\XX)$. The two properties below follow from Proposition \ref{prop_of_tv_norm}. 
 
\begin{prop}\label{classifier_with_err_as_fct_of_tv}
Let $(\XX, \BB)$ be a Borel measurable space and $\mu_1$, $\mu_2$ be two finite measures on $\XX$. Then, there exists a binary classifier $c: \XX \rightarrow \{-1,1\}$ such that 
\begin{equation*}
\varepsilon(c ; \mu_1,\mu_2) = \frac12 \left( \mu_1(\XX) + \mu_2(\XX) - ||\mu_1-\mu_2|| \right).
\end{equation*}
\end{prop}

\begin{proof}[Proposition \ref{classifier_with_err_as_fct_of_tv}]
Let $\lambda$ be a common dominating measure for the $\mu_i$'s (ie., $\mu_i \leq \lambda$) on $(\XX, \BB)$ and let $m_i$ be the respective Radon-Nikodym derivatives of $\mu_i$ with respect to $\lambda$, for $i = 1,2$. Let $M$ be the measurable set defined by $M = \{x \in \XX; \, m_1(x) \geq m_2(x) \}$ and let $c = \indicator_M - \indicator_{M^{\comp}}$. Therefore, $\varepsilon(c; \mu_1,\mu_2) = \mu_1(M^{\comp}) + \mu_2(M)$. By Proposition \ref{prop_of_tv_norm}, we have
\begin{align*}
||\mu_1 - \mu_2|| = \int c \diff (\mu_1-\mu_2) &= \int_M \diff(\mu_1-\mu_2) - \int_{M^{\comp}} \diff(\mu_1-\mu_2) \\
&= \mu_1(M) - \mu_2(M) -\mu_1(M^{\comp}) + \mu_2(M^{\comp}) \\
&= \mu_1(\XX) + \mu_2(\XX) - 2 \left( \mu_1(M^{\comp}) + \mu_2(M) \right).
\end{align*} 
Hence, $\dps \varepsilon(c, \mu_1, \mu_2) = \frac12 \left( \mu_1(\XX) + \mu_2(\XX) - ||\mu_1 - \mu_2|| \right)$.
\end{proof}

\begin{prop}\label{tv_lower_bound_of_error}
Let $(\XX, \BB)$ be a Borel measurable space and $\mu_1$, $\mu_2$ be two finite measures on $\XX$. Then for any measurable function $f:\XX \rightarrow \RR$ we have 
\begin{equation*}
\varepsilon(c_f; \mu_1,\mu_2) \geq \frac12 \left( \mu_1(\XX) + \mu_2(\XX) - ||\mu_1-\mu_2|| \right).
\end{equation*}
\end{prop}

\begin{proof}[Proposition \ref{tv_lower_bound_of_error}]
We keep the notations given in the proof of Proposition \ref{classifier_with_err_as_fct_of_tv}. 

For a measurable function $f:\XX \rightarrow \RR$, let $B=\{ x\in\XX; \, f(x) > 0\}$. By Definition \ref{error_quantity}, we have 
\begin{equation*}
\begin{split}
\varepsilon(c_f; \mu_1,\mu_2) &= \mu_1(B^{\comp}) + \mu_2(B) \\
&= \mu_1(B^{\comp} \cap M) + \mu_1(B^{\comp} \cap M^{\comp}) + \mu_2(B \cap M) + \mu_2(B \cap M^{\comp}).
\end{split}
\end{equation*}
Since $\, m_2 > m_1$ on $M^{\comp}$, we obtain
\begin{equation*}
\mu_2(B \cap M^{\comp}) = \int_{B \cap M^{\comp}} m_2 \diff \lambda \geq \int_{B \cap M^{\comp}} m_1 \diff \lambda = \mu_1(B \cap M^{\comp}).
\end{equation*}
Similarly, $\mu_1(B^{\comp} \cap M) \geq \mu_2(B^{\comp} \cap M)$.\\
Therefore, we have the following inequality:
\begin{align*}
\varepsilon(c_f; \mu_1,\mu_2) &= \mu_1(B^{\comp} \cap M^{\comp}) + \mu_2(B \cap M^{\comp}) + \mu_1(B^{\comp} \cap M) + \mu_2(B \cap M)\\
& \geq \mu_1(M^{\comp}) + \mu_2(M).
\end{align*}
Since $\mu_1(M^{\comp}) + \mu_2(M) = \varepsilon(c; \mu_1,\mu_2)$, the proof is complete.
\end{proof}

\noindent Recall that for any pair of probability measures $\mu_1, \, \mu_2$ on $(\XX, \BB)$, the total variation distance is $||\mu_1 - \mu_2||_{TV} = \frac12 ||\mu_1-\mu_2||$. Proposition \ref{classifier_with_err_as_fct_of_tv} and Proposition \ref{tv_lower_bound_of_error} then become:

\begin{prop}
Let $(\XX, \BB)$ be a measurable space and $\mu_1$, $\mu_2$ be two probability measures on $\XX$. Then, there exists a binary classifier $c: \XX \rightarrow \{-1,1\}$ such that  
\begin{equation*}
\varepsilon(c; \mu_1,\mu_2) = 1 - ||\mu_1-\mu_2||_{TV}.
\end{equation*}
\end{prop}

\begin{prop}
Let $(\XX, \BB)$ be a measurable space and $\mu_1$, $\mu_2$ be two probability measures on $\XX$. Then, for any measurable function $f: \XX \rightarrow \RR$, we have
\begin{equation*}
\varepsilon(c_f; \mu_1,\mu_2) \geq 1 - ||\mu_1-\mu_2||_{TV}.
\end{equation*}
\end{prop}

\noindent In the example below, we associate the Bayes decision function introduced in Section 2.1 of \textit{A Probabilistic Theory of Pattern Recognition} \cite{Prob_theory_of_Pattern} to the optimal measurable classifier of Proposition \ref{classifier_with_err_as_fct_of_tv}. We keep the notation used in \cite{Prob_theory_of_Pattern} even if it does not correspond to the one used in this thesis.

\begin{egg}\label{Bayes_error_has_partition_error}
Let $\nu \in P(\Omega)$, where $P(\Omega)$ denotes the space of Borel probability measures on $\Omega$. Let $(h,\varphi)$ be a pair of Borel maps from $\Omega$ taking their respective values in the Borel measurable space $\XX$ and $\{0,1\}$. Let $\mu = h(\nu)$ denote the push-forward measure on $\XX$ and $\eta: \XX \rightarrow [0,1]$ be the regression of $\varphi$ on $h$, that is $\eta(x) = \nu(\varphi = 1/ h = x),$ for $x \in \XX$.

\noindent Recalling the Disintegration Theorem \ref{Disintegration_thm}, we know that there exists a Borel map $x \in \XX \mapsto \nu_x \in P(\XX)$ such that $\nu_x(h^{-1}(x)) = 1$, $\mu$-a.e. and $\dps \nu = \int \nu_x \diff \mu(x)$. Then, $\eta(x) = \dps\int_{\Omega} \varphi \diff \nu_x$, for $x \in \XX$. \\

\noindent In section 2.1 of \cite{Prob_theory_of_Pattern}, the authors define a decision function $g$ as a measurable function $g: \XX \rightarrow \{0,1\}$ given by 
\begin{equation*}
g(x) = \begin{cases}
1 &\quad\textnormal{ if $x \in B$}\\
0 &\quad\textnormal{ if $x \in B^{\comp}$}
\end{cases}
\qquad \mbox{for some } B \in \BB,
\end{equation*}
where $\BB$ is the Borel $\sigma$-algebra associated to $\XX$.\\
We associate to the decision function $g$, its probability of error $L(g)$ given by $\nu\big(\{ \omega\in \Omega ; \, g \circ h(\omega) \not = \varphi(\omega) \}\big)$, that we denote by $\nu(g \circ h \not = \varphi)$.

With all the relevant notations from section 2.1 of \cite{Prob_theory_of_Pattern}, we can now associate the Bayes decision function to the optimal measurable classifier. Let $\mu_1,\mu_2$ be two measures on $\XX$ given by $\mu_1 = \eta \mu$ and $\mu_2 = (1 - \eta) \mu$ and let $B = \{x \in \XX; g(x) = 1\} = g^{-1}(\{1\})$. Then, \\
\noindent \underline{Claim}: $L(g) = \varepsilon(c_g; \mu_1,\mu_2)$, where $c_g$ is the binary classifier associated to $g$.\\
 
\noindent \underline{Proof of Claim}: By definition, we have 
\begin{equation*}
L(g) = \nu(\varphi \not= g \circ h) = \nu(\{\omega; \, \varphi(\omega) = 1, g \circ h (\omega) = 0 \}) + \nu(\{\omega; \, \varphi(\omega) = 0, g \circ h (\omega) = 1 \}).
\end{equation*}
The first term of the sum above can be written as follow:
\begin{align*}
\nu(\{\omega; \, \varphi(\omega) = 1, g \circ h (\omega) = 0 \}) &= \int_{\XX} \nu_x (\{\omega; \, \varphi(\omega) = 1, g(x) = 0 \}) \diff \mu(x)\\
&= \int_{\XX} \indicator_{g^{-1}(0)}(x) \nu_{\XX}(\{\omega;\, \varphi(\omega) =1\}) \diff \mu(x) \\
&=  \int_{\XX} \indicator_{g^{-1}(0)}(x) \eta(x) \diff \mu(x) \\
&= \int_{\XX} \indicator_{g^{-1}(0)} \diff \mu_1(x) \\
&= \mu_1(g^{-1}(0)).
\end{align*}
Similarly, we obtain
\begin{equation*}
\nu(\{\omega; \, \varphi(\omega) = 0, g \circ h (\omega) = 1 \}) = \int_{\XX} \indicator_{g^{-1}(1)} (1- \eta(x)) \diff \mu(x) = \int_{\XX} \indicator_{g^{-1}(1)} \diff \mu_2(x) = \mu_2(g^{-1}(1)).
\end{equation*}
As $\varepsilon(c_g; \mu_1,\mu_2) = \mu_1(g^{-1}(0)) + \mu_2(g^{-1}(1))$, the claim is proved. \\

\noindent Following section 2.1, in \cite{Prob_theory_of_Pattern}, the Bayes decision function $g^*: \XX \rightarrow \{0,1\}$ is defined by
\begin{equation*}
g^*(x) = \begin{cases}
1 &\quad\textnormal{ if $\eta(x) > \frac12$}\\
0 &\quad\textnormal{ if $\eta(x) \leq \frac12$}
\end{cases}
\end{equation*}
Hence, $g^* = \indicator_E$ where the set $E$ is defined by $E = \{x \in \XX; \, \eta(x) > \frac12\}$. \\
By construction, $\eta$ is the Radon-Nikodim derivative of $\mu_1$ and $1 - \eta$ is the Radon-Nikodim derivative of $\mu_2$. Moreover, the inequality $\eta(x) > \frac12$ can easily be rewritten as $1-\eta(x) < \eta(x)$ and hence $E$ can be written as $E = \{x \in \XX; \, 1-\eta(x) < \eta(x) \}$. As in Proposition \ref{classifier_with_err_as_fct_of_tv}, we have
\begin{equation*}
||\mu_1-\mu_2|| = \mu_1(\XX) + \mu_2(\XX) - \varepsilon(c_{g^*}; \mu_1,\mu_2).
\end{equation*}
As $\mu_1(\XX) + \mu_2(\XX) = \int \eta \diff \mu + \int (1- \eta) \diff \mu = 1$ we get $\varepsilon(c_{g^*}; \mu_1;\mu_2) = \frac12 (1 - ||\mu_1-\mu_2||)$ where$c_{g^*}$ is the Bayes binary classifier defined by
\begin{equation*}
c_{g^*}(x) = \begin{cases}
1 &\quad\textnormal{ if $\eta(x) > \frac12$}\\
-1 &\quad\textnormal{ if $\eta(x) \leq \frac12$}
\end{cases}
\end{equation*}
\end{egg}

\begin{rmk} Let us make two remarks:
\begin{enumerate}
\item[(i)] Let $\mu_1$, $\mu_2$ be two probability measures on $(\XX, \BB)$. As defined, for example, in Pollard \cite{Pollard_total_variation}, the expression $1 - ||\mu_1-\mu_2||_{TV}$ is called the \textit{affinity} between $\mu_1$ and $\mu_2$ and is denoted by $\alpha_1(\mu_1,\mu_2)$.
Denote by $\mu_1 \land \mu_2$ the largest measure on $(\XX, \BB)$ for which, for all $A \in \BB$, the following inequality holds:
\begin{equation*}
\mu_1 \land \mu_2 (A) \leq \min(\mu_1(A), \mu_2(A)).
\end{equation*}
For some dominating measure $\lambda$ of $\mu_1$ and $\mu_2$, the measure $\mu_1 \land \mu_2$ is such that 
\begin{equation*}
\frac{\diff(\mu_1 \land \mu_2)}{\diff \lambda} = \min \left(\frac{\diff \mu_1}{\diff \lambda}, \frac{\diff \mu_2}{\diff \lambda}\right), \quad \lambda \mbox{-a.e.}
\end{equation*}
and $\mu_1 \land \mu_2 (\XX) = ||\mu_1 \land \mu_2|| = \alpha_1(\mu_1,\mu_2)$.

\item[(ii)] Let $\nu$ be a Borel probability measure on $\Omega$. Let $(h,\varphi)$ be a pair of Borel maps from $\Omega$ taking their respective values in $\XX$ and $\{0,1\}$, and assume that $\nu\big(\{ \omega\in \Omega ; \, \varphi(\omega) = 0 \}\big) = \nu\big(\{ \omega\in \Omega ; \, \varphi(\omega) = 1 \}\big) = 1/2$.\\
As in Exemple \ref{Bayes_error_has_partition_error}, let $\mu = h(\nu)$ and $\eta: \XX \rightarrow [0,1]$ be the regression of $\varphi$ on $h$, let $\mu_1$ and $\mu_2$ be two measures on $\XX$ given by $\mu_1 = \eta \mu$ and $\mu_2 = (1 - \eta) \mu$. Then, $\mu_1(\XX) = \mu_2(\XX) = 1/2$.\\
Let us define two probability measures
\begin{equation*}
\tilde{\mu}_1 = \frac{1}{\mu_1(\XX)}\mu_1 \, \mbox{ and } \, \tilde{\mu}_2 =  \frac{1}{\mu_2(\XX)}\mu_2,
\end{equation*}
and let $g^*: \XX \rightarrow \{0,1\}$ be the Bayes decision function given by $g^*= \indicator_E$, where $E=\{x \in \XX; \eta(x) > 1/2\}$. Then, 
\begin{equation*}
\varepsilon(1-2g^*; \tilde{\mu}_1, \tilde{\mu}_2) = 1 - ||\tilde{\mu}_1- \tilde{\mu}_2||_{TV}.
\end{equation*}
\end{enumerate}
\end{rmk}

\vspace{5mm}

\noindent Consider a classifier $c_f$ linked to a classification function $f$. In the case of the particular measures $\mu_+$ and $\mu_-$ defined in Equation \ref{def_mu_+_et_mu_-_theorique} and specific loss functions, there exist an interesting relation between its risk function $R(c_f)$ and its quantity $\varepsilon$ defined in Definition \ref{error_quantity}. First, let us recall the definition of $\mu_+$ and $\mu_-$:

\begin{defn}\label{rappel_def_mu_+_et_mu_-_theorique}
For $(\XX,\BB)$ and $(\{-1,1\}, \PPP\{-1,1\})$ two Borel spaces, we consider $(\XX\times\{-1,1\}, \BB\otimes\PPP\{-1,1\})$ the product space with the natural product $\sigma$-algebra, a probability measure $\mu$ on $\XX\times\{-1,1\}$ and the canonical functions $\pi:\XX\times\{-1,1\}\to\{-1,1\}$ and $\pi_x:\XX\times\{-1,1\}\to\XX$. We define two new measures $\mu_+$ and $\mu_-$ on $\XX$ such that:
\begin{equation*}
\mu_+ = \frac{1}{\mu_1(\XX)}\,\pi_x(\mu_1) \, \mbox{ and } \, \mu_- = \frac{1}{\mu_{-1}(\XX)}\,\pi_x(\mu_{-1}),
\end{equation*}
where $\mu_i$ are measures in $\XX\times\{-1,1\}$ such that $\mu_i(A) = \mu(A \cap \pi^{-1}(\{i\}))$.
\end{defn}

\noindent We can now write the following proposition:

\begin{prop}\label{risk_to_varepsilon}
Let $(\XX\times\{-1,1\}, \BB)$ be a measurable space and let $\mu$ be a probability measure on $\XX\times\{-1,1\}$. Consider $\mu_1,\mu_{-1}$ and $\mu_+, \mu_-$ be as defined in Definition \ref{rappel_def_mu_+_et_mu_-_theorique}. Let $f:\XX\to\RR$ be a real-valued function and $c_f$ be the classifier linked to the classification function $f$. Then, for all $(x,y) \in \XX\times\{-1,1\}$,
\begin{equation*}
\int_{\XX\times\{-1,1\}} \Bigl(1-H(yf(x)) \Bigr) \diff \mu(x,y) = \varepsilon(c_f;\mu_+,\mu_-)
\end{equation*}
where $H$ is the Heaviside function and $L_{\alpha}(u) = \indicator_{u\leq\alpha}$.
\end{prop}

\begin{proof}[Proposition \ref{risk_to_varepsilon}]
Using equality (\ref{eqn_risk_avec_mu+_mu-}) with $\Phi(yf(x)) = 1 - H(yf(x))$ we can write:
\begin{align*}
R(g_f) &= \int_{\XX} \Bigl( 1 - H(f(x)) \Bigl) \diff \mu_+ + \int_{\XX} \Bigl( 1 - H(-f(x)) \Bigl) \diff \mu_- \\
&=\mu_+(\{x\in\XX: f(x)\leq 0 \}) + \mu_-(\{x\in\XX: f(x) > 0 \}) \\
&= \varepsilon(c_f;\mu_+,\mu_-).
\end{align*}
\end{proof}

Proposition \ref{risk_to_varepsilon} shows that the quantity $\varepsilon$ defined in Definition \ref{error_quantity} is linked to the risk functional of $c_f$. Using this result as a starting point, we define the notion of $\epsilon$-error for a binary classifier $c_f$:

\begin{defn}\label{def_err_epsilon}
Let $(\XX, \BB)$ be a measurable space, $\mu_1$ and $\mu_2$ be two finite measures on $\XX$, and $f:\XX \rightarrow \RR$ be a measurable function. For any $\epsilon \geq 0$, the $\epsilon$-error of $f$ with respect to $(\mu_1,\mu_2)$ is given by
\begin{equation*}
\err(f, \epsilon ; \mu_1,\mu_2) = \min \left(\varepsilon(c_f, \epsilon ; \mu_1,\mu_2),\varepsilon(c_f, \epsilon ; \mu_2,\mu_1) \right),
\end{equation*}
where $\varepsilon(c_f, \epsilon;\mu_i,\mu_j) =  \mu_i(\{x \in \XX ; \, f(x) \leq \epsilon/2 \}) + \mu_j(\{x \in \XX ; \, f(x) > -\epsilon/2\})$
\end{defn}

\begin{rmk} There are two remarks worth making:
\begin{enumerate}
\item[(i)] $\err(c_f, \epsilon; \mu_1, \mu_2) \geq \err(c_f, 0;\mu_1,\mu_2)$ 
\item[(ii)] $\err(c_{-f}, \epsilon; \mu_1, \mu_2) = \mu_1(\{x \in \XX ; \, f(x) \leq -\epsilon/2\}) + \mu_2(\{x \in \XX ; \, f(x) > \epsilon/2\})$, hence we have \\
\hspace*{-1cm} $\begin{aligned} \err(c_{-f}, \epsilon; \mu_1, \mu_2) +  \mu_1(\{x \in \XX ; \, f(x) = \epsilon/2\}) = \err(c_f, \epsilon; \mu_1, \mu_2) + \mu_2(\{x \in \XX ; \, f(x) = -\epsilon/2\}). \end{aligned}$
\end{enumerate}
\end{rmk}

\begin{defn}[$(\epsilon,\delta)$-classification function]
Let $(\XX,\BB)$ be a measurable space and $\mu_1$ and $\mu_2$ be two finite measures on $\XX$. For $\epsilon$ and $\delta \geq 0$, a measurable function $f: \XX \rightarrow \RR$ is a $(\epsilon,\delta)$-classification function if 
\begin{equation*}
\err(f, \epsilon ; \mu_1,\mu_2) \leq \delta.
\end{equation*}
\end{defn}

Let $\XX$ be a Polish space and $\mu_1$ and $\mu_2$ be two Borel probability measures on $\XX$. The next proposition links $(\epsilon, \delta)$-classification function to the two distribution functions of the probability measures $f(\mu_1)$ and $f(\mu_2)$. This result will be useful in the next section \ref{Lip_classifiers_and_KR_dist}.
 
\begin{prop}\label{borne_inf_pour_Vallender}
Let $(\XX, \BB)$ be a measurable space and let $\mu_1$ and $\mu_2$ be two probability measures on $\XX$. Let $f: \XX \to\RR$ be a $(\epsilon,\delta)$-classification function and $F_1$, $F_2$ be the respective distribution functions of the probability measures $f(\mu_1)$ and $f(\mu_2)$. Then, 
\begin{equation*}
\int|F_1(t) - F_2(t)| \diff t \geq (1-\delta)\epsilon.
\end{equation*}
\end{prop}

\begin{proof}[Proposition \ref{borne_inf_pour_Vallender}]
We first suppose that $\varepsilon(c_f,\epsilon; \mu_1,\mu_2) \leq \varepsilon(c_f, \epsilon; \mu_2,\mu_1) \leq \delta$.\\
Using $F_1$ and $F_2$, we have 
\begin{equation*}
\varepsilon(c_f, \epsilon; \mu_1,\mu_2) = 1 - \big( F_2(-\epsilon/2) - F_1(\epsilon/2)\big) \mbox{ and thus } F_2(-\epsilon/2) - F_1(\epsilon/2) \geq 1- \delta.
\end{equation*}
Now, we show that, for all $-\epsilon/2 \leq s \leq \epsilon/2$, $F_2(s) \geq (1-\delta) + F_1(s)$. Since $F_i$ is monotone and non-decreasing we have: $F_2(s) \geq F_2(-\epsilon/2)$ and $F_1(s) \leq F_1(-\epsilon/2)$. Therefore, for all $-\epsilon/2 \leq s \leq \epsilon/2$, we obtain: $F_2(s)\geq F_2(-\epsilon/2) \geq (1-\delta) + F_1(\epsilon/2) \geq (1-\delta) + F_1(s)$.\\

\noindent Hence, $\dps \int^{\epsilon/2}_{-\epsilon/2} F_2(t) - F_1(t) \diff t \geq \epsilon(1-\delta) $.\\

\noindent Likewise, for the case $\varepsilon(c_f,\epsilon; \mu_2,\mu_1) \leq \varepsilon(c_f, \epsilon; \mu_1,\mu_2) \leq \delta$, we obtain
\begin{equation*}
\int^{\epsilon/2}_{-\epsilon/2} F_1(t) - F_2(t) \diff t \geq \epsilon(1-\delta)
\end{equation*}
Therefore, $\dps  \int |F_2(t) - F_1(t)| \diff t \geq \int^{\epsilon/2}_{-\epsilon/2} |F_2(t) - F_1(t)| \diff t \geq \epsilon(1-\delta)$.
\end{proof}

\begin{lem}\label{prop_min_of_rho1_rho2}
Let $(\XX, \BB)$ be a measurable space and let $\mu_1$ and $\mu_2$ be two probability measures on $\XX$. Let $f: \XX \rightarrow \RR$ be a measurable function and $F_1$, $F_2$ be the respective distribution functions of the probability measures $f(\mu_1)$ and $f(\mu_2)$. Then, for any $t \in \RR$,
\begin{equation*}
\min (\varepsilon(c_{f,t}; \mu_1,\mu_2), \varepsilon(c_{f,t}; \mu_2,\mu_1)) = 1 - |F_1(t) - F_2(t)|.
\end{equation*}
\end{lem}

\begin{proof}[Lemma \ref{prop_min_of_rho1_rho2}]
By the definition of $F_1$ and $F_2$, we have:
\begin{align*}
\mu_1(\{ x \in \XX \, | \,f(x) > t\}) = \mu_1(\{x \in \XX \, | \, f(x) \leq t\}^{\comp}) = \mu_1(\XX) - F_1(t) =  1 - F_1(t) \\
\mu_2(\{ x \in \XX \, | \,f(x)> t\}) = \mu_2(\{ x \in \XX \, | \, f(x) \leq t\}^{\comp}) = \mu_2(\XX) - F_2(t) = 1 - F_2(t).
\end{align*}
\noindent Using $F_1$ and $F_2$, we can write $\varepsilon(c_{f,t}; \mu_1,\mu_2)$ and $\varepsilon(c_{f,t}; \mu_2,\mu_1)$ as follow:             
\begin{align*}
\varepsilon(c_{f,t}; \mu_1,\mu_2) &= \mu_1(\{ x \in \XX \,| \, f(x) \leq t \}) + \mu_2(\{  x \in \XX \,| \,f(x) > t\}) \\
  &= \mu_2(\XX) - F_2(t) + F_1(t) \\
  &= 1 + \big(F_1(t) - F_2(t) \big) \\
\varepsilon(c_{f,t}; \mu_2,\mu_1) &=   \mu_1(\{ x \in \XX \,| \, f(x) > t \}) + \mu_2(\{  x \in \XX \,| \,f(x) \leq t\}) \\
 &= \mu_1(\XX) - F_1(t) +F_2(t) \\
 &= 1 - \big(F_1(t) - F_2(t) \big).
\end{align*}
\noindent Hence, $\min (\varepsilon(c_{f,t}; \mu_1,\mu_2), \varepsilon(c_{f,t}; \mu_2,\mu_1)) = 1 - |F_1(t) - F_2(t)|$.
\end{proof}

\section{Lipschitz classifiers and the Kantorovich-Rubinstein distance}\label{Lip_classifiers_and_KR_dist}

Generally, the datasets we want to classify are subsets of Euclidean spaces, and more generally of metric spaces. Therefore, in the rest of this section, $\XX$ will be equipped with a distance $d_{\XX}$. More precisely, as we will use notions and results of Chapter \ref{Chapitre_KR_distance}, $\XX$ will be a Polish space, with its Borel structure and the distance $d$ will be lower semi-continuous. With a metric space, we can focus on Lipschitz classification functions.

\begin{defn}[$(\epsilon,\delta)$-Lipschitz classifier]
Let $\XX$ be a Polish space and $d \in \DD_{\XX}$ be a bounded lower semi-continuous distance on $\XX$. Let $\mu_1$ and $\mu_2$ be two finite measures on $(\XX,\BB)$.
\begin{enumerate}
\item[(i)] A Lipschitz classifier $c_f: \XX \rightarrow \RR$ is a binary classifier such that its classification function $f \in L_1(\XX, \mu_i)$, for $i = 1,2$, and $f$ is 1-Lipschitz with respect to the distance $d$.
\item[(ii)] For $\epsilon$ and $\delta \geq 0$, a classifier $c_f: \XX \rightarrow \RR$ is a $(\epsilon,\delta)$-Lipschitz classifier if it is a Lipschitz classifier such that $\err(c_f, \epsilon; \mu_1,\mu_2)\leq \delta$. 
\end{enumerate}
\end{defn}

\begin{rmk}
Let $\XX$ be a Polish space, and $\BB(\XX)$be its Borel $\sigma$-algebra. Let $d$ be a lower semi-continuous distance on $\XX$. Then a 1-Lipschitz function $f: \XX \rightarrow \RR$ is not necessarily $\BB(\XX)$-measurable.\\
For example, let $\XX$ be the real line $\RR$, with its natural Polish topology and let $d$ be the discrete distance on $\RR$. By Lemma \ref{discrete_distance_lower_semi_continuous}, we know that $d$ is lower semi-continuous. For any subset $E \subset \RR$, the characteristic function $\indicator_E$ is 1-Lipschitz, but is Borel measurable only if $E \in \BB(\RR)$.
\end{rmk}

Let $\XX$ be a Polish space and $\mu_1$ and $\mu_2$ be two Borel probability measures on $\XX$. In the next theorem, we show that an $(\epsilon, \delta)$-Lipschitz classifier determines a lower-bound of the Kantorovich-Rubinstein distance of $\mu_1$ and $\mu_2$:

\begin{theo}\label{given_classifier}
Let $\XX$ be a Polish space and $d \in \DD_{\XX}$ be a bounded lower semi-continuous distance on $\XX$. Let $\mu_1$ and $\mu_2$ be two probability measures on $\XX$.\\
If, for $0 \leq \epsilon \leq 1$, $c_f$ is a $(\epsilon, \delta)$-Lipschitz classifier, then $W_{\XX}(\mu_1,\mu_2) \geq \epsilon (1- \delta).$
\end{theo}

\begin{proof}[Theorem \ref{given_classifier}]
Let $f$ be a $(\epsilon, \delta)$-Lipschitz classification function. Applying Theorem \ref{inequality of KR distances with Lipschitz function} and Theorem \ref{KR_distance_on_R}, we have
\begin{equation*}
W_{\XX}(\mu_1,\mu_2) \geq W_{\RR}(\mu_1,\mu_2) = \int_{-\infty}^{+\infty} |F_{\mu_1}(x) - F_{\mu_2}(x)| \diff x.
\end{equation*}
where $F_i$ denotes the distribution functions of the measure $f(\mu_i)$.\\
Applying Proposition \ref{borne_inf_pour_Vallender} completes the proof.
\end{proof}

\begin{prop}\label{KR_distance_smaller_than_F_1-F_2}
Let $\XX$ be a Polish space and $d \in \DD_{\XX}$ be a bounded lower semi-continuous distance on $\XX$. Let $\mu_1$ and $\mu_2$ be two probability measures on $\XX$. If $\Delta$ denotes the $d$-diameter of $\XX$, then, for any $\eta > 0$, there exists a 1-Lipschitz measurable function $f: \XX \to \RR$ such that
\begin{equation*}
W_{\XX}(\mu_1,\mu_2) \leq \Delta (1+ \eta) \err(c_f; \mu_1,\mu_2).
\end{equation*}
\end{prop}

\begin{proof}[Proposition \ref{KR_distance_smaller_than_F_1-F_2}]
By Proposition \ref{equality_MK_dist_for_1-Lip_real_functions}, there exist an optimal function $g:\XX\to\RR$ such that 
$W_{\XX}(\mu_1,\mu_2) = W_{\RR}(g(\mu_1),g(\mu_2))$. On $\RR$ equipped with the Euclidean distance, we apply Theorem \ref{KR_distance_on_R} and obtain
\begin{equation*}
W_{\XX}(\mu_1,\mu_2)= \int_{-\infty}^{+\infty} |G_2(t) - G_1(t)| \diff t,
\end{equation*}
where $G_i(t) = \mu_i(\{x \in \XX ;\, g(x) \leq t \})$ denotes the distribution functions of the measure $g(\mu_i)$, for $i=1,2$.\\
Since $\diam(\XX)<+\infty$ and $g$ is 1-Lipschitz, $\diam(g(\XX))\le \Delta$. Thus, $|G_2(t) - G_1(t)| \leq 1$, for any $t \in \RR$ and $|G_2(t) - G_1(t)| = 0$ if $t \not\in [a, b]$ with $a<b,\,b-a\leq\Delta$. We can write:
\begin{equation*}
W_{\XX}(\mu_1,\mu_2)= \int_0^{\Delta} |G_{\mu_2}(t) - G_{\mu_1}(t)| \diff t.
\end{equation*}

We denote, for simplicity, $\sup_t |G_2(t) - G_1(t)|$ by $S$. For all $\eta>0$, there exist a $t_{\circ}\in [a,b]$ such that $S-\eta < |G_2(t_{\circ}) - G_1(t_{\circ})| \leq S$. Therefore, 
\begin{equation*}
W_{\XX}(\mu_1,\mu_2) \leq \Delta S \leq \Delta (1 + \eta) |G_2(t_{\circ}) - G_1(t_{\circ})|.
\end{equation*}
Set $f = g - t_{\circ}$ and $F_i$ the distribution function of the probability measure $f(\mu_i)$. For all $t \in \RR, F_i(t) = G_i(t - t_{\circ})$ and thus $G_i(t_{\circ}) = F_i(0)$. Then, $f$ is a 1-Lipschitz, measurable function such that 
\begin{equation*}
W_{\XX}(\mu_1,\mu_2) = \int_{\XX} f \diff (\mu_1-\mu_2) \leq \Delta (1 + \eta) |F_2(0) - F_1(0)|.
\end{equation*}
Applying Lemma \ref{prop_min_of_rho1_rho2} completes the proof. 
\end{proof}

In Proposition \ref{KR_distance_smaller_than_F_1-F_2}, we have shown the existence of $(\epsilon, \delta)$-Lipschitz classifiers whose $\epsilon$-error $\delta$ depends on $W_{\XX}(\mu_1,\mu_2)$, when $\epsilon = 0$.
In Theorem \ref{theo_upper_bound_for_delta}, we consider the case of $(\epsilon, \delta)$-Lipschitz classifiers for $\epsilon \geq 0$. It is a more general result, but we obtain a less stringent upper bound. We first give a proposition that we use in the proof of Theorem \ref{theo_upper_bound_for_delta}. 
 
\begin{prop}\label{prop_pre_thm_upper_bound_for_delta}
Let $\XX$ be a Polish space and $d \in \DD_{\XX}$ be a bounded lower semi-continuous distance on $\XX$. Let $\mu_1$ and $\mu_2$ be two probability measures on $\XX$.\\
If $\Delta$ denotes the $d$-diameter of $\XX$ then there exists a 1-Lipschitz classification function $f$ such that for all $0\leq s <t \leq 1$,
\begin{equation*}
s \mu_2(\{x\in \XX; f(x) > s\Delta\}) + (1-t)\mu_1(\{x \in \XX; f(x) \leq t\Delta\}) \leq 1 - \frac{W_{\XX}(\mu_1,\mu_2)}{\Delta}.
\end{equation*}
Therefore, for $0< s \leq 1/2$,
\begin{equation*}
\mu_2(\{x\in \XX; f(x) > s\Delta\}) + \mu_1(\{x \in \XX; f(x) \leq (1-s)\Delta\}) \leq \frac{1}{s} \left(1- \frac{W_{\XX}(\mu_1,\mu_2)}{\Delta}\right).
\end{equation*}
\end{prop}

\begin{proof}[Proposition \ref{prop_pre_thm_upper_bound_for_delta}]
By Proposition \ref{writing_KR_dist_with_dist_funct} , there exists a 1-Lipschitz measurable function $g: \XX \rightarrow [0, \Delta]$ such that
\begin{equation*}
W_{\XX}(\mu_1,\mu_2) = \int_0^{\Delta} (G_2(s) - G_1(s)) \diff s,
\end{equation*}
where $G_i(s) = \mu_i(\{ x\in\XX; \, g(x) \leq s\})$, for $i = 1,2$. \\
For $0 \leq s < t \leq 1$, we have
\begin{align} \label{equation_upper_bound_of_KR_distance}
W_{\XX}(\mu_1,\mu_2) &= \left( \int_0^{s\Delta} + \int_{s\Delta}^{t\Delta} + \int_{t\Delta}^{\Delta}\right) (G_2(s) - G_1(s)) \diff s \nonumber \\[4mm]
&\leq s \Delta \, G_2(s\Delta) + (t-s)\Delta + \Delta (1-t) \big(1 - G_1(t\Delta)\big). 
\end{align}
Dividing by $\Delta$ on both sides and rearranging the terms yields
\begin{equation*}
\frac{W_{\XX}(\mu_1,\mu_2)}{\Delta}  \leq  s \, G_2(s\Delta) + (t-s) + 1 - G_1(t\Delta) - t  \big(1 - G_1(t\Delta)\big).
\end{equation*}
For $0 \leq r \leq \Delta$, let $A_r = \{x \in \XX; g(x)\leq r\Delta\}$. We can now write:
\begin{align*}
\frac{W_{\XX}(\mu_1,\mu_2)}{\Delta} & \leq s \mu_2(A_s) + (t-s) + (1-t) \left(1- \mu_1(A_t)\right) \\[0.5cm]
1 - \frac{W_{\XX}(\mu_1,\mu_2)}{\Delta} & \geq s \left(1 - \mu_2(A_s^{\comp})\right) + (1-s) - (t -s) + (1-t) \left(1 - \mu_1(A_t)\right) \\[0.5cm]
1 - \frac{W_{\XX}(\mu_1,\mu_2)}{\Delta} & \geq s \mu_2(A_s^{\comp}) + (1-t) \mu_1(A_t).
\end{align*}
For $s \in (0,1/2]$ and $t = 1-s$, we obtain
\begin{equation*}
\mu_2(A_s^{\comp}) + \mu_1(A_{1-s}) \leq \frac{1}{s} \left(1- \frac{W_{\XX}(\mu_1,\mu_2)}{\Delta}\right).
\end{equation*}
\end{proof}
 
\begin{theo}\label{theo_upper_bound_for_delta}
Let $\XX$ be a Polish space and $d \in \DD_{\XX}$ be a bounded lower semi-continuous distance on $\XX$. Let $\mu_1$ and $\mu_2$ be two probability measures on $\XX$.\\
If $\Delta$ denotes the $d$-diameter of $\XX$ then, for $0 \leq \varrho < 1$, there exists a $(\varrho\Delta, \delta)$-Lipschitz classifier $c_{f,\varrho\Delta}$ such that 
\begin{equation*}
\delta \leq \frac{2}{1 - \varrho} \left(1 - \frac{W_{\XX}(\mu_1,\mu_2)}{\Delta}\right).
\end{equation*}
\end{theo}

\begin{proof}[Theorem \ref{theo_upper_bound_for_delta}]
Let $s = \frac{1 - \varrho}{2}$. By Proposition \ref{prop_pre_thm_upper_bound_for_delta}, there exists a Lipschitz classification function $g:\XX \to [0,\Delta]$ such that 
\begin{multline*}
\mu_2\left(\{x\in \XX; g(x) > \frac{\Delta}{2}-\frac{\varrho\Delta}{2}\}\right) + \mu_1\left(\{x \in \XX; g(x) \leq \frac{\Delta}{2} + \frac{\varrho\Delta}{2}\}\right)\\
 \leq \frac{2}{1-\varrho} \left(1- \frac{W_{\XX}(\mu_1,\mu_2)}{\Delta}\right).
\end{multline*}
 By setting $f = g - \frac{\Delta}{2}$, the proof is complete.
\end{proof}

In the case where the two measures $\mu_1$ and $\mu_2$ are not probability measures but bounded measures such that $\mu_1(\XX) = \mu_2(\XX) = \gamma$, Proposition \ref{prop_pre_thm_upper_bound_for_delta} and Theorem \ref{theo_upper_bound_for_delta} can be made slightly more general:

\begin{prop}
Let $\XX$ be a Polish space and $d \in \DD_{\XX}$ be a bounded lower semi-continuous distance on $\XX$. Let $\mu_1$ and $\mu_2$ be two finite measures on $\XX$ such that $\mu_1(\XX) = \mu_2(\XX) = \gamma$.\\
If $\Delta$ denotes the $d$-diameter of $\XX$ then there exists a 1-Lipschitz classification function $f$ such that for all $0 \leq s <t \leq 1$,
\begin{equation*}
s \mu_2(\{x\in \XX; f(x) > s\Delta\}) + (1-t)\mu_1(\{x \in \XX; f(x) \leq t\Delta\}) \leq \gamma - \frac{W_{\XX}(\mu_1,\mu_2)}{\Delta}.
\end{equation*}
Therefore, for $0< s \leq 1/2$,
\begin{equation*}
\mu_2(\{x\in \XX; f(x) > s\Delta\}) + \mu_1(\{x \in \XX; f(x) \leq (1-s)\Delta\}) \leq \frac{1}{s} \left(\gamma- \frac{W_{\XX}(\mu_1,\mu_2)}{\Delta}\right).
\end{equation*}
\end{prop}

\begin{theo}\label{new_version_theo_upper_bound_of_delta}
Let $\XX$ be a Polish space and $d \in \DD_{\XX}$ be a bounded lower semi-continuous distance on $\XX$. Let $\mu_1$ and $\mu_2$ be two finite measures on $\XX$ such that $\mu_1(\XX) = \mu_2(\XX) = \gamma$.\\
If $\Delta$ denotes the $d$-diameter of $\XX$ then, for $0 \leq \varrho < 1$, there exists a $(\varrho\Delta, \delta)$-Lipschitz classifier $c_{f,\varrho\Delta}$ such that 
\begin{equation*}
\delta \leq \frac{2}{1 - \varrho} \left(\gamma - \frac{W_{\XX}(\mu_1,\mu_2)}{\Delta}\right).
\end{equation*}
\end{theo}

\section{A more general notion of error}

In section \ref{section_Def_of_measurable_classifiers}, we give the definition of an $\epsilon$-error for a classification function $f$. The Definition \ref{def_err_epsilon} can be considered quite ``rough" as it does not take into consideration the value of the function $f$ but only its sign. It therefore seems natural to consider a more general notion of error:

\begin{defn}
Let $(\XX, \BB)$ be a measurable space and let $\mu_1$ and $\mu_2$ be two finite measures on $\XX$. Let $f:\XX\to\RR_+$ be a measurable function such that $f \in L_1(\XX,\mu_i)$, $i=1,2$. \\
Then, for $t \geq 0$, let $\EEE(f,t; \mu_1,\mu_2)$ be defined by
\begin{equation*}
\EEE(f,t;\mu_1,\mu_2) = \int_{B_t^{\comp}} f \diff \mu_1 + \int_{B_t} f \diff \mu_2,
\end{equation*}
where $B_t = \{x \in \XX; f(x)\leq t\}$.
\end{defn}

\begin{rmk} There are three remarks worth making:
\begin{enumerate}
\item[(i)] If $\mu_i(\{x \in \XX: f(x)=0\})= 0$, then, $\dps \EEE(f,0;\mu_1,\mu_2) = \int_{\XX} f \diff \mu_2$ 
\item[(ii)] If $f: \XX\to[0,\Delta]$, then $\dps \EEE(f,\Delta;\mu_1,\mu_2) = \int_{\XX} f \diff \mu_1$
\item[(iii)] If $f: \XX\to[0,\Delta]$, then 
\begin{align*}
\EEE(\Delta-f,t;\mu_1,\mu_2) &= \int_{B_t^{\comp}} \Delta-f \diff \mu_1 + \int_{B_t}  \Delta-f \diff \mu_2 \\
&= \Delta(\mu_1(B_t^{\comp}) + \mu_2(B_t)) - \EEE(f,t;\mu_1,\mu_2).
\end{align*}
Therefore, $\EEE(f,t;\mu_1,\mu_2) + \EEE(\Delta-f,t;\mu_1,\mu_2) = \Delta \bigl(\mu_2(\XX)-(\mu_2(B_t)-\mu_1(B_t))\bigr)$ and $\EEE(f,t;\mu_2,\mu_1) + \EEE(\Delta-f,t;\mu_2,\mu_1) = \Delta \bigl(\mu_1(\XX)-(\mu_1(B_t)-\mu_2(B_t))\bigr)$.
\end{enumerate}
\end{rmk}

\begin{defn}\label{dfn_Err}
Let $(\XX, \BB)$ be a measurable space and let $\mu_1$ and $\mu_2$ be two finite measures on $\XX$. Let $f:\XX\to\RR_+$ be a measurable function such that $f \in L_1(\XX,\mu_i)$, $i=1,2$. \\
Then, for $t \geq 0$, the $\epsilon$-error of $f$ with respect to $(\mu_1,\mu_2)$ is given by
\begin{equation*}
\Err(f,\epsilon:\mu_1,\mu_2) = \min\Bigl(\EEE(f,t;\mu_1,\mu_2) + \EEE(\Delta-f,t;\mu_1,\mu_2), \EEE(f,t;\mu_2,\mu_1) + \EEE(\Delta-f,t;\mu_2,\mu_1)\Bigr).
\end{equation*}
\end{defn}

\noindent In the particular case where $\mu_1$ and $\mu_2$ are probability measures, we obtain that 
\begin{equation*}
\Err(f,\epsilon; \mu_1,\mu_2) = \Delta (1- |\mu_2(B_t)-\mu_1(B_t)|)
\end{equation*}
since, for any $r\in\RR$, $\max(r,-r) = |r|$.\\
By Lemma \ref{prop_min_of_rho1_rho2} we get, for $t\in[0,\Delta]$, $\Err(f,\epsilon;\mu_1,\mu_2) = \err(c_{f,t; \mu_1,\mu_2})$.\\

\noindent Now, we consider a particular case: Let $\mu_1$ and $\mu_2$ be two probability measures on $\XX$. Let $\Delta$ denotes the $d$-diameter of $\XX$ and $f: \XX \to [0,\Delta]$ be a measurable function such that, for $t\in[0,\Delta]$, $F_2(t) - F_1(t) \geq 0$, where each $F_i(t) = \mu_i(\{x \in \XX; f(x)\leq t\})$ is continuous. Let $t_{\circ} \in [0,\Delta]$ be such that $\sup\{F_2(t)-F_1(t); t\in[0,\Delta]\} = F_2(t_{\circ})-F_1(t_{\circ})$. We obtain the following proposition and its immediate corollary:

\begin{prop}
Let $\XX$ be a Polish space and $d \in \DD_{\XX}$ be a bounded lower semi-continuous distance on $\XX$. Let $f: \XX \to [0,\Delta]$ be as defined above and $g= \Delta - f$. Then,
\begin{equation*}
\EEE(f,t_{\circ};\mu_1,\mu_2) = \int_{B_{t_{\circ}}} f \diff \mu_1 + \int_{B_{t_{\circ}}^{\comp}} f \diff \mu_2 \quad\mbox{ and }\quad \EEE(g,\Delta-t_{\circ};\mu_2,\mu_1) = \int_{B_{t_{\circ}}^{\comp}} f \diff \mu_1 + \int_{B_{t_{\circ}}} f \diff \mu_2,
\end{equation*}
where $B_{t_{\circ}} = \{x \in \XX; f(x) \in [0,t_{\circ}]\}$. Therefore,
\begin{equation*}
\EEE(f,t_{\circ};\mu_1,\mu_2) + \EEE(g,\Delta-t_{\circ};\mu_2,\mu_1) = \Delta \left( 1 - (F_2(t_{\circ}) - F_1(t_{\circ}))\right).
\end{equation*}
\end{prop}

\begin{cor}
Let $\XX$ be a Polish space and $d \in \DD_{\XX}$ be a bounded lower semi-continuous distance on $\XX$. Then, there exist 1-Lipschitz classification functions $f: \XX \to [0,\Delta]$ and $g= \Delta - f$ such that:
\begin{equation*}
\EEE(f,t_{\circ};\mu_1,\mu_2) + \EEE(g,\Delta-t_{\circ};\mu_2,\mu_1) \leq \Delta \left(1 - \frac{W_{\XX}(\mu_1,\mu_2)}{\Delta}\right).
\end{equation*}
\end{cor}

\cleardoublepage

\chapter{Applications of the Kantorovich-Rubinstein Distance in Machine Learning}\label{KR_score_appplied_to_gwas}

Chapter \ref{KR_score_appplied_to_gwas}, aims to show that the use of the Kantorovich-Rubinstein distance as a feature-selection criterion function is an efficient and natural thing to do.

The first section makes the case that the Kantorovich-Rubinstein distance is a very good feature-selection criterion function because, by construction, it encompasses both topological and metric information of the two sample distributions. Hence, the Kantorovich-Rubinstein distance gives us a good bound for the classification errors without having to go through the whole classification process (as shown in the results of Chapter \ref{KR_distance_upper_bound_of_classifiers}).

The second section of the chapter constructs a Kantorovich-Rubinstein distance based function $J$. We then prove interesting properties of $J$ (monotonicity and additivity) that enable us to apply search algorithms from the Branch-and-Bound family. The advantage of such search engines is that they are much quicker than an exhaustive search while remaining optimal.

\section{The Kantorovich-Rubinstein distance as a descriptor of sample complexity - l.9}

The no free lunch theorem for machine learning states that, averaged over all possible data-generating distributions, no classification algorithm can consistently have the lowest classification error. Said differently: \\
``The most sophisticated algorithm we can conceive of has the same average performance (over all possible tasks) as merely predicting that every point belongs to the same class'' \cite{Deep_Learning_book}.\\
Hence, since there is no classification algorithm that can have the smallest misclassification error over all data-generating distributions, there might be, for particular datasets, specific classification algorithms that are better suited. The difficulty is to find these better-suited algorithms: Suppose we are given an independent and identically distributed sample $S_n = \{(x_1,y_1), \ldots, (x_n,y_n)\}$ with respect to an unknown probability measure $\mu$ on $\XX\times\{-1,1\}$ and $S_x$ be the projection of $S_n$ on $\XX$. If, for a chosen classifier $g$, the empirical misclassification $\hat{R}(g)$ over $S_n$ is large, one can only speculate which of these two explanations is the correct one: is it because the Bayes error $R^*$ is also large or because the choice of the classifier $g$ is inadequate? There is therefore a need for a characterisation of $S_n$. In the case of the measurable space $(\RR^d\times\{-1,1\}, \BB)$, with $d\in \NNN$, the paper \textit{Complexity Measures of Supervised Classification Problems} \cite{Complexity_Ho_and_Basu} by Ho and Basu was seminal in defining the notion of complexity of $S_n$ and finding a set of descriptors that are not directly dependent on a classifier $g$. We first define the notion of complexity. The complexity of $S_n$ is defined using the Kolmogorov complexity concept: the Kolmogorov complexity of $S_n$ is characterised by the length of the shortest algorithm necessary to fully describe the relationship between each point in $x_i \in S_x$ and their respective label. The longer is the algorithm, the more complex is $S_n$. The worst case scenario would require to list all the $x_i \in S_x$ along with their label. However, if there exist some form of regularity in $S_n$, a more compact algorithm can be obtained. In practice, the Kolmogorov complexity is algorithmically incomputable. In their paper, Ho and Basu approximate the Kolmogorov complexity with statistical indicators and geometrical descriptors drawn from $S_n$. They refer to these indicators and descriptors as \textit{complexity measures}. A subset of the complexity measures is called the \textit{geometrical complexity measures}. It contains descriptors and indicators that describe the regularities and irregularities of the boundary that separates the classes of $S_n$. Ho and Basu assume that the family of these descriptors and indicators are sufficient to give a good approximation of $S_n$'s complexity as most classifiers in $\RR^d$ can be characterised by geometrical descriptions of their decision regions \cite{Complexity_Ho_and_Basu}.\\
Following Ho and Basu, many papers were published with new descriptors. A thorough survey was done in 2019 by Lorena \textit{et al.} \cite{Complexity_Review}. The survey groups geometrical complexity measures according to 4 categories:
\begin{enumerate}
\item \textbf{Feature-based measures}, which characterise how informative the available features are to separate the classes; 
\item \textbf{Linearity measures}, which try to quantify whether the classes can be linearly separated; 
\item \textbf{Neighbourhood measures}, which characterise the presence and density of same or different classes in local
neighbourhoods; 
\item \textbf{Network measures}, which extract structural information from the dataset by modeling it as a graph.
\end{enumerate}
\vspace{2mm}
\noindent Lorena \textit{et al.} describes the descriptors of category 2 as follows: \\
``These descriptors try to quantify to what extent the classes are linearly separable, that is, if it is possible to separate the classes by a hyperplane. They are motivated by the assumption that a linearly separable problem can be considered simpler than a problem requiring a non-linear decision boundary.'' \\
There is a major drawback to the linearity descriptors: they are computed using quantities obtained from a linear classifier (i.e., a separating hyperplane). As is standard nowadays, the linear classifier considered is the one constructed using the Support Vector Machines (SVM) algorithm. SVM is considered one of the most theoretically well motivated and practically most effective classification algorithms. Notwithstanding the quality of SVM, this is a situation in which one constructs a classifier to verify if the dataset can be classified. That is, as said above, precisely what one tries to avoid. What would be preferable are predictors that can be computed solely from the dataset while quantifying whether the classes can be separated by a specific family of classifiers. With the results obtained in the previous chapters, we can do just that: construct a predictor based solely on the dataset that predicts to what extend we can separate the classes using 1-Lipschitz classifiers:

\begin{clm}\label{Claim_complexity_KR}
Let $\XX$ be a Polish space and $d \in \DD_{\XX}$ be a bounded lower semi-continuous distance on $\XX$. Let $\mu$ be a probability measure on $\XX\times\{-1,1\}$ and $S_n = \{(x_1,y_1), \ldots, (x_n,y_n)\}$ be an independent and identically distributed sample with respect to $\mu$ and $S_x$ be the projection of $S_n$ on $\XX$. Keeping the notations of Equation (\ref{def_mu_+_et_mu_-_theorique}), we define:
\begin{equation*}
\mu_+ = \frac12 \frac{1}{|S_+|} \sum_{x\in S_+} \delta_x \quad \mbox{ and } \quad \mu_- = \frac12 \frac{1}{|S_-|} \sum_{x\in S_-} \delta_x,
\end{equation*}
where $S_+=\{x_i \in S_x : (x_i,1) \in S_n\}$ and $S_-=\{x_i \in S_x : (x_i,-1) \in S_n\}$.\\
We say that $W_{\XX}(\mu_+,\mu_-)$ is a good complexity measure for $S_n$ with respect to 1-Lipschitz classifiers if the ratio $W_{\XX}(\mu_+,\mu_-)/\Delta$ being close to 1/2 implies that there exists a 1-Lipschitz classifier for which most $x\in S_x$ have a large enough confidence.\\
We claim that the Kantorovich-Rubinstein distance $W_{\XX}(\mu_+,\mu_-)$ between the measures $\mu_+$ and $\mu_-$ is a good complexity measure for $S_n$ with respect to 1-Lipschitz classifiers.
\end{clm}

\noindent We present below what we believe to be a strong theoretical argument to support Claim \ref{Claim_complexity_KR}:

Keeping the notations of Definition \ref{alpha_translated_loss}, we recall that the risk functional of the classifier $c_f$ with the $\alpha$-translated zero-one loss is given by 
\begin{equation*}
R(c_f) = \int_{\XX\times\{-1,1\}} \indicator_{yf(x)\leq\alpha} \diff\mu(x,y).
\end{equation*}
Thus, $R(c_f) = \mu \left(\{(x,y) \in \XX\times\{-1,1\}: yf(x)\leq \alpha\}\right)$. The $\alpha$-translated zero-one loss penalizes both misclassified points and correctly classified points with confidence smaller than $\alpha$. The penalty is 1. Thus, the risk $R(c_f)$ measures the set of points that are both misclassified and correctly classified but with a confidence smaller than $\alpha$. Hence, if $R(c_f)$ is small, the measure of the 3 sets $\{(x,-1): f(x) \geq \alpha\}$, $\{(x,1): f(x) \leq -\alpha\}$ and $\{x \in \XX: -\alpha \leq f(x) \leq \alpha\}$ are also small.\\
In the case of a finite sample $S_n$, the empirical risk functional is: 
\begin{equation*}
\hat{R}(c_f) = \frac{1}{n} \sum_{i=1}^n\indicator_{y_if(x_i)\leq\alpha}. 
\end{equation*}
Thus, a small $\hat{R}(c_f)$ implies that the cardinality of the 3 sets $\{x_i \in S_+: f(x_i) \geq \alpha\}$, $\{x_i \in S_-: f(x_i) \leq -\alpha\}$ and $\{x_i \in S_n: -\alpha \leq f(x_i) \leq \alpha\}$ is small. This means that the classification function f is such that there are few misclassified points and, moreover, there are few points of $S_x$ in $f^{-1}([-\alpha,\alpha])$. Said differently, the 1-Lipschitz classifier $c_f$ is such that most $x\in S_x$ have a confidence larger than $\alpha$.\\
To support our claim, we thus need to show that a large Kantorovich-Rubinstein distance between $\mu_+$ and $\mu_-$ implies that $\hat{R}(c_f)$ will be small:\\
Let $\varphi = \alpha/\Delta$. Then, by Theorem \ref{new_version_theo_upper_bound_of_delta}, there exits a $(\alpha,\delta)$-Lipschitz classifier $c_{f,\alpha}$ such that 
\begin{equation*}
\err(c_{f,\alpha}; \mu_+,\mu_-) \leq \frac{2}{1-\varphi} \left(\frac12 - \frac{W_{\XX}(\mu_+,\mu_-)}{\Delta}\right).
\end{equation*}
Note that the right hand side of the inequality is always positive since $W_{\XX}(\mu_+,\mu_-) \leq \frac12 \Delta$. \\
By Proposition \ref{risk_to_varepsilon}, we know that, for the classification function of $c_f$,
\begin{equation*}
\frac{1}{n} \sum_{i=1}^n\indicator_{y_if(x_i)\leq\alpha} = \err(c_{f,\alpha};\mu_+,\mu_-).
\end{equation*}
And thus we obtain:
\begin{equation*}
\hat{R}(c_f) \leq \frac{2}{1-\varphi} \left(\frac12 - \frac{W_{\XX}(\mu_+,\mu_-)}{\Delta}\right).
\end{equation*}

\section{The Kantorovich-Rubinstein distance and feature-selection algorithms - l.64}\label{section_sur_J}

Many machine learning problems become exceedingly difficult when the sample set lives in a high-dimensional space. These difficulties are referred to as \textit{the curse of dimensionality}. These difficulties stem from the fact that, in high-dimensional space, a sample is most likely sparse. Most classification algorithms from the ``pre deep learning" era are designed to encourage the construction of smooth classification functions. The idea being that a smooth classification function would have little overfitting. But when your data is sparse, it becomes extremely complicated to generalise well. To avoid the curse of dimensionality, one can either add assumptions about the unknown data generating distribution (as in the case in deep learning and manifold learning) or can reduce the dimensionality of the learning space. There are two different approaches to dimensionality reduction: the feature selection technique and the feature extraction technique. Suppose $S_n = \{(x_1,y_1), \ldots, (x_n,y_n)\}$ is a sample in $\XX\times\{-1,1\}$ with $\dim(\XX) = r$, and $S_x$ be the projection of $S_n$ on $\XX$. For feature selection, given $0<a<r$, a score $J$ is assigned to all subsets $A = \{i_1,\ldots,i_a\} \subset \{1, \ldots,r\}$. The sample $S_x$ is then canonically projected on $\XX_{A^*} = \XX_{i_{a^*}} \times \ldots \times \XX_{i_{a^*}}$, where $A^*$ is the subset with the optimal value of $J$. A feature selection algorithm is in fact a canonical projection on $a$ chosen dimensions of $\XX$.\\
On the other hand, feature extraction algorithm includes transformations of the space $\XX$ before projecting on a space of lower dimension \cite{pattern_recognition_review}. Note that both techniques can be used on the same set of observations $\XX$.

Well known examples of feature extraction techniques include Principal Component Analysis (PCA) and Random Projections. PCA is extensively covered in Chapter 12 of \textit{Pattern Recognition and Machine Learning} \cite{Bishop} while Random Projections is well summarized in section 4.2 of H. Duan's Master's thesis \cite{Hubert_thesis}. Another feature extraction technique worth mentioning is the \textit{Borel isomorphic dimensionality reduction technique}. It was introduced by V. Pestov in \cite{Pestov_Borel_dimensionality_reduction} and further studied by S. Hatko in \cite{Stan_Honour_project}.

One can often obtain a set of transformed features generated by feature extraction that provide a better discriminative ability than the best subset chosen by feature selection. The drawback is that feature extracted subsets often lose their original physical/biological meaning which makes it less useful for applications. For that reason, we restrict our interest to feature selection algorithms and more precisely the Branch \& Bound family of algorithms. A quick overview of feature selection algorithms is given in Appendix \ref{Feature_selection_algo_appndix}. We recommend you read it before reading subsection \ref{branch_and_bound_subsection}.

\subsection{Branch \& Bound algorithms - l.73}\label{branch_and_bound_subsection}

This section gives a brief overview of the initial Branch \& Bound algorithm and touches on its newer and improved versions. The section relies heavily on the paper by Somol \textit{et al.} \cite{Fast_Branch&Bound_Algo_for_optimal_feature_selection} and Frank \textit{et al.} \cite{Frank_et_al_Distance_Based__BandB_Feature_Selection_algorithm}. As explained in the introduction of section \ref{section_sur_J}, the goal of feature selection algorithms is to select, from a set $E$ of $r$ elements, an optimal subset of cardinality $k$. The notion of optimality is measured by an evaluation criterion function $J$. From Appendix \ref{Feature_selection_algo_appndix}, we know that, in the case of a set $E$ with large cardinality, it is totally unrealistic to expect to compute $J$ for each subset of cardinality $k$. Indeed, the exhaustive search is impractical even for problems of small cardinality $r$ as the optimization space of all subsets of cardinality $k<r$ is of combinatorial complexity. To bypass the issue of complexity, two strategies has been used. The first one is to relax the rigour of optimality. (ie. reducing the size of the optimization space). The second strategy is the introduction of a feature selection criterion function (fscf) with specific properties which help to identify sections of the search optimization space that can be left unexplored. We use the standard notation $J: \PPP(E) \mapsto \RR$ for the fscf. The Branch \& Bound family of algorithms employs the second strategy by making use of a fscf  $J$ that satisfies the set inclusion monotonicity property. 
\begin{defn}\label{monotone_def}
Let $E$ be a set of cardinality $r$ and $\PPP(E)$ be the power set of $E$. Consider $f: \PPP(E) \mapsto \RR$. The set inclusion monotonicity property assumes that, for two subsets, $E_1,E_2 \in \PPP(E)$,
\begin{equation*}
f(E_1) \leq f(E_2) \mbox{ if } E_1 \subset E_2.
\end{equation*}
\end{defn}

It is the use of a fscf $J$ satisfying the inclusion monotonicity property that allows to remove parts of the search space that cannot possibly contain the optimal solution. Given a fscf satisfying the inclusion monotonicity property, one can use an algorithm from the Branch \& Bound family. The Branch \& Bound principle can be summarized as follows:\\
First, the algorithm constructs a solutions tree $T(E,k)$ where the root represents the set $E$ of cardinality $r$ and the $C(r,k)$ leaves represent all the subsets of $k$ elements. In the solution tree, any node at depth $d$ represents a subset of $E$ with a cardinality of $r-d$ elements. Hence, the leaves are at depth $r-k$. The generation of a solution tree can be regarded as a recursive procedure that builds solution trees from depth $d=0$ to $d = r - k$. After obtaining a solution tree, all string-structure subtrees (subtrees in which non-leaves nodes have only one child node) can be pruned to obtain a minimum solution tree. Note that the generation of a solution tree (and of minimum solution tree) is independent of the choice of fscf. For more details on this recursive procedure, see Yu \textit{et al} \cite{Efficient_B&B_for_fture_slctn_Yu&Yuan}.\\
Once the minimum solution tree has been built, one traverses the tree to find the optimal subset. The algorithm keeps the fscp of the currently best leaf-node in memory (denoted as the bound). Anytime the fscp of some internal node is found to be lower than the bound, the whole subtree may be cut off and thus many computations can be omitted. Somol \textit{et al.} \cite{Fast_Branch&Bound_Algo_for_optimal_feature_selection} have observed that: 1) nearer to the root, the fscf computation is usually slower as evaluated feature subsets are larger and 2) nearer to the root, subtree cut-offs are less frequent.
Note that, for a given minimum solution tree, the speed of the Branch \& Bound algorithm depends on the choice of the fscf. Indeed, "traversing the solution tree" requires to compute the fscf at many (if not most) of the tree nodes.\\ 
Over time, considerable effort has been invested into the acceleration of the Branch \& Bound algorithm. 
The article \textit{Fast Branch \& Bound Algorithms for Optimal Feature Selection} by Somol \textit{et al.} \cite{Fast_Branch&Bound_Algo_for_optimal_feature_selection} is a good review of the many improvements since the original Branch \& Bound algorithm developed by Narendra and Fukunaga \cite{Original_B&B_algorithm}.
 
\subsection{The use of the Kantorovich-Rubinstein distance to construct an evaluation criterion function - l.86}

Our goal is to use the Branch \& Bound algorithm as a dimensionality reduction prior to using a classification algorithm. Hence, a natural choice for the criterion function $J$ is a function that satisfies the set inclusion monotonicity property defined in \ref{monotone_def} and characterises the complexity of $S_n$ in $\XX_{A} = \XX_{i_1} \times \ldots \times \XX_{i_a}$, with $A = \{i_1,\ldots,i_a\} \subset \{1, \ldots,r\}$. We could then be cautiously optimistic that our ``optimal'' subset $A^* = \{i_1^*,\ldots,i_a^*\}$ would allow the projection of $S_n$ into $\XX_{A^*}\times\{-1,1\}$ to be easily classifiable. We now construct a Kantorovich-Rubinstein distance-based function $J$ and prove two interesting properties of it: monotonicity and additivity.\\

Let $r \in \NNN$ be a positive integer. For $1 \leq k \leq r$, let $\XX_k$ be a Polish space, $\BB_k = \BB(\XX_k)$ be its Borel $\sigma$-algebra and $d_k$ be a lower semi-continuous distance on $\XX_k$ (written $d_k \in \DD_{\XX_k}$). The product space $\XX = \XX_1 \times \ldots \times \XX_r$ endowed with the product topology is a Polish space and its Borel $\sigma$-algebra is $\BB(\XX) = \BB_1 \otimes \ldots \otimes \BB_r$. We denote by $\tilde{d}_1$ (respectively $\tilde{d}_{\infty}$) the $\ell_1$ (respectively $\ell_{\infty}$) distance on $\XX$. As seen in Definition \ref{definition_metric_on_cartesian_space_recall}, we have, for $x,y \in \XX$,
\begin{equation*}
\tilde{d}_1(x,y) = \sum_{i = 1}^r d_k(x_k,y_k) \, \mbox{ and }\, \tilde{d}_{\infty}(x,y) = \max_{1\leq k \leq r} d_k(x_k,y_k).
\end{equation*}
By Lemma \ref{lsc_of_product_distance}, we know that $\tilde{d}_1$ and $\tilde{d}_{\infty}$ are lower semi-continuous on $\XX$.\\
Consider $\XX_{A} = \XX_{i_1} \times \ldots \times \XX_{i_a}$ where $\{i_1,\ldots, i_a\}$ belongs to the power set $\PPP\{1,\ldots, r\}$. Recall that the canonical projection $\pi_{a} : \XX \rightarrow \XX_A$ is defined by $\pi_a(x) = (x_{i_1},x_{i_2},\ldots ,x_{i_a})$ and the $\ell_1$-distance on $\XX_A$ is defined by
\begin{equation*}
\tilde{d}_{1,a}(x,y) = \sum_{k= 1}^a d_{i_k}(x_{i_k},y_{i_k}).
\end{equation*}
 
\noindent We can now define our evaluation function:

\begin{defn}\label{Definition_function_J}[l.112]
Let $r$ be a positive integer. For $1 \leq k \leq n$, let $\XX_k$ be a Polish space and $d_k \in \DD_{\XX_k}$. Consider the product space $\XX = \XX_1 \times \ldots \times \XX_r$ endowed with the product topology.\\
We denote by $J: P_{\XX} \times P_{\XX} \times \PPP(\{1, \ldots, r\}) \longrightarrow \RR_+$ the evaluation function defined by:
\begin{equation*}
J(\mu_1, \mu_2, A) = W_{\XX_A}\big( \pi_a(\mu_1), \pi_a(\mu_2)\big),
\end{equation*} 
where $W_{\XX_A}$ is the Kantorovich-Rubinstein distance associated to $(\XX_A, \tilde{d}_{1,a})$.\\

\noindent For $(\mu_1,\mu_2) \in P_{\XX} \times P_{\XX}$ fixed, we define the function $J_{(\mu_1,\mu_2)} : \PPP(\{1,\ldots, r\}) \rightarrow \RR_+$ given by the formula $J_{(\mu_1,\mu_2)} (A) = J(\mu_1,\mu_2, A)$.
\end{defn}

\noindent The criterion function $J$ defined above has two interesting properties:

\begin{prop}
Let $r$ be a positive integer. For $1 \leq k \leq r$, let $\XX_k$ be a Polish space and $d_k \in \DD_{\XX_k}$. Let $A=\{i_1,\ldots, i_a\} \subset B= \{j_1, \ldots, j_b\}$ be non-empty subsets of $\{1, \ldots, r\}$. Consider the product space $\XX = \XX_1 \times \ldots \times \XX_r$ endowed with the product topology and let $\mu_1$ and $\mu_2$ be two probability measures on $P_{\XX}$.\\
Then, $J_{(\mu_1,\mu_2)} (A) \leq J_{(\mu_1,\mu_2)} (B)$. \\
Such a function $J_{(\mu_1,\mu_2)}$ is said to be monotone.
\end{prop}

\begin{proof}
As $J_{(\mu_1,\mu_2)} (A) = W_{\XX_A}\big( \pi_a(\mu_1), \pi_a(\mu_2)\big)$ and $J_{(\mu_1,\mu_2)} (B) = W_{\XX_B}\big( \pi_b(\mu_1), \pi_b(\mu_2)\big)$, a direct application of Corollary \ref{inequality of KR distances with projection on Y_s and Y_s+r} yields the desired result.
\end{proof}

\begin{prop}[l.125]
Let $r$ be a positive integer. For $1 \leq k \leq r$, let $\XX_k$ be a Polish space, $d_k \in \DD_{\XX_k}$ and $\mu_1^{(k)},\,\mu_2^{(k)} \in P_{\XX_k}$ be two probability measures. Consider the product space $\XX = \XX_1 \times \ldots \times \XX_r$ endowed with the metric $d$ and $\mu_1 = \mu_1^{(1)} \otimes \ldots \otimes \mu_1^{(r)}$ and $\mu_2 = \mu_2^{(1)} \otimes \ldots \otimes \mu_2^{(r)}$, the corresponding product measure on $\XX$. Then, for any $A, B \in \PPP(\{1,\ldots, r\})$, $A \cap B = \emptyset$, 
\begin{equation*}
J_{(\mu_1,\mu_2)}(A \cup B) = J_{(\mu_1,\mu_2)}(A) + J_{(\mu_1,\mu_2)}(B).
\end{equation*}
Such a function $J_{(\mu_1,\mu_2)}$ is said to be additive.
\end{prop}

\begin{proof}
Without loss of generality, it is enough to prove that additivity holds for singleton. Let $A = \{a\}$ and $B = \{b\}$. As $\mu$ is a product measure, we have, by construction, that $\pi_a(\mu) = \mu_a$, $\pi_b(\mu) = \mu_b$ and $\pi_{a \cup b}(\mu) = \mu_a \otimes \mu_b$. Using Theorem \ref{thm_mesure_produit}, we obtain:
\begin{align*}
J_{(\mu_1,\mu_2)}(A \cup B) &= W_{\XX_A \times \XX_B} (\pi_{a \cup b}(\mu_1), \pi_{a \cup b}(\mu_2))\\
&= W_{\XX_A \times \XX_B} (\mu_1^{(a)}\otimes\mu_1^{(b)}, \mu_2^{(a)}\otimes\mu_2^{(b)}) \\
&= W_{\XX_A}(\mu_1^{(a)}, \mu_2^{(a)}) + W_{\XX_B}(\mu_1^{(b)}, \mu_2^{(b)}) \\
&= J_{(\mu_1,\mu_2)}(A) + J_{(\mu_1,\mu_2)}(B).
\end{align*}
\end{proof}
\cleardoublepage


\chapter{Conclusion}\label{chapitre_conclusion}

The idea to use the Kantorovich-Rubinstein distance for dimensionality reduction was presented to Hubert Duan and I by Professor Vladimir Pestov. In his Master's thesis \cite{Hubert_thesis}, H. Duan constructed the ``simplest'' dimensionality reduction algorithm: for a sample $S_n$ of points in $\XX\times\{-1,1\}$ (where $\XX$ is a product space $\XX = \XX_1 \times \ldots \times \XX_d$) and $S_x$ the projection of $S_n$ on $\XX$, he considered the $d$ projection maps $\pi_i:\XX\to\XX_i$ and associated the zero-one distance to each $\XX_i$. Such a choice of distance made his $d$ evaluation functions $J_i$ equal to the total variation distance on the $d$ samples $\pi_i(S_x)$. He then fixed a threshold $\alpha$ and collected the coordinates $\{i_1, \ldots, i_a\} \subset \{1,\ldots, d\}$ with an evaluation function greater than $\alpha$. This allowed him to define a dimension reduction map $T:\XX\to\XX_A$ where $\XX_A = \XX_{i_1}\times\ldots\times\XX_{i_a}$.\\
In conjunction with the Random Forest classification algorithm, H. Duan tested his dimensionality reduction algorithm on a portion of the OHGS-2 dataset. He obtained promising results.\\

In this thesis, we had two main objectives. The first objective was to build a strong theoretical foundation to justify the use of the Kantorovich-Rubinstein distance to construct dimensionality reduction techniques applicable to GWAS datasets such as the OHGS dataset. The second objective was to generalise H. Duan dimensionality reduction algorithm in two ways. First by considering not only the $d$ projection maps $\pi_i:\XX\to\XX_i$ but the projections over all possible combination of the $d$ coordinates. Secondly, by considering any distance $d_i$ associated to each $\XX_i$. \\

As this thesis had two main objectives, each objective will have its own list of suggestions for future research.

\noindent In the case of the theoretical foundations to justify the use of the Kantorovich-Rubinstein distance, here are additional ideas that could be developed:
\begin{enumerate}
\item[(i)] Here, we keep the notations of Chapter \ref{Classification_problem_background} and Chapter \ref{KR_distance_upper_bound_of_classifiers}. In Chapter \ref{Classification_problem_background}, we give an interpretation of the quantities $yf(x)$ and $\Phi(yf(x))$: $yf(x)$ is the confidence margin of the prediction $g_f(x)$ and $\Phi(yf(x))$ is the loss incurred by $c_f$ at the point $(x,y)$. In the two particular cases of $\Phi(yf(x)) = 1-H(yf(x))$ and $\Phi(yf(x))= L_{\alpha}(yf(x))$ (where $H$ is the Heaviside function and $L_{\alpha}$ is the $\alpha$-translated zero-one loss), we showed that $R(g_f) = \err(c_{f,\alpha};\mu_+,\mu_-)$. This equality does not stand for all loss functions. In particular, one cannot write $R(g_f) = \Err(f,\alpha;\mu_+,\mu_-)$ when we consider the $\alpha$-margin loss function i.e., $\Phi_{\alpha}(yf(x))= \min \left(1, \max\ \left(0, 1- \frac{u}{\alpha}\right)\right)$. It would therefore be interesting to construct all loss functions that allow the risk functional to be written as an error. It would require to find the structure of all loss functions $\Phi:\XX \times \{-1,1\}\to\RR$ that satisfy the equalities
\begin{equation*}
\int_{\XX} \Phi_{\alpha}(f(x)) \diff\mu_+ = \int_{B} \Phi_{\alpha}(f(x)) \diff\mu_+ \mbox{ and } \int_{\XX} \Phi_{\alpha}(-f(x)) \diff\mu_- = \int_{B^{\comp}} \Phi_{\alpha}(-f(x)) \diff\mu_-
\end{equation*}
where $B \subset \XX$. \\
\item[(ii)] Here, we keep the notations of Chapter \ref{KR_distance_upper_bound_of_classifiers}. In the proof of Proposition \ref{prop_pre_thm_upper_bound_for_delta}, we construct an upper bound for the Kantorovich-Rubinstein distance (see Equation (\ref{equation_upper_bound_of_KR_distance})). This upper bound is constructed using a geometrical approach: since the Kantorovich-Rubinstein distance $W_{\XX}(\mu_1,\mu_2)$ can be expressed as the area between the distribution functions $G_1$ and $G_2$ (where $G_i(s) = \mu_i (\{x\in g(x): g(x)\leq s\})$, we built a covering of $W_{\XX}(\mu_1,\mu_2)$ with three rectangles. The sum of each rectangle's area could then be associated to the notion of error $\err(c_{f,\alpha};\mu_1,\mu_2)$ defined in \ref{def_err_epsilon}. We think that we can build a better covering of $W_{\XX}(\mu_1,\mu_2)$ using geometrical areas that are given in Equation (\ref{equation_geometric_version_of_KR_distance}) and Equation (\ref{equation_geometric_version_of_KR_distance_Delta-f}) of Appendix \ref{KR_distance_geometric_construction}. The sum of these areas would be associated with the notion of error $\Err(f,\epsilon:\mu_1,\mu_2)$ defined in \ref{dfn_Err}.
\end{enumerate} 

\noindent In the case of the second objective (i.e., the generalisation of H. Duan's algorithm), the natural next step is to test our new algorithm on the OHGS-2 dataset. For this thesis, many programs were written to study the Kantorovich-Rubinstein distance as a tool of dimensionality reduction. These programs are written in R. They are accessible upon request. These programs are separated in two categories. The first category regroups programs preparing the OHGS-2 dataset for computations and computing the Kantorovich-Rubinstein distance. The second category regroups the analysis programs. The computation should be improved in two ways. First, the Kantorovich-Rubinstein computation needs to be upgraded with the latest (and thus fastest) algorithms. Today, there exists many efficient algorithms to compute approximations of the Kantorovich-Rubinstein distance. The major breakthrough came in 2013 when Cuturi \cite{cuturi2013sinkhorn} showed that the empirical Kantorovich-Rubinstein distance can be regularised by an entropic term. The regularisation turns the Kantorovich-Rubinstein distance into a strictly convex problem which can be solved very quickly with the Sinkhorn-Knopp's matrix scaling algorithm. Such a convex problem exhibits linear convergence and can be trivially parallelised. Since Cuturi's paper, many improved algorithms have been published. In particular, the greedy coordinate descent algorithm \textit{Greenkhorn} developed by Altschuler \textit{et al.} \cite{altschuler2018nearlinearn} runs in near-linear time (i.e., runs in time $\OO(d^2)$ where $d$ is the space dimension). The expertise in Kantorovich-Rubinstein distance computation exists at the department of Mathematics and Statistics of the University of Ottawa. Secondly, we would need to write the code for an efficient Branch \& Bound algorithm. Paired with the Greenkhorn algorithm, we would obtain a state-of-the-art dimensionality reduction algorithm for which we could fix the number of dimensions we want. From the biology perspective, this would allow us to choose the number $n$ of snps to consider and then obtain the most significant combination of $n$ snps in a timely fashion.

\cleardoublepage

%
%
%
%
%
%
%
%
%
%
%
\appendix


\chapter{Elementary Definitions and Results}\label{Elementary_results_appendix}

This Appendix features some known definitions and results used in the thesis but not directly related to optimal transport or classification problems. The Appendix is separated in two sections that are totally independent of each other. The first section states the disintegration theorem that is used often in the thesis. The second section is a quick overview of the total variation distance.  
\section{Disintegration Theorem}

\begin{theo}[Disintegration Theorem - l.7]\label{Disintegration_thm}
Let $\XX$ and $\YY$ be two standard Borel spaces and $P(\XX)$ and $P(\YY)$ be their respective collection of Borel probability measures. Let $\mu \in P(\YY)$, let $\pi : \YY \rightarrow \XX$ be a Borel-measurable function, and let $ \nu \in P(\XX)$ be the push-forward measure $\nu  = \pi * (\mu) = \mu(\pi^{-1})$.\\
Then there exists a $\nu$-almost everywhere uniquely determined family of probability measures $\{\mu_x\}_{x\in \XX} \subset P(\YY)$ such that :
 \begin{itemize}
 \item the function $x \mapsto \mu_x$ is Borel measurable, in the sense that  $x\mapsto \mu _x(B)$ is a Borel measurable function for each Borel measurable set $B\subset \YY$;
 
\item $\mu_x$ "lives on" the fiber $\pi^{-1}(x)$: for $\nu$-almost all $x\in \XX$, $\displaystyle{\mu_x \left(\YY\setminus\pi^{-1}(x)\right) = 0}$.

Hence $\mu_x(B) = \mu_x\left( B \cap \pi^{-1}(x)\right)$, for $B \in \YY$.

\item for every Borel measurable function $f : \YY \rightarrow [0, +\infty]$,
\begin{equation*}
\int_{\YY} f(y)\diff\mu(y) = \int_{\XX} \int_{\pi^{-1}(x)} f(y)\diff\mu_x(y) \diff\nu(x)
\end{equation*}
In particular, for any event $B\subset \YY$, taking $f$ to be the indicator function of $B$,
\begin{equation*}
\mu(B) = \int_{\XX} \mu_x (B) \diff \nu(x).
\end{equation*}
\end{itemize}
\end{theo}

\begin{proof}[Theorem \ref{Disintegration_thm}]
For a proof of the Disintegration Theorem, one can consult Kechris (\cite{Descriptive_set_theory_Kechris}, p.115).
\end{proof}
 
\section{Total Variation Norm}\label{Total_variation_norm_appendix - l.30}

\begin{defn}[signed measure]
Let $(\XX, \FF)$ be a measurable space and let $\mu$ be a function on $\FF$ with values in $(-\infty, +\infty)$.\\
If $\mu$ is countably additive and satisfies $\mu(\emptyset) = 0$, it is a finite signed measure.
\end{defn}
Note that a signed measure is a function that result if the requirement of non-negativity is removed from the definition of a measure.

\begin{theo}[Hahn decomposition theorem - l.39]
Let $(\XX, \FF)$ be a measurable space, and let $\mu$ be a signed measure on $(\XX,\FF)$. Then there are disjoint subsets $P$ and $N$ of $\XX$ such that $P$ is a positive set for $\mu$, $N$ is a negative set for $\mu$ and $\XX = P \cup N$.
\end{theo}

\begin{proof}
For a proof of the Hahn decomposition theorem, one can consult Theorem 4.1.4 in \textit{Measure Theory} by Cohn \cite{measure_theory_Cohn}.
\end{proof}

\begin{cor}\label{Jordan_decomposition}[Jordan decomposition theorem - l.46]
Every signed measure is the difference of two positive measures, at least one of which is finite. Indeed, for any measure $\mu$, one can write $\mu = \mu_+ - \mu_-$, with $\mu_+ (A) = \mu (A \cap P)$, $\mu_- (A) = - \mu (A \cap N)$ where $P$ is a positive set for $\mu$ and $N$ is a negative set for $\mu$. \\
The representation $\mu = \mu_+ - \mu_-$ is called the Jordan decomposition of $\mu$.
\end{cor}

\begin{proof}
For a proof of the Jordan decomposition theorem, one can consult Corollary 4.1.5 in \textit{Measure Theory} by Cohn\cite{measure_theory_Cohn}.
\end{proof}

The Definition \ref{total_variation_norm_def} given below is given in \textit{Measure Theory} by Cohn. 

\begin{defn}[Total Variation norm - l.59]\label{total_variation_norm_def}
Let $(\XX, \BB)$ be a measurable space, and let $\mu$ be a signed measure on $\XX$. \\
The variation of the signed measure $\mu$ is the positive measure $|\mu|$ defined by \\
$|\mu| = \mu_+ + \mu_-$. \\
The total variation norm $||\mu||$ of the signed measure $\mu$ is defined by \\
$||\mu|| = |\mu|(\XX)$.
\end{defn}

\begin{rmk} As written in David Pollard's manuscript \cite{Pollard_total_variation}, the total variation norm can equivalently be defined as $||\mu|| = \sup(|\mu|(A) + |\mu|(A^{\comp}))$, where the supremum is taken over all $A \in \BB$.
\end{rmk}

\begin{egg}[Total variation norm for atomic measures]
Let $S$ be a finite subset of $\XX$. Consider the atomic signed measure $\dps \mu = \sum_{x \in S} \mu_x \delta_x$. Let us define, on $\XX$, two positive measures 
\begin{equation*}
\mu_+ = \sum_{x \in S_+} \mu_x \delta_x \quad \mbox{and} \quad \mu_- = \sum_{x \in S_-} (-\mu_x) \delta_x,
\end{equation*}
where $S_+ = \{ x \in S; \, \mu_x \geq 0\}$ and $S_- = \{ x \in S; \, \mu_x \leq 0\}$. Then,
 \begin{equation*}
|\mu| = \sum_{x \in S} |\mu_x| \delta_x \quad \mbox{and} \quad ||\mu|| = |\mu|(\XX) = \sum_{x \in S} |\mu_x|.
\end{equation*}
\end{egg}

\begin{prop}\label{prop_of_tv_norm}[l.81]
Let $(\XX, \BB)$ be a measurable space, and let $||\mu||$ be the total variation norm of a signed measure $\mu$. Then,
\begin{enumerate}[topsep = 4mm, itemsep = 5mm, parsep = 2mm]
\item[(i)] $\dps ||\mu|| = \int_{\XX} |m| \, \diff \lambda$, where $\dps m = \frac{\diff \mu}{\diff \lambda}$ for some dominating measure $\lambda$, 
\item[(ii)] $ \dps ||\mu|| = \frac1k \sup \left\{ \left| \int f \diff \mu \right|; \, f: \XX \rightarrow [-k,k], \, f \mbox{ measurable} \right\}$. \par
\noindent The supremum over all measurable $f$, $|f| \leq k$, is attained with the function $g$ defined by $g = k\indicator_M - k\indicator_{M^{\comp}}$, where $M = \{x \in \XX; \,m \geq 0\}$.
\end{enumerate}
\end{prop}

\begin{proof}
\begin{enumerate}[topsep = 7mm]
\item[(i)] Since $\mu$ is absolutely continuous with respect to $\lambda$, $\mu_+$ and $\mu_-$ are also absolutely continuous with respect to $\lambda$. It is well known that the Radon-Nikodym derivative for $\mu_+$ and $\mu_-$ are, respectively, the positive and negative parts of $m$, denoted by $m_+$ and $m_-$. Hence, 
\begin{equation*}
|\mu| = \int m_+ \diff \lambda + \int m_- \diff \lambda =  \int |m| \diff \lambda \, \mbox{ and thus }  \, ||\mu|| = \int_{\XX} |m| \diff \lambda.
\end{equation*}
\item[(ii)] Since $|f| \leq k$, we have
\begin{equation*}
\left| \int f \diff \mu \right| = \left| \int f m \diff \lambda \right| \leq \int |f| |m| \diff \lambda \leq k \int |m| \diff \lambda.
\end{equation*}
Hence $ \dps\sup \left\{ \left| \int f \diff \mu \right|; \, f: \XX \rightarrow [-1,1], \, f \mbox{ measurable} \right\} \leq k ||\mu||$.
\end{enumerate}
\noindent Let $M = \{x \in \XX; \,m \geq 0\}$ and define the function $f$ by $f = k\indicator_M - k\indicator_{M^{\comp}}$. Then,
\begin{equation*}
\int fm \diff \lambda = \int km_+\diff \lambda - \int (-km_-) \diff \lambda = \int km_+ + km_- \diff \lambda = k \int |m| \diff \lambda.
\end{equation*}
Therefore, $\dps ||\mu|| = \frac1k \sup \left\{ \left| \int f \diff \mu \right|; \, f: \XX \rightarrow [-k,k], \, f \mbox{ measurable} \right\}$.
\end{proof}

\begin{prop}\label{total_variation_norm_when_measure_of_X_is_zero}[l.110]
Let $(\XX, \BB)$ be a Borel measurable space, $\mu$ be a signed measure on $(\XX,\BB)$ such that $\mu(\XX) = 0$ and $\mu_+$, respectively $||\mu||$, be the positive part, respectively the total variation norm, of $\mu$. Then, 
\begin{enumerate}[topsep = 4mm, itemsep = 4mm]
\item[(i)] $\dps ||\mu|| = 2 \int_{\XX} m_+ \diff \lambda$,
\item[(ii)] $||\mu|| = 2 \sup_A \mu(A)$.
\end{enumerate}
\end{prop}

\begin{proof}
\begin{enumerate}
\item[(i)] Since $\mu(\XX) = 0$, $\dps \int_{\XX} m_+ \diff \lambda = \int_{\XX} m_- \diff \lambda$. Hence, 
\begin{equation*}
||\mu|| = \int_{\XX} |m| \diff \lambda = 2  \int_{\XX} m_+ \diff \lambda = 2 \int_{\XX} m_- \diff \lambda.
\end{equation*}
\item[(ii)] Since $\mu(\XX) = 0$, $\mu_+(\XX) = \mu_-(\XX)$. Hence, $||\mu|| = 2 \mu_+(\XX)$. We thus need to show that $\mu_+(\XX) = \sup \mu(A)$ where the supremum is taken over $A \in \BB$. \\
By the Hahn decomposition, we can write
\begin{equation*}
\mu(A) = \mu((A \cap P) \cup (A \cap N)) = \mu(A \cap P) - \mu(A \cap N).
\end{equation*}
Since $\mu(A \cap N)$ is a non-negative measure and $P$ is a subset of $\XX$, we have 
\begin{equation*}
\mu(A) \leq \mu(A \cap P) \leq \mu(P) = \mu_+(\XX).
\end{equation*}
As the supremum is the smallest upper bound, we have $\sup \mu(A) \leq \mu_+(\XX)$.\\
Now, suppose that $\sup \mu(A) < \mu_+(\XX)$ and let $\epsilon = \mu_+(\XX) - \sup \mu(A)$. Since the measure $\mu$ is regular, for the positive set $P \subset \XX$, there exist $A_{\epsilon/2}$ such that $\mu(P) - \mu(K_{\epsilon/2}) \leq \epsilon/2$. But $\mu(P) = \mu_+(\XX)$ hence we have $\mu_+(\XX) - \mu(A_{\epsilon/2}) \leq \epsilon/2$. Thus $\sup \mu(A) < \mu(A_{\epsilon/2})$, which contradicts our assumption of strict inferiority.
\end{enumerate}
\end{proof}

\begin{prop}\label{equality_tv_sum_over_support_of_measures}[probably not useful - l.141]
Consider $(\XX, \BB)$, a Borel probability space. Let $\mu_1$, $\mu_2$ be two atomic probability measures on $\XX$. Then,
\begin{equation}\label{equality_tv_sum_over_X}
\sup_{A} |\mu_1(A) - \mu_2(A)| =  \sum_{x \in \XX} |\mu_1(x) - \mu_2(x)|.
\end{equation}
where the supremum is taken over all $A \in \BB$.
\end{prop}

Note that the support of an atomic probability measure on $\XX$ contains at most a countable number of elements. Hence we can rewrite equation \ref{equality_tv_sum_over_X} as 
\begin{equation*}
\sup_{A} |\mu_1(A) - \mu_2(A)| = \frac12 \sum_{i =1}^n |\mu_1(x_i) - \mu_2(x_i)|.
\end{equation*}

\begin{proof}[Proposition \ref{equality_tv_sum_over_support_of_measures} - l.155]
Let $A \in \BB$ and define $B=\{ x \in \XX : \, \mu_1(x) \geq \mu_2(x) \}$. Then
\begin{align*}
\mu_1(A) - \mu_2(A) &= \big( \mu_1(A \cap B) + \mu_1(A \cap B^{\comp}) \big) - \big( \mu_2(A \cap B) + \mu_2(A \cap B^{\comp}) \big) \\
&= \big( \mu_1(A \cap B) - \mu_2(A \cap B) \big) + \big( \mu_1(A \cap B^{\comp}) - \mu_2(A \cap B^{\comp}) \big) 
\end{align*}

But for any $x \in B^{\comp}$, we have $\mu_1(x)-\mu_2(x) < 0$. Hence,
\begin{align*}
\mu_1(A) - \mu_2(A) &\leq \mu_1(A \cap B) - \mu_2(A \cap B) \\
& \leq \big( \mu_1(A \cap B) - \mu_2(A \cap B) \big) + \big( \mu_1(A^{\comp} \cap B) - \mu_2(A^{\comp} \cap B) \big) \\
&= \mu_1(B) - \mu_2(B).
\end{align*}

Likewise, using the fact that, for any $x \in B$ we have $\mu_2(x)-\mu_1(x) \leq 0$, we obtain
\begin{equation*}
\mu_2(A) - \mu_1(A) \leq \mu_2(A \cap B^{\comp}) - \mu_1(A \cap B^{\comp}) \leq \mu_2(B^{\comp}) - \mu_1(B^{\comp}).
\end{equation*}

Now, if $\mu_1(A) - \mu_2(A) \geq 0$, we have $|\mu_1(A) - \mu_2(A)| \leq \mu_1(B) - \mu_2(B)$. If $\mu_1(A) - \mu_2(A) < 0$, we have $|\mu_1(A) - \mu_2(A)| \leq \mu_2(B^{\comp}) - \mu_1(B^{\comp})$. Hence we obtain
\begin{align*}
|\mu_1(A) - \mu_2(A)| &\leq \frac12  \big(  \mu_1(B) - \mu_2(B) + \mu_2(B^{\comp}) - \mu_1(B^{\comp}) \big) \\[3mm]
|\mu_1(A) - \mu_2(A)| &\leq \frac12  \Big( \sum_{x \in B}  \mu_1(x) - \mu_2(x) + \sum_{x \in B^{\comp}} \mu_2(x) - \mu_1(x) \Big) \\[3mm]
|\mu_1(A) - \mu_2(A)| &\leq \frac12  \Big( \sum_{x \in B}  |\mu_1(x) - \mu_2(x)| + \sum_{x \in B^{\comp}} |\mu_1(x) - \mu_2(x)| \Big) \\[3mm]
|\mu_1(A) - \mu_2(A)| &\leq \frac12 \sum_{x \in \XX}  |\mu_1(x) - \mu_2(x)|.
\end{align*}

As the supremum over all $A \in \BB$ is the smallest upper bound, we have
\begin{equation*}
||\mu_1-\mu_2||_{TV} \leq \sum_{x \in \XX}  |\mu_1(x) - \mu_2(x)|.
\end{equation*}

Now, if we choose $A$ to be equal to $B$, we have the equality
\begin{equation*}
|\mu_1(A) - \mu_2(A)| = \frac12  \big(  \mu_1(A) - \mu_2(A) + \mu_2(A^{\comp}) - \mu_1(A^{\comp}) \big). 
\end{equation*}

Thus, for $A=B$, the upper bound is attained, i.e. 
\begin{equation*}
|\mu_1(A) - \mu_2(A)| = \frac12 \sum_{x \in \XX}  |\mu_1(x) - \mu_2(x)|.
\end{equation*}

Therefore, 
\begin{equation*}
\sup_{A} |\mu_1(A) - \mu_2(A)| = \sum_{x \in \XX}  |\mu_1(x) - \mu_2(x)|.
\end{equation*}
\end{proof}

\cleardoublepage

\chapter{A Quick Overview of Feature Selection Algorithms}\label{Feature_selection_algo_appndix}

This Appendix introduces concepts and algorithms of feature selection, and groups different algorithms with a categorisation based on search strategies and evaluation criteria. The information and the text presented relies heavily on the work of Liu and Yu \cite{Liu_Yu_Integrating_ftre_selection_algo_for_classification}, Jovic \textit{et al.} \cite{feature_selection_review_croatia} and Liu \textit{et al.} \cite{Liu_et_al_Feature_selec_ever_evolving}. \\
Feature selection algorithm is a process that selects a subset of an original set of features. Typically, a feature selection process consists of four basic steps: subset generation, subset evaluation, stopping criterion and result validation. In this Appendix, we only consider the first three steps as the fourth one is
 
\subsection{Subset Generation}
 
To generate a subset, one has to build a search algorithm which specifies a candidate subset for evaluation at each branching step. The nature of a search engine is determined by two choices: the choice of search strategy and the choice of search starting point (or points). Note that the search starting point influences the search direction.
 
 There are four search starting points that imply four search directions: one may start with an empty set (and then add features); one may start with the complete set (and then remove features); one may start from both sides (and simultaneously consider smaller and larger feature subsets; or one may start with a subset generated by a heuristic function. \\
Regarding the search strategies, one can categorize them in three groups: complete search, sequential search and randomized search. 
 
\noindent A complete search guarantees to find the optimal result with respect to the chosen evaluation criterion. The Exhaustive search is complete but is computationally prohibitive as a set of $r$ elements has $C(r,k)$ subsets of $r$ elements. Using appropriate feature selection evaluation functions (as defined in Chapter \ref{KR_score_appplied_to_gwas}) one can reduce the search space without jeopardizing the completeness of the search engine. 
 
A Sequential search is essentially trading its optimality for computational efficiency. Indeed, sequential search algorithm are usually $O(n^2)$ or less. A sequential search adds (respectively, removes) $p$ features in one step and removes (respectively, adds) $q$ features in the next step (with $p>q$). These search processes are suboptimal due to the fact that the best pair of features need not contain the best single feature \cite{pattern_recognition_review}. In general, good, larger feature sets do not necessarily include good small sets. 
 
Randomized search methods incorporate randomness into their search procedure in order to avoid being trapped in local optima of the search space. The randomness is introduce in the search engine by choosing the first, many, or all subset(s) randomly.
 
 \subsection{Subset Evaluation}
 
Each subset generated by the search algorithm needs to be evaluated by a feature selection evaluation function, also known as a feature selection criterion function. Clearly, the notion of ``goodness'' of a subset is dependent on the choice of the evaluation function: one optimal subset selected using one evaluation function may not be optimal with another evaluation function. We note that the mathematical ``requirements'' for evaluation functions are quite limited. Thus, in practice, the real difficulty is to construct a feature selection evaluation function that represents the intrinsic characteristics of interest in our subset.

\subsection{Stopping Criteria}\label{stopping_criteria_subsection}

A stopping criteria determines when the feature selection process should be terminated. According to Liu and Yu \cite{Liu_Yu_Integrating_ftre_selection_algo_for_classification}, some frequently used stopping criteria are:
\begin{enumerate}
\item The search completes
\item Some given bound is reached, where a bound can be
a specified number (minimum number of features or
maximum number of iterations).
\item Subsequent addition (or deletion) of any feature
does not produce a better subset.
\item A sufficiently good subset is selected (e.g., a subset
may be sufficiently good if its classification error rate is less than the allowable error rate for a given task).
\end{enumerate}

\subsection{Result Validation}
In real-world applications, the result validation step consists in monitoring the change of classification performance with respect to the subset of features obtained after the feature selection process has stopped (that is, after the third step \ref{stopping_criteria_subsection}).
Note that result validation is useful only in the case of independent evaluation functions. In the case of dependent evaluation functions, result validation is redundant.

\cleardoublepage

\chapter{Biological and Technical Background} \label{Biological_and_Technical_Background_appendix}

The purpose of Appendix \ref{Biological_and_Technical_Background_appendix} is to introduce the biological and technical concepts which are necessary to understand how a GWAS dataset is constructed. The appendix is organized as follows. The first section gives a short summary on the biology of coronary artery disease. The second and third sections introduce the concepts of DNA and SNPs from genetics, followed by a brief definition of a GWAS in section 4. Section 5 explains why the structure of the dataset OHGS-2 is as such. It is the most important section as one of the central ideas in this thesis rely on the particular structure of OHGS-2.

\section{Coronary Artery Disease}

The cardiovascular system (CVS), composed of the heart, blood, and blood vessels is an organ system that allows blood to circulate through the body. In general, oxygenated blood travels in arteries whereas deoxygenated blood flows in veins. In addition to delivery of  oxygen to cells, blood is essential for nutrient delivery, metabolic waste collection, acid-base status and circulation other molecules such as immune cells and hormones. Blood is pumped through the body by the heart which gets its vascularization through coronary arteries. 

Coronary artery disease (CAD) is defined as the stenosis (narrowing) or blocking of coronary arteries. The pathophysiology of CAD is linked to atherosclerosis in coronary arteries. Atherosclerosis is defined as the accumulation of atherosclerotic plaques within the walls of arteries. The earliest but asymptomatic manifestation of these plaques, fatty streaks, can be seen by age 20. Following injury to the endothelium (vessel inner lining), an inflammatory response is triggered leading to accumulation of low density lipoproteins (a type of cholesterol) and leukocytes (white blood cells) within the vessels walls. These leukocytes eventually transform into foam cells which are prominent inflammatory participants. Over time, as these fatty streaks progress to plaques, the lipid (fat) pools are covered by a protective layer of smooth muscle cells. At a given time $t$, these plaques can be either ``stable" or ``vulnerable" when the lipid pool is large, the fibrous cap is thin and abundant inflammatory cells are present. Even though, over a patient's life, most plaques remain asymptomatic, complications may occur. These complications can lead to different clinical pictures depending on the type of plaque and their location. As this thesis focuses on CAD, we restrict our interest to atherosclerotic plaques within the coronary arteries. In the coronary arteries, stable plaques that narrow the vessel lumen lead to intermittent shortness of breath and chest pain on exertion (stable angina pectoris). Vulnerable plaques are subject to rupture, leading to acute thrombus (clot) formation. The magnitude of the thrombotic response that follows a plaque's rupture is extremely variable. It only occasionally evolves into a life-threatening luminal clot \cite{Bentzon_Mech_plaque_formation_rupture}. If immediate reperfusion is not achieved, major luminal thrombus can lead to the death of the affected tissues and severe impairment of heart function. This process is the myocardial infarction (MI), also known as a heart attack. MI can lead to cardiac arrest and death.

\section{DNA and SNPs}

Deoxyribonucleic acid (DNA) is the molecule that contains the genetic information unique to each specie. A nucleotide is the building block of DNA. It is composed of a sugar group called deoxyribose, a phosphate group and a base, the varying unit. Four types of bases exist: Cytosine ($C$), Guanine ($G$), Adenosine ($A$) and Thymine ($T$). A DNA strand is formed by the alignment of nucleotides via covalent bounds between the sugar group of one nucleotide and the phosphate group of the following nucleotide, leading to an alternating sugar-phosphate backbone. The DNA double helix is formed by two strands of DNA coiled around each other. These two strands are held together by (hydrogen) bonds between bases. Adenine always forms a base pair with Thymine and Cytosine always binds to Guanine. 

In humans, the majority of the 3.08 billion base pairs of DNA is stored in the cells' nucleus subdivided into 46 different molecules, corresponding to the 46 chromosomes. Each chromosome contains 44 to 246 million base pair, corresponding to 1.5 to 8.4 cm of DNA. Depending on the phase of the cell's cycle, DNA is compacted to different degrees to fit in a 90 um cell nucleus. In its most compact form during cell division, condensed chromosomes measure from 1.3 to 10 um. 

Most cells of a human body are called somatic cells and contains 46 chromosomes. By contrast, germ cells contain only 23 chromosomes. A cell with 46 chromosome is called a diploid cell, whereas a cell with 23 chromosomes in an haploid cell. Diploid cells contain 2 copies of the same chromosome whereas haploid cells contain only one copy. During fecundation, when a maternal (oocyte) and paternal (spermatozoid) germ cell come together (fertilization) the resulting zygote inherits 2 copies of each chromosome, one from its mother and one from its father, making a 46 chromosome cell. 

As mentioned above, DNA contains the biological information that defines a specie and within the same species, an individual. At the molecular level, a gene is a segment of DNA that codes for a specific protein (with some exceptions). Between genes, there is abundant non coding DNA, once called junk DNA. Although it does not contain information to synthesize a protein, non-coding DNA plays a key role in regulation of gene expression, cell growth, cell division etc. Between any two individuals, the genome is at least 99.5 \% identical. The genetic variations between two individuals are the consequences of germline mutations (mutations occurring in the gametes) in their respective ancestry: Embryo are derived from the fusion of an egg and a sperm hence a mutation in at least one of the parent's fused gametes is then found in each nucleated cell of their offspring. Mutations usually arise from unrepaired DNA damage, DNA replications errors or mobile genetic elements \cite{Defining_mutation_and_poly}. Geneticists distinguish between three major classes of mutations: single nucleotide substitutions, insertion/deletion of one or many (up to twenty) nucleotides and large genomic rearrangements. Note that there also exist rare types of mutation mechanism that do not belong to any of the three classes \cite{Mecanismes_consequences_mutations}. A variation in the DNA sequence that occurs in a population with a frequency of 1 \% or higher is termed a polymorphism. Polymorphsims, like mutations, can be of one or more nucleotide changes. The commonest polymorphism is the single nucleotide polymorphism (SNP). All existing variants of a SNP in a population are referred to as alleles. The allele that is observed more frequently in the population is termed the major allele, while the less frequent variant is the minor allele. Most (over 99.6\%) SNPs are bi-allelic but 271'934 tri-allelic SNPs were identified in the 1000 Genomes Phase III variation data \cite{Triallelic_SNP}. Humans, as diploids, always have a pair of alleles. If one has two copies of the major allele, he is said to be homozygous major; if he has two copies of the minor allele, he is said two be homozygous minor and having one copy of each allele makes him heterozygous. For simplicity, geneticists often label the two alleles as $A$ and $B$. Therefore, an individual's  genotype for a given gene is either $AA$, $BB$ or $AB$. 

\section{DNA Microarrays}

DNA microarrays rely on hybridization, the biochemical principe  that nucleotide bases bind to their complimentary partners (A binds to T and C binds to G). 
DNA microarrays are silicon biochips on which a collection of microscopic DNA spots are attached. Each DNA spot contains picomoles of a specific DNA sequence, called probes.

There are many different types of arrays used for different applications. To detect particular SNPs among alleles within a population, geneticists use SNP arrays. There are 2 main array manufacturers that commercially produce SNP arrays: Affymetrix and Illumina. We will concentrate on the Affymetrix technology as the data analysed in this these was produce with the Affymetrix 6.0 chip.\\
The Affymetrix 6.0 chips allow for the detection of 906,600 SNPs. Every SNP site is interrogated by a set of 6 to 8 probes; 3 or 4 replicates of the same probe for each of the 2 alleles.
For Affymetrix SNP arrays, a probe is a short 25 base pair fragment of DNA design to be complementary to the sequence harbouring a particular SNP and featuring one of the 2 alleles of that SNP. The SNP, being centered on the probe, is base pair number 13. When using an Affymetrix 6.0 chip, the acquired DNA is first purified, then amplified. The amplification is done via repeated reverse transcription. Those multiple single-strand DNA copies are then coupled with a label. The label is typically a fluorophore, a fluorescent chemical compound. A fluorophore has the property to emit light that it has just previously absorbed. Generally, the emitted light has a longer wavelength than the absorbed light. For SNP arrays, the fluorophores absorb ultra-violet radiation and emits light in the visible region of the wavelength spectrum. Now that the single strand DNA copies are labelled, they are broken down in small fragments. These fragments are then mixed with a propriety hybridization solution and the mixture is poured in the pinholes of the microarray. The pinholes are then sealed and the microarray is mixed in order for the small fragments to be uniformly distributed over the microarray probes. By the hybridization principle, labelled single-strand fragment of DNA will bind preferentially to probes that are complementary to themselves. After a defined hybridization period, the microarray is washed off and dried. This eliminates all non-binded DNA fragments and only the strongly hybridized ones remain. The microarray is then scanned by a machine that emits ultra-violet radiation and produces a measure of the signal intensity associated with each probe set.
 The underlying principle is that the signal intensity depends upon the number of targeted fragments perfectly hybridized to their complementary probe. For example, suppose a patient is homozygous AA for a given SNP. Then, the DNA fragments in the hybridization solution only contains allele A. Therefore, the fragments will only bind to the 4 probes harbouring the complementary sequence to allele A. Likewise, for a homozygous BB patient, the DNA fragment would only bind to the 4 probes harbouring the complementary sequence to allele B. For a heterozygous AB patient, DNA fragments would bind to all 8 probes.\\ 
In reality, this is an oversimplification of the facts. First, the signal intensity also depends upon the amount of targeted fragments in the hybridization mixture. Indeed, even if the perfect complimentary probes of a given target fragment exist on the microarray, the probability that they will hybridize is lower than 1 and decreases as the concentration of the targeted fragments in the  hybridization mixture decreases.Also, for the 2 different allelic version of the probes, the 24 other bases remain exactly the same. Thus, the 2 versions of the probes are identical on 96\% of their DNA sequence. Hence, the probability that a fragment with allele A will bind to a probe that contains allele B is rather high. Moreover, the hydrogen bounds between the DNA bases C and G are stronger than the bound between A and T. Therefore, a target sequence rich in CG bounds will be more likely to bind to the probe of the opposite allele. Despite these limitations, commercial SNP array can now genotype SNPs with an accuracy over 95\% \cite{LaFramboise_microarray}.
 
\section{Chiamo: a Genotype Calling Algorithm} \label{chiamo_section}

In order to yield SNP genotype inferences from the raw intensity associated with each probe set, extensive processing and analysis is required. This computational analysis is done by genotype calling algorithms. With each new version of a SNP microarray, the computational community (biostatisticians, bioinformaticians) responds with a new algorithmic development. Theses developments are driven both by the microarray manufacturers and academics. It is interesting to note that there has been a productive synergistic relationship between advances in biological understanding, computational methodology and technological advancements. Progress in each of these three areas has spurred progress in the other two. For example, the choice of SNPs to include on a microarray is influenced by the genotype calling algorithm: the SNPs chosen are the SNPs for which the current algorithm performs best \cite{LaFramboise_microarray}. 

Genotype calling algorithm are composed of two steps. The first step is the normalisation of the raw data and the second step is the calling algorithm per se. The output of the calling algorithm is, as stated above, an inference on the individual's genotype at the SNP of interest. Such an inference is known as the genotype call.
Genotype calling algorithms are applied after normalization of the raw data. Since Genotype calling algorithms examine probe intensities, it is crucial to standardize these intensities to take into account technological differences such as platform fabrication. Normalization aims to correct these biases by homogenizing the intensity distribution of each array. Currently, the consensus settles on a non-parametric method that ensures that all arrays have the identical probe intensity distribution. That method is called Quantile Normalization. 
Once the raw data has been normalized, it can be genotyped. In the case of the Affimetrix SNP array 6.0, the default genotype calling algorithm is Birdseed, developed by J. Korn \textit{et al.} \cite{birdseed}. In 2011, Robert Davies, at the UOHI, modified the normalisation step of Chiamo, a genotype calling algorithm developed for Affimetrix 500k arrays in order to apply it to SNP array 6.0 data. Robert Davies then compared the outputs of both Chiamo and Birdseed and concluded that ``Chiamo is slightly more accurate than Birdseed"  \cite{Robbie_instructions_OHI}. No further results or explanation are given for this statement.

Here is a brief overview of Robert Davies' modified version of Chiamo. The genotype calling algorithm Chiamo \cite{supp_info_chiamo} was developed in 2007 by by Peter Donnelly, Jonathan Marchini, Chris Spencer and Yik Ying Teo, all members of The Wellcome Trust Case Control Consortium at the time. As stated above, each SNP on the SNP array 6.0 is interrogated by 6 or 8 probes - 3 or 4 replicates of the same probe for each of the two alleles. Hence, intensity data for each SNP consists of two sets of 4 repeated measurements: The raw data for the $s^{th}$ SNP of the $i^{th}$ array (thus individual) can be denoted as $I_{isk} = (I_{isk}^A, I_{isk}^B)$, where $k$ represents the $k^{th}$ probe. Hence $k=1, \ldots, K$ with $k \in \{3,4\}$. First, the vectors $I_{sk}$ are normalised using the pre-processing step of quantile normalization. Second, the quantile normalized intensities  are log transformed to reduce their skewness. Let $Y_{iks} = (Y_{isk}^A, Y_{isk}^B)$ denote the log transformed intensities. The third step is the probe set summarization during which the signals are combined across all $k$ probes using an arithmetic mean to create a single pair of intensities $X_{is} = (X_{is}^A, X_{is}^B)$ for the $i^{th}$ individual at SNP $s$. That is:
\begin{equation*}
X_{is}^A = \frac{1}{n_k} \sum_{k=1}^{n_k} Y_{iks}^A \quad \mbox{and} \quad X_{is}^B = \frac{1}{n_k} \sum_{k=1}^{n_k} Y_{iks}^B.
\end{equation*}

\noindent Chiamo is applied after quantile normalization of the data from each SNP. For a given SNP $s$, it uses a two-dimensional four-stage-hierarchical Bayesian Gaussian mixture model to call genotypes. Note that both the component parameters and the mixture weights are estimated with a four-stage Bayesian hierarchical model. It is natural for Chiamo to be based on a mixture model since a mixture model is a probabilistic model for representing the presence of subpopulations within an overall population. The distribution for the set of bi-variate intensity vectors $X_{is}$ is a two-dimensional four-stage-hierarchical Bayesian Gaussian mixture distribution with four components. It is important to keep in mind that a $n$-dimensional Gaussian mixture distribution with $m$ components need not be a $m$-stage-hierarchical Bayesian distribution. The Gaussian mixture distribution is given by a linear combination of four two-dimensional normals of mean $\mu_{k,s} \in \RR^2$ and covariance $\Sigma_{k,s} \in \MM_{2 \times 2}$. That is, 
\begin{equation}
X_{is}| \theta_s \sim \sum_{k=1}^4 \lambda_{k,s}\, N_2(\mu_{k,s}, \Sigma_{k,s}),
\end{equation}
where $\theta_s$ is the vector of parameters (note that  $\lambda_{k,s}$, $\mu_k$ and $\Sigma_k$ are parameters). 

To specify the identity of the mixture component of observation $x_{is}$, one introduces a four-dimensional binary random variable $Z_{is}$ having a one-in-four representation in which a particular element $z_k$ is equal to 1 and all other elements are equal to zero. The random variable $Z_{is}$ is called a latent random variable. The values of $Z_{is}$ thus satisfy $z_k \in \{0,1\}$ and $\sum_k z_k =1$ and $Z_{is}$ follows a multinomial distribution of parameters $(1, \lambda_s)$ where $\lambda_s \in \RR^4$. With the latent variable $Z_{is}$, we obtain the following likelihood function for $X_{is}$:
\begin{equation*}
X_{is}| \theta_s, Z_{is} \sim N_2(\mu_{Z_{is},s}, \Sigma_{Z_{is},s}) \quad \mbox{equivalent to} \quad X_{is}| \theta_s, Z_{is} \sim \prod_{k=1}^4 N_2(\mu_{k,s}, \Sigma_{k,s})^{z_k}.
\end{equation*}

Donnely \textit{et al.} \cite{supp_info_chiamo} use the latent variable $Z_{is}$ to denote the genotype call for the $i^{th}$ patient of snp $s$. There is four different possible calls: one for each of the genotypes $\{AA, AB,BB\}$ as well as a null call to capture the clear outliers and add robustness to the model fit of the other three genotype components. This format allows for genotype uncertainty. Statistically, the probability of each genotype call, given the pair of intensities $X_{is}$, is the posterior distribution $P(Z_{is} | X_{is}, \hat{\theta}_s)$, where $\hat{\theta}_s$ is the maximum a posteriori estimate of $\theta_s$. The exact iterative method to obtain the MAP estimate of the parameters $\theta_s$ and to compute the posterior distribution is not specified in the paper by Donnely \textit{et al.} \cite{supp_info_chiamo} but it is very likely that the Expectation-Maximization (EM) algorithm or one of its variants was employed. The EM algorithm seems the most plausible algorithm since it explicitly computes the posterior distribution of the latent variables in its step E. Indeed, the EM algorithm is an iterative process with the following steps:
\setlist{nolistsep}
\begin{itemize}[noitemsep]
\item{E step:} Evaluate the posterior distribution of the latent variables $P(Z_{is} | \theta_{n-1}, X_{is})$, where $\theta_{n-1}$ are the parameter values obtain at level $n-1$.
\item{M step:} Determine the revised parameter estimate $\theta_n$ given by $\theta_n = \argmax_{\theta} \mathcal{Q}(\theta_{n-1}, \theta_{n})$,
where $\mathcal{Q}$ represents the expectation of the complete-data log likelihood evaluated with respected to $\theta$:
\begin{equation*}
\mathcal{Q}(\theta,\theta_{n-1}) = \sum_i P(Z_{is} | X_{is}, \theta_{n-1}) \ln P(X_{is},Z_{is} | \theta).
\end{equation*}
\end{itemize}
For an in-depth coverage of the EM algorithm (as well as Gaussian mixture models) refer to chapter 9 in \textit{Pattern Recognition and Machine Learning} \cite{Bishop}.\\
A short-coming of the EM algorithm is that the sequence of complete-data log likelihoods may converge to a local maximum depending on the initial parameter values. To avoid that problem, Donnely \textit{et al.} ran the EM algorithm twelve times with twelve different random starts. That is, they ran the algorithm with twelve difference set of initial parameter values $\theta_{\circ}$. \\

A visual representation of the data is obtained by plotting each pair of normalized summary probe intensities $(X_{is}^A, X_{is}^B)$ in the plane. If the given SNP has been genotyped well, one should see three clear, distinct, clusters on the plot that would correspond to the three genotypes. The next step is to colour each point to indicate how the genotype calling algorithm Chiamo classifies that individual (either as a homozygote for one of the two alleles, a heterozygote, or a null (missing) call). Such coloured scatter plot in the plane is called cluster plot. A cluster plot is therefore a graphical representation of the results of both the genotyping and genotype calling of a SNP, with each point representing one individual. See Figure 3 and Supplementary Figure 1 in \cite{birdseed}, for example of cluster plots. \\

It is important to keep in mind that the output of Chiamo is not a genotype call per se but the probability of each genotypes. Hence, for a given SNP $s$, the output for patient $i$ is a vector of $\RR^4$ where the $k^{th}$ component is $P(Z_{is}=z_k | X_{is}, \hat{\theta})$ with $Z_{is}=z_k$ being an abusive notation for the vector with the $k^{th}$ component equal to 1 and all others zero. Since $Z_{is}$ follows a 4-multinomial distribution and thus $\sum_k P(Z_{is} = z_k) = 1$, the probability $P(Z_{is} = z_4 | X_{is}, \hat{\theta})$ of a null call is not explicitly given in the output. The output of a complete genotype file from Chiamo has a one-line-per-snp format: The first 5 entries of each line should be the SNP id, rs id of the SNP, base-pair position of the SNP, the allele coded A and the allele coded B. The SNP id can be used to denote the chromosome number of each SNP. The next three numbers on the line should be the probabilities of the three genotypes AA, AB and BB at the SNP for the first individual in the cohort. The next three numbers should be the genotype probabilities for the second individual in the cohort, and so on. 
Therefore, from a mathematical stand point, the output of a complete genotype file is a matrix of dimension $n \times 3m$ where $n$ is the number of SNPs and $m$ is the number of patients. Thus, for a row $i$, each triple $(a_{i,3k}, a_{i,3k+1}, a_{i,3k+2})$ represents the probabilities of patient $k$ having genotypes $AA, AB$ and $BB$ for SNP $i$. Note that, for a patient $k$, the sum $a_{i,3k} + a_{i,3k+1} + a_{i,3k+2}$ is not necessarily equal to 1 since the null call is not given in the output.

\section{Quality Control}\label{QC_section}

Quality Control (QC) in the case of GWAS, is the process of ensuring that the data obtained by a genotype calling algorithm is of acceptable quality. QC removes results that are likely to be inaccurate.  It is performed before any analysis on the data since errors in genotype calling have the potential to introduce ``systematic biases into genetic case-control association studies, leading to an increase in the number of false positive associations" \cite{Anderson_QC_structure}.

\noindent QC is a multi-layered process that removes patients or SNPs. It is important to note that the ``layering" is not commutative: the order of the steps affects the end result. The different steps can be grouped in two categories: the ``per-patient" QC and the ``per-SNP" QC. In their paper \textit{Data Quality Control in Genetic Case-Control Association Studies}, Anderson \textit{et al.} lay out a very clear protocol that details all the steps typically carried out during QC for a GWAS dataset. In the context of this thesis, it is important to remember that ``step 0" of QC categorizes, for each patient $k$ and SNP $i$, the triple $(a_{i,3k}, a_{i,3k+1}, a_{i,3k+2})$ as either ``called" or ``missing" (or ``uncalled"). For a fixed threshold $\alpha \in [0,1]$, a triple $(a_{i,3k}, a_{i,3k+1}, a_{i,3k+2})$ is considered called if $a_{i,3k} + a_{i,3k+1} + a_{i,3k+2} \geq \alpha$ and missing if $a_{i,3k} + a_{i,3k+1} + a_{i,3k+2} < \alpha$. In the landmark paper \textit{Genomewide Association Study of 14000 cases of seven common diseases and 3000 shared controls} \cite{supp_info_chiamo}, the threshold $\alpha$ has a value of $0.9$, following an ``analysis of the relationship between concordance and missing data rates". The data for this analysis is not shown. After their QC, the called rate for their data was $99.63\%$. Hence, the vast majority of the triples were called. Uncalled triples were removed from the dataset.

\section{Information on OHGS Dataset}

Genome Wide Association Studies (GWAS) are defined as observational studies of a genome wide set of genetic variants in individuals to investigate if any variant is statistically more prevalent for a given trait. The power of SNP arrays to interrogate a significant number of SNPs both rapidly and cheaply has made SNPs the preferred genetic variants for GWAS. The most common design for GWAS is the case-control setup in which a chosen trait is used to split a large sample of individuals in two groups: the case group ``affected" by the trait and the control group, without the trait. For each SNP, one then investigates if the allele frequency, or the frequency of a combination of allele, is significantly different between the case and control group in order to detect evidence of association with the trait. The statistical methodology most commonly used for testing genetic association of case-control data is covered in \textit{A Tutorial on statistical methods for population association studies} \cite{tutorial_on_stat_method_for_gwas}.

The Ottawa Heart Genomics Study (OHGS) is a GWAS based out of the University of Ottawa Heart Institute (UOHI). The OHGS was divided in two separate studies:  OHGS-1 and OHGS-2. For the purpose of this thesis we will focus on OGHS-2. The OHGS-2 is of case-control design with 1929 cases and 1978 controls.  Subjects were recruited from the University of Ottawa Heart Institute (UOHI) lipid clinic, catheterization laboratory, or from the Cleveland Clinic. The inclusions criteria for cases where given in R. Davies Master's thesis \cite{Robbie_thesis}:

``Inclusion criteria for cases was set as having had either a myocardial infarction, coronary revascularization (coronary angioplasty/percutaneous coronary intervention or coronary artery bypass graft) or had angiography or computed tomography angiography demonstrating stenosis of at least 50\% in at least one coronary artery. Age limits for cases were originally set at  $\leq 55$ for men and $\leq 65$ for women; however, several cases were included which did not meet this criteria: 22 men aged $\geq 56$, 4 women aged $\geq 66$, 6 men of unknown age at onset, 4 women of unknown age at onset. Cases were  excluded if they had diabetes mellitus or overt hyperlipidemia and if they had nonEuropean ancestry. Two sets of inclusion criteria were set up for controls, due to different acquisition protocols. One set of controls were recruited with inclusion criteria set as being healthy and to have a lack of cardiovascular disease history. For the other set of controls, recruited through the catheterization lab at the UOHI, inclusion criteria was set as having an angiogram which showed that none of the coronary arteries had a stenosis encompassing greater than 50\% of the vessel. Age cutoffs for controls were originally set as men aged 65 and older, women 70 and older; however, several controls were included which did not meet this criteria: 5 men aged $\leq 64$ and 30 women aged $\leq 69$.''\\

For both cases and control, the genetic profile of each individual was obtained with Affimetrix 6.0 arrays. Although 6.0 arrays genotype 946,000 SNPs, the .CEL files containing the raw dataset only contains 894,240 SNPs, from chromosomes 1 to 22. Note that SNPs from the two sex chromosomes X and Y are not considered. The SNP count for each chromosome is given in the table below.

The raw dataset was then genotyped by R.W. Davies using his own modification of the genotype calling algorithm Chiamo.
The genotype files were saved under OHGS\_B2\_i\_6.0.gen (where $ i = 1,\ldots, 22 $) and stored on the OHI\_GA\_DRIG, a blade server physically residing within the department of IT at UOHI.

Recall that (see section \ref{chiamo_section}), for a given SNP, Chiamo outputs a patient's genotype as a triple representing the probability of each genotype. Consequently, the dataset is stored in two matrices. The matrix for the cases is of dimension $1929 \times 2682720$. The matrix for the controls is of dimension $1978\times 2682720$ (where 2682720 = 3 $\times$ 894240). For a given row $i$, each triple $(a_{i,3k}, a_{i,3k+1}, a_{i,3k+2})$, for $k = 1, \ldots, 894240$, represents the probabilities of patient $i$ having genotype {\em homozygous major}, {\em heterozygous}, and {\em homozygous minor} for SNP $k$. Since there are only 3 possible genotypes, $a_{i,3k}, a_{i,3k+1}, a_{i,3k+2} \in [0,1]$ and $a_{i,3k} + a_{i,3k+1} +a_{i,3k+2} = 1$. 

\begin{table}
\begin{center}
\begin{tabular}{|c||cccccccc|}
\hline
Chromosome & 1 & 2 & 3 & 4 & 5 & 6 & 7 & 8 \\
\hline
\# of SNPs & 73571 & 75918 & 62268 & 57582 & 57971& 57687& 48380 & 50026\\
\hline \hline
Chromosome & 9 & 10 & 11 & 12 & 13 & 14 & 15 & 16\\
\hline
\# of SNPs  &42785 &49600 & 45927 & 43802 & 34979 & 28936 & 26907 & 28552\\
\hline \hline
Chromosome & 17 & 18 & 19 & 20 & 21 & 22 & & \\
\hline
\# of SNPs & 21319 & 27212 & 12422 & 23488 &12924 & 11984 & &  \\
\hline
\end{tabular}
\end{center}
\caption{Information on the number of SNPs, from each of Chromosomes 1 to 22, in the OHGS-2 dataset.}
\label{table_summary_snps_per_chromosome}
\end{table}

\cleardoublepage

\chapter{A geometrical perspective of the Kantorovich-Rubinstein Distance}\label{KR_distance_geometric_construction}
In this thesis, we have often use the fact that, for a finite optimal function $f:\XX \rightarrow [0,\Delta]$, the Kantorovich-Rubinstein distance $W_{\XX}(\mu_1,\mu_2)$ can be expressed as the area between the distribution functions $F_1$ and $F_2$ (where $F_i(s) = \mu_i (\{x\in f(x): f(x)\leq s\})$. In this appendix, using basic properties of product measures, we construct the geometric regions of the $[0,\Delta] \times [0,1]$ rectangle which, added together, give the region between the distribution functions $F_1$ and $F_2$.

\vspace{4mm}
\noindent Let $\gamma \in \RR$ and $E_{\gamma} = \left\{x \in \XX; \, f(x) \geq \gamma \right\}$. Then we have: 
\begin{align*}
W_{\XX}(\mu_1,\mu_2) &= \int_{E_{\gamma}} f \diff(\mu_1 - \mu_2) + \int_{E_{\gamma}^{\comp}} f \diff(\mu_1 - \mu_2) \\
&= \int_{E_{\gamma}} f \diff \mu_1 - \int_{E_{\gamma}} f \diff \mu_2 \, + \,  \int_{E_{\gamma}^{\comp}} f \diff\mu_1  - \int_{E_{\gamma}^{\comp}} f \diff\mu_2.
\end{align*}
Now, suppose that $f:\XX \rightarrow [0,\Delta]$ and $\gamma \in [0, \Delta]$. Thus, $\dps W_{\XX}(\mu_1,\mu_2) =  \int_0^{\Delta} f \diff(\mu_2 - \mu_1)$. \\
We first consider $\dps \int_{E_{\gamma}} f \diff \mu_i$.\\
Using properties of product measures (see p.162, \cite{measure_theory_Cohn}):
\begin{equation*}
\int_{E_{\gamma}} f \diff \mu_i = \int _0^{\gamma} \mu_i \left( \{x \in E_{\gamma} ;\, f(x) > y\} \right) \diff y
+ \int _{\gamma}^{\Delta} \mu_i \left( \{x \in E_{\gamma} ;\, f(x) > y\} \right) \diff y.
\end{equation*}
Consider the first term: $\dps \int _0^{\gamma} \mu_i \left( \{x \in E_{\gamma} ;\, f(x) > y\} \right) \diff y$. The variable $y$ is such that $y \leq \gamma$. Then:
\begin{align*}
\{x \in E_{\gamma} ;\, f(x)>y\} &= \{x \in [0,\Delta] ;\, f(x)>y \land f(x)> \gamma \} \\
&= \{x \in [0,\Delta] ;\, f(x)>\gamma \}.
\end{align*}
Hence, $\dps \int _0^{\gamma} \mu_i \left( \{x \in E_{\gamma} ;\, f(x) > y\} \right) \diff y = \int_0^{\gamma} 1 - F_i(\gamma) \diff y$. Geometrically, it corresponds to the rectangle with a length between $F_i(\gamma)$ and 1, and a width between 0 and $\gamma$. \\

Consider the second term: $\dps \int _{\gamma}^{\Delta} \mu_i \left( \{x \in E_{\gamma} ;\, f(x) > y\} \right) \diff y$. The variable $y$ is such that $y >\gamma$. Then:
\begin{align*}
\{x \in E_{\gamma} ;\, f(x)>y\} &= \{x \in [0,\Delta] ;\, f(x)>y \land f(x)> \gamma \} \\
&= \{x \in [0,\Delta] ;\, f(x)>y \}.
\end{align*}
Hence, $\dps \int _{\gamma}^{\Delta} \mu_i \left( \{x \in E_{\gamma} ;\, f(x) > y\} \right) \diff y = \int_{\gamma}^{\Delta} 1 - F_i(y) \diff y$. Geometrically, it corresponds to the area between the curves $F_i$ and 1 on $[\gamma,\Delta]$. \\

We now consider $\dps \int_{E_{\gamma}^{\comp}} f \diff \mu_i$.\\
Using properties of product measures (see p.162 of \cite{measure_theory_Cohn})
\begin{equation*}
\int_{E_{\gamma}^{\comp}} f \diff \mu_i = \int _0^{\gamma} \mu_i \left( \{x \in E_{\gamma}^{\comp} ;\, f(x) > y\} \right) \diff y
+ \int _{\gamma}^{\Delta} \mu_i \left( \{x \in E_{\gamma}^{\comp} ;\, f(x) > y\} \right) \diff y.
\end{equation*}
Consider the first term: $\dps \int _0^{\gamma} \mu_i \left( \{x \in E^{\comp}_{\gamma} ;\, f(x) > y\} \right) \diff y$. The variable $y$ is such that $y \leq \gamma$. Then: 
 \begin{align*}
\{x \in E^{\comp}_{\gamma} ;\, f(x)>y\} &= \{x \in [0,\Delta] ;\, f(x)\leq \gamma \land f(x)> y \} \\
&= \{x \in [0,\Delta] ;\, y <f(x)\leq\gamma \}.
\end{align*}
Hence, $\dps\int _0^{\gamma} \mu_i \left( \{x \in E^{\comp}_{\gamma} ;\, f(x) > y\} \right) \diff y = \int _0^{\gamma} F_i(\gamma) - F_i(y) \diff y$. Geometrically, it corresponds to the area between the curves $F_i$ and the horizontal line of equation $F_i(\gamma)$ on $[0,\gamma]$. \\

Consider the second term: $\dps \int _{\gamma}^{\Delta} \mu_i \left( \{x \in E_{\gamma}^{\comp} ;\, f(x) > y\} \right) \diff y$. The variable $y$ is such that $y > \gamma$. Then
\begin{align*}
\{x \in E_{\gamma}^{\comp} ;\, f(x)>y\} &= \{x \in [0,\Delta] ;\, f(x)\leq\gamma \land f(x)> y \mbox{ with } y\geq \gamma \} \\
&= \{x \in [0,\Delta] ;\, f(x)\leq \gamma < y < f(x) \} = \emptyset.
\end{align*}
Hence, $\dps \int _{\gamma}^{\Delta} \mu_i \left( \{x \in E_{\gamma}^{\comp} ;\, f(x) > y\} \right) \diff y = 0$.\\

\noindent Using the fact that 
\begin{equation*}
W_{\XX}(\mu_1,\mu_2) = \int_{E_{\gamma}} f \diff \mu_1 + \int_{E_{\gamma}^{\comp}} f \diff\mu_1 \, - \, \left( \int_{E_{\gamma}} f \diff \mu_2 +  \int_{E_{\gamma}^{\comp}} f \diff\mu_2 \right),
\end{equation*}

We can write the K-R distance:
\begin{align}\label{equation_geometric_version_of_KR_distance}
W_{\XX}(\mu_1,\mu_2) &= \left(\int_0^{\gamma} 1 -F_1(\gamma) \diff y + \int_{\gamma}^{\Delta} 1 - F_1(y) \diff y \right) + \int_0^{\gamma} F_1(\gamma) - F_1(y) \diff y \nonumber \\
&-  \Biggl(\left(\int_0^{\gamma} 1 -F_2(\gamma) \diff y + \int_{\gamma}^{\Delta} 1 - F_2(y) \diff y \right) + \int_0^{\gamma} F_2(\gamma) - F_2(y) \diff y \Biggl).
\end{align}

\noindent Now, using the fact that $\mu_i(\{x \in \XX; \, f(x) > y\}) = 1 - \mu_i(\{x \in \XX; \, f(x) \leq y\})$, we have 

\begin{align*}
\int_{E_{\gamma}} f \diff \mu_i &= \int _0^{\Delta} \mu_i \left( \{x \in E_{\gamma} ;\, f(x) > y\} \right) \diff y \\
&= \Delta - \int _0^{\Delta} \mu_i \left( \{x \in E_{\gamma} ;\, f(x) \leq y\} \right) \diff y\\
&= \Delta - \left( \int _0^{\gamma} \mu_i \left( \{x \in E_{\gamma} ;\, f(x) \leq y\} \right) \diff y + \int _{\gamma}^{\Delta} \mu_i \left( \{x \in E_{\gamma} ;\, f(x) \leq y\} \right) \diff y \right)
\end{align*}

Consider the first term: $\dps \int _0^{\gamma} \mu_i \left( \{x \in E_{\gamma} ;\, f(x) \leq y\} \right) \diff y$. The variable $y$ is such that $y \leq \gamma$. Then:
\begin{align*}
\{x \in E_{\gamma} ;\, f(x) \leq y\} &= \{x \in [0,\Delta] ;\, f(x)\leq y \land f(x) > \gamma \mbox{ with } y \leq \gamma \} \\
&= \{x \in [0,\Delta] ;\, f(x) \leq y \leq\gamma < f(x) \} = \emptyset
\end{align*}
Hence, $\dps \int _0^{\gamma} \mu_i \left( \{x \in E_{\gamma} ;\, f(x) \leq y\} \right) \diff y = 0.$ \\

Consider the second term: $\dps \int _{\gamma}^{\Delta} \mu_i \left( \{x \in E_{\gamma} ;\, f(x) \leq y\} \right) \diff y$. The variable $y$ is such that $y >\gamma$. Then:
\begin{align*}
\{x \in E_{\gamma} ;\, f(x) \leq y\} &= \{x \in [0,\Delta] ;\, f(x) \leq y \land f(x)> \gamma \} \\
&= \{x \in [0,\Delta] ;\, \gamma < f(x) \leq y \}.
\end{align*}
Hence, $\dps \int _{\gamma}^{\Delta} \mu_i \left( \{x \in E_{\gamma} ;\, f(x) \leq y\} \right) \diff y = \int_{\gamma}^{\Delta} F_i(y) - F_i(\gamma) \diff y$. \\
Geometrically, it corresponds to the area between the curves $F_i$ and the horizontal line of equation $F_i(\gamma)$ on $[\gamma,\Delta]$. \\

We now consider $\dps \int_{E_{\gamma}^{\comp}} f \diff \mu_i$.\\
\begin{equation*}
\dps \int_{E_{\gamma}^{\comp}} f \diff \mu_i = \Delta - \left( \int _0^{\gamma} \mu_i \left( \{x \in E_{\gamma}^{\comp} ;\, f(x) \leq y\} \right) \diff y + \int _{\gamma}^{\Delta} \mu_i \left( \{x \in E_{\gamma}^{\comp} ;\, f(x) \leq y\} \right) \diff y \right)
\end{equation*}

Consider the first term: $\dps \int _0^{\gamma} \mu_i \left( \{x \in E_{\gamma}^{\comp} ;\, f(x) \leq y\} \right) \diff y$. The variable $y$ is such that $y \leq \gamma$. Then:
\begin{align*}
\{x \in E_{\gamma}^{\comp} ;\, f(x) \leq y\} &= \{x \in [0,\Delta] ;\, f(x) \leq \gamma \land f(x) \leq y \mbox{ with } y \leq \gamma \} \\
&= \{x \in [0,\Delta] ;\, f(x) \leq y \}.
\end{align*}
Hence, $\dps \int _0^{\gamma} \mu_i \left( \{x \in E_{\gamma}^{\comp} ;\, f(x) \leq y\} \right) \diff y = \int_0^{\gamma} F_i(y) \diff y$.\\
Geometrically, it corresponds to the area under the curve $F_i$  on $[0,\gamma]$. \\

Consider the second term: $\dps \int _{\gamma}^{\Delta} \mu_i \left( \{x \in E_{\gamma}^{\comp} ;\, f(x) \leq y\} \right) \diff y$. The variable $y$ is such that $y \geq \gamma$. Then:
\begin{align*}
\{x \in E_{\gamma}^{\comp} ;\, f(x) \leq y\} &= \{x \in [0,\Delta] ;\, f(x) \leq \gamma \land f(x) \leq y \mbox{ with } y \geq \gamma \} \\
&= \{x \in [0,\Delta] ;\, f(x) \leq \gamma \}.
\end{align*}
Hence, $\dps \int _{\gamma}^{\Delta} \mu_i \left( \{x \in E_{\gamma}^{\comp} ;\, f(x) \leq y\} \right) \diff y = \int_{\gamma}^{\Delta} F_i(\gamma) \diff y$.\\
Geometrically, it corresponds to the rectangle between the x-axis and the horizontal line $F_i(\gamma)$ on $[\gamma, \Delta]$. \\

\noindent We can write the K-R distance:
\begin{align}\label{equation_geometric_version_of_KR_distance_Delta-f}
W_{\XX}(\mu_1,\mu_2) &= \Delta - \int_{\gamma}^{\Delta} F_1(y) - F_1(\gamma) \diff y + \Delta - \left(\int_0^{\gamma} F_1(y) \diff y + \int_{\gamma}^{\Delta} F_1(\gamma) \diff y \right) \nonumber \\
&- \Biggl( \Delta - \int_{\gamma}^{\Delta} F_2(y) - F_2(\gamma) \diff y + \Delta - \left(\int_0^{\gamma}F_2(y) \diff y + \int_{\gamma}^{\Delta} F_2(\gamma) \diff y \right) \Biggl) \nonumber \\
&= \int_{\gamma}^{\Delta} F_2(y) - F_2(\gamma) \diff y + \int_0^{\gamma} F_2(y) \diff y + \int_{\gamma}^{\Delta} F_2(\gamma) \diff y \nonumber \\
&- \Biggl( \int_{\gamma}^{\Delta} F_1(y) - F_1(\gamma) \diff y + \int_0^{\gamma} F_1(y) \diff y + \int_{\gamma}^{\Delta} F_1(\gamma) \diff y \Biggl).
\end{align}

\cleardoublepage


%
%
%
%
%


\bibTexBibliography{Biblio_Gael} 


%
%
%

\end{document}